\documentclass[12pt,leqno]{memo-l}
\usepackage{graphicx,psfrag,amssymb,amsmath,amsthm}
\usepackage{epstopdf}

\makeindex

\newtheorem{theorem}{Theorem}

\newtheorem{lemma}{Lemma}[chapter]
\newtheorem{corollary}[lemma]{Corollary}
\newtheorem{proposition}[lemma]{Proposition}
\newtheorem{conjecture}{Conjecture}[chapter]
\newtheorem*{claim}{Claim}

\numberwithin{equation}{chapter}

\providecommand{\norm}[1]{\lVert#1\rVert}

\def \proof{\medskip\noindent{\em Proof. }}
\def \qed{\hfill$\square$}

\def \Quad{{\quad\quad\quad\quad\quad\quad\quad\quad\quad\quad\quad
\quad\quad\quad\quad\quad\quad\quad\quad\quad\quad\quad\quad\quad
\quad\quad\quad\quad\quad\quad\quad\quad\quad\quad\quad\quad\quad
\quad\quad\quad}}

\def\dist{{\rm dist}}
\def\eps{{\varepsilon}}

\def\length{{\rm length}}
\def\mes{{\rm mes}}

\def\Area{{\rm Area}}

\def\Const{{\rm Const}}

\def\Div{{\rm div}}

\def\EXP{{\mathbb{E}}}

\def\BAN{\mathbb{B}}

\def\bbH{\mathbb{H}}

\def\naturals{\mathbb{N}}
\def\Tor{\mathbb{T}}
\def\reals{\mathbb{R}}
\def\integers{\mathbb{Z}}

\def\RmII{{I\!\!I}}
\def\RmIII{{I\!\!I\!\!I}}
\def\RmIV{{I\!V}}

\def\bA{\mathbf{A}}
\def\bB{\mathbf{B}}

\def\bL{\mathbf{L}}
\def\bM{\mathbf{M}}

\def\bP{\mathbf{P}}
\def\bQ{\mathbf{Q}}
\def\bR{\mathbf{R}}
\def\bS{\mathbf{S}}
\def\bT{\mathbf{T}}
\def\bV{\mathbf{V}}
\def\bb{\mathbf{b}}

\def\be{\mathbf{e}}
\def\bk{\mathbf{k}}

\def\bn{\mathbf{n}}
\def\br{\mathbf{r}}
\def\bs{\mathbf{s}}

\def\bv{\mathbf{v}}
\def\bw{\mathbf{w}}

\def\bz{\mathbf{z}}
\def\btheta{{\boldsymbol{\theta}}}
\def\bTheta{{\boldsymbol{\Theta}}}
\def\btau{{\boldsymbol{\tau}}}
\def\bchi{{\boldsymbol{\chi}}}
\def\bvarepsilon{{\boldsymbol{\varepsilon}}}

\def\brA{{\bar A}}
\def\brB{{\bar B}}
\def\brC{{\bar C}}
\def\brD{{\bar D}}

\def\brJ{{\bar J}}

\def\brL{{\bar L}}
\def\brM{{\bar M}}
\def\brQ{{\bar Q}}

\def\brV{{\bar V}}

\def\brc{{\bar c}}

\def\brw{{\bar w}}
\def\brx{{\bar x}}

\def\bralpha{{\bar \alpha}}

\def\brrho{{\bar\rho}}
\def\brtau{{\bar\tau}}

\def\brsigma{{\bar \sigma}}

\def\cA{\mathcal{A}}
\def\cB{\mathcal{B}}
\def\cC{\mathcal{C}}
\def\cD{\mathcal{D}}
\def\cI{\mathcal{I}}
\def\cJ{\mathcal{J}}
\def\cF{\mathcal{F}}
\def\cG{\mathcal{G}}
\def\cH{\mathcal{H}}
\def\cE{\mathcal{E}}
\def\cK{\mathcal{K}}
\def\cL{\mathcal{L}}
\def\cM{\mathcal{M}}
\def\cN{\mathcal{N}}
\def\cO{\mathcal{O}}
\def\cP{\mathcal{P}}
\def\cQ{\mathcal{Q}}
\def\cR{\mathcal{R}}
\def\cS{\mathcal{S}}
\def\cT{\mathcal{T}}
\def\cU{\mathcal{U}}
\def\cV{\mathcal{V}}
\def\cW{\mathcal{W}}
\def\cX{\mathcal{X}}

\def\dcD{{\partial\cD}}
\def\dcP{{\partial\cP}}

\def\fM{\mathfrak{M}}
\def\fR{\mathfrak{R}}
\def\fS{\mathfrak{S}}
\def\fU{\mathfrak{U}}

\def\hA{{\hat A}}
\def\hB{{\hat B}}
\def\hD{{\hat D}}
\def\hE{{\hat E}}

\def\hL{{\hat L}}
\def\hQ{{\hat Q}}
\def\hbQ{{\hat{\bQ}}}
\def\hS{{\hat S}}
\def\hcS{{\hat \cS}}

\def\hV{{\hat V}}
\def\hW{{\hat W}}
\def\hbV{{\hat{\bV}}}

\def\hd{{\hat d}}

\def\hht{{\hat t}}

\def\halpha{{\hat\alpha}}

\def\hgamma{{\hat\gamma}}

\def\hrho{{\hat\rho}}

\def\tA{{\tilde A}}
\def\tcA{{\tilde \cA}}
\def\ttcA{{\tilde{\tilde{\mathcal{A}}}}}

\def\tF{{\tilde F}}
\def\tcF{{\tilde \cF}}
\def\tcG{{\tilde \cG}}

\def\tR{{\tilde R}}
\def\tQ{{\tilde Q}}
\def\tS{{\tilde S}}

\def\tV{{\tilde V}}
\def\tW{{\tilde W}}
\def\tY{{\tilde Y}}

\def\tc{{\tilde c}}

\def\tr{{\tilde r}}
\def\tu{{\tilde u}}
\def\ttt{{\tilde t}}

\def\tbeta{{\tilde\beta}}
\def\tgamma{{\tilde\gamma}}
\def\ttgamma{{\tilde{\tilde{\gamma}}}}
\def\tdelta{{\tilde\delta}}

\def\trho{{\tilde\rho}}

\def\ttheta{{\tilde\theta}}

\def\tOmega{{\tilde{\Omega}}}
\def\tcB{{\tilde{\mathcal{B}}}}
\def\tcR{\tilde{\mathcal{R}}}
\def\tPi{{\tilde \Pi}}
\def\ttPi{{\tilde{\tilde{\Pi}}}}
\def\ttu{{\tilde{\tilde{u}}}}

\def\la{\langle}
\def\ra{\rangle}

\def\beq{\begin{equation}}
\def\eeq{\end{equation}}

\begin{document}

\title{Brownian Brownian Motion -- I}

\author{N. Chernov and D. Dolgopyat}
\date{}

\maketitle

\renewcommand{\thepage}{\roman{page}}

\setcounter{page}{4}

\setcounter{tocdepth}{2}

\tableofcontents

\thispagestyle{plain}

\centerline{\textbf{Abstract}}\bigskip

A classical model of Brownian motion consists of a heavy molecule
submerged into a gas of light atoms in a closed container. In this
work we study a 2D version of this model, where the molecule is a
heavy disk of mass $M \gg 1$ and the gas is represented by just one
point particle of mass $m=1$, which interacts with the disk and the
walls of the container via elastic collisions. Chaotic behavior of
the particles is ensured by convex (scattering) walls of the
container. We prove that the position and velocity of the disk, in
an appropriate time scale, converge, as $M\to\infty$, to a Brownian
motion (possibly, inhomogeneous); the scaling regime and the
structure of the limit process depend on the initial conditions. Our
proofs are based on strong hyperbolicity of the underlying dynamics,
fast decay of correlations in systems with elastic collisions
(billiards), and methods of averaging theory.
%\end{abstract}

\footnotetext[1]{Received by editor May 15, 2005; and in revised
form September 18, 2006.} 
 \footnotetext[4]{Nikolai Chernov: Department of
Mathematics, University of Alabama at Birmingham, Birmingham, AL
35294.} \footnotetext[5]{Dmitry Dolgopyat: Department of
Mathematics, University of Maryland, College Park 20742.}

\newpage

\renewcommand{\thepage}{\arabic{page}}

\setcounter{page}{1}

\chapter{Introduction} \label{SecI}
\setcounter{section}{1} \setcounter{subsection}{0}

\subsection{The model}
We study a dynamical system of two particles -- a hard disk of
radius $\br>0$ and mass $M\gg 1$ and a point particle of mass
$m=1$. Our particles move freely in a two-dimensional container
$\cD$ with concave boundaries and collide elastically with each other and with the walls
(boundary) of $\cD$. Our assumptions on the shape of the container
$\cD$ are stated in Section~\ref{subsecDass}.

Let $Q(t)$ denote the center and $V(t)$ the velocity of the heavy
disk at time $t$. Similarly, let $q(t)$ denote the position of the
light particle and $v(t)$ its velocity. When a particle collides
with a scatterer, \index{Scatterer} the normal component of its
velocity reverses. When the two particles collide with each other,
the normal components of their velocities change by the rules
\begin{equation}
  \label{CR1} v_{\rm new}^\perp=-\frac{M-1}{M+1}\, v_{\rm
  old}^\perp +\frac{2M}{M+1}\, V_{\rm old}^\perp
\end{equation}
and
\begin{equation}
  \label{CR2}  V_{\rm new}^\perp=\frac{M-1}{M+1}\, V_{\rm
  old}^\perp +\frac{2}{M+1}\, v_{\rm old}^\perp ,
\end{equation}
while the tangential components remain unchanged. The total kinetic
energy is conserved, and we fix it so that
\begin{equation}
   \label{Energy}  \|v\|^2 + M\|V\|^2 = 1.
\end{equation}
This implies $\|v\|\leq 1$ and $\|V\|\leq 1/\sqrt{M}$.

This is a Hamiltonian system, and it preserves Liouville measure
on its phase space. Systems of hard disks in closed containers are
proven to be completely hyperbolic and ergodic under various
conditions \cite{BLPS,Sm1,Sm2,Sm3}. These results do not cover our
particular model, but we have little doubt that it is hyperbolic
and ergodic, too. In this paper, though, we do not study ergodic
properties.

We are interested in the evolution of the system during the initial
period of time before the heavy disk experiences its first collision
with the border $\dcD$. This condition restricts our analysis to an
interval of time $(0,cM^a)$, where $c,a>0$ depend on $Q(0)$ and
$V(0)$, see Chapter~\ref{SecSR}. During this initial period, the
system does not exhibit its ergodic behavior, but it does exhibit a
diffusive behavior in the following sense. As
(\ref{CR1})--(\ref{CR2}) imply,
\beq \label{vvVV}
   \bigl| \|v_{\rm new}\| -\|v_{\rm old}\|
   \bigr| \leq 2 / \sqrt{M}
   \quad\text{and}\quad
   \|V_{\rm new}-V_{\rm old}\| \leq 2/M,
\eeq
hence the changes in $\|v\|$ and $V$ at each collision are much
smaller than their typical values, which are $\|v\| =\cO(1)$ and
$\|V\| =\cO (1/\sqrt{M})$. Thus, the speed of the light particle,
$\|v(t)\|$, remains almost constant, and the heavy particle not only
moves slowly but its velocity $V(t)$ changes slowly as well (it has
{\em inertia}). We will show that, in the limit $M\to\infty$, the
velocity $V(t)$ can be approximated by a Brownian motion, and the
position $Q(t)$ by an integral of the Brownian motion.

Our system is one of the simplest models of a particle moving in a
fluid. This is what scientists called Brownian motion about one
hundred years ago. Now this term has a more narrow technical
meaning, namely a Gaussian process with zero mean and stationary
independent increments. Our paper is motivated by the Brownian
motion in its original sense, and this is why we call it
``Brownian Brownian motion''.
% As we shall see, our simple model captures some features
% of this physical process, but there are certain specifics.

Even though this paper only covers a two particle system (where the
``fluid'' is represented by a single light particle), we believe
that our methods can extend to more realistic fluids of many
particles. We consider this paper as a first step in our studies
(thus the numeral one in its title) and plan to investigate more
complex models in the future, see our discussion of open problems in
Chapter~\ref{SecOP}.
% In the rest of this paper, the term Brownian
% motion has its modern meaning -- the standard Gaussian model.

\subsection{The container}\label{subsecDass}
In this paper we assume that $\cD$ is a dispersing \index{Dispersing
billiards} billiard table with finite horizon \index{Finite horizon}
and smooth boundary:
\medskip

Assumption A1: $\cD$ is a dispersing \index{Dispersing billiards}
billiard table, i.e.\ its boundary $\dcD$ is concave; this
guarantees chaotic motion of the particle colliding with $\dcD$.
\medskip

Assumption A2: $\cD$ has finite \index{Finite horizon} horizon,
which means the point particle cannot travel longer than a certain
finite distance $L_{\max} < \infty$ without collisions (even if we
remove the hard disk from $\cD$); this prevents superdiffusive
(ballistic) motion of the particle \cite{Bl}.
\medskip

Assumption A3: $\cD$ has $C^3$ smooth boundary (without corner
points).
\medskip

Containers satisfying all these assumptions can be constructed as
follows. Let $\Tor^2$ be the unit torus and $\BAN_1, \dots, \BAN_r
\subset \Tor^2$ some disjoint convex regions with $C^3$ smooth
boundaries, whose curvature never vanishes. Then
$$
          \cD = \Tor^2 \setminus \cup_{i=1}^r \BAN_i.
$$
The obstacles $\BAN_1, \dots, \BAN_r$ act as scatterers,
\index{Scatterer} our light particle bounces between them (like in a
pinball machine). They must also block all collision-free flights of
the particle to ensure the finite horizon \index{Finite horizon}
assumption.

The motion of a single particle in such domains $\cD$ has been
studied by Ya.~Sinai \cite{S2}, and this model is now known as
dispersing \index{Dispersing billiards} billiard system. It is
always hyperbolic and ergodic \cite{S2}, and has strong statistical
properties \cite{BSC2,Y}.
%, in particular, exponential mixing rates and a central limit theorem.

Our assumptions on $\cD$ are fairly restrictive. It would be
tempting to cover simpler containers -- just a rectangular box, for
example. We believe that most of our results would carry over to
rectangular boxes (perhaps, with certain adjustments). However, a
two-particle system in a rectangular container, despite its apparent
simplicity, would be much more difficult to analyze, because the
corresponding billiard system is not chaotic. For this reason
rectangular containers are currently out of reach. On the other hand
if the boundary of $\cD$ is convex then with positive probability
the particles will never meet (see \cite{La}), so some assumptions
on the shape of $\cD$ are necessary.

\subsection{Billiard approximations}
\label{subsecBA} We denote phase points by $x=(Q,V,q,v)$ and the
phase space by $\cM$. Due to the energy conservation
(\ref{Energy}), dim$\,\cM=7$. The dynamics $\Phi^t\colon\cM\to\cM$
can be reduced, in a standard way, to a discrete time system -- a
collision map -- as follows.

We call $\Omega = \partial\cM$ the {\em collision space}. Let
$\cP(Q)$ denote the disk of radius $\br$ centered on $Q$, then
$\Omega=\{ (Q,V,q,v)\in\cM\colon\ q \in \dcD \cup \, \dcP (Q) \}$.
At each collision, we identify the precollisional and
postcollisional velocity vectors. Technically, we will only
include the {\em postcollisional} vector in $\Omega$, so that
\begin{align*}
  \Omega &= \Omega_\cD \cup \Omega_\cP, \\
 \Omega_\cD &= \bigl\{(Q,V,q,v)\in\cM\colon\ q\in\dcD,\ \
   \la v, n \ra \geq 0 \bigr\}, \\
   \Omega_\cP &= \bigl\{(Q,V,q,v)\in\cM\colon\ q\in\dcP(Q),\ \
   \la v-V, n \ra \geq 0 \bigr\}.
\end{align*}
where $\la\cdot,\cdot\ra$ stands for the scalar product of vectors
and $n$ denotes a normal vector to $\dcD\cup\dcP(Q)$ at $q$
pointing into $\cD\setminus\cP(Q)$. The first return map
$\cF\colon\Omega\to\Omega$ is called the {\em collision map}. It
preserves a smooth probability measure $\mu$ on $\Omega$ induced
by the Liouville measure on $\cM$.

It will be convenient to denote points of $\Omega$ by $(Q,V,q,w)$,
where
\beq
        w = \left \{ \begin{array}{ccl}
        v & {\rm for} & q\in\dcD\\
        v-V & {\rm for} & q\in\dcP(Q)
        \end{array} \right .
          \label{w}
\eeq
so that $\Omega$ can be represented by a unified formula
\beq
   \Omega = \bigl\{(Q,V,q,w)\colon\ q\in\dcD\cup\dcP(Q),\ \
   \la w, n \ra \geq 0 \bigr\}
      \label{Omegaw}
\eeq
For every point $(Q,V,q,w)\in\Omega$ we put
\beq
    \brw=\frac{w}{\|w\|}\,\bs_V, \qquad
    \bs_V = \sqrt{1-M\|V\|^2}
       \label{brw}
\eeq
(note that $\|\brw\| = \bs_V = \|v\|$ due to (\ref{Energy}), hence
$\bs_V$ only depends on $\|V\|$). For each pair $(Q,V)$ we denote
by $\Omega_{Q,V}$ the cross-section of $\Omega$ obtained by fixing
$Q$ and $V$. By using (\ref{brw}) we can write
\beq
   \Omega_{Q,V} = \bigl\{(q,\brw)\colon\ q\in\dcD\cup\dcP(Q),\ \
   \la \brw, n \ra \geq 0,\ \
   \|\brw\|= \bs_V \bigr\}
      \label{OmegaQVbrw}
\eeq

Now let us pick $t_0\geq 0$ and fix the center of the heavy disk at
$Q=Q(t_0)\in\cD$ and set $M=\infty$. Then the light particle would
move with a constant speed $\|v(t)\|=\bs_V$, where $V=V(t_0)$, in
the domain $\cD\setminus\cP(Q)$ with specular reflections at
$\dcD\cup\dcP(Q)$. Thus we get a billiard-type dynamics, which
approximates our system during a relatively short interval of time,
until our heavy disk moves a considerable distance. We may treat our
system then as a small perturbation of this billiard-type dynamics,
and in fact our entire analysis is based on this approximation.

The collision map $\cF_{Q,V}$ of the above billiard system acts on
the space (\ref{OmegaQVbrw}), where $\brw$ denotes the
postcollisional velocity of the moving particle. The map
$\cF_{Q,V}\colon\Omega_{Q,V}\to\Omega_{Q,V}$ preserves a smooth
probability measure $\mu_{Q,V}$ as described in
Chapter~\ref{SecSPE}.

The billiard-type system $(\Omega_{Q,V},\cF_{Q,V},\mu_{Q,V})$ is
essentially independent of $V$. By a simple rescaling (i.e.\
renormalizing) of $\brw$ we can identify it with
$(\Omega_{Q,0},\cF_{Q,0},\mu_{Q,0})$, which we denote, for brevity,
by $(\Omega_{Q},\cF_{Q},\mu_{Q})$, and it becomes a standard
billiard system, where the particle moves at unit speed, on the
table $\cD\setminus\dcP (Q)$. This is a dispersing \index{Dispersing
billiards} (Sinai) table, hence the map $\cF_Q$ is hyperbolic,
ergodic and has strong statistical properties, including exponential
decay of correlations and the central limit theorem \cite{S2,Y}.

Equations (\ref{CR1})--(\ref{CR2}) imply that the change of the
velocity of the disk due to a collision with the light particle is
\begin{equation}
  \label{CR2a}
  V_{\rm new}- V_{\rm old}
  = -\frac{2\bigl( v_{\rm new}^\perp
  - V_{\rm new}^\perp\bigr)}{M+1}
  = -\frac{2\, w^\perp}{M+1}
\end{equation}
Since $w=\brw\,\|v-V\|/\|v\|$, we have
\beq
  \label{CR2b}
  V_{\rm new}- V_{\rm old}
  = -\frac{2\, \brw^\perp}{M} + \delta
\eeq
where
\beq
     |\delta|\leq\,\Const\,\biggl(\frac{\norm{V}}{M\norm{v}}
        +\frac{1}{M^2}\biggr)
          \label{chi}
\eeq
is a relatively small term. Define a vector function on $\Omega$ by
\beq
\label{MEx}
   \cA = \left \{ \begin{array}{ccl}
   -2\brw^\perp & {\rm for} & q\in\dcP(Q) \\
   0 & {\rm for} & q\in\dcD\setminus\dcP(Q)
   \end{array} \right .
\eeq
Obviously, $\cA$ is a smooth function, and due to a rotational
symmetry
$$
    \int \cA\, d\mu_{Q,V}=0
$$
for every $Q,V$. Hence the central limit theorem for dispersing
\index{Dispersing billiards} billiards \cite{BS,Y} implies the
convergence in distribution
\begin{equation}
\label{EqVar}
    \frac{1}{\sqrt n}\sum_{j=0}^{n-1} \cA\circ\cF_{Q,V}^j\to
     \cN(0, \brsigma^2_{Q,V}(\cA)).
\end{equation}
as $n\to\infty$, where $\brsigma^2_{Q,V}(\cA)$ a symmetric
positive semidefinite matrix given by the Green-Kubo formula
\index{Green-Kubo formula}
\begin{equation}  \label{EqSigmabar}
  \brsigma^2_{Q,V}(\cA) = \sum_{j=-\infty}^{\infty}
  \int_{\Omega_{Q,V}}
  \cA\, \left (\cA\circ\cF_{Q,V}^j\right )^T\,d\mu_{Q,V}.
\end{equation}
(this series converges because its terms decay exponentially fast
as $|j|\to\infty$, see \cite{Y,C2}). By setting $V=0$ we define a
matrix $\brsigma^2_{Q}(\cA)\colon=\brsigma^2_{Q,0}(\cA)$. Since
the restriction of $\cA$ to the sets $\Omega_{Q,V}$ and
$\Omega_{Q,0}=\Omega_Q$ only differ by a scaling factor $\bs_V$,
see (\ref{brw}), we have a simple relation
\beq
  \label{EqSigmabar2}
  \brsigma^2_{Q,V}(\cA) = (1-M\norm{V}^2)\,\, \brsigma^2_Q(\cA).
\eeq
Define another matrix
\beq
\label{EqSigma}
     \sigma^2_Q(\cA)=\brsigma^2_Q(\cA)/\brL,
\eeq
where
\begin{equation}
  \label{FPL}
   \brL=\pi\, \frac{\Area(\cD)-\Area(\cP)}
   {\length(\dcD)+\length(\dcP)}
\end{equation}
is the mean free path of the light particle in the billiard
dynamics $\cF_Q$, see \cite{C1} (observe that $\brL$ does not
depend on $Q$). Lastly, let $\sigma_Q(\cA)$ be the symmetric
positive semidefinite square root of $\sigma^2_Q(\cA).$

Let us draw some conclusions, which will be entirely heuristic at
this point (they will be formalized later). In view of
(\ref{CR2a})--(\ref{EqVar}), we may expect that the total change
of the disk velocity $V$ in the course of $n$ consecutive
collisions of the light particle with $\dcD\cup\dcP$ can be
approximated by a normal random variable $\cN\Big (0,
n\brsigma^2_{Q,V}(\cA)/M^2\Big )$. During an interval $(t_0,t_1)$,
the light particle experiences $n\approx \brL^{-1} \norm{v}
(t_1-t_0)$ collisions, hence for the total change of the disk
velocity we expect another normal approximation
\begin{equation}
   \label{Vt1t2}
   V(t_1)-V(t_0) \sim \cN\Big (
     0,\norm{v} (t_1-t_0)\,\sigma^2_{Q,V}(\cA)/M^2\Big )
\end{equation}
(due to the inertia of the heavy disk, we expect $Q(t)\approx
Q(t_0)$, $\|V(t)\| \approx \|V(t_0)\|$, and thus $\norm{v(t)}
\approx\, \bs_V(t_0)$ for all $t_0<t<t_1$). A large part of our
paper is devoted to making the heuristic approximation
(\ref{Vt1t2}) precise. \newpage

\chapter{Statement of results} \label{SecSR}
\setcounter{section}{2}\setcounter{subsection}{0}

Suppose the initial position $Q(0)=Q_0$ and velocity $V(0)=V_0$ of
the heavy particle are fixed, and the initial state of the light
particle $q(0),v(0)$ is selected randomly, according to a smooth
distribution in the direct product of the domain
$\cD\setminus\cP(Q_0)$ and the circle $\norm{v(0)}^2 =
1-M\norm{V_0}^2$ (alternatively, $q(0)$ may be chosen from
$\dcD\cup\dcP(Q_0)$ and $v(0)$ from the semicircle containing all
the postcollisional velocity vectors). The shape of the initial
distribution will not affect our results.

We consider the trajectory of the heavy particle $Q(t),V(t)$ during
a time interval $(0,cM^a)$ with some $c,a>0$ selected below. We
scale time by $\tau = t/M^a$ and, sometimes, scale space in a way
specified below, to convert $\{Q(t), V(t)\}$ to a pair of functions
$\{\cQ(\tau), \cV(\tau)\}$ on the interval $0<\tau<c$. The random
choice of $q(0),v(0)$ induces a probability measure on the space of
functions $\cQ(\tau), \cV(\tau)$, and we are interested in the
convergence of this probability measure, as $M\to\infty$, to a
stochastic process $\{\bQ(\tau), \bV(\tau)\}$. We prove three major
results in this direction corresponding to three different regimes
in the dynamics of the massive disk.

\subsection{Heavy disk in `equilibrium' (linear motion)} \label{subsecSR1}
First, let the initial velocity of the heavy particle be of order
$1/\sqrt{M}$. Specifically, let us fix a unit vector $u_0 \in S^1$
and $\chi\in (0,1)$, set
\beq \label{iniTm1}
        V_0= M^{-1/2} \chi\, u_0
\eeq
and fix $Q_0\in\cD$ arbitrarily (but so that dist$(Q_0, \dcD) >
\br$). Note that if the heavy disk moved with a constant velocity,
$V_0$, without colliding with the light particle, it would hit
$\dcD$ at a certain moment $c_0 M^{1/2}$, where $c_0>0$ is
determined by $Q_0$, $u_0$ and $\chi$. We restrict our analysis to
a time interval $(0,cM^{1/2})$ with some $c<c_0$. During this
period of time we expect, due to (\ref{Vt1t2}), that the overall
fluctuations of the disk velocity will be $\cO(M^{-3/4})$. Hence
we expect $V(t)= V_0 +\cO (M^{-3/4})$ and
$Q(t)= Q_0 +tV_0 +\cO(t M^{-3/4})=Q_0 +tV_0 +\cO(M^{-1/4})$
for $0<t<cM^{1/2}$. This leads us to a time scale
\beq \label{timescaleTm1}
     \tau =tM^{-1/2}
\eeq
and a space scale
\beq \label{spacescaleTm1}
    \cQ(\tau)=
    M^{1/4} \left[Q(\tau M^{1/2})-Q_0-\tau M^{1/2}V_0 \right]
\eeq
and, respectively,
\beq \label{velscaleTm1}
   \cV(\tau)=
    M^{3/4} \left[V(\tau M^{1/2})-V_0\right]
\eeq
We can find an asymptotic distribution of $\cV(\tau)$ by using our
heuristic normal approximation (\ref{Vt1t2}). Let
$$
    Q^\dag (\tau) = Q_0+\tau M^{1/2}V_0 = Q_0 +\tau \chi\, u_0
$$
Then for any $\tau\in(0,c)$ we have $Q(\tau M^{1/2})\to
Q^\dag(\tau)$, as $M\to\infty$, hence
\beq
    \label{sigmaappr}
   \sigma^2_{Q(\tau M^{1/2})}(\cA) \to
   \sigma^2_{Q^\dag(\tau)}(\cA).
\eeq
Anticipating that $\norm{v(t)}\approx \sqrt{1-\chi^2}$ for all
$0<t<cM^{1/2}$ we can expect that for small $d\tau>0$ the number
of collisions $N(d\tau)$ is approximately equal to $d\tau
\brL^{-1} \sqrt{1-\chi^2}$ and the momenta exchange during each
collision is close to $\sqrt{1-\chi^2}\cA$, see (\ref{MEx}). This
should give us
$$
  \cV(\tau+d\tau)-\cV(\tau) \sim
  \cN\bigl(0, D(d\tau)\bigr)
$$
where
\begin{align}
 \label{Vtaud}
   D(d\tau) &\approx N(d\tau)\, (1-\chi^2)\,
   \brsigma^2_{Q^\dag(\tau)}(\cA)\nonumber\\
   &= d\tau\,\sqrt{(1-\chi^2)^3}\,
   \sigma^2_{Q^\dag(\tau)}(\cA)
\end{align}
Integrating (\ref{Vtaud}) over $(0,\tau)$ yields
\beq
     \label{Vtauint}
   \cV(\tau) \sim
  \cN\Big (0,\sqrt{(1-\chi^2)^3} \,\int_0^\tau
  \sigma^2_{Q^\dag(s)}(\cA)\, ds\Big )
\eeq
The following theorem (proved in this paper) makes this conclusion
precise:

\begin{theorem}
\label{Root}
Under the conditions A1--A3 and (\ref{iniTm1})
the random process $\cV(\tau)$ defined on the interval
$0\leq \tau \leq c$ by (\ref{timescaleTm1}), (\ref{velscaleTm1})
weakly converges, as
$M\to\infty$, to a Gaussian Markov random process $\bV(\tau)$ with
independent increments, zero mean and covariance matrix
\beq
      \label{CovV}
    {\rm Cov}(\bV(\tau)) = \sqrt{(1-\chi^2)^3}
    \int_0^\tau \sigma^2_{Q^\dag(s)}(\cA)\, ds
\eeq
The process $\bV(\tau)$ can be, equivalently, defined by
\begin{equation}
    \label{DblIntBM}
   \bV(\tau)=\sqrt[4]{(1-\chi^2)^3}\,\int_0^\tau
   \sigma_{Q^\dag(s)}(\cA)\, d\bw(s)
\end{equation}
where $\bw(s)$ denotes the standard two dimensional Brownian
motion. Accordingly, the function $\cQ(\tau)$ defined by
(\ref{spacescaleTm1}) converges weakly to a Gaussian random
process $\bQ(\tau)=\int_0^\tau \bV(s)\, ds$, which has zero mean
and covariance matrix
\beq
      \label{CovQ}
    {\rm Cov}(\bQ(\tau)) = \sqrt{(1-\chi^2)^3}
    \int_0^\tau (\tau-s)^2\,\sigma^2_{Q^\dag(s)}(\cA)\, ds
\eeq
\end{theorem}

We note that the limit velocity process is a (time inhomogeneous)
Brownian motion, while the limit position process is the integral of
that Brownian motion.

\subsection{Heavy disk at rest (slow acceleration)} \label{subsecSR2}
The next theorem deals with the more difficult case of
\beq \label{iniTm2}
         V_0=0
\eeq
(here again, $Q_0$ is chosen arbitrarily). Since the heavy disk is now
initially at rest, it takes it longer to build up speed and travel to
the border $\dcD$. We expect, due to (\ref{Vt1t2}), that the disk
velocity grows as $\norm{V(t)} =\cO(\sqrt{t}/M)$, and therefore its
displacement grows as $\norm{Q(t)-Q_0} =\cO(t^{3/2}/M)$. Hence, for
typical trajectories, it takes $\cO(M^{2/3})$ units of time for the
disk to reach $\dcD$. It is convenient to modify the dynamics of the
disk making it stop (``freeze'') when it comes too close to $\dcD$. We
pick a small $\delta_0>0$ and stop the disk at the moment
\beq
   t_\ast=\min\{t>0\colon\ \dist(Q(t),\dcD)=\br+\delta_0\},
     \label{tast}
\eeq
hence we obtain a modified dynamics $Q_\ast(t)$, $V_\ast(t)$ given
by
$$
  Q_\ast(t)=\left \{
  \begin{array}{lc} Q(t) & {\rm for}\ \ t<t_\ast\\
  Q(t_\ast) & {\rm for}\ \ t>t_\ast \end{array}\right .\qquad
  V_\ast(t)=\left \{
  \begin{array}{cc} V(t) & {\rm for}\ \ t<t_\ast\\
  0 & {\rm for}\ \ t>t_\ast \end{array}\right .
$$
With this modification, we can consider the dynamics on a time interval
$(0,cM^{2/3})$ with an arbitrary $c>0$. Our time scaling is
\beq \label{timescaleTm2}
         \tau=tM^{-2/3}
\eeq
and there is no need for any space scaling, i.e.\ we set
\beq \label{spacescaleTm2}
      \cQ(\tau)=Q_\ast(\tau M^{2/3})
       \qquad \text{and} \qquad
      \cV(\tau) =M^{2/3}V_\ast(\tau M^{2/3}).
\eeq

We note that the heavy disk, being initially at rest, can move
randomly in any direction and follow a random trajectory before
reaching $\dcD$. Thus, the matrix $\sigma^2_{Q(\tau
M^{2/3})}(\cA)$ does not have a limit (in any sense), as
$M\to\infty$, because it depends on the (random) location of the
disk. Hence, heuristic estimates of the sort
(\ref{sigmaappr})--(\ref{Vtauint}) are now impossible, and the
limit distribution of the functions $\{\cQ(\tau), \cV(\tau)\}$
cannot be found explicitly. Instead, we will show that any weak
limit of these functions, call it $\{\bQ(\tau), \bV(\tau)\}$,
satisfies two stochastic differential equations (SDE)
\beq
   \label{TIBMa}
     d\bQ = \bV\, d\tau,\qquad
     d\bV = \sigma_{\bQ}(\cA)\, d\bw(\tau)
\eeq
with initial conditions $\bQ_0 = Q_0$ and $\bV_0 = 0$. Thus, the limit
behavior of the functions $\cQ(\tau)$ and $\cV(\tau)$ will only be
described implicitly, via (\ref{TIBMa}).

In order to guarantee the convergence, though, we need to make sure
that the initial value problem \eqref{TIBMa} has a unique solution
$\{\bQ(\tau), \bV(\tau)\}$. SDE of this type have unique solutions
if the matrix $\sigma_Q(\cA)$, as a function of $Q$, is
differentiable \cite[Section IX.2]{RY} but they  may have multiple
solutions if $\sigma_Q$ is only continuous. We do not expect our
matrix $\sigma_Q(\cA)$ to be differentiable, though. Recent
numerical experiments \cite{BDL} suggest that dynamical invariants,
such as the diffusion \index{Diffusion matrix} matrix, are not
differentiable if the system has singularities.

To tackle this problem, we prove (see Section~\ref{subsecUSDE}) that
the SDE (\ref{TIBMa}) has a unique solution provided $\sigma_Q(\cA)$
is log-Lipschitz \index{Log-Lipschitz continuity} continuous in the
following sense:
\beq
        \norm{\sigma_{Q_1}(\cA)-\sigma_{Q_2}(\cA)}\leq
        \Const\,\norm{Q_1-Q_2} \,
          \Big |\ln\norm{Q_1-Q_2}\,\Big |
             \label{sigQ1Q2}
\eeq
(this condition is weaker than Lipschitz continuity but stronger than
H\"older continuity with any exponent $<1$).
%(this condition is nearly optimal -- if the logarithmic factor here was
%raised to a power $>1$, the uniqueness could no longer be guaranteed).

Thus we need to establish (\ref{sigQ1Q2}), which constitutes a novel
and rather difficult result in billiard theory. Its proof occupies a
sizable part of our paper (Chapter~\ref{SecRDM}) and requires two
additional assumptions on the scatterers \index{Scatterer} $\BAN_i$.
First, the collision map $\cF_Q\colon \Omega_Q\to\Omega_Q$ must be
$C^3$ smooth (rather than $C^2$, which is commonly assumed in the
studies of billiards), hence the boundaries $\partial\BAN_i$ must be
at least $C^4$. This additional smoothness of $\cF_Q$ allows us to
prove that
\beq
        \norm{\sigma_{Q_1}^2(\cA)-\sigma_{Q_2}^2(\cA)}\leq
        \Const\,\norm{Q_1-Q_2} \,
          \Big |\ln\norm{Q_1-Q_2}\,\Big |
             \label{sigQ1Q2add}
\eeq
which is slightly weaker than (\ref{sigQ1Q2}). To convert
(\ref{sigQ1Q2add}) to (\ref{sigQ1Q2}) we need the matrix
$\sigma^2_Q(\cA)$ be nonsingular for every $Q$, so that the function
$\sigma^2 \mapsto \sigma$ is smooth. To this end we only find a
criterion (see Section~\ref{subsecA2a}), in terms of periodic orbits
of $\cF_Q$, for the nonsingularity of $\sigma^2_Q(\cA)$. We believe
it is satisfied for typical configurations of scatterers
\index{Scatterer} $\BAN_i$, but we do not prove it here.

Thus we formulate our additional assumptions:\medskip

\noindent Assumption A3': The boundaries $\partial\BAN_i$ of all
scatterers are $C^4$ smooth; \medskip

\noindent Assumption A4: $\sigma^2_Q(\cA)>0$ for all $Q\in\cD$
such that $\dist(\cP(Q),\dcD)\geq\delta_0$. \medskip

Next we state the convergence theorem:

\begin{theorem}
\label{ThVelBM} Under the conditions A1, A2, A3', A4
the random processes $\{\cQ(\tau), \cV(\tau)\}$ defined on the interval
$0\leq \tau\leq c$ by
(\ref{iniTm2})--(\ref{spacescaleTm2})
weakly converge to a stochastic process $\{\bQ(\tau),
\bV(\tau)\}$, which constitutes a unique solution of the following
stochastic differential equations with initial conditions:
\beq
   \label{TIBM}
     \begin{array}{ll}
     d\bQ = \bV\, d\tau, & \bQ_0 = Q_0 \\
     d\bV = \sigma_{\bQ}(\cA)\, d\bw(\tau),\ \ \ \  & \bV_0 = 0
     \end{array}
\eeq
which are stopped the moment $\bQ(\tau)$ comes to within
the distance $\br+\delta_0$ from the border $\dcD$
(here again $\bw(\tau)$ is the standard two dimensional Brownian
motion).
\end{theorem}

\subsection{Heavy disk of small size} \label{subsecSR3}
We now turn to our last major result. The problems that plagued us
in the previous theorem can be bypassed by taking the limit $\br\to
0$, in addition to $M\to\infty.$ (That is we assume that the heavy
particle is microscopically large but mactroscopically small.) In
this case the matrix $\sigma^2_Q(\cA)$ will be asymptotically
constant, as we explain next. Recall that
$\sigma^2_Q(\cA)=\brsigma^2_Q(\cA)/\brL$, where $\brL$ does not
depend on $Q$. Next, $\brsigma^2_Q(\cA)$ is given by the Green-Kubo
formula (\ref{EqSigmabar}). Its central term, corresponding to
$j=0$, can be found by a direct calculation:
\begin{equation}
    \label{j=0} \int_{\Omega_Q} \cA\,\cA^T\, d\mu_Q =
    \frac{8\pi\br}{3\,(\length(\dcD)+\length(\dcP))}\,I,
\end{equation}
see Section~\ref{subsecA2b}, and it is independent of $Q$. Hence,
the dependence of $\sigma^2_Q(\cA)$ on $Q$ only comes from the
correlation terms $j\neq 0$ in the series (\ref{EqSigmabar}). Now,
when the size of the massive disk is comparable to the size of the
domain $\cD$, the average time between its successive collisions
with the light particle is of order one, and so those collisions
are strongly correlated. By contrast, if $\br\approx 0$, the
average time between successive interparticle collisions is
$\cO(1/\br)$, and these collisions become almost independent.
Thus, in the Green-Kubo formula (\ref{EqSigmabar}), the central
term (\ref{j=0}) becomes dominant, and we arrive at
\begin{equation}
   \label{rsmall}
   \sigma^2_Q(\cA)=\frac{8\br}{3\, \Area(\cD)}\, I+o(\br).
\end{equation}
see Section~\ref{subsecA2b} for a complete proof.

Next, the time scale introduced in the previous theorem has to be
adjusted to the present case where $\br\to 0$. Due to
(\ref{Vt1t2}) and (\ref{rsmall}), we expect that the disk velocity
grows as $\norm{V(t)} =\cO(\sqrt{\br t}/M)$, and its displacement
as $\norm{Q(t)-Q_0} =\cO(t^{3/2}\br^{1/2}/M)$. Hence, for typical
trajectories, it takes $\cO(\br^{-1/3}M^{2/3})$ units of time
before the heavy disk hits $\dcD$. It is important that during
this period of time $\norm{V(t)}=\cO(r^{1/3}M^{-2/3})\approx 0$,
hence $\norm{v}$ remains close to one. We again modify the
dynamics of the disk making it stop (freeze up) the moment it
becomes $\delta_0$-close to $\dcD$ and consider the so modified
dynamics of the heavy disk $Q_\ast(t)$, $V_\ast(t)$ on time
interval $(0,c\,\br^{-1/3}M^{2/3})$ with a constant $c>0$. Our
time scale is now
\beq \label{timescaleTm3}
   \tau=t\,\br^{1/3}M^{-2/3}
\eeq
and we set
\beq \label{spacescaleTm3}
    \cQ(\tau) =Q_\ast(\tau\, \br^{-1/3}M^{2/3})
\eeq
and hence
\beq \label{velscaleTm3}
    \cV(\tau) =\br^{-1/3}M^{2/3}V_\ast(\tau\,\br^{-1/3} M^{2/3}).
\eeq

It is easy to find an asymptotic distribution of $\cV(\tau)$ by
using our heuristic normal approximation (\ref{Vt1t2}) in a way
similar to (\ref{sigmaappr})--(\ref{Vtauint}). Since the matrix
$\sigma^2_Q(\cA)$ is almost constant, due to (\ref{rsmall}), we
expect that $\cV(\tau) \to \cN(0,\sigma_0^2\tau I)$, where
\beq
        \sigma_0^2 = \frac{8}{3\,\Area(\cD)}.
\eeq
The following theorem shows that our heuristic estimate is
correct:

\begin{theorem}
\label{ThSmall} Under the conditions A1--A3
there is a function
$M_0=M_0(\br)$ such that if $\br\to 0$ and $M\to\infty$, so that
$M>M_0(\br)$, then the processes $\{\cQ(\tau), \cV(\tau)\}$
defined by (\ref{iniTm2}), (\ref{timescaleTm3})--(\ref{velscaleTm3})
converge weakly on the interval $0<\tau<c$:
\beq
      \label{Vlimsmall}  \cV(\tau) \to \sigma_0\, \bw_{\cD} (\tau)
\eeq
and
\beq
     \label{Qlimsmall}
    \cQ(\tau) \to Q_0 + \sigma_0  \int_0^\tau \bw_{\cD} (s)\, ds
\eeq
where $\bw_{\cD} (\tau)$ is a standard two dimensional Brownian
motion subjected to the following modification: we set $\bw_{\cD}
(\tau)=0$ for all $\tau>\tau_\ast$, where $\tau_\ast$ is the
earliest moment when the right hand side of (\ref{Qlimsmall})
becomes $\delta_0$-close to $\dcD$.
\end{theorem}

\subsection{Comparison to previous works}
\label{subsecCPW} There are two directions of research which our
results are related to. The first one is the averaging
 \index{Averaging} theory of differential equations and the second is
the study of long time behavior of mechanical systems.

The averaging theory deals with systems characterized by two types
of dynamic variables, \emph{fast} and \emph{slow}. The case where
the fast variables make a Markov process, which does not depend on
the slow variables, is quite well understood \cite{FW}. The
results obtained in the Markov case have been extended to the
situation where the fast motion is made by a hyperbolic dynamical
system in \cite{K1, Pn, D2}. By contrast, if the fast variables
are coupled to the slow ones, as they are in our model, much less
is known. Even the case where the fast motion is a diffusion
process was settled quite recently \cite{PV}. Some results are
available for coupled systems where the fast motion is uniformly
hyperbolic \cite{Ba2, K2}, but they all deal with relatively short
time intervals, like the one in our relatively simple Theorem
\ref{Root}. The behavior during longer time periods, like those in
our Theorems~\ref{ThVelBM} and \ref{ThSmall}, remains virtually
unexplored. This type of behavior is hard to control, it appears
to be quite sensitive to the details of the problem at hand; for
example, the uniqueness of the limiting process in our
Theorem~\ref{ThVelBM} relies upon the smoothness of the auxiliary
function $\sigma^2_Q(\cA)$, cf.\ also Theorem~3 in \cite{PV}.

Let us now turn to the second research field. While the phenomenological theory
of the Brownian motion is more  than a hundred years old (see \cite{Ei, N} for historic
background) the mathematical understanding of how this theory can be derived from
the microscopic Hamiltonian laws is still limited. Let us describe some available results
refering the reader to the surveys \cite{CDol, Sol, ST2} for more information.

Probably, the simplest
mechanical model where one can observe a non-trivial statistical
behavior is a periodic Lorentz gas. Bunimovich and Sinai \cite{BS}
(see also an improved version in \cite{BSC2}) were the first to
obtain a Brownian motion approximation for the {\it position} of a
particle traveling in the periodic Lorentz gas with finite
\index{Finite horizon} horizon. Their results hold for arbitrarily
long intervals of time with respect to the equilibrium measure; such
approximations are fairly common for chaotic dynamical systems
\cite{DP}. On the contrary, we construct a Brownian motion
approximation for relatively short periods of time, and in the
context of Theorems~\ref{ThVelBM}--\ref{ThSmall} our system is far
from equilibrium.

On the other hand, models of Brownian motion where one massive
(tagged) particle is surrounded by an ideal gas of light particles
have been studied in many papers, see, e.g., \cite{CDK, DGL1, DGL2,
Ha, Ho, SSol, Spi, ST1}. We do not discuss these papers here
referring the reader to the surveys \cite{Sol, ST2}. We observe that
even though the methods of these papers do not play an important
role in our proofs they should be useful for possible multiparticle
extensions (see Section~\ref{subsecGNP}). The models in the above
cited papers are more realistic in explaining the actual Brownian
motion, however, there are still some unresolved questions. For
example, it is commonly assumed that the gas is (and remains) in
equilibrium, but there is no satisfactory mathematical explanation
of why and how it reaches and maintains that state of equilibrium.
(Moreover the results for out-of-equilibrium systems can differ from
physical predictions made under the equilibrium assumption. See e.g.
\cite{Sz, MS}.) We do not make (and do not need) such assumptions.

In fact, the equilibrium assumption is hard to substantiate. In
ideal gases, where no direct interaction between gas particles takes
place, equilibrium can only establish and propagate due to indirect
interaction via collisions with the heavy particle and the walls.
This process requires many collisions of each gas particle with the
heavy one, but the existing techniques are incapable of tracing the
dynamics beyond the time when each gas particle experiences just a
few collisions, cf.\ \cite{CLS,CLSlong}. Our model contains a single
light particle, but we are able to control the dynamics up to
$cM^{a}$ collisions (and actually longer, the main restriction of
our analysis is the lack of control over $\sigma_Q(\cA)$ as the
heavy disk approaches the border $\dcD$, see Chapter~\ref{SecOP}).
We hope that our method can be used to analyze systems of many
particle as well.

We refer the reader to the surveys \cite{HBS, Sp, Sol, BLRB} for
descriptions of other models where macroscopic equations have been
derived from deterministic microscopic laws. One of the main
difficulties in deriving such equations is that microscopic
equations of motion are time reversible, while the limit
macroscopic equations are not, and therefore we cannot expect the
convergence everywhere in phase space. Normally, the convergence
occurs on a set of large (and, asymptotically, full) measure,
which is said to represent ``typical'' phase trajectories of the
system. In each problem, one has to carefully identify that large
subset of the phase space and estimate its measure. If the system
has some hyperbolic behavior, that large set has quite a complex
fractal structure.

The existing approaches to the problem of convergence make use of
certain families of measures on phase space, such that the
convergence holds with probability $\approx 1$ with respect to each
of those measures. In the context of hyperbolic dynamical systems,
the natural choice is the family of measures having smooth
conditional distributions on unstable manifolds \cite{S1,R0,PsSn}
(which is the characteristic property of Sinai-Ruelle-Bowen
measures). Such measures work very well for averaging
\index{Averaging} problems when the hyperbolicity is uniform and the
dynamics is entirely smooth \cite{K1, D1, D2}. However, if the
system has discontinuities and unbounded derivatives (as it happens
in our case), the analysis of its behavior near the singularities
becomes overwhelmingly difficult. Still, we will prove here that
this general approach applies to systems with singularities.
\newpage

\chapter{Plan of the proofs}
\label{SecPP} \setcounter{section}{3}\setcounter{subsection}{0}

\subsection{General strategy} \label{subsecGS}
Our heuristic calculations of the asymptotic distribution of
$\cV(\tau)$ in Section~\ref{subsecSR1} were based on the normal
approximation (\ref{Vt1t2}), hence our main goal is to prove it. A
natural approach is to fix the heavy disk at $Q=Q(t)$ and
approximate the map $\cF\colon\Omega\to\Omega$ by the billiard map
$\cF_{Q,V}\colon\Omega_{Q,V}\to\Omega_{Q,V}$, which is known
\cite{BS,Y} to obey the central limit theorem (\ref{EqVar}). This
approach, however, has obvious limitations.

On the one hand, our map $\cF$ has positive Lyapunov exponents,
hence its nearby trajectories diverge exponentially fast, so the
above approximation (in a strict sense) only remains valid during
time intervals $\cO(\ln M)$, which are far shorter than we need.
Thus, some sort of averaging is necessary to extend the CLT to
longer time intervals. Here comes the second limitation: the central
limit theorem for dispersing \index{Dispersing billiards} billiards
(\ref{EqVar}) holds with respect to the billiard measure
$\mu_{Q(t_0),V(t_0)}$, while we have to deal with the initial
measure $\mu_{Q_0,V_0}$ and its images under our map $\cF$, the
latter might be quite different from $\mu_{Q(t_0),V(t_0)}$.

To overcome these limitations, we will show that the measures
$\cF^n (\mu_{Q_0,V_0})$ can be well approximated (in the weak
topology) by averages (convex sums) of billiard measures
$\mu_{Q,V}$:
\begin{equation}
\label{ConsLocEq}
     \cF^n (\mu_{Q_0,V_0})\sim
     \int \mu_{Q,V}\, d\lambda_{n} (Q,V)
\end{equation}
where $\lambda_{n}$ is some factor measure on the $QV$ space.
Furthermore, it is convenient to work with an even larger family of
(auxiliary) \index{Auxiliary measures} measures, which we introduce
shortly, and extend the approximation (\ref{ConsLocEq}) to each
auxiliary measure $\mu'$:
\begin{equation}
\label{ConsLocEqA}
        \cF^n (\mu')\sim \int \mu_{Q,V}\, d\lambda_{n} (Q,V)
\end{equation}
for large enough $n$. In Section~\ref{subsecMS} below we make this
approximation precise.

The proof of the \index{Equidistribution} `equidistribution'
(\ref{ConsLocEq})--(\ref{ConsLocEqA}) follows a shadowing
 \index{Shadowing} type argument developed in the theory of uniformly
hyperbolic systems without singularities \cite{Ba1, D1, KKPW, R2}.
However, a major extra effort is required to extend this argument to
systems with singularities, like ours. In fact, the largest error
terms in our approximation (\ref{ConsLocEq}) come from the orbits
passing near singularities.

There are two places where we have trouble establishing
(\ref{ConsLocEq})--(\ref{ConsLocEqA}) at all. First, if the
velocity of the light particle becomes small, $\norm{v}\approx 0$,
then our system is no longer a small perturbation of a billiard
dynamics (because the heavy disk can move a significant distance
between successive collisions with the light particle). Second, if
the heavy disk comes too close to the border $\partial\cD$, then
the mixing properties of the corresponding billiard dynamics
deteriorate dramatically (roughly speaking, because the light
particle can be trapped in a narrow tunnel between $\dcP$ and
$\dcD$ for a long time). In this case the central limit theorem
could only provide a satisfactory normal approximation to the
billiard dynamics (in which the heavy disk is fixed) over very
large times, but then the position and velocity of the heavy disk
may change too much, rendering the billiard approximation itself
useless.

Accordingly, we will fix a small $\delta_1 < \delta_0$, and most of
the time we work in the region
\beq
       \label{Upsilon}
    (Q,V)\in\Upsilon_{\delta_1}\colon=
    \{M\norm{V}^2<1-\delta_1,\ \ \
    \dist(Q, \partial\cD)>\br +\delta_1\}
\eeq
(the first inequality guarantees that $\|v\| > \delta_1^{1/2}>0$).
We will show that violation of the first restriction (i.e.\ $\|v\|
\leq \delta_1^{1/2}$) is improbable on the time scale we deal with,
but possible violations of the other restriction (i.e.\
$\dist(Q,\partial\cD) \leq \br + \delta_1$) will force us to stop
the heavy disk whenever it comes too close to the border $\dcD$.

Our paper can be divided, roughly, into two parts, nearly equal in
size but quite different in mathematical content. In the first,
``dynamical'' part (Chapters~\ref{SecSPE} and \ref{SecRDM} and
Appendices) we analyze the mechanical model of two particles,
construct auxiliary \index{Auxiliary measures} measures, prove the
\index{Equidistribution} equidistribution
(\ref{ConsLocEq})--(\ref{ConsLocEqA}) and the log-Lipschitz
\index{Log-Lipschitz continuity} continuity of the diffusion
\index{Diffusion matrix} matrix (\ref{sigQ1Q2}). In the second,
``probabilistic'' part (Chapters~\ref{ScME}--\ref{ScSLP}) we prove
the convergence to stochastic processes as claimed in Theorems
~\ref{Root}--\ref{ThSmall}; there we use various (standard and
novel) moment-type techniques\footnote{Note that the time scale in
\cite{CE} corresponds to that of our Theorem~\ref{ThVelBM}. Indeed,
in the notation of \cite{CE}, $\varepsilon=1/\sqrt{M}$ is the
typical velocity of the heavy particle at equilibrium, hence their
``4/3 law'' becomes our ``2/3 law''. Let us also mention the papers
\cite{Kh, K3} studying the exit problem from a neighborhood of a
``non-degenerate equilibrium'' perturbed by a small noise. In these
terms, our Theorem \ref{ThVelBM} deals with a ``degenerate
equilibrium'', but the heuristic argument used to determine the
correct scaling is similar to that of \cite{Kh, K3}. } \cite{IL,
CE}. The arguments in Chapters~\ref{ScME}--\ref{ScSLP} do not rely
on the specifics of the underlying dynamical systems, hence if one
establishes results similar to (\ref{ConsLocEq})--(\ref{ConsLocEqA})
and (\ref{sigQ1Q2}) for another system, one would be able to derive
analogues of our limit theorems by the same moment estimates.

The proofs of Theorems~\ref{Root}--\ref{ThSmall} follow similar
lines, but Theorem~\ref{ThVelBM} requires much more effort than the
other two (mainly, because there are no explicit formulas for the
limiting process, so we have to proceed in a roundabout way). We
divide the proof of Theorem~\ref{ThVelBM} between three sections:
the convergence to equilibrium in the sense of
(\ref{ConsLocEq})--(\ref{ConsLocEqA}) is established in
Chapter~\ref{SecSPE}, the log-Lipschitz \index{Log-Lipschitz
continuity} continuity of the diffusion \index{Diffusion matrix}
matrix (\ref{sigQ1Q2}) in Chapter~\ref{SecRDM}, and the moment
estimates specific to the scaling of Theorem~\ref{ThVelBM} are done
in Chapter~\ref{ScME}. The modifications needed to prove the easier
theorems~\ref{Root} and \ref{ThSmall} are described in
Chapters~\ref{ScFSP} and \ref{ScSLP}, respectively.

\subsection{Precise definitions} \label{subsecPD}
First we give the definition of auxiliary \index{Auxiliary measures}
measures. Recall that our primary goal is control over measures
$\cF^n (\mu_{Q,V})$ for $n\geq 1$. The measure $\mu_{Q,V}$ is
concentrated on the surface $\Omega_{Q,V}$. Let us coarse-grain this
measure by partitioning $\Omega_{Q,V}$ into small subdomains
$D\subset\Omega_{Q,V}$ and representing $\mu_{Q,V}$ as a sum of its
restrictions to those domains. The image of a small domain
$D\subset\Omega_{Q,V}$ under the map $\cF^n$ gets strongly expanded
in the unstable direction of the billiard map $\cF_{Q,V}$, strongly
contracted in the stable direction of $\cF_{Q,V}$, slightly deformed
in the transversal directions, and possibly cut by singularities of
$\cF$ into several pieces. Thus, $\cF^n(D)$ looks like a union of
one-dimensional curves that resemble unstable manifolds of the
billiard map $\cF_{Q,V}$, but may vary slightly in the transversal
directions. Thus, the measure $\cF^n(\mu_{Q,V})$ evolves as a
weighted sum of smooth measures on such curves.

Motivated by this observation we introduce our family of auxiliary
\index{Auxiliary measures} measures. A {\em standard pair} is
$\ell=(\gamma,\rho)$, where \index{Standard pair}
$\gamma\subset\Omega$ is a $C^2$ curve, which is $C^1$ close to an
unstable curve $\gamma_{Q,V} \subset \Omega_{Q,V}$ for the billiard
map $\cF_{Q,V}$ for some $Q,V$, and $\rho$ is a smooth enough
probability density on $\gamma$. The precise description of standard
pairs is given in Chapter~\ref{SecSPE}, here we only mention the
properties of standard pairs most essential to our analysis. For a
standard pair $\ell$, we denote by $\gamma_\ell$ its curve, by
$\rho_{\ell}$ its density, and by $\mes_\ell$ the measure on
$\gamma$ with the density $\rho_\ell$.

We define auxiliary \index{Auxiliary measures} measures via convex
sums of measures on standard pairs, which satisfy an additional
``length control'':
\medskip

\noindent {\bf Definition}. An \emph{auxiliary measure} is a
probability measure $m$ on $\Omega$ such that \index{Auxiliary
measures}
\beq
       m=m_1+m_2,\qquad |m_2|<M^{-50}
         \label{aux1}
\eeq
and $m_1$ is given by
\beq
   m_1 = \int \mes_{\ell_{\alpha} } \, d\lambda(\alpha)
      \label{aux2}
\eeq
where $\{\ell_{\alpha} =(\gamma_{\alpha} ,\rho_{\alpha})\}$ is a
family of standard pairs such that $\{\gamma_{\alpha}\}$ make a
measurable partition of $\Omega$ ($m_1$-mod 0), and $\lambda$ is
\index{Standard pair}
some factor measure satisfying
\beq
   \lambda \Bigl(\alpha\colon\
   \length(\gamma_{\alpha} )<M^{-100} \Bigr) =0
      \label{aux3}
\eeq
which imposes a ``length control''. We denote by $\fM$ the family of
auxiliary measures. \index{Auxiliary measures}
\medskip

It is clear that our family $\fM$ contains the initial smooth
measure $\mu_{Q_0,V_0}$, as well as every billiard measure
$\mu_{Q,V}$ for $Q,V$ satisfying $\norm{V}<1/\sqrt{M}$ and
$\dist(Q,\partial\cD)>\br+\delta_0/2$. Indeed, one can easily
represent any of these measures by its conditional distributions on
the fibers of a rather arbitrary smooth foliation of the
corresponding space $\Omega_{Q,V}$ into curves whose tangent vectors
lie in unstable cones (see precise definitions in
Chapter~\ref{SecSPE}).

Next we need to define a class of functions
$\fR=\{A\colon\Omega\to\reals\}$ satisfying two general (though
somewhat conflicting) requirements. On the one hand, the functions
$A\in\fR$ should be smooth enough on the bulk of the space
$\Omega$ to ensure a fast (in our case -- exponential) decay of
correlations under the maps $\cF_{Q,V}$. On the other hand, the
regularity of the functions $A\in\fR$ should be compatible with
that of the map $\cF$, so that for any $A\in\fR$ the function
$A\circ\cF$ would also belong to $\fR$.

We will see in Chapter~\ref{SecSPE} that our map
$\cF\colon\Omega\to\Omega$ is not smooth, its singularity set
$\cS_1= \partial\Omega \cup \cF^{-1} (\partial\Omega)$ consists of
points whose next collision is grazing. Due to our finite horizon
\index{Finite horizon} assumption, $\cS_1\subset\Omega$ is a finite
union of compact $C^2$ smooth submanifolds (with boundaries). We
note that while the map $\cF$ depends on the mass $M$ of the heavy
disk, its singularity set $\cS_1$ does not. The complement
$\Omega\setminus\cS_1$ is a finite union of open connected domains,
we call them $\Omega_k$, $1\leq k\leq k_0$. The restriction of the
map $\cF$ to each $\Omega_k$ is $C^2$. The derivatives of $\cF$ are
unbounded, but their growth satisfies the following inequality:
\beq
    \norm{D_x \cF} \leq \,L_{\cF}\cdot [\dist (x,\cS_1)]^{-1/2}
      \label{DxcF}
\eeq
where $L_{\cF}>0$ is independent of $M$, see a proof in
Section~\ref{subsecSUV}. In addition, the restriction of $\cF$ to
each $\Omega_k$ can be extended by continuity to the closure
$\bar{\Omega}_k$, it then loses smoothness but remains H\"older
continuous:
\beq
      \forall k\quad \forall x,y\in\bar{\Omega}_{k}\quad
      \norm{\cF(x)-\cF(y)}\leq K_{\cF}\,[\dist(x,y)]^{1/2}
        \label{cFHolder}
\eeq
where $K_{\cF}>0$ is independent of $M$, see a proof in
Section~\ref{subsecSUV}.

These facts lead us to the following definition of $\fR$:

\medskip\noindent{\bf Definition}. A function $A\colon\Omega\to\reals$
belongs to $\fR$ iff\\
(a) $A$ is continuous on $\Omega \setminus \cS_1$. Moreover, the
continuous extension of $A$ to the closure of each connected
component $\Omega_{k}$ of $\Omega\setminus\cS_1$ is H\"older
continuous with some exponent $\alpha_A\in (0,1]$:
$$
  \forall k\quad \forall x,y\in\bar{\Omega}_{k}\quad
   |A(x)-A(y)|\leq K_A\,[\dist(x,y)]^{\alpha_A}
$$
(b) at each point $x\in \Omega\setminus\cS_1$ the function $A$ has
a local Lipschitz constant
\beq
   {\rm Lip}_x(A) \colon= \limsup_{y\to x}
   |A(y) - A(x)|/\dist(x,y)
     \label{Ax'}
\eeq
which satisfies the restriction
$$
     {\rm Lip}_x(A) \leq\,L_A
     \, [\, \dist(x,\cS_{1})]^{-\beta_A}
$$
with some $L_A>0$ and $\beta_A<1$. The quantities $\alpha_A\leq
1$, $\beta_A<1$, and $K_A,L_A>0$ may depend on the function $A$.
\medskip

Note that the set $\cS_1$, and hence the class $\fR$, are
independent of $M$. On the contrary, the singularity set
$$
   \cS_n = \cup_{i=0}^{n} \cF^{-1} (\partial\Omega)
$$
of the map $\cF^n$ depends on $M$ for all $n\geq 2$. In Appendix~B
(Section~\ref{subsecA3}) we will prove the following:

\begin{lemma} Let $n\geq 1$ and $B_1,B_2\in\fR$. Then the function
$A=B_1\,(B_2\circ\cF^{n-1})$ has the following properties:\\ {\rm
(a)} $A$ is continuous on $\Omega \setminus \cS_{n}$. Moreover,
the continuous extension of $A$ to the closure of each connected
component $\Omega_{n,k}$ of the complement
$\Omega\setminus\cS_{n}$ is H\"older continuous with some exponent
$\alpha_A\in (0,1]$:
$$
    \forall k \qquad
   \forall x,y\in\bar{\Omega}_{n,k}\quad
   |A(x)-A(y)|\leq K_A\,[\dist(x,y)]^{\alpha_A}
$$
{\rm (b)} at each point $x\in \Omega\setminus\cS_{n}$ the local
Lipschitz constant (\ref{Ax'}) of $A$ satisfies the restriction
$$
     {\rm Lip}_x(A) \leq\,L_A
     \, [\, \dist(x,\cS_{n})]^{-\beta_A}
$$
with some $\beta_A<1$. Here $\alpha_A$, $\beta_A$, $K_A$, and
$L_A$ are determined by $n$ and $\alpha_{B_i}$, $\beta_{B_i}$,
$K_{B_i}$, $L_{B_i}$ for $i=1,2$, but they do not depend on $M$.
\label{lmfRprop}
\end{lemma}

For any function $A\colon\Omega\to\reals$ and a standard pair
\index{Standard pair}
$\ell=(\gamma,\rho)$ we shall write
$$
      \EXP_\ell(A)=\int_\gamma A(x) \rho(x)\, dx.
$$
We also define a projection $\pi_1(Q,V,q,v)=(Q, V)$ from $\Omega$
to the $QV$ space.

\subsection{Key technical results}
\label{subsecMS} With the above notation we are ready to state
several propositions that give precise meaning to the heuristic
formulas \eqref{ConsLocEq} and \eqref{ConsLocEqA}. According to
(\ref{Upsilon}), we will only deal with standard pairs
$\ell=(\gamma,\rho)$ satisfying two restrictions:
\begin{equation}
     \label{Away}
     M \|\brV^2\| \leq 1-\delta_1 \quad
     \text{and} \quad
     \dist(\brQ,\partial\cD)>\br+\delta_1
\end{equation}
for some $(\brQ,\brV) \in \pi_1 (\gamma)$.

Next we fix a small $\delta_{\diamond} \ll \delta_1$ and will only
deal with $\cF^n$ where $n$ satisfies $0 \leq n\leq
\delta_\diamond\sqrt{M}$. This guarantees that the vital
restrictions (\ref{Away}) will not be grossly violated, i.e.\
$\|v\|$ will stay close enough to a positive constant, cf.\
(\ref{vvVV}), and $Q$ will stay away from the boundary $\dcD$.

In all our propositions, $K$ will denote sufficiently large
constants, i.e.\ all our statements will hold true if $K>0$ is large
enough (the value of $K$ can easily be chosen the same in all our
estimates, so we will use the same plain symbol to avoid unnecessary
indexation).

Our first proposition shows that the class of auxiliary
\index{Auxiliary measures} measures is ``almost'' invariant under
the dynamics (we can only claim ``almost'' invariance because
(\ref{Away}) will eventually be violated).

\begin{proposition}[Propagation]
\label{PrDistEq1} If $\ell=(\gamma,\rho)$ satisfies (\ref{Away}),
then for all $n$ satisfying
$$
  K\,|\ln\length(\gamma)| \leq n \leq \delta_\diamond\sqrt{M}
$$
and any integrable function $A$ we have
\begin{equation}
\label{EqProp}
      \EXP_\ell \left(A\circ\cF^n\right) = \sum_\alpha
      c_\alpha\, \EXP_{\ell_\alpha}(A)
\end{equation}
where $c_\alpha>0$, $\sum_\alpha c_\alpha=1$, and \index{Standard
pair} $\ell_\alpha=(\gamma_\alpha ,\rho_\alpha)$ are standard pairs
(the components of the image of $\ell$ under $\cF^n$ with induced
conditional measures); besides
\beq \label{Kvarepsilon}
   \sum_{\length(\gamma_\alpha)<\varepsilon}
   c_\alpha \leq K \varepsilon
\eeq
for all $\varepsilon>0$, the map $\cF^{-n}$ is smooth on each
$\gamma_\alpha$, and
\beq
   \forall y', y''\in \gamma_\alpha
   \quad \dist[\cF^{-m}(y'), \cF^{-m}(y'')]\leq
     K\vartheta^m
      \label{expclose}
\eeq
for all $1\leq m\leq n$ and some constant $\vartheta \in (0,1)$.
\end{proposition}

The condition $K\,|\ln\length(\gamma)| \leq n$ is necessary to give
short \index{Standard pair} standard pairs enough time to expand and
satisfy (\ref{Kvarepsilon}).

The next proposition basically shows that if $\ell$ satisfies
(\ref{Away}), then
$$
    \cF^n (\mes_{\ell}) \approx \mu_{\brQ, \brV}
$$
(here we do not need to consider convex combinations yet) for all
$n$ in the range
\beq \label{nnnn}
   \ln M \lesssim
    n \lesssim M^b,
   \qquad \text{where}\quad 0<b<1/2
\eeq
The lower bound on $n$ guarantees that $\cF^n_{\brQ, \brV}
(\mes_{\ell})$ is almost uniformly distributed (`equidistributed')
in $\Omega_{\brQ,\brV}$, and the \index{Equidistribution} upper
bound on $n$ prevents $Q$ and $V$ from changing significantly during
$n$ iterations (implying that $\cF_{\brQ, \brV}^n$ will be still a
good approximation to $\cF^n$). We will use functions $A\colon
\Omega\to\reals$ such that
\beq \label{ABB}
     A=B_1\, (B_2\circ\cF^{n_A-1}),
     \qquad B_1,B_2\in\fR
\eeq
for some small fixed $n_A\geq 1$ (independent of $M$). All the
constants denoted by $K$ will now depend on $n_A$ as well.

\begin{proposition}[Short term equidistribution]
\label{PrDistEq2} Let $\ell=(\gamma,\rho)$ satisfy (\ref{Away}) and
\index{Equidistribution} $A$ satisfy (\ref{ABB}). Then for all $n$
satisfying
$$
   K\, |\ln\length(\gamma)| \leq n \leq
       \delta_{\diamond}\sqrt{M}
$$
and all $m \leq \min\{ n/2,K\ln M\}$ we have
$$
   \EXP_\ell (A\circ\cF^n)=
   \mu_{\brQ,\brV}(A)+ \cO(\cR_{n,m} + \theta^m),
$$
where $\theta \in (0,1)$ is a constant and
$$
    \cR_{n,m} = \norm{\brV} (n+m^2) + (n^2+m^3)/M.
$$
\end{proposition}

Even though Proposition~\ref{PrDistEq2} is formulated for a wider
range of $n$ than given by (\ref{nnnn}), it will only be useful when
(\ref{nnnn}) holds, otherwise the error terms will be too big.

If we want to extend Proposition~\ref{PrDistEq2} to $n$ beyond the
upper bound in (\ref{nnnn}) and still keep the error terms small, we
will have to deal with possible significant variation of the
coordinates $Q$ and $V$ over the set $\cF^n (\gamma_{\ell})$. That
can be done by using convex combinations of measures $\mu_{Q,V}$ (in
the spirit of (\ref{ConsLocEq})), but it will be sufficient for us
to restrict the analysis to a simpler case of functions $A$ whose
average $\mu_{Q,V}(A)$ does not depend on $Q$ or $V$.

\begin{corollary}[Long term equidistribution]
\label{PrDistEq3} Let $\ell=(\gamma,\rho)$ satisfy (\ref{Away}), $A$
\index{Equidistribution} satisfy (\ref{ABB}), and, additionally,
$\brA=\mu_{Q,V}(A)$ be independent of $Q,V$. If
$$
   K\, |\ln\length(\gamma)| \leq n \leq
     \delta_{\diamond}\sqrt{M},
$$
then for all $j$ satisfying
$$
   K\, |\ln\length(\gamma)| \leq j \leq
   n - K\, |\ln\length(\gamma)|
$$
and $m\leq\min\{j/2,K\ln M\}$ we have
\begin{equation}
\label{ConstExp}
   \EXP_\ell ( A\circ \cF^n)
   =\brA+\cO(\cR_{n,j,m}+\theta^{m}),
\end{equation}
where
$$
    \cR_{n,j,m} = \EXP_\ell\left(\norm{V_{n-j}}\right)(j+m^2)
    + (j^2+m^3)/M,
$$
and $V_{n-j}$ denotes the $V$ component of the point
$\cF^{n-j}(x)$, $x\in\gamma$.
\end{corollary}

Even though Corollary \ref{PrDistEq3} is formulated for a wide range
of $j$ and $m$, it will only be useful when
$$
   \ln M \lesssim
    j,m \lesssim M^b,
   \qquad \text{where}\quad 0<b<1/2,
$$
otherwise the error terms become too big. But the main number of
iterations, $n$, can well grow up to $\delta_{\diamond}\sqrt{M}$, in
this sense the corollary describes `long term equidistribution'. To
derive Corollary~\ref{PrDistEq3} we apply \index{Equidistribution}
Proposition~\ref{PrDistEq1} with $n-j$ in place of $n$ and then
apply Proposition~\ref{PrDistEq2} (with $j$ in place of $n$) to each
$\alpha$ in \eqref{EqProp}, see Remark in the end of
Chapter~\ref{SecSPE}.

Lastly we state one more technical proposition necessary for the
proof of Theorem~\ref{ThVelBM}. We formulate it in probabilistic
terms (however, it is known to be equivalent to the fact that
limiting factor measure $\lambda_{n}(Q,V)$ of (\ref{ConsLocEq})
satisfies an associated partial differential equation):

\begin{proposition}
\label{DLbill}
\label{Weak}
{\rm (a)} Let $M_0>0$ and $a>0$. The families of random processes
$Q_{\ast}(\tau M^{2/3})$ and $M^{2/3}V_{\ast}(\tau M^{2/3})$ such
that $M\geq M_0,$ and the initial condition $(Q_0,V_0,q(0),v(0))$
is chosen randomly with respect to a measure in $\fM$ such that
almost surely $\norm{V_0}\leq a M^{-2/3}$, are tight and any limit
process
$(\bQ(\tau), \bV(\tau))$ satisfies (\ref{TIBM}).\\
{\rm (b)} If the matrix $\sigma_Q(\cA)$ satisfies (\ref{sigQ1Q2}),
then the equations (\ref{TIBM}) are well posed in the sense that
any two solutions with the same initial conditions have the same
distribution.
\end{proposition}

The proofs of Propositions~\ref{PrDistEq1} and \ref{PrDistEq2} and
Corollary~\ref{PrDistEq3} are given in Chapter~\ref{SecSPE}. The
heart of the proof is contained in Sections~\ref{subsecPA} and
\ref{subsEP}, whereas Sections~\ref{subsecSUV}--\ref{subsecSP}
extend some known results for classical billiards to our
two-particle model. We remark that the estimates in
Propositions~\ref{PrDistEq2} and \ref{PrDistEq3} are likely to be
less than optimal, but they suffice for our purposes because we
restrict our analysis to time periods $\cO(M^{2/3})$, which is much
shorter than the ergodization time (the latter is apparently of
order $M$, as one can see via a heuristic analysis similar to that
in Section~\ref{subsecBA}). Therefore to investigate the long time
behavior of our system, the estimates of
Propositions~\ref{PrDistEq2} and \ref{PrDistEq3} might have to be
sharpened (see Section~\ref{subsLTS}), but we do not pursue this
goal here.

Proposition~\ref{Weak} is proved in Chapter~\ref{ScME}. In Chapters
\ref{ScFSP} and \ref{ScSLP} we describe the modifications needed to
prove Theorems~\ref{Root} and \ref{ThSmall} respectively. Chapter
\ref{ScFSP} is especially short since the material there is quite
similar to \cite[Sections 13 and 14]{D2}, except that here some
additional complications are due to the fact that we have to deal
with a continuous time system.
\newpage

\chapter[Standard pairs]{Standard pairs and equidistribution}
\label{SecSPE} \setcounter{section}{4}\setcounter{subsection}{0}
\index{Equidistribution} \index{Standard pair}

The main goals of this section are the construction of standard
pairs and the proofs of Statements~\ref{PrDistEq1}, \ref{PrDistEq2}
and \ref{PrDistEq3}.

\subsection{Unstable vectors} \label{subsecSUV}
Our analysis will be restricted to the region (\ref{Upsilon}). We
first discuss the flow $\Phi^t$ in the full (seven-dimensional)
phase space $\cM$ in order to collect some preliminary estimates.

Let $x=(Q,V,q,v)\in \cM$ be an arbitrary point and
$$dx=(dQ,dV,dq,dv) \in \cT_x \cM$$ a tangent vector. Let $dx(t) =
D\Phi^t (dx)$ be the image of $dx$ at time $t$. We describe the
evolution of $dx(t)$ for $t>0$.

Between successive collisions, the velocity components $dV$ and
$dv$ remain unchanged, while the position components evolve
linearly:
\beq
   dQ(t+s) = dQ(t)+s\, dV(t),
   \ \ \ \ \ \
   dq(t+s) = dq(t)+s\, dv(t)
      \label{Qqts}
\eeq
At collisions, the tangent vector $dx(t)$ changes discontinuously,
as we describe below.

First, we need to introduce convenient notation. For any unit
vector $n\in\reals^2$ (usually, a normal vector to some curve), we
denote by $\bP_n$ the projection onto $n$, i.e.\ $\bP_{n} (u) =
\la u,n \ra \, n$, and by $\bP^{\perp}_n$ the projection onto the
line perpendicular to $n$, i.e.\ $\bP^{\perp}_{n}(u) = u -
\bP_n(u)$. Also, $\bR_n$ denotes the reflection across the line
perpendicular to $n$, that is
$$
      \bR_n(u) =
      -\bP_{n}(u) + \bP^{\perp}_{n}(u)
       = u - 2\la u,n\ra \, n
$$
For any vector $w\neq 0$, we write $\bP^{\perp}_w$ for
$\bP^{\perp}_{w/\norm{w}}$ , for brevity.

Now, consider a collision of the light particle with the wall
$\dcD$, and let $n$ denote the inward unit normal vector to $\dcD$
at the point of collision. The components $dQ$ and $dV$ remain
unchanged because the heavy disk is not involved in this event.
The basic rule of specular reflection at $\dcD$ reads $v^+ = \bR_n
(v^-)$ (the superscripts ``$+$" and ``$-$" refer to the
postcollisional and precollisional vectors, respectively). Note
that $\|v^+\| = \|v^-\|$. Accordingly, the tangent vectors $dq$
and $dv$ change by
$$
   dq^+ = \bR_n (dq^-)
$$
and
$$
    dv^+ = \bR_n \left ( dv^-\right ) + \bTheta^+(dq^+)
$$
where
$$
    \bTheta^+ =
    \frac{2\cK\,\|v^+\|^2}{\la v^+,n\ra}
    \,\bP^{\perp}_{v^+}
$$
Here $\cK > 0$ denotes the curvature of the boundary $\dcD$ at the
point of collision. Note that $\|dq^+\| = \|dq^-\|$. Also,
$\bTheta^+(dq^+) = \bTheta^-(dq^-)$ where
\beq
   \bTheta^- = \frac{2\cK\,\|v^+\|^2}{\la v^+,n\ra}
    \,\bR_n \circ \bP^{\perp}_{v^-}
      \label{Theta-0}
\eeq
Also, the geometry of reflection implies $\la v^+,n\ra > 0$.

Next, consider a collision between the two particles. At the
moment of collision we have $q\in \dcP (Q)$, i.e.\ $\|q-Q\|=\br$.
Let $n=(q-Q)/\br$ be the normalized relative position vector. Then
the laws of elastic collision (\ref{CR1})--(\ref{CR2}) can be
written as
\begin{align*}
  v^+ &= v^- - \frac{2M}{M+1}\,
  \bP_n(v^--V^-)\\
   &= \bR_n(v^-) + \frac{2M}{M+1}\,
  \left (\frac 1M\, \bP_n(v^-) + \bP_n(V^-)\right)\\
   V^+
  &= V^- + \frac{2}{M+1}\,
  \bP_n(v^--V^-)
\end{align*}
Let $w = v-V$ denote the relative velocity vector, cf.\ (\ref{w}).
Then
$$
   w^+ = w^--2\,\bP_n(w^-) = \bR_n(w^-)
$$
and hence $\|w^+\|=\|w^-\|$. The components $dq$ and $dQ$ of the
tangent vector $dx$ change according to
\begin{equation}
\label{Rdq}
  dq^+ = \bR_n(dq^-) + \frac{2M}{M+1}\,
  \left (\frac 1M\, \bP_n(dq^-) + \bP_n(dQ^-)\right)
\end{equation}
\begin{equation}
\label{RdQ}
  dQ^+= \bR_n(dQ^-) + \frac{2M}{M+1}\,
  \left ( \frac 1M\, \bP_n(dq^-) + \bP_n(dQ^-)\right )
\end{equation}
$$  = dQ^- + \frac{2}{M+1}\, \bP_n(dq^--dQ^-) $$

Note that
$$
  dq^+ - dQ^+ = \bR_n (dq^- - dQ^-)
$$
and so $\|dq^+ - dQ^+\| = \|dq^- - dQ^-\|$. Next, the components
$dv$ and $dV$ of the tangent vector $dx$ change by
\begin{align*}
  dv^+ &= \bR_n(dv^-) + \frac{2M}{M+1}\,
  \left (\frac 1M\, \bP_n(dv^-) + \bP_n(dV^-)\right)\\
  &\quad +\frac{M}{M+1}\,\bTheta^+(dq^+-dQ^+)
\end{align*}
and
\begin{align*}
  dV^+ &= dV^- + \frac{2}{M+1}\,
  \bP_n(dv^--dV^-) \\
  &\quad -\frac{1}{M+1}\,\bTheta^+(dq^+-dQ^+)
\end{align*}
where
$$
   \bTheta^+ = \frac{2\cK\,\|w^+\|^2}{\la w^+,n\ra}
   \,\bP^{\perp}_{w^+}
$$
Here $\cK=1/\br$ is the curvature of $\dcP (Q)$. Note that
$\bTheta^+(dq^+ - dQ^+) = \bTheta^-(dq^- - dQ^-)$, where
$$
   \bTheta^- = \frac{2\cK\,\|w^+\|^2}{\la w^+,n\ra}
   \, \bR_n \circ \bP^{\perp}_{w^-}
$$
Also, the geometry of collision implies $\la w^+,n\ra > 0$, since
we have chosen $n$ to point {\em toward} the light particle.

All the above equations can be verified directly. Alternatively, one
can use the fact that the system of two particles of different
masses $M\neq m$ reduces to a billiard in a four dimensional domain
by the change of variables $\tilde{Q} = Q\sqrt{M}$, $\tilde{V} =
V\sqrt{M}$, $\tilde{q} = q\sqrt{m}$, and $\tilde{v} = v\sqrt{m}$
(the latter two are trivial since $m=1$). This reduction is standard
\cite{SS}, and then the above equations can be derived from the
general theory of billiards \cite{C1,KSS,SS}. We omit the proof of
the above estimates. \index{Standard pair}

Now, since the total kinetic energy is fixed (\ref{Energy}), the
velocity components $dv$ and $dV$ of the tangent vector $dx$ satisfy
\beq
      \la v,dv\ra + M\la V,dV\ra = 0
         \label{Vort}
\eeq
In addition, the Hamiltonian character of the dynamics implies
that if the identity
\beq
      \la v,dq\ra + M\la V,dQ\ra = 0
         \label{Qort}
\eeq
holds at some time, it will be preserved at all times (future and
past). From now on, we assume that all our tangent vectors satisfy
(\ref{Qort}).

There is a class of tangent vectors, which we will call {\em
unstable vectors}, \index{Unstable vectors} that is invariant under
the dynamics. It is described in the following proposition:

\begin{proposition}
The class of tangent vectors $dx$ with the following properties
remains invariant under the forward dynamics:
\begin{itemize}
\item[\rm (a)] $\la dq,dv\ra \geq (1-C^{-1})\,\|dq\|\,\|dv\|$ \item[\rm
(b)] $\|dQ\| \leq \frac CM \, \|dq\|$ \item[\rm (c)] $\|dV\| \leq
\frac CM \, \|dv\|$ \item[\rm (d)] $\la dq,v\ra \leq C\|V\| \,
\|dq\| \leq \frac{C}{\sqrt{M}} \, \|dq\|$ \item[\rm (e)] $\la
dv,v\ra \leq C\|V\| \, \|dv\| \leq \frac{C}{\sqrt{M}} \, \|dv\|$
\item[\rm (f)] $\|dq\| \leq C\|dv\|$ \item[\rm (g)] $\|dv\| \leq
\frac{C}{|t^-(x)|} \, \|dq\|$
\item[\rm (h)] Equations (\ref{Vort}) and (\ref{Qort}) hold.
\end{itemize}
Here $C>1$ is a large constant, and $t^-(x) = \max\{t\leq 0\colon
\Phi^t(x) \in \Omega \}$ is the time of the latest collision along
the past trajectory of $x$. \label{prabc}
\end{proposition}

The proof of this proposition is based on the previous equations
and some routine calculations, which we omit. \qed \medskip

We emphasize that our analysis has been done in the region
(\ref{Upsilon}) only, hence the above invariance holds as long as
the system stays in $\Upsilon_{\delta_1}$; the constant $C$ here
depends on the choice of $\delta_1>0$, and we expect $C \to \infty$
as $\delta_1 \to 0$.

We also note that though unstable \index{Unstable vectors} vectors
make a multidimensional cone in the tangent space to $\Omega$, this
cone is essentially one-dimensional, its `opening' in the $Q$ and
$V$ directions is $\cO(1/M)$. In the limit $M \to \infty$ we simply
obtain the one-dimensional unstable cone for the classical billiard
map.

Unstable \index{Unstable vectors} vectors have strong (uniform in
time) expansion property:

\begin{proposition}
Let $dx$ be an unstable tangent vector and $dx(t) = D\Phi^t (dx)$
its image at time $t>0$. Then the norm $\| dx(t) \|$ monotonically
grows with $t$. Furthermore, there is a constant $\vartheta <1$
such that for any two successive moments of collisions $t<t'$ of
the light particle with $\dcD\cup\dcP(Q)$ we have
\beq
     \|dx(t+0)\| \leq \vartheta\, \|dx(t'+0)\|
       \label{theta}
\eeq
The notation $t+0$, $t'+0$ refer to the postcollisional vectors.
\label{prhyper}
\end{proposition}

The proof easily follows from the previous equations. In fact,
\beq
   \vartheta^{-1} = 1 + L_{\min}\cK_{\min}\\
     \label{thetamin}
\eeq
where $\cK_{\min}>0$ is the smaller of $1/\br$ and the minimal
curvature of $\dcD$, $L_{\min}$ is the smaller of the minimal
distance between the scatterers \index{Scatterer} and $\delta_1$,
the minimal distance from the heavy disk to the scatterers allowed
by (\ref{Upsilon}).

At the moments of collisions it is more convenient (for technical
reasons) to use the vector $w$ defined by (\ref{w}), instead of
$v$, and respectively $dw$ instead of $dv$. Then the vector $w$
changes by the same rule $w^+ = \bR_n(w^-)$ for both types of
collisions (at $\dcD$ and $\dcP(Q)$). At collisions with $\dcD$,
the vector $dw=dv$ will change by the rule
$$
    dw^+ = \bR_n \left ( dw^-\right ) + \bTheta^+(dq^+)
$$
while at collisions with the heavy disk, the vector $dw=dv-dV$
will change by a similar rule
$$
    dw^+ = \bR_n \left ( dw^-\right ) + \bTheta^+(dq^+-dQ^+)
$$
Furthermore, the expressions for $\bTheta^+$ and $\bTheta^-$ will
be identical for both types of collisions. The geometry of
collision implies $\la w,n \ra \geq 0$ for both types of
collisions.

It is easy to see that the inequalities (a)--(g) in
Proposition~\ref{prabc} remain valid if we replace $v$ by $v-V$
and $dv$ by $dv-dV$ at any phase point, hence they apply to the
vectors $w$ and $dw$ at the points of collision. The inequality
(\ref{theta}) will also hold in the norm on $\Omega$ defined by
\beq
    \|dx\|^2 = \|dQ\|^2+\|dV\|^2+\|dq\|^2+\|dw\|^2
       \label{dxnorm}
\eeq

\medskip\noindent{\em Remark}. Our equations show that the
postcollisional tangent vector $$(dQ^+, dV^+,dq^+,dw^+)$$ depends on
the precollisional vector $$(dQ^-,dV^-,dq^-,dw^-)$$ smoothly, unless
$\la w^+ ,n \ra = 0$. This is the only singularity of the dynamics,
it corresponds to grazing collisions (also, colloquially, called
``tangential collisions'').

\medskip\noindent{\em Remark}. Our equations imply that the derivative
of the collision map $\cF$ defined in Section~\ref{subsecBA} is
bounded by $\|D_x\cF\|\leq\Const/\la w^+ ,n \ra$, where $\Const$
does not depend on $M$. It is easy to see that
$\dist(x,\cS_1)=\cO(\la w^+ ,n \ra^2)$, hence we obtain
(\ref{DxcF}). Now (\ref{cFHolder}) easily follows by integrating
(\ref{DxcF}).
\medskip

\subsection{Unstable curves}\label{subsecSUC}
We call a smooth curve $\cW \subset \cM$ an {\em unstable curve} (or
a \index{u-curves (unstable curves)} {\em u-curve}, for brevity) if,
at every point $x\in\cW$, the tangent vector to $\cW$ is an unstable
\index{Unstable vectors} vector. By Propositions~\ref{prabc} and
\ref{prhyper} the future image of a u-curve is a u-curve, which may
be only piecewise smooth, due to singularities, and every
\index{u-curves (unstable curves)} u-curve is expanded by $\Phi^t$
monotonically and exponentially fast in time.

Now we extend our analysis to the collision map
$\cF\colon\Omega\to\Omega$. For every point $x \in \cM$ we denote
by $t^+(x) =\min\{ t\geq 0\colon \Phi^t(x) \in \Omega \}$ and
$t^-(x) =\max\{ t\leq0\colon\Phi^t(x) \in \Omega \}$ the first
collision times in the future and the past, respectively. Let
$\tilde{\pi}^{\pm}(x) = \Phi^{t^{\pm}(x)}(x) \in \Omega$ denote
the respective ``first collision'' projection of $\cM$ onto
$\Omega$. Note that $\cF(x)=\tilde{\pi}^+(\Phi^\varepsilon x)$ for
all $x\in\Omega$ and small $\varepsilon>0$.

For any unstable curve $\cW\subset\cM$, the projection $W =
\tilde{\pi}^- (\cW )$ is a smooth or piecewise smooth curve in
$\Omega$, whose components we also call {\em unstable curves} or
{\em u-curves}. Let $dx = (dQ, dV, dq, dw)$ be the postcollisional
tangent vector to $\cW$ at a moment of collision (we remind the
reader that dim$\, \cM =7$ and dim$\, \Omega =6$). Its projection
under the derivative $D\tilde{\pi}^-$ is a tangent vector $dx' =
(dQ', dV', dq', dw')$ to the \index{u-curves (unstable curves)}
u-curve $W = \tilde{\pi}^-(\cW) \subset \Omega$. Observe that
$dV'=dV$, $dw'=dw$, $dQ'=dQ-tV$, and $dq'=dq-tv$, where $t$ is
uniquely determined by the condition $dx'\in\cT_x\Omega$. Now some
elementary geometry and an application of Proposition~\ref{prabc}
give

\begin{proposition} \label{prdxdx'}
There is a constant $1<C<\infty$ such that $\|dQ'\|\leq
C\|dv\|/\sqrt{M}$ and $C^{-1}\|dv\|\leq\|dq'\|\leq C\|dv\|$.
Therefore,
$$
  \|dx'\|^2 = \bigl[ \|dq'\|^2 + \|dv'\|^2\bigr]\,
  \bigr [ 1 + \cO(1/\sqrt{M}) \bigr ] ,
$$
and $C^{-1} \leq \|dx'\| / \|dx\| \leq C$.
\end{proposition}

We introduce two norms (metrics) on \index{u-curves (unstable
curves)} u-curves $W\subset\Omega$. First, we denote by
$\length(\cdot)$ the norm on $W$ induced by the Euclidean norm
$\|dx'\|$ on $\cT_x(\Omega)$. Second, if $W = \tilde{\pi}^- (\cW )$,
we denote by $|\cdot|$ the norm on $W$ induced by the norm
(\ref{dxnorm}) on the postcollisional tangent vectors $dx$ to $\cW$
at the moment of collision. Due to the last proposition, these norms
are equivalent in the sense
\beq
       C^{-1} \leq \frac{\length(W)}{|W|} \leq C
         \label{length++}
\eeq

%\medskip\noindent{\em Remark} (on the lengths of unstable curves).
By (\ref{length++}), we can replace $\length(\gamma)$ with
$|\gamma|$ in the assumptions of Propositions~\ref{PrDistEq1} and
\ref{PrDistEq2}, as well as in many other estimates of our paper. We
actually prefer to work with the $|\cdot|$-metric, because it has an
important {\em uniform expansion} property: the map $\cF$ expands
every \index{u-curves (unstable curves)} u-curve in the
$|\cdot|$-metric by a factor $\geq\vartheta^{-1}>1$, see
(\ref{thetamin}) (while the $\length(\cdot)$ metric lacks this
property).\medskip

Observe that the $Q,V$ coordinates vary along \index{u-curves
(unstable curves)} u-curves $W\subset\Omega$ very slowly, so that
u-curves are almost parallel to the cross-sections $\Omega_{Q,V}$ of
$\Omega$ defined by (\ref{OmegaQVbrw}). Each $\Omega_{Q,V}$ can be
supplied with standard coordinates. Let $r$ be the arc length
parameter along $\dcD \cup \dcP (Q)$ and $\varphi\in [-\pi/2,\pi/2]$
the angle between the outgoing relative velocity vector $\brw$ and
the normal vector $n$. The orientation of $r$ and $\varphi$ is shown
in Fig.~\ref{FigOrient}. Topologically, $\Omega_{Q,V}$ is a union of
cylinders, in which the cyclic coordinate $r$ runs over the
boundaries of the scatterers \index{Scatterer} and the disk $\dcP
(Q)$, and $\varphi\in [-\pi/2, \pi/2]$. We need to fix reference
points on each scatterer and on $\dcP (Q)$ in order to define $r$,
and then the coordinate chart $r,\varphi$ in $\Omega_Q$ will
actually be the same for all $Q,V$. We denote by $\Omega_0$ that
unique $r,\varphi$ coordinate chart.

\begin{figure}[htb]
    \centering
    \psfrag{vm}{$v^-$}
    \psfrag{vp}{$v^+$}
    \psfrag{n}{$n$}
    \psfrag{f}{$\varphi$}
    \psfrag{r}{$r$}
    \includegraphics{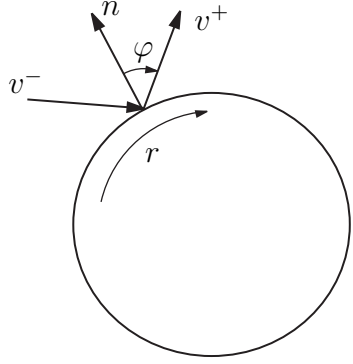}
    \caption{A collision of the light particle with a scatterer:
    the orientation of $r$ and $\varphi$}
    \label{FigOrient}
\end{figure}

Note that $r$ and $\varphi$ are defined at every point
$x\in\Omega$, hence they make two coordinates in the
(six-dimensional) space $\Omega$. Since $\cos\varphi = \la
w,n\ra/\|w\|$, the singularities of the map $\cF$ correspond to
$\cos \varphi =0$, i.e.\ to $\varphi=\pm\pi/2$ (which is the
boundary of $\Omega_0$). Let $\pi_0$ denote the natural projection
of $\Omega$ onto $\Omega_0$. Note that, for each $Q,V$ the
projection $\pi_0\colon \Omega_{Q,V} \to \Omega_0$ is one-to-one.
Then the map
$$
     \pi_{Q,V} \colon= \left ( \pi_0|_{\Omega_{Q,V}}\right )^{-1}
     \circ \pi_0
$$
defines a natural projection $\Omega \to \Omega_{Q,V}$
(geometrically, it amounts to moving the center of the heavy disk
to $Q$, setting its velocity to $V$, and rescaling the vector $w$
at points $q\in\dcP(Q)$ by the rule (\ref{brw})).

We turn back to \index{u-curves (unstable curves)} u-curves
$\cW\subset\cM$. For any such curve, $W = \pi_0 (\tilde{\pi}^- (\cW
))$ is a smooth or piecewise smooth curve in $\Omega_0$, whose
components we also call {\em u-curves}. Any such curve is described
by a smooth function $\varphi = \varphi(r)$. Let $dx = (dQ, dV, dq,
dv)$ be the postcollisional tangent vector to $\cW$ at a moment of
collision. Its projection under the derivative $D( \pi_0 \circ
\tilde{\pi}^-)$ is a tangent vector to $W$, which we denote by
$(dr,d\varphi)$.

To evaluate $(dr,d\varphi)$, we introduce two useful quantities,
$\cE$ and $\cB$, at each collision point. We set $\cE = \|
\bP^{\perp}_{w}(dq)\|$ if the light particle collides with $\dcD$,
and $\cE =\| \bP^{\perp}_{w}(dq-dQ)\|$ if it collides with the
disk. Then we set
$$
  \cB = \frac{\| \bP^{\perp}_{w}(dw)\|}
  {\cE\,\| w\|}
$$

\begin{proposition}
In the above notation,
$$
   |dr| = \cE / \cos\varphi
   \quad{\rm and}\quad
   d\varphi/dr= \cB\cos\varphi - \cK,
$$
where $\cK>0$ is the curvature of $\dcD\cup\dcP$ at the point of
collision. There is a constant $C>1$ such that for any
\index{u-curves (unstable curves)} u-curve $W\subset\Omega$ and any
point $x \in W$
$$
       \frac{2\cK}{\cos\varphi} \leq \cB
       \leq \frac{2\cK}{\cos\varphi} + C
$$
and
\beq
       C^{-1} \leq \frac{d\varphi}{dr} \leq C
          \label{CCu}
\eeq
In particular, $d\varphi/dr>0$, hence $\pi_0 (W)$ is an increasing
curve in the $r,\varphi$ coordinates. Lastly,
$$
       C^{-1} \leq \frac{(dr)^2+(d\varphi)^2}{\|dx\|^2} \leq C
$$
\label{prdrdp}
\end{proposition}

The proof is based on elementary geometric analysis, and we omit
it. \qed \medskip

Next we study the evolution of \index{u-curves (unstable curves)}
u-curves under the map $\cF$. Let $W_0 \subset \Omega$ be a u-curve
on which $\cF^n$ is smooth for some $n \geq 1$. Then $W_i = \cF^i (
W_0 )$ for $i\leq n$ are \index{u-curves (unstable curves)}
u-curves. Pick a point $x_0 \in W_0$ and put $x_i=\cF^i (x_0)$ for
$i\leq n$. For each $i$, we denote by $r_i$, $\varphi_i$, $\cK_i$,
$\cB_i$, etc.\ the corresponding quantities, as introduced above, at
the point $x_i$.

Also, for any \index{u-curves (unstable curves)} u-curve $W \subset
\Omega$ and $k \geq 1$ we denote by $\cJ_W \cF^k (x)$ the Jacobian
of the map $\cF^k \colon W \to \cF^k (W)$ at the point $x\in W$ in
the norm $| \cdot |$, i.e.\ the local expansion factor of the curve
$W$ under $\cF^k$ in the $|\cdot|$-metric.

\begin{proposition}
There is a constant $1<C<\infty$ such that
$$
   1 + \frac{C^{-1}}{\cos\varphi_{i+1}} < \cJ_{W_i}\cF(x_i)
   < 1 + \frac{C}{\cos\varphi_{i+1}}
$$
and
\beq
     \cJ_{W_0}\cF^n(x_0)=
     \cJ_{W_0}\cF(x_0)\cdots \cJ_{W_{n-1}}\cF(x_{n-1})
     \geq \vartheta^{-n}
        \label{JJtheta}
\eeq
where $\vartheta<1$ is given by (\ref{thetamin}). \label{prJexp}
\end{proposition}

\proof This follows from Proposition~\ref{prdrdp} by direct
calculation. \qed \medskip

\noindent{\em Remark}. By setting $M=\infty$ in all our results we
obtain their analogues for the billiard-type dynamics in $\cD
\setminus \cP(Q)$, in which the disk $\cP(Q)$ is fixed and the light
particle moves at a constant speed $\|\brw\|=\bs_V$ given by
(\ref{brw}). Most of them are known in the studies of billiards. In
particular, we recover standard definitions of unstable
\index{Unstable vectors} vectors and unstable curves for the
billiard-type map $\cF_{Q,V}\colon\Omega_{Q,V}\to\Omega_{Q,V}$.
Proposition~\ref{prdrdp} implies that the $\|dx\|$ norm on
$\Omega_{Q,V}$ becomes
\begin{align}
     \|dx\|^2 &= \|dq\|^2 + \|dw\|^2\nonumber\\
     &= (dr\cos\varphi)^2 + \bs_V^2(d\varphi+\cK\, dr)^2
       \label{dxnormQV}
\end{align}
In this norm, the map $\cF_{Q,V}$ expands every unstable curve by
a factor $\geq\vartheta^{-1}>1$.

As usual, reversing the time (changing $\cF_{Q,V}$ to
$\cF^{-1}_{Q,V}$) gives the definition of stable vectors and
stable curves (or {\em s-curves} for brevity) in the space
$\Omega_{Q,V}$. Those are decreasing in the $r,\varphi$
coordinates and satisfy the bound
\beq
        -C<d\varphi/dr<-C^{-1}<0
          \label{CCs}
\eeq
The corresponding norm on stable vectors/curves is defined on {\em
precollisional} tangent vectors and is expressed by
\beq
     \|dx\|^2_{\rm stable} = (dr\cos\varphi)^2 + \bs_V^2(d\varphi-\cK\, dr)^2
       \label{dxnormQVst}
\eeq
which differs from (\ref{dxnormQV}) by the sign before $\cK$. In
the norm (\ref{dxnormQVst}), the map $\cF_{Q,V}$ contracts every
s-curve by a factor $\leq\vartheta<1$.

\subsection{Homogeneous unstable curves}\label{subsecSHUC}
To control distortions of \index{u-curves (unstable curves)}
u-curves by the map $\cF$, we need to carefully partition the
neighborhood of the singularity set $\partial\Omega = \{\cos \varphi
= 0\}$ into countably many surrounding sections (shells). This
procedure has been introduced in \cite{BSC2} and goes as follows.
Fix a large $k_0\geq 1$ and for each $k\geq k_0$ define two
``homogeneity strips'' in $\Omega_0$ \index{Homogeneity strips}
$$
    \bbH_k=\{(r,\varphi)\colon \pi/2-k^{-2}<\varphi <\pi/2-(k+1)^{-2}\}
$$
and
$$
    \bbH_{-k}=\{(r,\varphi)\colon -\pi/2+(k+1)^{-2}<\varphi < -\pi/2+k^{-2}\}
$$
We also put
\beq \label{bbH0}
    \bbH_0=\{(r,\varphi)\colon -\pi/2+k_0^{-2}<\varphi < \pi/2-k_0^{-2}\}
\eeq
Slightly abusing notation, we will also denote by $\bbH_{\pm k}$ the
preimages $\pi_0^{-1} (\bbH_{\pm k}) \subset \Omega$ and call them
\index{Homogeneity sections} {\em homogeneity sections}. A
\index{u-curves (unstable curves)} u-curve $W\subset \Omega$ is said
to be {\em weakly homogeneous} if $W$ belongs to one section
$\bbH_k$ for some $|k| \geq k_0$ or for $k=0$.

Let $W\subset\bbH_k$ be a weakly homogeneous \index{u-curves
(unstable curves)} u-curve, $x=(r,\varphi)\in W$, and $|\Delta
\varphi|$ be the projection of $W$ onto the $\varphi$ axis. Due to
(\ref{CCu}), we have
\beq
        \label{cos23}
   |W|\leq\Const\, |\Delta\varphi| \leq\Const\, (|k|+1)^{-3}
   \leq\Const\, \cos^{3/2}\varphi.
\eeq

Now let $W_0 \subset \Omega$ be a \index{u-curves (unstable curves)}
u-curve on which $\cF^n$ is smooth, and assume that the u-curve
$W_i=\cF^i(W_0)$ is weakly homogeneous for every $i=0,1,\dots,n$.
Consider two points $x_0, x_0' \in W_0$ and put $x_i= \cF^i (x_0)$
and $x_i' = \cF^i (x_0')$ for $1\leq i\leq n$. We denote by $r_i$,
$\varphi_i$, $\cK_i$, $\cB_i$, etc.\ the corresponding quantities,
as introduced in Section~\ref{subsecSUC}, at the point $x_i$, and by
$r_i'$, $\varphi_i'$, $\cK_i'$, $\cB_i'$, etc.\ similar quantities
at the point $x_i'$.

For any curve $W$ we denote by $W(x,x')$ the segment of $W$
between the points $x,x' \in W$ and by $\measuredangle (x,x')_{W}$
the angle between the tangent vectors to the curve $W$ at $x$ and
$x'$.

\begin{proposition}[Distortion bounds]
 \index{Distortion bounds}
Under the above assumptions, if the following bound holds for
$i=0$ with some $C_0 =c >0$, then it holds for all $i=1,\dots,n-1$
with some $C_i = C>c$ (i.e., $C_i$ is independent of $i$ and $n$)
$$
    \left | \ln \frac{\cJ_{W_i}\cF(x_{i})}
     {\cJ_{W_i}\cF(x_{i}')} \right |
     \leq C_i\, \frac{|W_{i+1}(x_{i+1},x_{i+1}')|}{|W_{i+1}|^{2/3}}
$$
Moreover, in this case
$$
    \left | \ln \frac{\cJ_{W_0}\cF^n(x_0)}
     {\cJ_{W_0}\cF^n(x_0')} \right |
     \leq C\, \frac{|W_{n}(x_{n},x_{n}')|}{|W_{n}|^{2/3}}
$$
\label{prdist}
\end{proposition}

\begin{proposition}[Curvature bounds]
Under the above assumptions, if the following bound holds for
$i=0$ with some $C_0 =c >0$, then it holds for all $i=1,\dots,n$
with some $C_i = C>c$ (independent of $i$ and $n$)
$$
    \measuredangle(x_i,x_i')_{W_i} \leq
    C_i\, \frac{|W_{i}(x_{i},x_{i}')|}{|W_{i}|^{2/3}}
$$
\label{prcurv}
\end{proposition}

The proofs of these two propositions are quite lengthy. They are
given in Appendix~C. It is also shown there that sufficiently smooth
unstable curves satisfy distortion \index{Distortion bounds} bounds
for a large enough $c>0$. From now on we fix a sufficiently large
$c>0$ and the corresponding (perhaps even larger) $C>0$ that
guarantee the abundance of curves satisfying distortion and
curvature bounds (this is a standard procedure in the study of
chaotic billiards, see e.g.\ \cite{C3}).

We now consider an arbitrary \index{u-curves (unstable curves)}
u-curve $W \subset \Omega$ and partition it and its images under
$\cF^n$, $n\geq 1$, into weakly homogeneous u-curves (called
H-components) as follows:

\medskip\noindent{\bf Definition (H-components)}.
Given a \index{u-curves (unstable curves)} u-curve $W \subset
\Omega$, we call nonempty sets $W \cap \bbH_k$ (for $k=0$ and $|k|
\geq k_0$) the \emph{H-components of} \index{H-components} $W$. Note
that $W$ intersects each hyperplane $\{\varphi = \pm (\pi/2 -
k^{-2})\}$ separating homogeneity \index{Homogeneity sections}
sections at most once, due to (\ref{CCu}), hence each H-component is
a weakly homogeneous u-curve. Next suppose, inductively, that the
H-components $W_{n,j}$, $j\geq 1$, of $\cF^n(W)$ are constructed.
Then the H-components of $\cF^{n+1}(W)$ are defined to be the
H-components of the u-curves $\cF(W_{n,j})$ for all $j\geq 1$.
\medskip

Observe that the H-components of $\cF^n(W)$ are obtained naturally
\index{H-components} if we pretend that the boundaries of the
homogeneity \index{Homogeneity sections} sections act as additional
singularities of the dynamics.

Next, observe that if the curve $W_0$ satisfies the distortion
\index{Distortion bounds} bound and the curvature bound for $i=0$,
then so does any part of it (because $|W_0|$ and $|W_1|$ decrease if
we reduce the size of the curve, thus the bounds in
Propositions~\ref{prdist} and \ref{prcurv} remain valid). Therefore,
if a weakly homogeneous \index{u-curves (unstable curves)} u-curve
$W_0$ satisfies the distortion bound and curvature bound for $i=0$,
then every H-component of its image $\cF^n(W)$, $n\geq 1$ satisfies
these bounds as well. This allows us to restrict our studies to
weakly homogeneous u-curves that satisfy the distortion and
curvature bounds:

\medskip\noindent{\bf Definition (\index{H-curves} H-curves)}. A weakly homogeneous
\index{u-curves (unstable curves)} u-curve $W_0$ is said to be {\em
homogeneous} (or an {\em H-curve}, for brevity) if it satisfies the
above distortion \index{Distortion bounds} bound and curvature
bound.
\medskip

We note that, in the notation of Proposition~\ref{prdist},
$$
  |W_{n}(x_{n},x_{n}')| /|W_{n}|^{2/3}\leq |W_n|^{1/3}\leq\,\Const,
$$
hence the distortions of \index{H-curves} H-curves under the maps
$\cF^n$, $n\geq 1$, are uniformly bounded, in particular, for some
constant $\tbeta>0$
\beq
       e^{-\tbeta}\,\frac{|W_{0}(x_{0},x_{0}')|}{|\cF^{-n}(W_{n})|}
      \leq \frac{|W_{n}(x_{n},x_{n}')|}{|W_{n}|}
      \leq e^{\tbeta}\,
      \frac{|W_{0}(x_{0},x_{0}')|}{|\cF^{-n}(W_{n})|}.
         \label{distrough}
\eeq
Moreover, Proposition~\ref{prcurv} implies
\beq
      \measuredangle(x,x')_W
      \leq e^{\tbeta} \qquad \forall x,x'\in W.
         \label{curvstrong}
\eeq

Now consider an \index{H-curves} H-curve $W_0$ such that
$W_i=\cF^i(W_0)$ is an H-curve for every $i=1,\dots,n$. Let $\mes_0$
be an absolutely continuous measure on $W_0$ with some density
$\rho_0$ with respect to the measure induced by the $|\cdot|$-norm.
Then $\mes_i=\cF^i(\mes_0)$ is a measure on the curve $W_i$ with
some density $\rho_i$ for each $i=1,\dots,n$. As an immediate
consequence of Proposition~\ref{prdist} and (\ref{distrough}), we
have

\begin{corollary}[Density bounds]
Under the above assumptions, if the following bound holds for
$i=0$ with some $C_0=c>0$, it holds for all $i=1,\dots,n$ with
some $C_i = C>c$ (independent of $i$ and $n$)
\beq
    \left | \ln \frac{\rho_i(x_i)}
     {\rho_i(x_i')} \right |
     \leq C_i\,\frac{|W_{i}(x_{i},x_{i}')|}{|W_{i}|^{2/3}}
        \label{densbound}
\eeq
\label{crrho}
\end{corollary}

Observe that if the density bound holds for $i=0$ on the curve
$W_0$, then it holds on any part of it (because $|W_0|$ decreases if
we reduce the size of the curves, so the bound (\ref{densbound})
remains valid). Therefore, if $W_0$ is an arbitrary \index{H-curves}
H-curve with a density $\rho_0$, then the map $\cF^n$, $n\geq 1$,
induces densities on the H-components of $\cF^n(W_0)$
\index{H-components} that satisfy the above density bound. Hence we
can restrict our studies to densities satisfying (\ref{densbound}):

\medskip\noindent{\bf Definition}. Given an \index{H-curves} H-curve $W_0$,
we say that $\rho_0$ is a {\em homogeneous density} if it
satisfies (\ref{densbound}).
\medskip

Note that (\ref{densbound}) remains valid whether we normalize the
corresponding densities or not. Also, because $|W(x,x')| /
|W|^{2/3} < |W|^{1/3} < \Const$, we have a uniform bound
\beq
      e^{-\tbeta}\leq\frac{\rho(x)}{\rho(x')}
      \leq  e^{\tbeta}\qquad \forall x,x'\in W,
          \label{rho1rho2}
\eeq
where $\tbeta=C_0\max_W |W|^{1/3}$.

We will require \index{H-curves} H-curves to have length shorter
than a small constant $\tdelta>0$ (to achieve this, large H-curves
can be always partitioned into H-curves of length between
$\tdelta/2$ and $\tdelta$), so that $\tbeta$ in (\ref{distrough}),
(\ref{rho1rho2}) and (\ref{curvstrong}) is small enough. This will
make our H-curves almost straight lines, the map $\cF$ on them will
be almost linear, and homogeneous densities will be almost constant.

\subsection{Standard pairs}
\label{subsecSP} Now we formally define standard pairs mentioned
earlier in Section~\ref{subsecMS}: \index{Standard pair}

\medskip\noindent{\bf Definition (Standard pairs)}.
A standard pair $\ell=(\gamma,\rho)$ is an \index{H-curves} H-curve
$\gamma\subset\Omega$ with a homogeneous probability density $\rho$
on it. We denote by $\mes=\mes_{\ell}$ the measure on $\gamma$ with
density $\rho$.
\medskip

Our previous results imply the following invariance of the class
of standard pairs:

\begin{proposition}
Let $\ell=(\gamma, \rho)$ be a standard pair. \index{Standard pair}
Then for each $n \geq 0$, we have $\cF^n(\gamma) = \cup_i
\gamma_{i,n}$ and $\cF^n(\mes_\ell) = \sum_i c_i \mes_{\ell_{i,n}}$
where $\sum_i c_i = 1$ and $\ell_{i,n} = (\gamma_{i,n}, \rho_{i,n})$
are standard pairs. The curves $\gamma_{i,n}$ are the H-components
\index{H-components} of $\cF^n(\gamma)$. Furthermore, any subcurve
$\gamma' \subset \gamma_{i,n}$ with the density $\rho_{i,n}$
restricted to it, is a standard pair. \label{prspair}
\end{proposition}

Recall that $\cF$ expands \index{H-curves} H-curves by a factor
$\geq\vartheta^{-1}>1$, which is a local property. It is also
important to show that \index{H-curves} H-curves grow in a global
sense, i.e.\ \index{H-components} given a small \index{H-curves}
H-curve $\gamma$, the sizes of the H-components of $\cF^n(\gamma)$
tend to grow exponentially in time until they become of order one,
on the average (we will make this statement precise). Such
statements are usually referred to as ``growth lemmas''
 \index{Growth lemma} \cite{BSC2,Y,C2,C3}, and we prove one below.

Let $\ell = (\gamma, \rho)$ be a standard pair and for $n\geq 1$ and
$x \in \gamma$, let $r_n(x)$ denote the distance from the point
$\cF^n (x)$ to the nearest endpoint of the H-component
\index{H-components} $\gamma_n(x) \subset \cF^n (\gamma)$ that
contains $\cF^n (x)$.

\begin{lemma}[``Growth lemma'']
\label{prgrow} \index{Growth lemma}
If $k_0$ in (\ref{bbH0}) is sufficiently large, then \\
{\rm (a)} There are constants $\beta_1 \in (0,1)$ and $\beta_2>0$,
such that for any $\varepsilon >0$
\beq
   \mes_\ell (x\colon\ r_n(x) < \varepsilon) \leq
   (\beta_1/\vartheta)^n\, \mes_\ell
   (x\colon r_0 < \varepsilon\vartheta^n)
   + \beta_2 \varepsilon
     \label{SHUgrown}
\eeq
{\rm (b)} There are constants $\beta_3,\beta_4>0$, such that if $n
\geq \beta_3\big| \ln |\gamma|\big |$, then for any
$\varepsilon>0$ we have $\mes_\ell (x\colon\ r_n(x) < \varepsilon)
\leq \beta_4 \varepsilon$.\\ {\rm (c)} There are constants
$\beta_5,\beta_6>0$, a small $\varepsilon_0 >0$, and $q\in (0,1)$
such that for any $n_2> n_1 > \beta_5\big| \ln| \gamma|\big |$ we
have
$$
   \mes_\ell \Bigl(x\colon\ \max_{n_1 < i < n_2}
   r_i(x) < \varepsilon_0 \Bigr) \leq \beta_6 q^{n_2-n_1}
$$
All these estimates are uniform in $\ell = (\gamma, \rho)$.
\end{lemma}

\proof The proof of (a) follows the lines of the arguments in
\cite{C2} and consists of three steps.

First, let $\gamma$ be an \index{H-curves} H-curve and
$\gamma_{i,1}$ all the \index{H-components} H-components of
$\cF(\gamma)$. For each $i$, denote by $\vartheta_{i,1}^{-1}$ the
minimal (local) factor of expansion of the curve
$\cF^{-1}(\gamma_{i,1})$ under the map $\cF$. We claim that
\beq
   \btheta_1 \colon = \lim_{\delta\to 0}\
   \sup_{\gamma\colon |\gamma|<\delta} \sum_i \vartheta_{i,1} <1
      \label{alpha01}
\eeq
(we call this a \emph{one-step expansion estimate} for the map
$\cF$).

To prove (\ref{alpha01}), we observe that a small \index{H-curves}
H-curve $\gamma$ may be cut into several pieces by the singularities
of $\cF$, which are made by grazing (tangential) collisions with the
scatterers \index{Scatterer} and the disk $\cP(Q)$. At each of them,
$\gamma$ is sliced into two parts -- one hits the scatterer (or the
disk) and gets reflected (almost tangentially) and the other misses
the collision (passes by). The reflecting part is further subdivided
into countably many H-components by the boundaries of the
\index{H-components} homogeneity \index{Homogeneity sections}
sections $\bbH_k$. Note that the reflecting part of $\gamma$ lies
entirely in $\cup_{k\geq k_0} \bbH_{\pm k}$ provided $|\gamma|$ is
small enough, which is guaranteed by taking $\liminf_{\delta\to 0}$.

Let $\gamma'$ be an H-component of $\cF(\gamma)$ falling into a
section $\bbH_{k}$ with some $|k| \geq k_0$. Since $\cos\varphi \sim
k^{-2}$ on $\gamma'$, the expansion factor of the preimage $\cF^{-1}
(\gamma')$ under the map $\cF$ is $\geq c k^2$ for some constant
$c>0$, due to Proposition~\ref{prJexp}. Thus all these
\index{H-components} H-components make a total contribution to
(\ref{alpha01}) less than $\sum_{k\geq k_0}
(ck^2)^{-1}\leq\Const/k_0$.

The part of $\gamma$ passing by without collision may be sliced
again by a grazing collision with another scatterer
\index{Scatterer} later on, thus creating another countable set of
reflecting H-components. This \index{H-components} can happen at
most $L_{\max}/L_{\min}$ times, where $L_{\min}$ is the minimal free
path of the light particle, guaranteed by (\ref{Upsilon}), and
$L_{\max}$ is the maximal free path of the light particle
($L_{\max}<\infty$ due to our finite horizon \index{Finite horizon}
assumption).

In the end, we will have $\leq L_{\max}/L_{\min}$ countable sets
of H-components resulting from almost grazing collisions and at
most one component of $\gamma$ that misses all the grazing
collisions and lands somewhere else on $\dcD \cup \dcP(Q)$. That
last component is only guaranteed to expand by a moderate factor
of $\vartheta^{-1}$. Thus, we arrive at
\beq  \label{bthetabound}
     \btheta_1 \leq \vartheta + \frac{L_{\max}}{L_{\min}}\,
     \frac{\Const}{k_0}
\eeq
Since $\vartheta<1$, the required condition $\btheta_1<1$ can be
ensured by choosing $k_0$ large enough. This completes the proof
of the one-step expansion estimate (\ref{alpha01}).

The second step in the proof of Lemma~\ref{prgrow} (a) is the
verification of (\ref{SHUgrown}) for $n=1$:
\beq
   \mes_\ell (x\colon\ r_1(x) < \varepsilon) \leq
   (\beta_1/\vartheta)\, \mes_\ell
   (x\colon r_0 < \varepsilon\vartheta)
   + \beta_2 \varepsilon.
     \label{SHUgrow1}
\eeq
We assume that $|\gamma| < \tdelta$, where $\tdelta$ is chosen so
that
$$
   \tilde{\btheta}_1 \colon =
   \sup_{\gamma\colon |\gamma|<\tdelta}
   \sum_i \vartheta_{i,1} <
   \bigl(1+\btheta_1 \bigr)/2 < 1.
$$
Now, for each H-component $\gamma_{i,1}$ of $\cF(\gamma)$, the set
\index{H-components} $\gamma_{i,1} \cap
\cF\bigl(\{r_1(x)<\varepsilon\}\bigr)$ is the union of two
subintervals of $\gamma_{i,1}$ of length $\varepsilon$ adjacent to
the endpoints of $\gamma_{i,1}$. Then the set
$\cF^{-1}(\gamma_{i,1}) \cap \{r_1(x)<\varepsilon\}$ is a subset of
the union of two subintervals of $\cF^{-1} (\gamma_{i,1})$ of length
$\vartheta_{i,1} \varepsilon$, therefore,
\begin{align}
   \mes_\ell(r_1(x)<\varepsilon)
     &\leq |\gamma|^{-1} e^{\tbeta}
     \sum_i 2 \varepsilon \vartheta_{i,1}
     \nonumber\\
     &\leq  2 \varepsilon |\gamma|^{-1}
     e^{\tbeta} \tilde{\btheta}_1
\end{align}
where the factor $e^{\tbeta}$ accounts for possible fluctuations
of the density $\rho(x)$ on $\gamma$, see (\ref{rho1rho2}). We can
make $\tbeta>0$ arbitrarily small by decreasing $\tdelta$, if
necessary, and guarantee that
$$
    \beta_1 \colon=e^{2\tbeta} \tilde{\btheta}_1 <1
    \qquad\text{and}\qquad
    \beta_1/\vartheta > 1
$$
(recall that the first bound is required by Lemma~\ref{prgrow};
the second one can be easily ensured because
$\vartheta<\tilde{\btheta}_1<1$). Now
the first term on the right hand side of (\ref{SHUgrow1}) is
bounded below by
\begin{align*}
   (\beta_1/\vartheta)\, \mes_\ell
   (x\colon r_0 < \varepsilon\vartheta)
   &\geq
   (\beta_1/\vartheta)\, \min\{1,2\varepsilon
   \vartheta |\gamma|^{-1} e^{-\tbeta} \} \\
   &=
   \min\{\beta_1/\vartheta,2
  \varepsilon |\gamma|^{-1} e^{\tbeta} \tilde{\btheta}_1 \}
\end{align*}
Since $\beta_1/\vartheta>1$, we obtain
\beq
   \mes_\ell (x\colon\ r_1(x) < \varepsilon) \leq
   (\beta_1/\vartheta)\, \mes_\ell
   (x\colon r_0 < \varepsilon\vartheta).
     \label{SHUgrow1a}
\eeq
This bound appears even better than (\ref{SHUgrow1}), but remember
it is only proved under the assumption $|\gamma| < \tdelta$. To make
this assumption valid, we require all our \index{H-curves} H-curves
to have length shorter than $\tdelta$, as already mentioned in the
end of the previous section. Accordingly, we have to partition the
\index{H-components} H-components of $\cF(\gamma)$ into pieces that
are shorter than $\tdelta$; this will enlarge the set $\{r_1(x) <
\varepsilon\}$ and result in the additional term
$\beta_2\varepsilon$ in (\ref{SHUgrow1}).

More precisely, let us divide each H-component $\gamma_{i,1}$ of
$\cF(\gamma)$ with length $> \tdelta$ into $k_i$ equal
subintervals of length between $\tdelta/2$ and $\tdelta$, with
$k_i \leq 2|\gamma_{i,1}| / \tdelta$. If $|\gamma_{i,1}| \leq
\tdelta$, then we set $k_i =0$ and leave $\gamma_{i,1}$ unchanged.
Then the union of the preimages of the $\varepsilon$-neighborhoods
of the new partition points has measure bounded above by
$$
   \leq 3\varepsilon|\gamma|^{-1} \sum_i k_i \vartheta_{i,1}
   \leq 6\varepsilon \tdelta^{-1} |\gamma|^{-1}
   \sum_i |\gamma_{i,1}| \vartheta_{i,1}
   \leq 7\varepsilon \tdelta^{-1},
$$
where we increased the numerical coefficient from 6 to 7 in order to
incorporate the factor $e^{\tbeta}$ resulting from the distortion
\index{Distortion bounds} bounds (\ref{distrough}). This completes
the proof of (\ref{SHUgrow1}) with $\beta_2 = 7\tdelta^{-1}$.

Lastly, the proof of (\ref{SHUgrown}) for $n>1$ goes by induction
on $n$. Assume that
\begin{align*}
   \mes_\ell \bigl(x\colon\ r_n(x) < \varepsilon\bigr) &\leq
   (\beta_1/\vartheta)^n\, \mes_\ell
   \bigl(x\colon r_0(x) < \varepsilon\vartheta^n\bigr) \\
   &\quad + 7\tdelta^{-1}\bigl(1+\beta_1+\cdots +\beta_1^{n-1}\bigr)
   \, \varepsilon .
\end{align*}
Then we apply (\ref{SHUgrow1}) with $\beta_2 = 7\tdelta^{-1}$ to
\index{H-components} each H-component of $\cF^n (\gamma)$ and obtain
\begin{align*}
   \mes_\ell \bigl(x\colon\ r_{n+1}(x) < \varepsilon\bigr) &\leq
   (\beta_1/\vartheta)\, \mes_\ell
   \bigl(x\colon r_n(x) < \varepsilon\vartheta\bigr)
   + 7\tdelta^{-1}\varepsilon \\
   &\leq (\beta_1/\vartheta)^{n+1}\, \mes_\ell
   \bigl(x\colon r_0(x) < \varepsilon\vartheta^{n+1}\bigr)\\
   &\quad + 7\tdelta^{-1}\bigl(1+\beta_1+\cdots +\beta_1^{n}\bigr)
   \, \varepsilon,
\end{align*}
which completes the induction step. Thus we get (\ref{SHUgrown})
for all $n\geq 1$ with $\beta_2 = 7\tdelta^{-1} / (1-\beta_1)$.

Part (b) of Lemma~\ref{prgrow} directly follows from (a). Indeed,
it suffices to set $\beta_3= 1/\min\{|\ln\beta_1|
,|\ln\vartheta|\}$, so that $\vartheta^n<|\gamma|$ and
$\beta_1^n<|\gamma|$, and notice that $\mes_\ell (x\colon r_0 <
\varepsilon\vartheta^n)<2 e^{\tbeta}\varepsilon\vartheta^n/
|\gamma|$ due to (\ref{rho1rho2}).

The proof of (c) requires a tedious bookkeeping of various short
H-components of the images of $\gamma$. Pick \index{H-components}
$\varepsilon_0<(1+\beta_4)^{-1}$ and divide the time interval
$[n_1,n_2]$ into segments of length $s\colon= [2\beta_3
|\ln\varepsilon_0|]$. We will estimate the measure of the set
$$
    \tilde{\gamma}=\Bigl\{x\in\gamma\colon \max_{1\leq i\leq K}
    r_{n_1+si}(x) <\varepsilon_0\Bigr\}
$$
where $K=(n_2-n_1)/s$. For each $x\in\tilde{\gamma}$ define a
sequence of natural numbers $S(x)=\{k_0,k_1,\dots,k_m\}$, with
$m=m(x)\leq K$, inductively. Set $k_0=1$ and given $k_0,\dots,k_i$
we put $t_i=k_0+\dots+k_i$ and consider the H-component
\index{H-components} $\gamma_i(x)$ of $\cF^{n_1+st_i}(\gamma)$ that
contains $\cF^{n_1+st_i}(x)$. We set $k_{i+1}=k$ if
$|\gamma_i(x)|\in [\varepsilon_0^{2k},\varepsilon_0^{2k-2})$. If it
happens that $t_i+k_{i+1}>K$, we reset $k_{i+1}=K-t_i+1$ and put
$m(x)=i+1$. Note that now $k_1+\dots+k_m=K$.

Next pick a sequence $S=\{k_0=1,k_1,\dots,k_m\}$ of natural
numbers such that $k_1+\dots +k_m=K$ and let $\tilde{\gamma}_S
=\{x\in\tilde{\gamma}\colon S(x)=S\}$. We claim that
\beq
   \mes_\ell (\tilde{\gamma}_S)\leq \beta_4^m\varepsilon_0^K
     \label{claimK}
\eeq
First, by part (b)
\beq
   \mes_\ell \bigl(k_1(x)=k\bigr)
   \leq \beta_4\varepsilon_0^k,\qquad k\geq 1
     \label{mesell1}
\eeq
(for $k\geq 2$ we actually have a better estimate $\mes_\ell
(k_1(x)=k) \leq \beta_4 \varepsilon_0^{2k-2}$). Then, inductively,
for each $i=0,\dots,m-2$ we use our previous notation $t_i$ and
$\gamma_i(x)$ and put $\tilde{\gamma}_i(x)=\gamma_i(x)$ if
$|\gamma_i(x)|<2\varepsilon_0$, otherwise we denote by
$\tilde{\gamma}_i(x) \subset \gamma_i(x)$ the
$\varepsilon_0$-neighborhood of an endpoint of $\gamma_i(x)$ that
contains the point $\cF^{n_1+st_i}(x)$. By
Proposition~\ref{prspair}, the curve $\tilde{\gamma}_i(x)$, with the
corresponding conditional measure on it (induced by \index{Standard
pair} $\cF^{n_1+st_i}(\mes_{\ell})$), makes a standard pair, call it
$\ell_i(x)$. Then again by part (b)
\beq
   \mes_{\ell_i(x)}\bigl(k_{i+2}(x)=k\bigr)
   \leq\beta_4\varepsilon_0^k,\qquad k\geq 1
     \label{meselli}
\eeq
(because $k_{i+2}(x)=k$ implies $|\gamma_{i+1}(x)| <
\varepsilon_0^{2k-2}$, which is enough for $k \geq 2$, and for
$k=1$ we have $|\tilde{\gamma}_{i+1}(x)|<\varepsilon_0$).
Multiplying (\ref{mesell1}) and (\ref{meselli}) for all
$i=0,\dots,m-2$ proves (\ref{claimK}).

Now, adding (\ref{claimK}) over all possible sequences
$S=\{1,k_1,\dots,k_m\}$ gives
$$
   \mes_\ell (\tilde{\gamma})\leq \sum_{m=1}^K
   \binom{K-1}{m-1}\beta_4^m\varepsilon_0^K
        \leq (1+\beta_4)^K\varepsilon_0^K
$$
where $\binom{K-1}{m-1}$ denotes the binomial coefficients coming
from counting the number of respective sequences $\{S\}$. This
completes the proof of Lemma~\ref{prgrow}. \qed\medskip

\noindent{\em Proof of Proposition~\ref{PrDistEq1}} is now
obtained by combining Proposition~\ref{prspair} with
Lemma~\ref{prgrow}. \qed
\medskip

We conclude this subsection with a few remarks. Let
$\gamma=\cup_{\alpha} \gamma_{\alpha}\subset\Omega$ be a finite or
countable union of disjoint \index{H-curves} H-curves with some
smooth probability measure $\mes_{\gamma}$ on it, whose density of
each $\gamma_{\alpha}$ is homogeneous. For every $\alpha$ and
$x\in\gamma_{\alpha}$ and $n\geq 0$ denote by $r_n(x)$ the distance
from the point $\cF^n(x)$ to the nearest endpoint of the
\index{H-components} H-component of $\cF^n (\gamma_{\alpha})$ to
which the point $\cF^n(x)$ belongs. The following is an easy
consequence of Lemma~\ref{prgrow} (a) obtained by averaging over
$\alpha$:
$$
   \mes_{\gamma}\bigl(x\colon\ r_n(x) < \varepsilon\bigr) \leq
   (\beta_1/\vartheta)^n\,
   \mes_{\gamma}\bigl(x\colon r_0 < \varepsilon\vartheta^n\bigr)
   + \beta_2 \varepsilon.
$$
Also, there is a constant $\beta_7>0$ such that if $\mes_{\gamma}
(x\colon\ r_0(x) < \varepsilon) \leq \beta_7 \varepsilon$ for any
$\varepsilon>0$, then $\mes_{\gamma} (x\colon\ r_n(x) <
\varepsilon ) \leq \beta_7 \varepsilon$ for all $\varepsilon>0$
and $n\geq 1$ (it is enough to set $\beta_7 = \beta_2 /
(1-\beta_1e^{\tbeta})$).

In addition, suppose for each $\alpha$ we fix a subcurve
$\gamma_{\alpha}' \subset \gamma_{\alpha}$. Put $\gamma' =
\cup_{\alpha} \gamma_{\alpha}'$ and denote by $\mes_{\gamma'}$ the
measure $\mes_{\gamma}$ conditioned on $\gamma'$. For every $\alpha$
and $x\in\gamma_{\alpha}'$ and $n\geq 0$ denote by $r_n'(x)$ the
distance from the point $\cF^n(x)$ to the nearest endpoint of the
H-component of $\cF^n (\gamma_{\alpha}')$ to which
\index{H-components} the point $\cF^n(x)$ belongs. Distortion \index{Distortion bounds} bounds
(\ref{distrough}) then imply
$$
    \mes_{\gamma'}\bigl(x\colon\ r_n'(x) <\varepsilon\bigr)
    \leq e^{\tbeta}\,\bigl[\mes_{\gamma}(\gamma')\bigr]^{-1}\,
    \mes_{\gamma}\bigl(x\colon\ r_n(x) <\varepsilon\bigr)
%      \label{gamma'}
$$
Lastly, we note that taking the limit $M \to \infty$ automatically
extends all our results to the billiard map $\cF_Q \colon \Omega_Q
\to \Omega_Q$ for any $Q$ satisfying (\ref{Upsilon}).

\subsection{Perturbative analysis}
\label{subsecPA}

Recall that the billiard-type map
$\cF_{Q,V}\colon\Omega_{Q,V}\to\Omega_{Q,V}$ is essentially
independent of $V$ and can be identified with
$\cF_Q\colon\Omega_Q\to\Omega_Q$ via $\pi_{Q,0}\circ\cF_{Q,V}
=\cF_Q\circ\pi_{Q,0}$. Furthermore, the spaces $\Omega_{Q}$ are
identified with the $r,\varphi$ coordinate space $\Omega_0$ by the
projection $\pi_0$. This gives us a family of billiard maps $\cF_Q$
acting on the same space $\Omega_0$. They preserve the same billiard
measure $d\mu_0=c^{-1}\cos\varphi\, dr\, d\varphi$, where
$c=2\,\length(\dcD)+4\pi\br$ denotes the normalizing factor.

We recall a few standard facts from billiard theory
\cite{BSC1,BSC2,C2,Y}. In the $r,\varphi$ coordinates,
\index{u-curves (unstable curves)} u-curves are increasing and
s-curves are decreasing, see (\ref{CCu}) and (\ref{CCs}), and they
are uniformly transversal to each other. For every $Q$ and integer
$m$, the map $\cF_Q^m$ is discontinuous on finite union of curves in
$\Omega_0$, which are stable for $m>0$ and unstable for $m<0$. The
discontinuity curves of $\cF_Q^m$ stretch continuously across
$\Omega_0$ between the two borders of $\Omega_0$, i.e.\ from
$\varphi=-\pi/2$ to $\varphi=\pi/2$ (they intersect each other, of
course).

We use the $|\cdot|$-norm, see Section~\ref{subsecSUC}, to measure
the lengths of stable and unstable curves. For a u-curve
$W\subset\Omega_0$ and a point $x\in\Omega_0$ we define dist$(x,W)$
to be the minimal length of s-curves connecting $x$ with $W$, and
vice versa (if there is no such connecting curve, we set
dist$(x,W)=\infty$). We define the ``Hausdorff distance'' between
two \index{u-curves (unstable curves)} u-curves
$W_1,W_2\subset\Omega_0$ to be
$$
     \dist(W_1,W_2)=\max\Bigl\{\sup_{x\in W_1}\dist(x,W_2),
     \sup_{y\in W_2}\dist(y,W_1)\Bigr\}
$$
(and similarly for s-curves). Let $W\subset\Omega_0$ be a stable or
unstable curve with endpoints $x_1$ and $x_2$, and
$\varepsilon<|W|/2$. For any two points $y_1,y_2\in W$ such that
$|W(x_i,y_i)|<\varepsilon$ for $i=1,2$, we call the middle segment
$W(y_1,y_2)$ an $\varepsilon$-{\em reduction} of $W$.

Next we show that, in a certain crude sense, the map $\cF_Q$
depends Lipschitz continuously on $Q$. Let $Q,Q'$ satisfy
(\ref{Upsilon}) and $\varepsilon =\|Q - Q'\|$. The following lemma
is a simple geometric observation:

\begin{lemma}
There are constants $c_2>c_1>1$ such that
\begin{itemize} \item[\rm (a)] The discontinuity curves of the map
$\cF_{Q'}$ in $\Omega_0$ are within the $c_1 \varepsilon$-distance
of those of the map $\cF_Q$. \item[\rm (b)] Let $W \subset\Omega_0$
be a \index{u-curves (unstable curves)} u-curve of length
$>2c_2\varepsilon$ such that $\cF_Q$ and $\cF_{Q'}$ are smooth on
$W$. Then there are two $c_2\varepsilon$-reductions of this curve,
$\tW$ and $\tW'$, such that ${\rm dist}(\cF_Q (\tW),\cF_{Q'}
(\tW'))<c_1\varepsilon$.
\end{itemize}
\label{lmclose0}
\end{lemma}

\begin{corollary}
There is a constant $c_3>c_2$ such that for any \index{H-curves}
H-curve $W\subset\Omega_0$ of length $>c_3\varepsilon$ there are two
finite partitions $W=\cup_{i=0}^I W_i =\cup_{i=0}^I W_i'$ of $W$
such that
\begin{itemize} \item[\rm (a)] $|W_0|<c_3\varepsilon$ and
$|W_0'|<c_3\varepsilon$ \item [\rm (b)] For each $i=1,\dots,I$, the
sets $\cF_Q(W_i)$ and $\cF_{Q'}(W_i')$ are \index{H-curves} H-curves
such that ${\rm dist} (\cF_Q(W_i), \cF_{Q'}(W_i'))< c_1\varepsilon$.
\end{itemize}
\label{crclose1}
\end{corollary}

\noindent{\em Proof}. The singularities of $\cF_Q$ divide $W$ into
$\leq K_1=1+L_{\max}/L_{\min}$ pieces, see the proof of
Lemma~\ref{prgrow} (a), and so do the singularities of $\cF_{Q'}$.
Removing pieces shorter than $2c_2\varepsilon$ and using
Lemma~\ref{lmclose0} gives us two partitions $W=\cup_{j=0}^J\hW_j
=\cup_{j=0}^J\hW_j'$ such that
\begin{itemize} \item[(a)] $|\hW_0|<2K_1c_2\varepsilon$ and
$|\hW_0'|<2K_1c_2\varepsilon$ \item [(b)] For each $j=1,\dots,J$,
the sets $\cF_Q(\hW_j)$ and $\cF_{Q'}(\hW_j')$ are \index{u-curves
(unstable curves)} u-curves such that ${\rm dist} (\cF_Q(\hW_j)
,\cF_{Q'}(\hW_j'))< c_1\varepsilon$.
\end{itemize}
Next, for any homogeneity strip $\bbH_k$ and $j\geq 1$, consider the
\index{Homogeneity strips} \index{H-curves} H-curves $\hW_{jk}=\cF_Q
(\hW_j)\cap \bbH_k$ and $\hW_{jk}' =\cF_{Q'} (\hW_j')\cap \bbH_k$.
It is easy to see that some $Cc_1\varepsilon$-reductions of these
curves, call them $W_{jk}$ and $W_{jk}'$, respectively, are
$c_1\varepsilon$-close to each other in the Hausdorff metric (here
$C>0$ is the bound on the slopes of u-curves and s-curves in
$\Omega_0$). Then we take the nonempty curves $\cF_Q^{-1}(W_{jk})$
and $\cF_{Q'}^{-1}(W_{jk}')$ for all $j$ and $k$ and relabel them to
define the elements of our partitions $W_i$ and $W_i'$,
respectively. Using the notation of Lemma~\ref{prgrow} we have
\begin{align*}
    \bigl| W\setminus\cup_i W_i\bigr| & \leq
   2K_1c_2\varepsilon + \bigl|\{x\in W\colon r_1(x)<Cc_1\varepsilon\}\bigr| \\
   & \leq  2K_1c_2\varepsilon + \beta_1\vartheta^{-1}
   \bigl|\{x\colon r_0(x)<Cc_1\vartheta\varepsilon\}\bigr|
   + \beta_2\varepsilon|W|
\end{align*}
(we apply the estimate in Lemma~\ref{prgrow} (a) to $\cF_Q$). The
resulting bound clearly does not exceed $c_3\varepsilon$ for some
$c_3>0$. A similar bound holds for $\left | W\setminus\cup
W_i'\right |$. \qed

\begin{corollary}
There is a constant $c_4>1$ such that for each integer $m$ the
discontinuity sets of the maps $\cF_{Q}^m$ and $\cF_{Q'}^m$ are
$c_4 \varepsilon$-close to each other in the Hausdorff metric.
\label{crclose2}
\end{corollary}

\proof Let $c_4=10c_2/(1-\vartheta)$. We prove this for $m>0$ (the
case $m<0$ follows by time reversal) using induction on $m$. For
$m=1$ the statement follows from Lemma~\ref{lmclose0} (a). Assume
that it holds for some $m\geq 1$. Now, if the statement fails for
$m+1$, then there is a point $x\in\Omega_0$ at which
$\cF_{Q'}^{m+1}$ is discontinuous and which lies in the middle of a
\index{u-curves (unstable curves)} u-curve $W$,
$|W|=2c_4\varepsilon$, on which $\cF_Q^{m+1}$ is smooth. We can
assume that $\cF_{Q'}$ is smooth on a $c_1\varepsilon$-reduction
$\hW$ of $W$, too, otherwise we apply Lemma~\ref{lmclose0} (a). Now
by Lemma~\ref{lmclose0} (b), there are $c_2\varepsilon$-reductions
of $\hW$, call them $\tW$ and $\tW'$, such that ${\rm
dist}(\cF_Q(\tW), \cF_{Q'}(\tW')) < c_1 \varepsilon$. Due to our
choice of $\varepsilon_4$ and the expansion of u-curves by a factor
$\geq\vartheta^{-1}$, the point $\cF_{Q'}(x)$ divides
$\cF_{Q'}(\tW')$ into two \index{u-curves (unstable curves)}
u-curves of length $>c_4\varepsilon+5c_2\varepsilon$ each. Since
$\cF_{Q'}^m$ is discontinuous at $\cF_{Q'}(x)$, our inductive
assumption implies that $\cF_Q^m$ is discontinuous on $\cF_Q(\tW)$,
a contradiction. \qed
\medskip

Let $Q\in \cD$ satisfy (\ref{Upsilon}) and $x\in\Omega$. Denote
\beq
   \varepsilon_n(x,Q) = \max_{0\leq i\leq n}
     \| Q - Q(\cF^i x) \| + 1/M
      \label{varepsn}
\eeq
where $Q(y)$ denotes the $Q$-coordinate of a point $y \in \Omega$.
For a \index{u-curves (unstable curves)} u-curve $W \subset \Omega$
we put
$$
  \varepsilon_n(W,Q) = \sup_{x\in W}\varepsilon_n (x,Q)
$$
Recall that the $Q$ coordinate varies by $<\Const/M$ on
\index{u-curves (unstable curves)} u-curves, so that the map $\cF^n$
acts on a \index{u-curves (unstable curves)} u-curve $W \subset
\Omega$ similarly to the action of $\cF_Q^n$ on its projection
$\pi_Q(W)$, if $\varepsilon_n (W,Q)$ is small.

Recall that we consider initial conditions satisfying (\ref{Away}).
In the lemmas below we require $n\leq \delta_\diamond \sqrt{M}$ to
prevent collision of the heavy particle with the walls, see
Section~\ref{subsecMS}.

The following two lemmas are
close analogies of the previous results (with, possibly, different
values of the constants $c_1,\dots,c_4$), and can be proved by
similar arguments, so we omit details:

\begin{lemma}
There exist constants $c_1,$ $c_3$ such that for any
\index{H-curves} H-curve $W\subset\Omega$ of length
$>c_3\varepsilon$ there are two finite partitions $W=\cup_{i=0}^I
W_i =\cup_{i=0}^I W_i'$ of $W$ such that
\begin{itemize} \item[\rm (a)] $|W_0|<c_3\varepsilon$ and
$|W_0'|<c_3\varepsilon$ \item [\rm (b)] For each $i=1,\dots,I$, the
sets $\cF_{Q}(\pi_{0}(W_i))$ and $\pi_{0}(\cF(W_i'))$ are
\index{H-curves} H-curves such that $\dist(\cF_{Q}(\pi_{0}(W_i)),
\pi_{0}(\cF(W_i')) < c_1\varepsilon$.
\end{itemize}
Here $\varepsilon=\varepsilon_1(W,Q)$. \label{lmclose2}
\end{lemma}

\begin{lemma}
There is a constant $c_4$ such that
for any discontinuity point $x$ of the map $\cF^n$,
$1\leq n\leq \delta_\diamond \sqrt{M}$, its
projection $\pi_{0}(x)$ lies in the $c_4 \varepsilon$-neighborhood
of some discontinuity curve of the map $\cF_{Q}^n$, were
$\varepsilon=\varepsilon_n(x,Q)$. \label{lmclose1}
\end{lemma}

\begin{lemma}
For any $1\leq n \leq \delta_\diamond \sqrt{M}$
the singularity set $\cS_n\subset\Omega$ of the
map $\cF^n$ is a finite union of smooth compact manifolds of
codimension one with boundaries. For every
$Q,V\in\Upsilon_{\delta_1}$ the manifold $\cS_n$ intersects
$\Omega_{Q,V}$ transversally (in fact, almost orthogonally), and
$\cS_n\cap\Omega_{Q,V}$ is a finite union of s-curves.
\label{lmsingQV}
\end{lemma}

\proof The first claim follows from our finite horizon \index{Finite
horizon} assumption.

%Next, consider the set $\tcS\subset\cM$ of
%trajectories that experience singularities (grazing collisions)
%anytime in the future. They make a manifold of codimension one
%invariant under the backward dynamics, i.e.\
%$\Phi^{-t}\tcS\subset\tcS$ for $t>0$. Denote by $dx^{\ast}
%=(dQ^\ast, dV^\ast, dq^\ast, dv^\ast)$ a skew-orthogonal vector to
%$\tcS$ at a point $x = (Q,V,q,v) \in \tcS$, i.e.\ a vector such
%that
%$$
%  \la dQ^\ast,dQ\ra - \la dV^\ast,dV\ra +
%  \la dq^\ast,dq\ra - \la dv^\ast,dv\ra = 0
%$$
%for every tangent vector $(dQ, dV, dq, dv)\in\cT_x\tcS$. Since the
%flow $\Phi^t$ preserves the symplectic form, the vector
%$D\Phi^{-t}(dx^{\ast})$ will be skew-orthogonal to $\tcS$ at the
%point $\Phi^t(x)$ for all $t>0$. It is easy to check that at the
%moment immediately preceding the grazing collision, $dQ^{\ast}
%=dV^{\ast} =0$ and $\la dq^{\ast},dv^{\ast}\ra\leq
%-c_0\,\|dq^{\ast}\|\,\|dv^{\ast}\|$ for some $c_0>0$. Hence, by
%using the estimates of Section~\ref{subsecSUV} we obtain that
%$\|dQ^{\ast}\| =\cO(\|dq^{\ast}\|/M)$, $\|dV^{\ast}\|
%=\cO(\|dv^{\ast}\|/M)$ and $\la dq^{\ast},dv^{\ast}\ra\leq
%-c_0\,\|dq^{\ast}\|\,\|dv^{\ast}\|$ on the entire manifold $\tcS$.
%Thus, skew-orthogonal (hence, tangent) vectors to
%$\tcS\cap\Omega_{Q,V}$ are s-vectors, hence $\tcS\cap\Omega_{Q,V}$
%is a union of s-curves. \qed

Next, recall that $\cS_n=\bigcup_{j=0}^{n-1} \cF^{-j} \cS_0.$
Consider for example a component $\hcS\subset \cS_0$ corresponding
to a grazing collision between particles (other components can be
treated similarly). In the whole 8-dimensional phase space $\hcS$
is given by the equations
\begin{align}
  \|Q-q\|  &=\br & \text{(collision)} \label{ECol}\\
  \la Q-q, V-v\ra &=0   & \text{(tangency)}  \label{ETan}\\
  MV^2+v^2   &=1   & \text{(energy conservation)} \label{EECon}
\end{align}
Recall that $\cF$ preserves the restriction to $\Omega$ of the
symplectic form
\begin{align*}
   &\omega\bigl((dQ_1, dV_1, dq_1, dv_1),
   (dQ_2, dV_2, dq_2, dv_2)\bigr) \\
   &\quad =M\,\la dQ_1, dV_2\ra - M\,\la dQ_2, dV_1\ra
   +\la dq_1, dv_2\ra - \la dq_2, dv_1 \ra
\end{align*}
By (\ref{ECol})--(\ref{EECon}) the tangent space $\cT\hcS$ in the
whole 8-dimensional space is the skew-orthogonal complement of the
linear subspace spanning three vectors
$$
\begin{array}{rccccccl}
e_1 & = & (     & 0,             & \frac{q-Q}{M}, & 0,   & Q-q & ) \\
e_2 & = & \big( & \frac{Q-q}{M}, & \frac{v-V}{M}, & q-Q, & V-v & ) \\
e_3 & = & (     & V,             &  0 ,           & v ,  & 0   & ) \\
\end{array}
$$

Equivalently,  in $\cT\Omega,$ the subspace $\cT\hcS$ can be
described as the skew-orthogonal complement of
$$
   \text{span}(e_1, e_2, e_3)\bigcap \cT\Omega=\reals e_1,
$$
where $\reals e_1 = \{ce_1,\, c\in\reals\}$. Observe that $e_1$ is
tangent to $\Omega$ because
$$
  \omega(e_1, e_3)=\la Q-q, v-V\ra =0
$$
by (\ref{ETan}). Hence $\cT(\cF^{-j} \hcS)$ is the skew-orthogonal
complement of $D\cF^{-j} (e_1)$ and so $\cT\bigl( (\cF^{-j}
\hcS)\bigcap \Omega_{Q,V} \bigr) $ is the skew-orthogonal
complement of $\pi \bigl( D\cF^{-j} (e_1)\bigr).$ Since $(\cF^{-j}
\hcS)\bigcap \Omega_{Q,V}$ is one-dimensional,
$$
   \cT\Bigl( (\cF^{-j} \hcS)\bigcap \Omega_{Q,V}\Bigr)
   =\reals\, \pi \bigl( D\cF^{-j} (e_1)\bigr).
$$
Our results in Section~\ref{subsecSUV} easily imply that  $\pi
\bigl( D\cF^{-j} (e_1)\bigr)$ is an s-vector for all $j \geq 1$,
so the lemma follows. \qed
\medskip

Now let $A$ be a function from Proposition~\ref{PrDistEq2}. For
each pair $(Q,V)$ we define a function $A_{Q,V}$ on $\Omega_0$ by
\beq
       A_{Q,V} = A\circ(\pi_0|_{\Omega_{Q,V}})^{-1}
         \label{AQV}
\eeq
Note that $A_{Q,V}$ has discontinuities on the set
$\cS_{Q,V}=\pi_0(\cS_{n_A}\cap\Omega_{Q,V})$. Also,
$$
   \brA (Q,V) = \int_{\Omega_{Q,V}} A(Q,V,q,w)\, d\mu_{Q,V}(q,w)
   = \int_{\Omega_0} A_{Q,V}(r,\varphi)\, d\mu_0
$$

\begin{lemma}
For any $(Q,V)$ and $(Q',V')$ we have
\begin{align*}
  \left | \brA (Q,V) - \brA (Q',V') \right |
     &\leq  C \, \Bigl( \| Q - Q'\| + \| V - V'\| \\
     &\quad + n_A \bigl(\|V\| + \|V'\|\bigr) + n_A^2/M \Bigr)
\end{align*}
for some $C=C(A)>0$. \label{lmAA}
\end{lemma}

\noindent{\em Proof}. If $A$ had a bounded local Lipschitz constant
(\ref{Ax'}) on the entire space $\Omega$, the estimate would be
trivial. However, the function $A$ is allowed to have singularities
on $\cS_{n_A}$ (the discontinuity set for the map $\cF^{n_A}$), and
the local Lipschitz constant Lip$_xA$ is allowed to grow near
$\cS_{n_A}$, according to Lemma~\ref{lmfRprop}. As a result, two
error terms appear, denoted by $E_1+E_2$, where $E_1$ comes from the
fact that the functions $A_{Q,V}$ and $A_{Q',V'}$ have different
singularity sets $\cS_{Q,V}$ and $\cS_{Q',V'}$, and $E_2$ comes from
the growing local Lipschitz constant near these singularity sets.

The error term $E_1$ is bounded by $2\|A\|_{\infty}\Area(G)$, where
$G$ is the region swept by the singularity set $\cS_{Q,V}$ as it
transforms to $\cS_{Q',V'}$ when $(Q,V)$ continuously change to
$(Q',V')$. According to Lemma~\ref{lmclose1}, these singularity sets
lie within the $c_4\varepsilon'$-distance from the discontinuity
lines of the maps $\cF^{n_A}_Q$ and $\cF^{n_A}_{Q'}$, respectively,
where
$$
   \varepsilon'=\Const\, \bigl[\,n_A
   \bigl(\|V\| + \|V'\|\bigr) + n_A^2/M\,\bigr]
$$
Due to our finite horizon \index{Finite horizon} assumption, the
discontinuity lines of $\cF^{n_A}_Q$ and $\cF^{n_A}_{Q'}$ have a
finite total length, and they lie within the $c_4\|Q- Q'\|$-distance
from each other by Corollary~\ref{crclose2}. Therefore, we can cover
$G$ by a finite union of stripes of width $2c_4\varepsilon'+c_4\|Q-
Q'\|$ bounded by s-curves roughly parallel to the discontinuity
lines of $\cF^{n_A}_Q$ and $\cF^{n_A}_{Q'}$. The union of these
stripes, call it $G_0$, has area bounded by $C(\varepsilon'+\|Q-
Q'\|)$, where $C=C(n_A)>0$, thus
$$
     \left | \int_{G_0}
      [A_{Q,V}(r,\varphi)-A_{Q',V'}(r,\varphi)]
      \, d\mu_0 \right | \leq
      2\|A\|_{\infty}C\bigl(\varepsilon'+\|Q- Q'\|\bigr)
$$
To estimate the error term $E_2$, we note that the Lipschitz
constants of the functions $A_{Q,V}$ and $A_{Q',V'}$ on the domain
$\Omega_0\setminus G_0$ are bounded by $CL_A\,\dist(x,
G_0)^{-\beta_A}$, where $x=(r,\varphi)\in\Omega_0$. Here the
distance, originally measured in the Lebesgue metric in
Lemma~\ref{lmfRprop}, can also be measured in the equivalent
$|\cdot|$ metric introduced above (the length of the the shortest
\index{u-curves (unstable curves)} u-curve connecting $x$ with
$\partial G_0$). Thus,
$$
   E_2\leq CL_A\,\Bigl( \| Q - Q'\| + \| V - V'\| \Bigr)
   \,\int_{\Omega_0\setminus G_0}
   [\,{\dist}(x,G_0)]^{-\beta_A}
   \, d\mu_0
$$
and the integral here is finite because $\beta_A<1$.  \qed
\medskip

\subsection{Equidistribution properties}
\label{subsEP} \index{Equidistribution} We use the following scheme
to estimate $\EXP_\ell (A \circ \cF^n)$, where $\ell = (\gamma, \rho
)$ is a standard pair in Proposition~\ref{PrDistEq2}.
\index{Standard pair}

Let $n_1,n_2$ (to be chosen later) satisfy $K\big |\ln |\gamma|\big
| < n_1 < n_2 <n$. For each point $x \in \gamma$ put
$$
   k(x) = \min_{n_1 < k < n_2}
   \{k\colon\ |\gamma_k(x)| \geq \varepsilon_0 \}
$$
where $\gamma_k (x)$ denotes the H-component of $\cF^k (\gamma)$
\index{H-components} that contains the point $\cF^k (x)$ (the
constant $\varepsilon_0$ was introduced in Lemma~\ref{prgrow}). In
other words, $k(x)$ is the first time, during the time interval
$(n_1, n_2)$, when the image of the point $x$ belongs in an
\index{H-curves} H-curve of length $\geq \varepsilon_0$. Clearly,
the set $\{\cF^{k(x)} (x) \colon \ x\in \gamma \}$ is a union of
\index{H-curves} H-curves of length $> \varepsilon_0$. We denote
those curves by $\gamma_j$, $j \geq 1$, and for each $\gamma_j$
denote by $k_j \in (n_1, n_2)$ the iteration of $\cF$ at which this
curve was created. Let $\rho_j$ be the density of the measure $\cF^{
k_j} (\mes_{\ell})$ conditioned on $\gamma_j$. Observe that
$(\gamma_j ,\rho_j)$ is a standard pair for every $j \geq 1$.
\index{Standard pair}

The function $k(x)$ may not be defined on some parts of $\gamma$,
but by Lemma~\ref{prgrow} we have
$$
       \mes_{\ell}\bigl(x\in\gamma\colon\ k(x)\ {\rm is}\
       {\rm not}\ {\rm defined}\,\bigr)
       \leq Cq^{n_2 - n_1}
$$
Hence
\beq
    \EXP_\ell (A \circ \cF^n) =
    \sum_j c_j \EXP_{\ell_j} (A \circ \cF^{n-k_j})
    + \cO (q^{n_2 - n_1})
      \label{expand}
\eeq
where $\sum_j c_j > 1 - Cq^{n_2 - n_1}$.

We now analyze each standard pair $(\gamma_j ,\rho_j)$ separately,
\index{Standard pair} and we drop the index $j$ for brevity. For
example, we denote by $\mes_{\ell}= \mes_{\ell_j}$ the measure on
$\gamma=\gamma_j$ with density $\rho=\rho_j$. Let $x\in \gamma$ be
an arbitrary point and $(Q,V) = \pi_1 (x)$ its coordinates. For each
$0 \leq i \leq n- k$ consider the map
\beq
    \cF_i\colon = \cF_{Q}^{n-k-i} \circ\, \pi_0 \, \circ \cF^i
      \label{cFi}
\eeq
on the curve $\gamma$ (here, as in the previous section, we
identify the domain of the map $\cF_Q$ with $\Omega_0$). Note that
$\cF_{n-k} =\pi_0\circ\cF^{n-k}$, and so
$ A\circ\cF^{n-k}=A_{Q_{n-k},V_{n-k}}\circ\cF_{n-k}, $
where the function $A_{Q,V}$ on $\Omega_0$ was
defined by (\ref{AQV}). Our further analysis is based on the
obvious identity
\begin{align}
\label{mainid}
    \EXP_\ell (A\circ \cF^{n-k}) - \brA (Q,V) &=
    \EXP_\ell (A_{Q,V}\circ \cF_0) - \brA (Q,V) \nonumber\\
    &\quad + \sum_{i=0}^{n-k-1}
    \bigl[\,\EXP_\ell (A_{Q,V}\circ \cF_{i+1}) -
    \EXP_\ell (A_{Q,V}\circ \cF_{i})\,\bigr] \\
    &\quad +\EXP_\ell\left((A_{Q_{n-k}, V_{n-k}}
    \circ \cF_{n-k})-(A_{Q, V}\circ \cF_{n-k})\right)
    \nonumber
\end{align}
We divide the estimate of (\ref{mainid}) into three parts (Propositions
\ref{prexpell}, \ref{prEE} and \ref{prh}).

\begin{proposition}
We have
$$
         \left | \EXP_\ell (A_{Q,V}\circ \cF_{0})
         - \brA (Q,V) \right |
         \leq C \theta_0^{n-k}
$$
for some constants $C>0$ and $\theta_0 <1$. \label{prexpell}
\end{proposition}

\noindent{\em Proof}. Since $\cF_{0} = \cF_{Q}^{n-k} \circ\, \pi_0$,
our proposition asserts the equidistribution for
\index{Equidistribution} dispersing \index{Dispersing billiards}
billiards, see Appendix~\ref{subsecA1}. Note that the
\index{u-curves (unstable curves)} u-curve $\gamma$ has length of
order one ($|\gamma|> \varepsilon_0$), hence there is no ``waiting
period'' during which the curve needs to expand -- the exponential
convergence starts right away. Furthermore, the convergence is
uniform in $Q$, i.e.\ $C$ and $\theta_0$ are independent of $Q$ and
$V$, see Extension~1 in Appendix~A. \qed
\medskip

\begin{proposition}
For each $0 \leq i \leq n-k-1$ we have
\beq
       \left | \EXP_\ell (A_{Q,V}\circ \cF_{i+1})
       - \EXP_\ell (A_{Q,V}\circ \cF_{i}) \right |
       \leq C \varepsilon_{\gamma}
         \label{EE}
\eeq
where $C>0$ is a constant and
\begin{align*}
    \varepsilon_{\gamma}
    \colon &= (n-k)\|V\| + (n-k)^2/M\\
    &\geq
    c\,\max_{0\leq j\leq n-k}\sup_{x\in\gamma}
    \bigl(\| Q - Q(\cF^j x)\|+\| V - V(\cF^j x)\|\bigr)
\end{align*}
where $c>0$ is a small constant. \label{prEE}
\end{proposition}

\noindent{\em Proof}. Estimates of this kind have been obtained for
Anosov diffeomorphisms \cite{D1} and they are based on shadowing
\index{Shadowing} type arguments. We follow this line of arguments
here, too, but face additional problems when dealing with
singularities.

%In fact, for Anosov diffeomorphisms the right hand side
%of (\ref{EE}) would be just $\cO (\varepsilon_{ \gamma})$, which is
%better than ours, the extra factor $n-k$ here comes from the
%singularities in our system.

We first outline our proof. We will construct two subsets
$\gamma^{\ast}_i ,\hgamma^{\ast}_{i}\subset\gamma$ and an absolutely
continuous map $H^{\ast} \colon \gamma^{\ast}_i \to
\hgamma^{\ast}_{i}$ (in fact, the map $h^{\ast} = \cF_{i+1} \circ
H^{\ast} \circ \cF_i^{-1}$ will be the holonomy \index{Holonomy map}
map between some H-components of $\cF_i(\gamma)$ and those of
$\cF_{i+1} (\gamma)$)
\index{H-components} that have three properties:
\begin{itemize}
\item[(H1)] $\mes_{\ell} (\gamma \setminus \gamma^{\ast}_i) < C
\varepsilon_{\gamma}$ and $\mes_{\ell} (\gamma \setminus
\hgamma^{\ast}_{i}) < C\varepsilon_{\gamma}$,
\item[(H2)]
$\EXP_{\ell^{\ast}_i} (|A_{Q,V} \circ \cF_i-A_{Q,V} \circ h^{\ast}
\circ \cF_i|) < C\varepsilon_{\gamma} \vartheta^{n-k-i}$,
\item[(H3)] the Jacobian
$\cJ_\ast (x)$ of the map $H^{\ast}$ satisfies
\begin{equation} \label{smallC}
      |\ln \cJ_\ast| \leq 2
\end{equation}
 and
\begin{equation} \label{smallL}
   \EXP_{\ell^{\ast}_i}(|\ln \cJ_\ast|) < C\varepsilon_{\gamma}
\end{equation}
\end{itemize}
(here $\EXP_{\ell^{\ast}_i}(f) = \int_{\gamma^{\ast}_i} f\,
d\mes_{\ell}$ for any function $f$). In the rest of the proof of
Proposition~\ref{prEE}, we will denote $A_{Q,V}$ by $A$ for
brevity.

Observe that (H1)--(H3) imply (\ref{EE}). Indeed,
we use the
change of variables $y=H^{\ast}(x)$ and get
\begin{equation}
\label{SUSin}
   \bigl| \EXP_\ell (A\circ \cF_{i+1} - A\circ \cF_{i})
       \bigr|
\end{equation}
\begin{align*}
       & \leq
\left(     \int_{\gamma-\gamma_{i}^\ast}
     \bigl|(A\circ \cF_{i})\, d\mes_{l_i}\bigr|
    +\int_{\gamma-\hgamma_{i}^\ast}
    \bigl|(A\circ \cF_{i+1})\bigr|\, d\mes_{l_i} \right)\\
    & + \EXP_{l_i^\ast}\bigl|(A\circ h^\ast\circ \cF_i)
    -(A \circ \cF_i)\bigr| \\
    & + \EXP_{l_i^\ast}\bigl|(A\circ h^\ast\circ \cF_i)
    (\cJ_\ast-1)\bigr|\\
    &=I+\RmII+\RmIII
\end{align*}
Next
\begin{align}
    |I| & \leq 2 C\eps_\gamma \|A\|_\infty & \quad \text{by (H1)} \\
    |\RmII| & \leq C\eps_\gamma & \quad \text{by (H2)}
\end{align}
To estimate $\RmIII$ observe that (\ref{smallC}) implies
$|\cJ_\ast-1|\leq\Const\, |\ln\cJ_\ast|$ on $\gamma_i^\ast$, so
\begin{equation}
    |\RmIII|\leq
    C\, \EXP_{\ell_i^\ast}\bigl(|\ln\cJ_\ast|\bigr)\,
    \|A\|_\infty \leq C \eps_\gamma \|A\|_\infty
     \qquad\text{by (H3)}
\end{equation}
This completes the proof of (\ref{EE}) assuming (H1)--(H3).

We begin the construction of the sets $\gamma^{\ast}_i$ and
$\hgamma^{\ast}_{i}$. First, the definition of both maps $\cF_i$ and
$\cF_{i+1}$, see (\ref{cFi}), involves the transformation of the
curve $\gamma$ to $\cF^i (\gamma )$. Let $\tgamma$ be an H-component
of $\cF^{i} (\gamma)$. If its length is \index{H-components}
$<c_3\varepsilon_\gamma$, we simply discard it (i.e., remove its
preimage in $\gamma$ from the construction of both $\gamma^{\ast}_i$
and $\hgamma^{\ast}_{i}$). If $|\tgamma| > c_3\varepsilon_\gamma$,
then Lemma~\ref{lmclose2} gives us two partitions $\tgamma
=\cup_{j=0}^J\tgamma_i =\cup_{j=0}^J\tgamma_j'$, such that for each
$j=1,\dots,J$ the sets $\cF_{Q}(\pi_{0}(\tgamma_j))$ and
$\pi_{0}(\cF(\tgamma_j'))$ are \index{H-curves} H-curves and
$\dist(\cF_{Q} (\pi_{0}(\tgamma_j)) ,\pi_{0}(\cF(\tgamma_j')) <
c_1\varepsilon_{\gamma}$. We remove the preimage of $\tgamma_0$ from
the construction of $\gamma^{\ast}_i$, and the preimage of
$\tgamma_0'$ from the construction of $\hgamma^{\ast}_{i}$. By
Lemma~\ref{prgrow}, the total $\mes_{\ell}$-measure of the (so far)
removed sets is $\cO (\varepsilon_\gamma)$. The remaining H-curves
in $\cF_{Q}(\pi_{0}(\cF^i(\gamma))$ and $\pi_{0}(\cF^{i+1}(\gamma))$
are now paired according to Lemma~\ref{lmclose2}.

Consider an arbitrary pair of curves $W'\subset\cF_{Q} (\pi_{0}
(\cF^i (\gamma))$ and $W''\subset\pi_{0} (\cF^{i+1} (\gamma))$
constructed above and remember that $\dist(W',W'') <
c_1\varepsilon_{\gamma}$. According to our definition of the maps
$\cF_i$ and $\cF_{i+1}$, both curves $W'$ and $W''$ will be then
iterated $n-k -i -1$ times under the same billiard map $\cF_{Q}$.
For each $x\in W'$ and $n\geq 0$ denote by $r_n(x)$ the distance
from the point $\cF_Q^n(x)$ to the nearest endpoint of the
H-component of $\cF_Q^n(\gamma')$ that contains the point
\index{H-components} $\cF_Q^n(x)$. Define
\begin{equation} \label{Wast}
    W'_\ast=\{x\in W':\quad r_n(x)\geq C\eps_\gamma \vartheta^n
    \quad \text{for all}\quad n\geq 0 \}
\end{equation}
where $C$ is a constant chosen as follows. Let $r^s(x)$ denote the
distance from $x$ to the nearest endpoint of the homogeneous stable
manifold $W^s_x$ for the map $\cF_{Q}$ passing through $x$. (A
homogeneous stable manifold $W^s \subset \Omega_Q$ is a maximal
curve such that $\cF_Q^n (W^s)$ is a homogeneous s-curve for each
$n\geq 0$.) By \cite[Appendix 2]{BSC2}, if $r^s(x)<\varepsilon$,
then for some $n\geq 0$ the point $\cF_Q^n(x)$ lies within the
$(\varepsilon \vartheta^n) $-neighborhood of either a singularity
set of the map $\cF_Q$ or the boundary of a homogeneity strip
\index{Homogeneity strips} $\bbH_{\pm k}$, $k\geq k_0$. Since the
singularity lines and the boundaries of homogeneity strips are
uniformly transversal to \index{u-curves (unstable curves)} u-curves
it follows that if $C$ in (\ref{Wast}) is large enough then for all
$x\in W'_\ast$ $W^s_x \cap W'' \neq \emptyset .$ Let $h \colon
W_{\ast}' \to W''$ denote the holonomy \index{Holonomy map} map
(defined by sliding along the stable manifolds $W^s_x$). We remove
the preimage of the set $W' \setminus W_{\ast}'$ from the
construction of $\gamma^{\ast}_i$, and the preimage of the set $W''
\setminus h(W_{\ast} ')$ -- from the construction of
$\hgamma^{\ast}_{i}$.

We need to estimate the measure of the sets just removed from the
construction. Denote by $\gamma' = \cup_{\alpha} \gamma_{\alpha}'
\subset \Omega_0$ the union of the above \index{H-curves} H-curves
$W' \subset \cF_{Q} (\pi_{0} (\cF^i (\gamma) )$ and by
$\mes_{\gamma'}$ the restriction of the measure $\cF_{Q} (\pi_{0}
(\cF^i (\mes_\ell)))$ to $\gamma'$.

\begin{claim}
$\mes_{\gamma'} \bigl( \cup_{W'}
(W' \setminus W_{\ast}') \bigr)
\leq \Const\, \varepsilon_\gamma$,
and a similar estimate holds for
$\cup_{W''} (W'' \setminus h(W_{\ast}'))$.
\end{claim}

\proof For any $n\geq 0$ and $\varepsilon>0$
\beq
      \mes_{\gamma'}\left(x\in\gamma'\colon
      r_n(x)<\varepsilon\right)<\beta\varepsilon
        \label{rnxass}
\eeq
where $\beta>0$ is some large constant, according to the remarks
in the end of Section~\ref{subsecSHUC} (they are stated for the
map $\cF$, but obviously apply to the billiard map $\cF_Q$ as
well).

% \beq
%      \mes_{\gamma'}\bigl(x\in\gamma'\colon\,
%      r^s(x)<\varepsilon\bigr)<\beta'\varepsilon
%        \label{rnxcon}
% \eeq
%for some large constant $\beta'>0$.
%Now we use (\ref{rnxass}) to obtain
Thus
$$
    \mes_{\gamma'}\bigl(\cup_{W'}
    (W' \setminus W_{\ast}') \bigr) \leq
      \sum_{n=0}^{\infty}C\beta\varepsilon_\gamma\vartheta^n=
      \frac{C\beta}{1-\vartheta}\,\varepsilon_\gamma.
$$
%which proves (\ref{rnxcon}).
This proves the estimate for $\mes_{\gamma'}\bigl(\cup_{W'} (W'
\setminus W_{\ast}') \bigr).$ To get a similar estimate for
$\cup_{W''} (W'' \setminus h(W_{\ast}'))$ we observe that the
$\cF_Q$ orbits of the points $x \in \cup_{W''} (W'' \setminus
h(W_{\ast}'))$ also come close to the singularities. Indeed, if
$r^s(x)\leq \Const\, \eps_\gamma$, then the orbit of $x$ comes
close to singularities by the previous discussion. If the opposite
inequality holds, then the orbit of $x$ should pass near a
singularity since otherwise we would have $h^{-1} (x) \in
W_\ast'.$ Now the result follows by (\ref{rnxass}). \qed
\medskip

This completes the construction of the sets $\gamma^{\ast}_i$ and
$\hgamma^{\ast}_{i}$ and the proof of (H1). The map $h^\ast\colon\
\cF_i(\gamma^{\ast}_i )\to\cF_{i+1} (\hgamma^{\ast}_{i})$ is the
induced holonomy \index{Holonomy map} map. It remains to prove (H2)
and (H3).

Put $d\colon= n-k -i -1$ for brevity. For any point $x' \in
W_{\ast}'$ and its ``sister'' $x'' =h(x) \in W''$, the points $z'=
\cF_{Q}^d (x') \in \cF_i (\gamma)$ and $z''= \cF_{Q}^d (x'') \in
\cF_{i+1} (\gamma)$ (related by $h^{\ast}(z')=z''$) will be
$(C\vartheta^d\varepsilon_\gamma)$-close, since $\cF_Q$ contracts
stable manifolds by a factor $\leq\vartheta<1$. In other words, the
trajectory of the point $x''$ shadows \index{Shadowing} that of $x'$
in the forward dynamics. Therefore, the values of the function
$A=A_{Q,V}$ will differ at the endpoints $z'$ and $z''$ by at most
$\cO (\vartheta^d \varepsilon_\gamma D(z',z''))$, unless they are
separated by a discontinuity curve of the function $A$. Here
$D(z',z'')=[\dist(W^s(z',z''),\cS_{Q,V})]^{-\beta_A}$, where
$W^s(z',z'')$ denotes the stable manifold connecting $z'$ with
$z''$, and $\cS_{Q,V}=\pi_0(\cS_{n_A}\cap\Omega_{Q,V})$ in
accordance with Lemma~\ref{lmfRprop} (b).

Let $W^{\diamond}$ be an H-component of $\cF_Q^d(W')$. Put
\index{H-components}
$W^{\diamond}_{\ast}=W^{\diamond}\cap\cF_Q^d(W_{\ast}')$ and
$\mes_i=\cF_i(\mes_{\ell})$. We need to estimate
$$
    \Delta(W^{\diamond})\colon =
    \int_{W^{\diamond}_{\ast}}|A(z')-A(z'')|\,d\,\mes_i
$$
The curve $W^{\diamond}$ crosses the discontinuity set $\cS_{Q,V}$
in at most $K_{n_A}$ points, cf.\ Lemma~\ref{lmsingQV}. If a pair of
points $z'$ and $z'' = h^\ast (z')$ is separated by a curve of
$S_{Q,V}$, then both $z'$ and $z''$ lie in the $(C\vartheta^d
\varepsilon_\gamma)$-neighborhood of that curve. Let $\fU_A$ denote
the $(C\vartheta^d \varepsilon_\gamma)$-neighborhood of $S_{Q,V}$.
The sets $W^\diamond \cap \fU_A$ and $W^\diamond_\ast \cap \fU_A$
have $|\cdot |$-measure less than $K_{n_A}$ times the $|\cdot |$
measure of the $(C \vartheta^d \varepsilon_\gamma)$-neighborhood of
the endpoints of $W^\diamond$ and $W^\diamond_\ast$, respectively.
Hence the contribution of these sets to $\Delta(W^{\diamond})$ will
be bounded by
$$
  K_{n_A} \mes_\ell \bigl( r_{n-k} (x)
  < C \vartheta^d \varepsilon_\gamma \bigr)
  \leq \Const \, \vartheta^d \varepsilon_\gamma
$$
where we used Lemma~\ref{prgrow}.

Next, the set $W^{\diamond} \setminus \fU_A$ is a union of
\index{H-curves} H-curves $W_1,\dots,W_k$ with some $k\leq K_{n_A}$.
For each $W_j$ we put $W_{j\ast}=W_j\cap\cF_Q^d(W_{\ast}')$ and
estimate
\begin{align*}
   \int_{W_j\ast}|A(z')-A(z'')|\,d\,\mes_i
     & \leq
   C'\vartheta^d\varepsilon_\gamma\,
   \frac{\mes_i(W_j)}{|W_j|}\,
   \int_0^{|W_j|} t^{-\beta_A}\, dt \\
     & \leq
   C''\vartheta^d \varepsilon_\gamma
     \frac{\mes_i(W_j)}{|W_j|^{\beta_A}}
\end{align*}
where $C',C''>0$ are some constants. Summing up over $j$ gives
\beq \label{AAWd}
   \int_{W^{\diamond}}|A(z')-A(z'')|\,d\,\mes_i
   \leq \Const\,K_{n_A}^{\beta_A}
   \vartheta^d\varepsilon_\gamma
   \frac{\mes_i(W^{\diamond})}{|W^{\diamond}|^{\beta_A}},
\eeq
where we first used the homogeneity of the measure $\mes_i$ to
estimate
$$
   \mes_i (W_j) \leq \Const\, |W_j|\,
   \frac{\mes_i (W^\diamond)}{|W^\diamond|}
$$
and then by Jensen's inequality obtain
$$
  \sum_j |W_j|^{1-\beta_A}
  \leq K_{n_A}^{\beta_A} |W^\diamond|^{1-\beta_A}.
$$

Next, summing over all the H-components of $\cF_Q^d(W')$ and all
\index{H-components} the curves $W'\subset\cF_{Q} (\pi_{0} (\cF^i
(\gamma))$ gives a bound
\begin{align}
\label{EqShift}
    \EXP_{\ell^{\ast}_i}
    (|A\circ\cF_i-A\circ h^{\ast}\circ\cF_i|)
    & \leq
    \sum_{W^{\diamond}\subset\cF_i(\gamma)}
   \Const\,K_{n_A}^{\beta_A} \vartheta^d\varepsilon_\gamma
     \frac{\mes_i(W^{\diamond})}{|W^{\diamond}|^{\beta_A}}
     \nonumber\\
     & \leq
   \Const\,K_{n_A}^{\beta_A} \vartheta^d\varepsilon_\gamma
\end{align}
where the last inequality follows from Lemma~\ref{prgrow} and
\beq
       \sum_{W^{\diamond}\subset\cF_i(\gamma)}
       \frac{\mes_i(W^{\diamond})}{|W^{\diamond}|^{\beta_A}}
       \leq \Const\int_\gamma
       [r_{n-k}(x)]^{-\beta_A}\, d\rho(x)
       \leq \, \Const
         \label{sumbetaA}
\eeq
(we remind the reader that $\beta_A<1$). This proves (H2).
It remains to prove
(H3).

\begin{figure}[htb]
    \centering
    \psfrag{x1}{$x'$}
    \psfrag{x2}{$x''$}
    \psfrag{x3}{$x_-'$}
    \psfrag{x4}{$x_-''$}
    \psfrag{y1}{$y'$}
    \psfrag{y2}{$y''$}
    \psfrag{z1}{$z'$}
    \psfrag{z2}{$z''$}
    \psfrag{g}{$\gamma$}
    \psfrag{gt}{$\tgamma$}
    \psfrag{p}{$\pi_Q$}
    \psfrag{W}{$W^s$}
    \psfrag{F}{$\cF$}
    \psfrag{Fi}{$\cF^i$}
    \psfrag{Fq}{$\cF_Q$}
    \psfrag{Fd}{$\cF_Q^d$}
    \psfrag{B}{$\Omega_0$ space $\ \ \Bigg\{$}
    \psfrag{T}{$\Bigg\}\quad\Omega$ space}
    \includegraphics{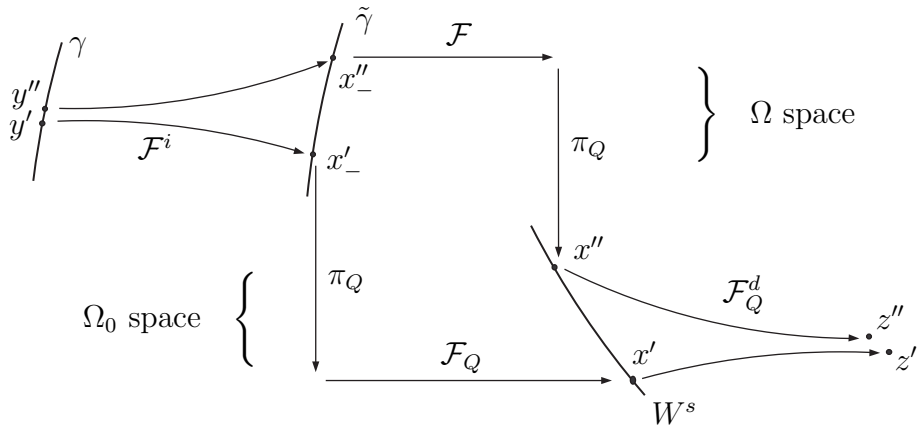}
    \caption{The construction in Proposition~\ref{prEE}}
    \label{FigShadow}
\end{figure}

Let $x_-', x_-''\in\tgamma$ be the preimages of $x',x''$,
respectively, i.e.\ $x' = \cF_{Q} (\pi_0 (x_-'))$ and $x'' = \pi_0
(\cF (x_- ''))$. Note that the distance between $x_- '$ and $x_- ''$
is $<C \varepsilon_\gamma$. Let $y' = \cF^{-i} (x_-')$ and $y'' =
\cF^{-i} (x_-'')$ be the preimages of our two points on the original
curve $\gamma$. Note that dist$(y', y'') \leq C\vartheta^i
\varepsilon_\gamma$ and $z' = \cF_i (y')$ and $z'' = \cF_{i+1}
(y'')$, see Fig.~\ref{FigShadow}. (We can say that the trajectory of
the point $y''$ shadows \index{Shadowing} that of $y'$ during all
the $n-k$ iterations.) Now $y''=H^{\ast}(y')$, where $H^{\ast}
=\cF_{i+1}^{-1}\circ h^{\ast}\circ\cF_i$. The Jacobian $\cJ_{\ast}$
of the map $H^{\ast}\colon\gamma\to\gamma$ satisfies
$$
   \ln \cJ_{\ast}(y') = \ln \frac{\cJ_{\gamma}\cF_i(y')}
       {\cJ_{\gamma}\cF_{i+1}(y'')} + \ln \cJ h^{\ast}(z')
$$
where $\cJ_{\gamma} \cF_i$ and $\cJ_{\gamma} \cF_{i+1}$ denote the
Jacobians (the expansion factors) of the maps $\cF_i$ and
$\cF_{i+1}$, respectively, restricted to the curve $\gamma$, and
$\cJ h^{\ast}$ is the Jacobian of the holonomy \index{Holonomy map}
map $h^{\ast}$.

\begin{lemma}
\label{LmJac}
We have
\beq
       \left | \ln \frac{\cJ_{\gamma}\cF_i(y')}
       {\cJ_{\gamma}\cF_{i+1}(y'')} \right |
       \leq \frac{C\varepsilon_\gamma}{|\tgamma|^{2/3}}
       + \sum_{r=0}^{d}
       \frac{C\vartheta^r\varepsilon_\gamma}{|\tgamma_r'|^{2/3}}
          \label{sumd}
\eeq
and
\beq
    \bigl |\ln \cJ h^{\ast}(z')\bigr |
    \leq \sum_{r=d}^{\infty}
       \frac{C\vartheta^r\varepsilon_\gamma}{|\tgamma_r'|^{2/3}}
\eeq
where $\tgamma_r'$ denotes the \index{H-components} H-component of
$\cF_{Q}^{r+1} (\pi_0 (\tgamma))$ containing the point $\cF_{Q}^r
(x')$.
\end{lemma}

This lemma is, in a sense, an extension of
Proposition~\ref{prdist}. Its proof is given in Appendix~C, after
the proof of Proposition~\ref{prdist}.

Now (\ref{smallC}) follows directly from Lemma \ref{LmJac} and the
definition of $\gamma^\ast_i$, cf.\ (\ref{Wast}). To complete the
proof of (H3), we need to establish (\ref{smallL}):
\begin{align*}
    \EXP_{\ell_i^{\ast}}(|\ln \cJ_\ast|)
    & \leq
    \sum_{\tgamma\subset\cF^i(\gamma)}
   C\varepsilon_\gamma
     \frac{\mes'(\tgamma)}{|\tgamma|^{2/3}}\\
     &\quad +\sum_{r=0}^{\infty}
     \sum_{\tgamma'\subset\cF_{Q}^{r+1}(\pi_0 (\cF^i(\gamma)))}
   C\vartheta^r\varepsilon_\gamma
     \frac{\mes_r'(\tgamma')}{|\tgamma'|^{2/3}}\\
     & \leq
   C\varepsilon_\gamma+\sum_{r=0}^{\infty}
   C\vartheta^r\varepsilon_\gamma
   = \Const\,\varepsilon_\gamma
         \label{EElong}
\end{align*}
where $\mes' = \cF^i(\mes_{\ell})$ and $\mes_r' = \cF_{Q}^{r+1}
(\pi_0 (\cF^i(\mes_{\ell})))$. Here we used the same trick as in
(\ref{sumbetaA}). The property (H3) is proved, and so is
Proposition~\ref{prEE}. \qed
\medskip

\begin{proposition}
\label{prh} There is a constant $C$ such that
$$
    \left|\EXP_\ell\left((A_{Q_{n-k}, V_{n-k}}\circ
    \cF_{n-k})-(A_{Q, V}\circ \cF_{n-k})\right)\right|
    \leq C \eps_\gamma.
$$
\end{proposition}

\proof The proof of this proposition follows exactly the same
arguments as the proof of (\ref{EqShift}) in the estimate of (H2)
so we omit it. \qed
\medskip

We now return to our main identity (\ref{mainid}) and obtain
\begin{align*}
    \EXP_\ell (A\circ \cF^{n-k}) - \brA (Q,V)
    & \leq
    C\theta_0^{n-k} + C(n-k)\varepsilon_\gamma \\
    & \leq
    C\theta_0^{n-k} + C(n-k)^2\|V\| + C(n-k)^3/M \\
\end{align*}
Equation (\ref{expand}) now yields
\begin{align*}
    \EXP_\ell (A \circ \cF^n)
    &=
    \sum_j c_j \brA (Q_j,V_j)
        + \cO (q^{n_2 - n_1})\\
       &\quad + \cO (\theta_0^{n-n_2})
        + \cO((n-n_1)^2)\max_{\gamma_j} \|V_j\|\\
       &\quad + \cO\left ((n-n_1)^3\right )/M
\end{align*}
where $(Q_j,V_j) \in \pi_1 (\gamma_j)$. Note that
$$
   \max_j \| V_j \| \leq \|V\| + Cn^2 / M
$$
Finally, we apply Lemma~\ref{lmAA} to estimate the value of $\brA
(Q_j, V_j)$:
$$
   \left | \brA(Q_j,V_j) - \brA(Q,V) \right |
   \leq Cn\,\|V\| + Cn^2/M
$$
and arrive at
\begin{align*}
   \left | \EXP_\ell(A \circ \cF^n) - \brA(Q,V) \right |
   &\leq C\,\|V\| \left [n+(n-n_1)^2\right ] \\
   &\quad +\,
   C\left [n^2+(n-n_1)^3\right ]/M\\
   &\quad +\, C\, q^{n_2 - n_1}+C\, \theta_0^{n-n_2}
\end{align*}
Now for any $m\leq \min\{n/2,K\ln M\}$ we can choose $n_1,n_2$ so
that $n-n_2 =n_2-n_1 = m$. This completes the proof of
Proposition~\ref{PrDistEq2}. \qed \medskip

\medskip\noindent{\em Remark}. Proposition~\ref{PrDistEq2} can be
easily generalized to any finite or countable union $\ell =
\cup_{\alpha} \gamma_{\alpha}$ of \index{H-curves} H-curves with a
smooth probability measure $\mes_{\ell}$ on it (as introduced in the
end of Section~\ref{subsecSP}) provided (i) they have approximately
the same $Q$ and $V$ coordinates and (ii) the lower bound on $n$
(that is, $n \geq K |\ln \length (\gamma)|$) is modified
accordingly. For the latter, let us assume that
$$
  \mes_{\ell} \bigl( r_0(x) < \varepsilon ) < \beta_7 \varepsilon
$$
for all $\varepsilon>0$, in the notation of the end of
Section~\ref{subsecSP}; i.e.\ the curves $\gamma_{\alpha}$ are
`long' on the average. Then the proposition would simply hold for
all $n \geq \Const$.
\medskip

To prove Corollary~\ref{PrDistEq3}, we decompose the set
$\cF^{n-j}(\gamma)$ into H-components according to
\index{H-components} Proposition~\ref{PrDistEq1}, then apply
Proposition~\ref{PrDistEq2} to these H-components (see the above
remark), and deal only with the last $j$ iterations of $\cF$. \qed

\noindent{\em Remark.}  As it was mentioned in Chapter~\ref{SecPP},
the estimates of Proposition~\ref{PrDistEq2}  and
Corollary~\ref{PrDistEq3} are apparently less than optimal. Even
though they suffice for our purposes, we outline possible
improvements of the key estimates \eqref{mainid} and \eqref{SUSin}
used in our analysis. Such improvements could be useful in other
applications.

\begin{itemize}
\item In our estimation of the term $I$ in \eqref{SUSin},
there is no need to discard  the orbits which pass close to
singularities for small $i$ (i.e., for $n-i\geq\Const \ln
\varepsilon_\gamma$), because the corresponding pieces expand in the
remaining time and their images become asymptotically
equidistributed.
\item In the same estimation, if $i$ is large then one can expect
that the singularities of $\cF$ on the set $\cF^i\gamma$ will be
asymptotically uniformly distributed along the singularities of
$\cF_Q$. Hence one can express their contribution in terms of some
integrals over $\cS.$ In fact, this idea is realized, in a different
situation, in Chapter~\ref{SecRDM}.
\item In the analysis of the last term in \eqref{mainid}, instead
of estimating $(Q_m, V_m)-(Q,V)$ by $\cO(\varepsilon_\gamma)$ we can
use more precise formulas
$$
   Q_m-Q=\sum_j V_j s_j \quad\text{and}\quad
   V_m-V\sim\sum_j \frac{(\cA\circ\cF^j)}{\sqrt{M}},
$$
where $s_j$ are intercollision times. Combining these with the
Taylor expansion of $A_{Q_m, V_m}-A_{Q,V}$ we obtain a series with
many terms of zero mean which lead to additional cancelation. A
similar method can be useful to improve our estimates on the terms
$\RmII$ and $\RmIII$ in \eqref{SUSin} (this is actually done in
\cite{R1} in the case of uniformly hyperbolic systems without
singularities), but it is technically quite complicated since $h$
and $J$ depend on infinite orbits.
\end{itemize}
\newpage

\chapter{Regularity of the diffusion matrix}
\label{SecRDM} \setcounter{section}{5}\setcounter{subsection}{0}
 \index{Diffusion matrix}

\subsection{Transport coefficients}
In this section we establish the log-Lipschitz \index{Log-Lipschitz
continuity} continuity, in the sense of (\ref{sigQ1Q2}), for the
diffusion matrix $\sigma_Q^2(\cA)$ given by (\ref{EqSigma}). Our
arguments, however, can be used for the analysis of other transport
coefficients in a periodic Lorentz gas (such as electrical
conductivity, heat conductivity, viscosity, etc.), so we precede the
proof of (\ref{sigQ1Q2}) by a general discussion.

Computing transport coefficients is one of the central problems in
linear response theory of statistical physics. The evolution of
various macroscopic quantities such as mass, momentum, heat, and
charge can be described by transport equations, which are very
general and can be derived from a few basic principles. They have
a wide range of applicability, in the sense that one equation can
describe transport in different media.
%The only difference between different materials is the numerical
%value of the corresponding transport coefficients.
However, the numerical values of transport coefficients are
material specific and cannot be found from general principles used
to derive the equations themselves. In physics, the values of
transport coefficients often have to be determined experimentally.
Obtaining the values of transport coefficients theoretically, from
the microstructure of the material, seems to be a difficult task.

The difficulty in computing transport coefficients may be partly due
to their erratic dependence on the parameters involved. It has been
noticed recently that transport coefficients are not differentiable
with respect to the model's parameters in several seemingly
unrelated cases: one-dimensional piecewise linear mappings
\cite{GK,GrK,Kl,KD1,KD2}, nonlinear baker transformations
\cite{GFD}, nonhyperbolic climbing-sine maps \cite{KK2}, billiard
particles bouncing against a corrugated wall \cite{HG}, and various
modifications of a periodic Lorentz gas \cite{BDL,HKG,KK1}. The only
common feature of these models is the presence of singularities in
the dynamics. Actually, for completely smooth chaotic systems, such
as Anosov diffeomorphisms, the transport coefficients are proven to
be differentiable \cite{R2, R1}.

In this section we analyze the diffusion \index{Diffusion matrix}
matrix $\sigma^2_Q (\cA)$ in a periodic Lorentz gas. Even though we
only derive an upper bound on its variation, our results and
analysis strongly suggest that it may be not differentiable with
respect to $Q$. A similar conjecture was stated in \cite{BDL}, where
another transport coefficient (electric conductivity) for the
periodic Lorentz gas was studied numerically and semi-heuristically.
The lack of differentiability of the electric conductivity was
traced in \cite{BDL} to singularities in the dynamics, and these are
the same singularities that cause divergence of certain terms in our
estimates. Eventually we hope to prove rigorously that transport
coefficients are not smooth, but so far this remains an open
problem. Let us also mention that the regularity of transport
coefficients is an issue for stochastic models of interacting
particles, see, e.g.\ \cite{V}.

Next we describe several specific problems related to transport
coefficients. We restrict our discussion to a periodic Lorentz gas
with finite \index{Finite horizon} horizon; other models are
discussed, e.g., in \cite{BLRB, HBS, Sp}.  \medskip

\noindent {\sc A. Diffusion.} Consider a single particle moving in a
periodic array of scatterers \index{Scatterer} in $\reals^2$. Let
$q(t)$ denote the position of the particle and $x(t)$ the projection
of its position and velocity onto the unit tangent bundle over the
fundamental domain (the latter is a torus minus the scatterers). Let
$x_n$ be the value of $x$ at the moment of the $n$-th collision and
$q_n$ the position of the particle in $\reals^2$ at this moment.
Then we have
$$
     q_n=\sum_{j=0}^{n-1} H(x_j),
$$
where $H(x_j)$ denotes the displacement (the change in position) of
the particle between the $j$th and the $(j+1)$st collisions
(obviously this difference does not depend on which lift of $x$ to
the plane we choose). The Central Limit Theorem for dispersing
\index{Dispersing billiards} billiards now gives the following:

\begin{theorem}[\cite{BSC2}]
If $x_0$ has a smooth initial density with respect to the Lebesgue
measure, then $q_n/\sqrt{n}$ converges,  as $n\to\infty$, to a
normal law $\cN(0, \brD^2)$ with non-singular covariance matrix
$\brD^2$ given by
\begin{equation} \label{DGK}
    \brD^2=\sum_{n=-\infty}^{\infty}
    \int_{\Omega} H(x_0) H(x_n)\, d\mu(x)
\end{equation}
where $\mu$ denotes the invariant measure on the collision space
$\Omega$.
\end{theorem}

Now standard methods allow us to pass from discrete to continuous
time (see \cite{Rt, DP, MT} or our Section~\ref{subsecMart}) and we
obtain

\begin{corollary}[\cite{BSC2}]
If $x(0)$ has a smooth initial density in the phase space, then
$q(t)/\sqrt{t}$ converges to $\cN(0, D^2)$ where
\begin{equation} \label{ContD}
     D^2=\brD^2/\brL .
\end{equation}
\end{corollary}

This result for a single particle system allows us to describe the
diffusion in the ideal gas of many noninteracting particles. For
example let $\rho_0$ be a smooth nonnegative function with a
compact support. Pick some $\eps>0$ and for every $m \in
\integers^2$ put $N_\eps = [\eps^{-1} \rho_0(\eps m)]$ independent
particles into the fundamental domain, which is centered at $m$,
so that each particle's position and velocity direction are
uniformly distributed with respect to the Lebesgue measure. Let
$\nu_{\eps,t}$ be the measure on $\reals^2$ given by
$\nu_{\eps,t}(B)=\eps^{-3} \times \#$(particles in $B/\eps$ at
time $t/\eps^2).$ Endow the space of measures with weak topology.
Then $\nu_{\eps,t}$ converges in probability, as $\eps\to 0$, to a
measure $\nu_t$ with density $\rho_t$, which is the convolution
$\rho_t = \rho_0 * \cN(0, D^2 t)$, i.e.\ $\rho_t$ satisfies the
diffusion equation
$$
   \frac{\partial\rho}{\partial t}=
   \frac{1}{2} \sum_{i,j} D_{ij}^2
   \frac{\partial^2 \rho}
   {\partial y_i \partial y_j},
$$
where $D^2$ is the matrix given by (\ref{ContD}), and $y_1,y_2$
denote the coordinates in $\reals^2$.

\medskip\noindent {\sc B. Electric conductance.} Consider the previous
model with a single particle moving in a periodic array of
scatterers, \index{Scatterer} and in addition assume that between
collisions the motion is governed by the equation
\begin{equation} \label{ET}
    \dot{v}=E-\frac{\la v,E \ra}{\|v\|^2}\, v
\end{equation}
where $E \in \reals^2$ is a fixed vector representing a constant
electric field; the second term in (\ref{ET}) is the so called
Gaussian thermostat, it models the energy dissipation (observe
that (\ref{ET}) preserves kinetic energy). Let $\cF_E \colon
\Omega \to \Omega$ denote the induced collision map.

\begin{theorem}[\cite{CELS}]
\label{ThOhm} {\rm (a)} For small $E$ there exists an
$\cF_E$-invariant ergodic measure $\mu_E$ such that for almost all
$x$ for all $A\in C(\Omega)$
$$
     \frac{1}{n} \sum_{j=0}^{n-1} A(\cF_E x)\to \mu_E(A),
$$
and $\mu_E$ is an SRB measure, i.e.\ its conditional distributions
on unstable manifolds are smooth.
\\
{\rm (b)} If $A$ is piecewise H\"older continuous, then
\begin{equation} \label{mesdiff}
     \mu_E(A)=\mu(A)+\omega(A,E)+o(\|E\|)
\end{equation}
where $\omega$ is linear in each variable.
\end{theorem}

Equation (\ref{mesdiff}) is typical for linear response theory in
statistical physics -- it describes the response of the system to
small perturbations of its parameters (here the parameter vector
$E$), up to a linear order.

As before we apply this result to the displacement of the moving
particle between consecutive collisions. Part (a) implies that for
almost all initial conditions there exists a limit
\begin{equation} \label{DCurr}
    \brJ(E)=\lim_{n\to\infty} \frac{q_n}{n}
\end{equation}
which we can interpreted as electrical current (the average speed
of the charged particle, see also below). Part (b) implies that
there exists a matrix $\brM$ such that for small $E$
\beq
        \brJ(E)=\brM E+o(\|E\|).
\eeq
In other words, $\brJ$ is a differentiable function of $E$ at
$E=0.$ Note, however, that numerical evidence \cite{BDL} indicates
that it is not always differentiable for $E \neq 0$.

As in Subsection~A above, (\ref{DCurr}) implies that there exists
a limit
\begin{equation} \label{CCurr}
     J(E)=\lim_{t\to\infty} \frac{q(t)}{t}=\frac{\brJ(E)}{\brL(E)}
\end{equation}
where $\brL (E)$ denotes the mean (with respect to $\mu_E$) free
path. Since $L(E)\to\brL$ as $E\to 0$, it follows that
$$
    J(E)=\frac{\brM E}{L}+o(\|E\|).
$$
Similarly to Subsection~A, this result can be applied to an ideal
gas. For example consider an infinitely long ``wire'' $W$ obtained
by identifying points of $\reals^2$ whose second coordinates
differ by an integer. Let $S= \{x=0\}$ be a vertical line cutting
$W$ in half. Put one particle to each fundamental domain in $W$
independently and uniformly distributed with respect to the
Lebesgue measure. Let $N_+(t)$ be the number of particles crossing
$S$ from left to right during the time interval $(0,t)$, and
$N_-(t)$ be the number of particles crossing $S$ from right to
left; denote $N(t) = N_+(t) - N_-(t)$. Then (\ref{CCurr}) implies
that almost surely there exists a limit
$$
     \lim_{t\to\infty} \frac{N(t)}{t}=\la J(E), e_1 \ra.
$$
Thus, the flow of particles in our wire is an electric current which
is for small fields approximately proportional to the ``voltage''
$\la E, e_1\ra$ -- we arrive at classical Ohm's law \index{Ohm's
law} of physics. To compute the coefficient in the corresponding
equation, we need to know the functional $\omega$ appearing in
(\ref{mesdiff}). It can be obtained by the following argument
(Kawasaki formula):
$$
    \mu_E(A)=\lim_{n\to\infty} \mu(A\circ \cF_E^n) ,
$$
$$
    \mu(A\circ \cF_E^n)=\mu(A)+\sum_{j=0}^{n-1}
    \bigl[\mu(A\circ \cF_E^{j+1}) -\mu(A\circ \cF_E^{j}) \bigr].
$$
To estimate the terms in the last sum let $y=\cF_E (x). $ Since
$\cF$ preserves the measure $\mu$, it follows that
$$
   \frac{d\mu(y)}{d\mu(x)}=1+\Div_\mu\left(\frac{\partial \cF_E}
   {\partial E} \right)+\cO(\|E\|^2).
$$
Hence
\begin{align*}
       & \int A(\cF_E^{j+1} x)\, d\mu(x)=
        \int A(\cF_E^{j} y)\,\, \frac{d\mu(x)}{d\mu(y)}\,\, d\mu(y)\\
       &\quad =\int A(\cF_E^{j} y)
       \left[1-\Div_\mu\left(\frac{\partial \cF_E}{\partial\, E}
       \right)\right]\, d\mu(y)+O\bigl(\|E\|^2\bigr).
\end{align*}
It follows that
$$
      \omega(A,E)=-\sum_{j=0}^\infty \int
      \Div_\mu\left(\frac{\partial \cF_E}{\partial E}
      \, \right)(y) \, A(\cF^j y)\, d\mu(y)
$$
expressing the derivative of $\mu_E$ as the sum of correlations.
In our case the divergence in question is easy to compute so we
obtain the relation $\bar{J}=\tfrac{1}{2} \bar{D}^2 E$, where
$\bar{D}^2$ is the matrix given by (\ref{DGK}). In other words,
$\bar{M} = \tfrac{1}{2} \bar{D}^2$, and, respectively,
$J=\frac{1}{2} D^2 E$, where $D^2$ is the matrix given by
(\ref{ContD}). This is known in physics as Einstein relation \index{Einstein relation}
\cite{CELS}.

\medskip\noindent {\sc C. Viscosity.} This transport coefficient
characterizes the flow of momentum in gases. By its very nature,
it can only be defined for systems with several interacting
particles, so we will not discuss it here (note, however, that a
very simplified version of viscosity in a gas with only two
molecules was introduced in \cite{BSp}).

\medskip\noindent {\sc D. Rayleigh gas.} \index{Rayleigh gas}
So far we have discussed
identical particles moving in a periodic configuration of fixed
scatterers, \index{Scatterer} but similar considerations apply to
Rayleigh gases, in which one or several big massive particles are
submerged into an ideal gas of light particles in the open space.
Another possibility is to take only one light particle but place it
in a semi-open container, such as a halfplane or a section of the
plane between two intersecting lines, cf.\ \cite{CDK}. In this case
an analysis similar to the one given in Section~\ref{subsecBA} leads
us to a diffusion equation for the big particle(s), but in contrast
with the single particle case \cite{DGL1}, a typical light particle
collides with several heavy ones before escaping to infinity. So the
coefficients of the corresponding transport equations are sum of
infinite series. Unfortunately our method cannot be applied to this
case yet, because of the lack of necessary results about mixing
properties of open billiards, but once such results become available
(see \cite{CMT} for a discussion of a simplified model), our method
could be used for the study of the well-posedness of transport
equations.

We summarize our discussion as follows:

\begin{itemize}
\item Transport coefficients are given by infinite correlation
sums. \item The regularity of the transport coefficients plays an
important role in proving well-posedness of the corresponding
transport equations. \item There is an experimental evidence that
for billiard problems transport coefficients are not smooth, but
this has yet to be established analytically.
\end{itemize}

We now turn to our primary goal -- proving the log-Lipschitz
\index{Log-Lipschitz continuity} continuity of the diffusion \index{Diffusion matrix} matrix,
as claimed by (\ref{sigQ1Q2}). Let $A$ and $B$ be smooth functions
on the $r\varphi$ coordinate chart $\Omega_0$ such that $\int
A\,d\mu_0=\int B\,d\mu_0=0$ (in fact, it is enough that one integral
vanishes). Recall that the spaces $\Omega_Q$ are identified with the
$\Omega_0$, hence our functions $A,B$ are defined on $\Omega_Q$. For
every $Q\in\cD$ such that $\dist(Q,\dcD)\geq\br+\delta$ we put
\begin{equation}
  \label{EqSigmabarA}
  \brsigma^2_{Q}(A,B) = \sum_{j=-\infty}^{\infty} \int_{\Omega_Q}
  A\, \left (B\circ\cF_{Q}^j\right )\,d\mu_{Q}.
\end{equation}
(Here we let $j$ change from $-\infty$ to $+\infty$
but of course our result is valid for one sided sums as well).

\begin{proposition}
\label{PrRegDif} Under Assumption A3', for all $Q_1\approx Q_2$ we
have
\beq
    \bigl|\brsigma^2_{Q_1}(A,B)-\brsigma^2_{Q_2}(A,B)\bigr|
      \leq    \Const\, \norm{Q_1-Q_2} \,
          \Big |\ln\norm{Q_1-Q_2}\,\Big |
        \label{sigQ1Q2AB}
\eeq
\end{proposition}

The bound (\ref{sigQ1Q2AB}), along with Assumption A4 on the
nonsingularity of the matrix $\sigma^2_Q(\cA)$, immediately
implies the required (\ref{sigQ1Q2}), so it remains to prove
Proposition~\ref{PrRegDif}.

\subsection{Reduction to a finite series}
We first discuss our approach to the problem. For smooth uniformly
hyperbolic systems, the dynamical invariants, such as diffusion
coefficients, are usually differentiable with respect to the
parameters of the model \cite{KKPW}. In dispersing \index{Dispersing
billiards} billiards, the presence of singularities makes the
dynamics {\em nonuniformly} hyperbolic, and for such systems, no
similar results are available. On the contrary, there is an
experimental evidence (supported by heuristic arguments) that
dynamical invariants are, generally, not differentiable, see
\cite{BDL}. Our proposition is the first positive result in this
direction.

Due to the identification of $\Omega_Q$ with $\Omega_0$, all the maps
$\cF_Q$ act on the same space $\Omega_0$ and preserve the same measure
$d\mu_0=c^{-1}\cos\varphi\, dr\, d\varphi$, where
$c=2\,\length(\dcD)+4\pi\br$ is the normalizing factor. Hence we can
treat $\cF_{Q_2}$ as a perturbation of $\cF_{Q_1}$ in (\ref{sigQ1Q2}).

There are two equivalent approaches to establish the regularity of
dynamical invariants for hyperbolic maps under perturbations. The
analytic approach consists of term-by-term differentiation of the
relevant infinite series, like our (\ref{EqSigmabarA}), with respect
to the parameters of the model (in our case, it is $Q$) and then
integrating by parts. The geometric approach is based on an explicit
comparison of the orbits under the two maps and using a
shadowing-type \index{Shadowing} argument.

The analytic method is shorter, if somewhat less transparent,
since it involves algebraic manipulations instead of geometric
considerations. However its range of applicability is rather
narrow, because it requires differentiability at all the relevant
values of parameters, whereas the geometric method is more
flexible and may handle less regular parameterizations. For this
reason we have used the geometric approach in the proof of
Proposition~\ref{PrDistEq2} because we treated the dynamics as a
small perturbation of the $M=\infty$ case, where we had no
analyticity (in $M$). For the proof of Proposition~\ref{PrRegDif}
we choose the analytic method, but we hope that after the proof of
Proposition~\ref{PrDistEq2} the geometric meaning of our
manipulations is clear.

The main difference between the proofs of
Propositions~\ref{PrRegDif} and \ref{PrDistEq2} is that in the
latter we had a luxury of discarding orbits which come close to
the singularities, but now we have to take them into account as
well. The contribution of those orbits is described by certain
triple correlation functions
$$
    \beta_{m,n}=\int_S (A\circ\cF_Q^{-m})\,B\,
    (C\circ\cF_Q^n)\,d\nu
$$
where $A,B,C$ are smooth functions and the measure $\nu$ is
concentrated on a singularity curve $S$ of the map $\cF_Q$. If
$\nu$ were a smooth measure on the entire space $\Omega_0$, then
we could use a local product structure to show, in a usual way, an
asymptotic independence of the future and the past, and the triple
correlations could be bounded as in \cite{CD}, so that
$\sum_{m,n\geq 0} |\beta_{m,n}| <\infty$. However, $\nu$ is
concentrated on a curve $S$, which has no local product structure,
and the series $\sum_{m,n\geq 0} |\beta_{m,n}|$ appears to
diverge. In fact, we are only able to get an estimate growing with
the number of terms: $\sum_{m+n\leq N} |\beta_{m,n}| = \cO(N)$,
and it is this estimate that gives us the logarithmic factor in
(\ref{sigQ1Q2AB}).

For brevity, we will write $\Omega=\Omega_0$ and $\mu=\mu_0$. Let
$$
      D_{A,B}^{(N)}(Q) = \sum_{n=1}^N
      \mu\bigl[A\, (B\circ\cF_Q^n)\bigr]
%    \int_{\Omega}A\,\left(B\circ\cF_Q^n\right )\, d\mu
$$
Our main estimate is
\beq
      \left|\frac{d D_{A,B}^{(N)}(Q)}{d Q}\right|
      \leq \Const_{A,B}\, N
         \label{MaindDN}
\eeq
where $d/dQ$ denotes the directional derivative along an arbitrary unit
vector in the $Q$ plane.

\medskip\noindent{\it Proof of Proposition \ref{PrRegDif}.} According to
a uniform exponential bound on correlations (Extension~1 in
Section~A.1),
$$
       \Bigl|\mu\bigl[ A\,(B\circ \cF_Q^n)\bigr]\Bigr|
       \leq \Const_{A,B} \, \theta^{|n|}
$$
for some $\theta<1$. Let $N=2\,\ln\|Q_1-Q_2\|/\ln\theta.$ Then for any
$Q$
$$
    \brsigma^2_Q=\mu(AB)
    +D_{A,B}^{(N)}(Q)+D_{B,A}^{(N)}(Q)
    +\cO\left(\|Q_1-Q_2\|^2\right)
$$
where the first term does not depend on $Q$. Hence
\begin{align*}
        \brsigma^2_{Q_1}-\brsigma^2_{Q_2}
        &=\bigl[D_{A,B}^{(N)}(Q_1)+D_{B,A}^{(N)}(Q_1)
        -D_{A,B}^{(N)}(Q_2)-D_{B,A}^{(N)}(Q_2)\bigr]\\
        &\quad +\cO\bigl(\|Q_1-Q_2\|^2\bigr)
\end{align*}
The main estimate (\ref{MaindDN}) implies that the expression in
brackets is bounded by $\Const_{A,B}\|Q_1 -Q_2\|\, N.$ \qed
\medskip

The rest of this section is devoted to proving the main estimate
(\ref{MaindDN}).

\subsection{Integral estimates: general scheme}
\label{subsecIEGS} Let
$$
     I_n=\frac{d}{dQ}\,\mu\bigl[ A\,(B\circ \cF_Q^n) \bigr]
$$
Since the rest of the proof deals with $\cF_Q$ for a fixed $Q$ we shall
omit the subscript from now on. We also put $A_n=A\circ\cF^n$. Note
that $A$ is smooth on $\Omega$ but $A_n$ has discontinuities on the
singularity set $\cS_n\subset\Omega$ of the map $\cF^n$. The curves of
$\cS_n$ change with $Q$ smoothly, so we have
$$
       I_n = I_n^{(c)}+I_n^{(d)}
$$
where the first term contains the derivative of the integrand
$$
    I_n^{(c)}=\int_{\Omega} A\,\frac{dB_n}{dQ}\, d\mu
$$
and the second one contains the boundary integrals
\beq
    I_n^{(d)}=\int_{\cS_n\setminus \cS_0}
    A\, (\Delta B_n)\,
    v^{\perp}\cos\varphi\, dl
      \label{Ind}
\eeq
where $\Delta B_n$ denotes the jump of $B_n$ across
$\cS_n\setminus\cS_0$, $v^{\perp}$ is the velocity of
$\cS_n\setminus\cS_0$ as it changes with $Q$ (in the normal
direction), and $dl$ the Lebesgue measure (length) on $\cS_n$.
Observe that $\cS_0=\partial\Omega$ does not change with $Q$,
hence it need not be included in $I_n^{(d)}$. We postpone the
analysis of the boundary terms until Section~\ref{subsecABT1}.

Now consider a vector field on $\Omega$ defined by
$$
        X=\frac{d\cF}{dQ}\circ \cF^{-1}
$$
For an $x=(r,\varphi)$ such that either $x$ or $\cF^{-1}(x)$ lies on
$\dcP(Q)\times [-\pi/2 ,\pi/2]$ $X$ vanishes, otherwise $X$ is an
unstable \index{Unstable vectors} vector with coordinates
\beq
   X=(dr_X,d\varphi_X)=\left(\frac{\sin(\varphi+\psi)}{\cos\varphi},
   \,\cK\,\frac{\sin(\varphi+\psi)}{\cos\varphi}\right)
          \label{Xfield}
\eeq
where $\cK>0$ is the curvature of the boundary $\dcD\cup\dcP(Q)$
at $x$, and $\psi$ is the angle between the normal to the boundary
at the point $x$ and the direction of our derivative $d/dQ$. Note
that $\|X\|=\cO(1/\cos\varphi)$ is unbounded but $\mu(\|X\|) <
\infty$, because the density of $\mu$ is proportional to
$\cos\varphi$. It is also easy to check that $\mu \bigl(
\|d\cF^k(X)\| \bigr) < \infty$ for all $k\geq 1$.

It is now immediate that
$$
      I_{n}^{(c)}=\sum_{k=0}^{n-1} I_{n,k}^{(c)}
$$
where
\beq
    \label{Inkc}
    I_{n,k}^{(c)}=\int_{\Omega} A\, \bigl[\left(\partial_{d\cF^k(X)}
    B\right)\circ\cF^n\bigr]\, d\mu
\eeq
Note that $d\cF^k(X)$ grows exponentially fast with $k$. To properly
handle these integrals, we will decompose $d\cF^k(X)$ into stable and
unstable components.

Let $E^u=E^u(x)$ be the field of unstable directions given by
equation $d\varphi/dr =\cK$, and $E^s=E^s(x)$ be the field of
stable directions given by equation $d\varphi/dr =-\cK$. The field
$E^u$ corresponds to (infinitesimal) families of trajectories that
are parallel before the collision at $x$, and $E^s$ corresponds to
families of trajectories that are parallel after the collision.
Note that in contrast with more common notation $E^u$ and $E^s$ are
not invariant under dynamics.
Rather they are smooth vector fields such that
$d\cF(E^s)=E^u$ and $X$ belongs in
$E^u$.

Let $\cG$ be a smooth foliation of $\Omega$ by \index{u-curves
(unstable curves)} u-curves that integrate the field $E^u$. Then
$\cG_m=\cF^m(\cG)$, for $m\geq 0$, is a piecewise smooth foliation
by u-curves that integrate the field $E^u_{m}=d\cF^m (E^u)$. Note
that the discontinuities of $\cG_m$ coincide with those of the map
$\cF^{-m}$. Let $\Pi^u_{m}$ and $\Pi^s_{m}$ denote the projections
onto $E^u_{m}$ and $E^s$, respectively, along $E^s$ and $E^u_m$. Let
$\Theta^{\ast}_{m}=\Pi^{\ast}_{m} \circ d\cF$, where $\ast=u,s$. For
$k> m\geq 0$, let
$$
    \Theta^s_{m,k}=\Theta^s_{k}\circ\dots\circ
     \Theta^s_{m+2}\circ \Theta^s_{m+1}
$$

\begin{lemma}
There is a constant $\theta<1$ and a function $u(x)$ on $\Omega$
such that for any nonzero vector $dx\in E^s$ and $m\geq 0$ we have
\beq
    \frac{\|\Theta^s_{m}(dx)\|}{\|dx\|} \leq
    \theta\, \frac{u(\cF(x))}{u(x)}
      \label{contract1}
\eeq
The function $u(x)$ is bounded away from zero and infinity:
$$
    0< u_{\min} < u(x) < u_{\max} < \infty,
$$
therefore, for any $0\neq dx \in E^s$
\beq
    \frac{\|\Theta^s_{m,k} (dx)\|}{\|dx\|} \leq
    \theta^{k-m}\, \frac{u(\cF^{k-m}(x))}{u(x)}
    \leq  \theta^{k-m}\, \frac{u_{\max}}{u_{\min}}
      \label{contract2}
\eeq
for any $k>m$. Lastly, $\|\Pi^s_m(X)\|\leq\Const$. \label{LmContract}
\end{lemma}

\proof We denote $x=(r,\varphi)$ and $\cF(x)=x_1=(r_1,\varphi_1)$.
Note that the vectors $dx=(dr,d\varphi)\in E^s(x)$ and $d\cF(dx)
=dx_1= (dr_1, d\varphi_1)\in E^u(x_1)$ correspond to an
(infinitesimal) family of trajectories that remain parallel
between the collisions at $x$ and $x_1$, and this family is
characterized by the vector $dq$ (orthogonal to the velocity
vector) and $dv=0$, in the notation of Section~\ref{subsecSUV},
and we have
$$
     \|dq\|= |\cos\varphi\, dr|= |\cos\varphi_1\, dr_1|
$$
Recall that the norm of s-vectors is defined by
(\ref{dxnormQVst}), where $\bs_V=1$ since $V=0$, hence
$$
  \|dx\|^2 = (4\cK^2+\cos^2\varphi)(dr)^2
$$
and for the s-vector $\Theta^s_{m}(dx)= dx_1^s= (dr_1^s,
d\varphi_1^s)\in E^s(x_1)$ we have
$$
  \|dx_1^s\|^2 = (4\cK_1^2 +\cos^2\varphi_1) (dr_1^s)^2
$$
where $\cK$ and $\cK_1$ denote the curvature of the boundary at
$x$ and $x_1$, respectively. Next, the vector $dx_1^s$ is the
projection of $dx_1$ onto $E^s(x_1)$ along $E^u_m(x_1)$, the
latter is given by equation $d\varphi/dr =\cK_1
+\cB_1\cos\varphi_1$, where $\cB_1$ is the curvature of the
precollisional family of trajectories corresponding to
$E^u_m(x_1)$.
By a direct inspection, see Fig.~\ref{FigDecomp}, we
have
\beq
   |dr_1^s| = |dr_1|\,
   \frac{\cB_1\cos\varphi_1}{2\cK_1+\cB_1\cos\varphi_1}
     \label{dr1sdr1}
\eeq
hence
\beq
     \frac{\|dx_1^s\|^2}{\|dx\|^2} =
     \frac{4\cK_1^2 +\cos^2\varphi_1}{4\cK^2+\cos^2\varphi}
     \times \frac{\cos^2\varphi}{(2\cB_1^{-1}\cK_1+\cos\varphi_1)^2}
       \label{longexact}
\eeq
Note that $0<\cB_1\leq 1/s$, where $s$ the free path length
between the points $x$ and $x_1$, hence $\cB_1 \leq 1/L_{\min}$.
Thus
\beq
     \frac{\|dx_1^s\|^2}{\|dx\|^2} \leq
     \frac{4\cK_1^2 +\cos^2\varphi_1}{4\cK^2+\cos^2\varphi}
     \times \frac{(c_0+\cos\varphi)^2}{(c_0+\cos\varphi_1)^2}
     \times \frac{\cos^2\varphi}{(c_0+\cos\varphi)^2}
       \label{longuu}
\eeq
where $c_0=2\,L_{\min}\cK_{\min}>0$. Now we put
$u(x)=(4\cK^2+\cos^2\varphi)^{1/2} /(c_0+\cos\varphi)$ and
$\theta=1/(c_0+1)$, which proves (\ref{contract1}). Replacing
$x_1$ by $x$ and $dx_1$ by $X=(dr_X,d\varphi_X)$, see
(\ref{Xfield}), in the above argument gives an estimate for the
vector $\Pi_m^s(X)=(dr^s,d\varphi^s)$:
\begin{align*}
  \|\Pi_m^s(X)\|^2 &= (4\cK^2+\cos^2\varphi)(dr^s)^2\\
  &= \frac{(\cB\cos\varphi)^2}{(2\cK+\cB\cos\varphi)^2}
  (4\cK^2+\cos^2\varphi)(dr_X)^2\\
  &\leq \frac{4\cK_{\max}^2+1}{4\cK_{\min}^2/\cB_{\max}^2}
    \qquad\qquad\qquad\qquad\qquad\square
\end{align*}

\begin{figure}[htb]
    \centering
    \psfrag{f}{$\varphi$}
    \psfrag{r}{$r$}
    \psfrag{x}{$x_1$}
    \psfrag{Es}{$E^s$}
    \psfrag{Eu}{$E^u$}
    \psfrag{Em}{$E^u_m$}
    \psfrag{F}{$d\cF(dx)$}
    \psfrag{Tu}{$\Theta^u_m(dx)$}
    \psfrag{Ts}{$\Theta^s_m(dx)$}
    \includegraphics{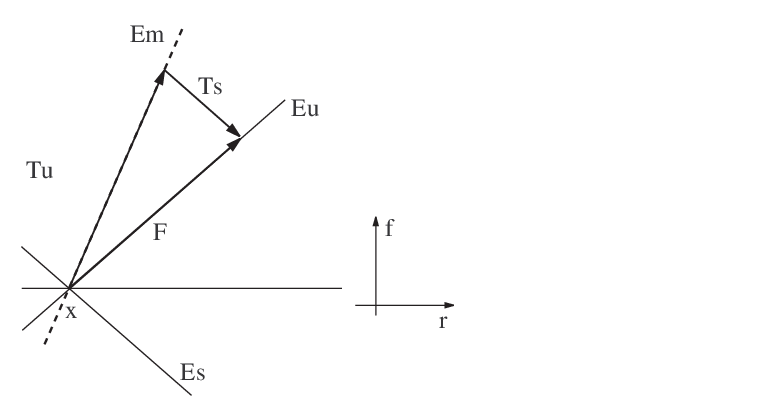}
    \caption{The decomposition $d\cF(dx)=
    \Theta^u_{m}(dx)+\Theta^s_{m}(dx)$}
    \label{FigDecomp}
\end{figure}

\medskip\noindent{\em Remark}. As (\ref{dr1sdr1}) implies,
$$
   \|d\cF(dx)\|\leq\Const\,\frac{\|\Theta_m^s(dx)\|}{\cos\varphi_1}
$$
hence
\beq
   \frac{\|\Theta_{k+1}^u\circ\Theta_{m,k}^s(dx)\|}{\|dx\|}
   \leq\Const\,\frac{\theta^{k-m}}{\cos\varphi_{k-m+1}}
     \label{dxusdxcos}
\eeq
where we denote $\cF^{k-m+1}(x) = (r_{k-m+1}, \varphi_{k-m+1})$.
Since $\cF^{-1}$ uniformly contracts u-vectors by a factor
$\cO(\cos\varphi)$, then
\beq
   \frac{\|d\cF^{-1}\circ\Theta_{k+1}^u\circ\Theta_{m,k}^s(dx)\|}{\|dx\|}
   \leq\Const\,\theta^{k-m}
     \label{dxusdx}
\eeq

\medskip\noindent{\em Remark}. For a future reference, we record a
slight improvement of the estimate (\ref{contract1}):
\beq
     \forall dx\in E^s\qquad
    \frac{\|\Theta^s_{m}(dx)\|}{\|dx\|} \leq
    \theta\, R(\cos\varphi)\, \frac{u(\cF(x))}{u(x)}
      \label{contract1a}
\eeq
where $R(\cos\varphi) = \min\{1,C_0\cos\varphi\}$ and
$C_0=1+c_0^{-1}$. This improvement follows from
(\ref{longuu}).\medskip

We now return to the integral (\ref{Inkc}). Let us decompose
$$
         X = \bralpha_{n-k,0} + \beta_{n-k,0}
$$
where
$$
   \bralpha_{n-k,0}=\Pi_{n-k}^u(X)\in E^u_{n-k},\qquad
   \beta_{n-k,0}=\Pi_{n-k}^s(X)\in E^s
$$
and, inductively, for $r\geq 1$,
$$
      d\cF^r(X) = \bralpha_{n-k,r} + \beta_{n-k,r}
$$
where
$$
      \bralpha_{n-k,r} = d\cF(\bralpha_{n-k,r-1}) + \Theta_{n-k+r}^u
      (\beta_{n-k,r-1}) \in E^u_{n-k+r}
$$
and
$$
     \beta_{n-k,r} =\Theta_{n-k+r}^s
      (\beta_{n-k,r-1}) \in E^s
$$
Observe that
$$
    \beta_{n-k,k} = \Theta^s_{n-k,n}(\beta_{n-k,0})
$$
and
$$
  \bralpha_{n-k,k} = d\cF^k(\bralpha_{n-k,0}) +
  \sum_{j=0}^{k-1} d\cF^j\circ\Theta^u_{n-j}(\beta_{n-k,k-j-1})
$$
Denote
$$
  \alpha_{n-k,k-j}=\Theta^u_{n-j}(\beta_{n-k,k-j-1})
  \in E^u_{n-j}
$$
If we also put, for convenience of notation, $\alpha_{n-k,0}
=\bralpha_{n-k,0}$ and denote
$$
  \alpha_{n-k,k-j}^{(m)} = d\cF^m(\alpha_{n-k,k-j})
  \qquad \forall\ m\in\integers
$$
then we obtain
\beq
    \label{dFkX}
   d\cF^k(X)=\sum_{j=0}^{k} \alpha_{n-k,k-j}^{(j)}
    +\beta_{n-k,k}
\eeq
Accordingly,
$$
    I_{n,k}^{(c)}=\sum_{j=0}^{k} I_{n,k,j}^{(u)} + I_{n,k}^{(s)}
$$
where
\beq
   I_{n,k,j}^{(u)}=
   \int_\Omega A\,\Bigl[(\partial_{\alpha_{n-k,k-j}^{(j)}}
   B) \circ \cF^n\Bigr]\, d\mu
     \label{Inkju0}
\eeq
and
\begin{equation}
\label{Inks}
   I_{n,k}^{(s)}=
   \int_\Omega A\,\Bigl[(\partial_{\beta_{n-k,k}}
   B)\circ \cF^n\Bigr] d\mu
\end{equation}

\begin{lemma}
\label{Lmalphabeta}
There is a constant $\theta<1$ such that for
all $0 \leq k < n$
\begin{align*}
  \|\alpha_{n-k,k-j}\| &\leq\Const\,\theta^{k-j}/\cos\varphi,\\
  \|\alpha_{n-k,k-j}^{(-m)}\| &\leq\Const\,\theta^{k-j+m}
  \qquad \forall\ m\geq 1
\end{align*}
and
$$
  \|\beta_{n-k,k}\|\leq\Const\,\theta^k.
$$
\end{lemma}

\proof Use Lemma~\ref{LmContract} and the subsequent remark. \qed

\begin{corollary}
\label{CrRD1}
$$
       \left| \sum_{n=1}^N\sum_{k=0}^{n-1}
       I_{n,k}^{(s)} \right|\leq \Const\, N.
$$
\end{corollary}

\proof  Estimating the integrand in (\ref{Inks}) by its absolute value we get

\noindent
$\left|I_{n,k}^{(s)}\right|\leq
\Const\,\|A\|_{\infty}\,\|B\|_{C^1}\,\theta^k. $ \qed

\subsection{Integration by parts}
The estimation of $I_{n,k,j}^{(u)}$ in (\ref{Inkju0}) requires
integration by parts. Changing variables $y=\cF^{n-j-1} x$ gives
\beq
        \label{Inkju}
   I_{n,k,j}^{(u)}=\int_{\Omega}
   A_{-(n-j-1)}\,
   \left(\partial_{\alpha_{n-k,k-j}^{(-1)}}
   B_{j+1}\right)\,d\mu
\eeq
(we have to work with $d\cF^{-1}\alpha_{n-k,k-j}$, instead of
$\alpha_{n-k,k-j}$ to avoid an infinite growth of the latter as
$\cos\varphi\to 0$, see Lemma~\ref{Lmalphabeta}).

Observe that the integrand in (\ref{Inkju}) is discontinuous on
the set $\cS_{-(n-j-1)}$ (due to $A_{-(n-j-1)}$ and the vector
field) and $\cS_{j+1}$ (due to $B_{j+1}$), hence we have to
integrate by parts on each connected domain
$D\subset\Omega\setminus (\cS_{-(n-j-1)}\cup\cS_{j+1})$.

Observe that $\alpha_{n-k,k-j}^{(-1)}\in E^u_{n-j-1}$, hence the
integral curves of this vector field are the fibers of the
foliation $\cG_{n-j-1}$. For every domain $D$, denote by
$\cG_D=\{\gamma_D\}$ the fibers of this foliation restricted to
$D$ and by $\rho$ the densities of the corresponding conditional
measures on them. Let $\lambda_D$ denote the factor measure. To
simplify our notation, we put $A_-=A_{-(n-j-1)}$, $B_+=B_{j+1}$
and $\alpha =\|\alpha_{n-k,k-j}^{(-1)}\|$. For any curve $\gamma$,
let $\int_\gamma C\, dx$ denote the integral of a function $C$
with respect to the arclength parameter on $\gamma$, and $C' =
\partial C / \partial x$ denote the derivative along $\gamma$. Then
the integration by parts gives
\begin{align}
   I_{n,k,j}^{(u)} &=
      \sum_D\int_{\cG_D} d\lambda_D \int_{\gamma_D}
   A_-\,\alpha\, B_+'\,\rho\,dx
   \nonumber\\
   &= I_{n,k,j}^{(b)}-I_{n,k,j}^{(i)}
   \label{intA-A+}
\end{align}
where
\beq
      \label{Inkjb}
    I_{n,k,j}^{(b)}=
    \int_{(\cS_{-(n-j-1)}\cup\cS_{j+1})\setminus\cS_0}
    \Delta\left[ A_-\,B_+\, \|\alpha^\perp\|_0\right]
    \cos\varphi\,dl
\eeq
is the boundary term, which will be analyzed in the next subsection
(note that we exclude the set $\cS_0=\{(r,\varphi)\colon
\cos\varphi=0\}$ since $\rho=0$ on $\cS_0$) and
\begin{equation}
\label{Inkji}
   I_{n,k,j}^{(i)}=\sum_D\int_{\cG_D} d\lambda_D \int_{\gamma_D}
   (\rho\,\alpha\,A_-)'\,B_+\,dx .
\end{equation}

Observe that
$$
  (\rho\,\alpha\,A_-)'
  =A_-'\,\alpha\,\rho+A_-\,\alpha\,(\ln\rho)'\,\rho
  +A_-\,\alpha'\,\rho,
$$
hence the last sum in (\ref{intA-A+}) equals
$$
    \int_{\Omega}\alpha\,A_-'\,B_+\,d\mu
    +\int_{\Omega}\alpha\,(\ln\rho)'\,A_-\,B_+\,d\mu
    +\int_{\Omega}\alpha'\,A_-\,B_+\,d\mu.
$$
These integrals will be estimated in the next two lemmas.

\begin{lemma}
\label{Lm3} For some constant $\theta<1$, we have
\beq
    \biggl|\int_{\Omega}\alpha\,A_-'\,B_+\,d\mu\biggr|
    \leq \Const\, \theta^{n-j},
       \label{3-1}
\eeq
\beq
    \left|\int_{\Omega}\alpha\,(\ln\rho)'\,A_-\,B_+\,d\mu\right|
    \leq \Const\, \theta^{k-j},
       \label{3-2}
\eeq
\beq
    \left|\int_{\Omega}\alpha'\,A_-\,B_+\,d\mu\right|
    \leq \Const\, \theta^{k-j}.
       \label{3-3}
\eeq
\end{lemma}

\proof Since $\cF^{-1}$ contracts unstable curves by a factor
$\leq\vartheta<1$, we have $\|A_-'\|\leq\Const\,\vartheta^{n-j}$, which
proves (\ref{3-1}).

Lemma~\ref{Lmalphabeta} implies that $\alpha\leq\Const\,\theta^{k-j}$,
and so
$$
    \left|\int_{\Omega}\alpha\,(\ln\rho)'\,A_-\,B_+\,d\mu\right|
    \leq \Const\, \theta^{k-j}
    \int_{\Omega}\left|(\ln\rho)'\right|\,d\mu
$$
To show that the last integral is finite, we first need to refine
our foliations $\cG_m$, $m\geq 0$. We divide the fibers of the
original foliation $\cG$ into \index{H-curves} H-curves (by cutting
\index{Homogeneity strips} them at the boundaries of the homogeneity
strips) and denote the resulting family of H-curves by $\tcG$. For
$m\geq 0$, let $\tcG_m$ denote the foliation of $\Omega$ into the
H-components of the sets $\cF^m(\gamma)$, $\gamma\in\tcG$, see
\index{H-components} Section~\ref{subsecSHUC} (note that $\tcG_m$ is
a refinement of $\cG_m$). Denote by $\gamma_{m}(x)$ the fiber of
$\tcG_{m}$ that contains the point $x$. Now by (\ref{densC3})
\beq
     \Bigl|[\ln\rho(x)]'\Bigr| \leq
     \frac{\Const}{|\gamma_{n-j-1}(x)|^{2/3}}
        \label{rho23}
\eeq
Lemma~\ref{prgrow} implies that $\mu\{x\colon\,
|\gamma_m(x)|<t\}\leq\Const\, t$ for every $m\geq 0$, hence
$$
       \int_{\Omega}\frac{1}{|\gamma_{n-j-1}(x)|^{2/3}}\,d\mu
       \leq \Const\int_0^1 t^{-2/3}\, dt\leq \, \Const
$$
which proves (\ref{3-2}). To derive (\ref{3-3}) we will show that
\beq
     |\alpha'(x)| \leq
     \frac{\Const\,\theta^{k-j}}
     {|\gamma_{n-j}(\cF(x))|^{2/3}}
        \label{alpha23}
\eeq
where $\theta<1$ is a constant. Then (\ref{3-3}) will follow by
$$
       \int_{\Omega}\frac{\theta^{k-j}}
       {|\gamma_{n-j}(\cF(x))|^{2/3}}\,d\mu
       \leq \Const\,\theta^{k-j}\int_0^1 t^{-2/3}\, dt
       \leq \, \Const\,\theta^{k-j}
$$
where we changed variable $y=\cF(x)$ and used the invariance of $\mu$.

It remains to prove (\ref{alpha23}). For $j=k$, we have $\alpha
=\|\alpha_{n-k,0}^{(-1)}\|$. The vector $\alpha_{n-k ,0}^{(-1)}$ is the
projection of $d\cF^{-1}X$ onto $E^u_{n-k-1}$ along $E^s_{-1}\colon
=d\cF^{-1}(E^s)$. Similarly to (\ref{Xfield}), we have
$$
   d\cF^{-1}X=(dr,d\varphi)=\left(\frac{\sin(\varphi-\psi)}{\cos\varphi},
   \,-\cK\,\frac{\sin(\varphi-\psi)}{\cos\varphi}\right)
$$
hence the vector field $(\cos\varphi)\, d\cF^{-1}X$ is $C^2$
smooth, with uniformly bounded first and second derivatives on
$\Omega$. For brevity, we will say that a function is {\em
uniformly $C^2$ smooth}, if its first and second derivatives are
bounded by some constants determined by the domain $\cD$, by
$\delta_0$ in (\ref{Upsilon}), and by our functions $A$ and $B$.
The field $E^u_{n-k-1}$ is uniformly $C^2$ smooth along the fibers
of $\tcG_{n-k-1}$ by Proposition~\ref{PrDistCurv}. The field $E^s$
is given by equation $d\varphi/dr =-\cK$, so it is uniformly $C^2$
smooth on $\Omega$. By using basic facts about billiards, cf.\
Appendices~A and B, and direct calculation we find that the vector
field $E^s_{-1}$ is given by equation
$$
   d\varphi/dr =-\cK -2\cK_1\cos\varphi/
     (2s\cK_1+\cos\varphi_1)
$$
where $x_1=(r_1,\varphi_1)=\cF(x)$ and $\cK_1$ denotes the
curvature of the boundary at the point $x_1$. Note that the lines
$E^s(x)$ and $E^s_{-1}(x)$ have slopes bounded away from $0$ and
$-\infty$, and the difference between these slopes is
\beq \label{EEangle}
     \measuredangle \bigl( E^s(x),E^s_{-1}(x) \bigr)
     = \cO(\cos\varphi)
\eeq
If $\hgamma(x)$ denotes the angle between $E^s(x)$ and
$E^s_{-1}(x)$, then $(\cos\varphi)^{-1}\hgamma(x)$ can be given by
a formal expression (in terms of $x$ and $x_1$) that would be a
uniformly $C^2$ smooth function of $x$ and $x_1$. However, if we
differentiate $(\cos\varphi)^{-1}\hgamma(x)$ with respect to $x$
along the fibers of the unstable foliation $\tcG_{n-k-1}$, then
$x_1$ becomes a function of $x$ such that
$|dx_1/dx|=\cJ(x)=\cO(1/\cos\varphi_1)$, where $\cJ(x)$ is the
Jacobian of the map $\cF\colon\gamma_{n-k-1} (x)\to\gamma_{n-k}
(x_1)$. Hence
$$
       \left| \frac{d}{dx}\,
       \left[(\cos\varphi)^{-1}\hgamma(x)\right]\right|
        \leq \frac{\Const}{\cos\varphi_1}
$$
This gives us an estimate for the derivative along the fibers of
$\tcG_{n-k-1}$:
\beq
   \left|\frac{d\alpha(x)}{dx}\right| =
   \left|\frac{d\alpha(x)}{dx_1}\,\frac{dx_1}{dx}\right|
   \leq \frac{\Const}{\cos\varphi_1}
   \leq \frac{\Const}{|\gamma_{n-k}(\cF(x))|^{2/3}}
     \label{alpha0'}
\eeq
where the last inequality follows from (\ref{cos23}), which proves
(\ref{alpha23}) for $j=k$. For a future reference, we also note
that
\beq
   \left|\frac{d^2\alpha(x)}{dx^2}\right| =
   \left|\frac{d[d\alpha/dx]}{dx_1}\,\frac{dx_1}{dx}\right|
   \leq \frac{\Const\, \cJ(x)}{|\gamma_{n-k}(\cF(x))|^{4/3}}
      \label{secder0}
\eeq

\begin{figure}[htb]
    \centering
    \psfrag{f}{$\varphi$}
    \psfrag{r}{$r$}
    \psfrag{x}{$x_{t-1}$}
    \psfrag{Es}{$E^s$}
    \psfrag{Eu}{$E^u$}
    \psfrag{E1}{$E^s_{-1}$}
    \psfrag{E2}{$E^u_{n-k+t-1}$}
    \psfrag{a}{$\alpha_{n-k,t}^{(-1)}$}
    \psfrag{b}{$\beta_{n-k,t-1}$}
    \includegraphics{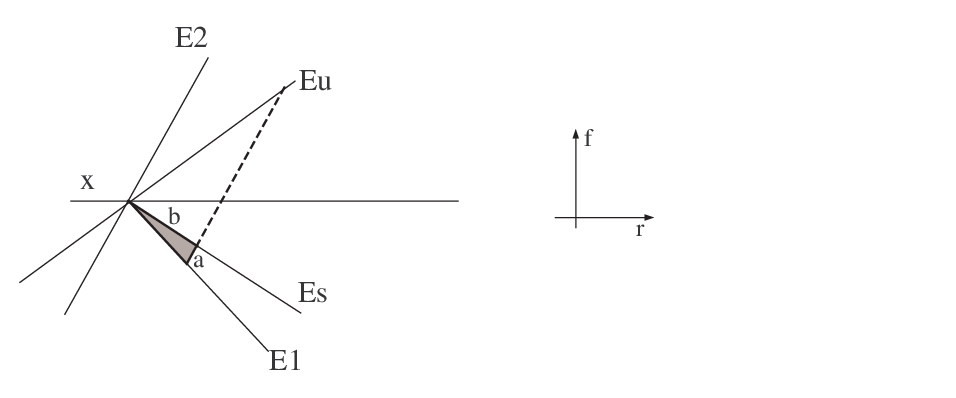}
    \caption{The vector $\alpha_{n-k,t}^{(-1)}$
    is parallel to $E^u_{n-k+t-1}$}
    \label{FigAlphabeta}
\end{figure}

To prove (\ref{alpha23}) for $j<k$ we use induction on $t\colon
=k-j$. Let $\alpha_t =\|\alpha_{n-k,t}^{(-1)}\|$ and $\beta_t
=\|\beta_{n-k,t}\|$. Consider the trajectory $x_t= (r_t ,
\varphi_t) =\cF^t (x)$ of a point $x$. Observe that the vector
$\alpha^{(-1)}_{n-k,t}$ is parallel to the line $E^u_{n-k+t-1}$;
and $\alpha_t$, $\beta_{t-1}$ are two sides of a triangle (shaded
on Fig.~\ref{FigAlphabeta}), in which one angle is $\cO( \cos
\varphi_{t-1} )$, cf.\ (\ref{EEangle}). Therefore,
\beq
    \alpha_t =E_{t-1}\beta_{t-1}\cos\varphi_{t-1}
\eeq
where the factor $E_{t-1}$ is bounded away from zero and infinity:
$$
   0<E_{\min} \leq E_{t-1} \leq E_{\max}<\infty;
$$
and $E_{t-1}$ can be given by a formal expression (in terms of
$x_{t-1}$, $x_t$, and the slope $\Gamma^u_{t-1} =d\varphi/dr$ of
the line $E^u_{n-k+t-1}$ at the point $x_{t-1}$) that would be a
uniformly $C^2$ smooth function of the variables $x_{t-1}$, $x_t$,
and $\Gamma^u_{t-1}$. Now we have
\beq
      \alpha_{t+1} = G_t\, \alpha_t,\qquad
      G_t\colon =\frac{E_{t}\beta_{t}\cos\varphi_{t}}
      {E_{t-1}\beta_{t-1}\cos\varphi_{t-1}}.
        \label{alphatt+1}
\eeq
It follows from (\ref{longexact}) that
$$
   H_t\colon =
   \frac{\beta_{t}}{\beta_{t-1}\cos\varphi_{t-1}}
$$
is bounded away from zero and infinity and can be given by a
formal expression (in terms of $x_{t-1}$, $x_t$, and the curvature
$\cB_t$ of the precollisional family of trajectories corresponding
to $E^u_{n-k+t}$ at the point $x_{t}$) that would be a uniformly
$C^2$ smooth function of $x_{t-1}$, $x_t$, and $\cB_{t}$. Then
\beq \label{GHEE}
     G_t=\frac{E_tH_t\cos\varphi_t}{E_{t-1}}
\eeq
is a uniformly $C^2$ smooth function of $x_{t-1}$, $x_t$,
$x_{t+1}$, $\Gamma_{t-1}$, $\Gamma_t$, and $\cB_t$ (the variable
$x_{t+1}$ comes only from $E_t$). We note that $\Gamma_{t-1}$ and
$\Gamma_t$ are $C^2$ smooth functions of $x_{t-1}$ and $x_t$,
respectively, and $\cB_t$ is a uniformly $C^2$ smooth function
along the corresponding fiber of the foliation $\cG_{n-k+t}$, see
Appendix~B.

Now we differentiate (\ref{GHEE}) with respect to $x_t$ along the
corresponding fiber of the foliation $\cG_{n-k+t}$ and use
$|dx_{t+1}/dx_t| = \cO(1/\cos\varphi_{t+1})$ and
$|dx_{t-1}/dx_t|<1$ to obtain
\beq
     \left|\frac{dG_t}{dx_t}\right|
     \leq \frac{\brC}{\cos\varphi_{t+1}}
\eeq
with a constant $\brC>0$. Next, differentiating the identity
$\alpha_{t+1}=G_t\alpha_t$ gives
\beq
     \left|\frac{d\alpha_{t+1}}{dx_t}\right|
     \leq \frac{\brC\alpha_t}{\cos\varphi_{t+1}}
     +\left|\frac{E_{t}\beta_{t}\cos\varphi_{t}}
      {E_{t-1}\beta_{t-1}\cos\varphi_{t-1}}
      \,\frac{1}{\cJ_{t-1}}
      \frac{d\alpha_{t}}{dx_{t-1}}\right|
        \label{alpha'recursive}
\eeq
where $\cJ_{t-1}=dx_t/dx_{t-1}$ is the Jacobian of the map
$\cF\colon\gamma_{n-k+t-1}(x_{t-1})\to\gamma_{n-k+t}(x_t)$. Note that
by (\ref{contract1a}),
\beq
   \frac{\beta_t}{\beta_{t-1}} \leq
   \frac{u_t}{u_{t-1}}\,R(\cos\varphi_{t-1})
     \label{betatou}
\eeq
where $u_t=u(x_t)$. Now we prove, by induction on $t$ that there exists
a constant $\theta<1$ and large constants $P,Q>0$ such that
\beq
     \left|\frac{d\alpha_{t}}{dx_{t-1}}\right| \leq
     \frac{E_{t-1}u_{t-1}\cos\varphi_{t-1}}{R(\cos\varphi_{t-1})}
     \,\left(P+\frac{Q}{\cos\varphi_t}\right)\,\theta^t
       \label{alpha'ind}
\eeq
Note that the first factor here is bounded away from zero and infinity:
$$
    0< \frac{E_{\min}u_{\min}}{C_0} \leq
    \frac{E_{t-1}u_{t-1}\cos\varphi_{t-1}}{R(\cos\varphi_{t-1})}
    \leq E_{\max}u_{\max} < \infty
$$
For $t=0$, the bound (\ref{alpha'ind}) follows from
(\ref{alpha0'}). (Note that the same angle is denoted by
$\varphi_1$ in (\ref{alpha0'}) and by $\varphi_t=\varphi_0$ in
(\ref{alpha'ind}) with $t=0$.) Combining
(\ref{alpha'recursive})--(\ref{alpha'ind}) gives
\begin{align}
     \left|\frac{d\alpha_{t+1}}{dx_t}\right|
     &\leq \frac{\brC\alpha_t}{\cos\varphi_{t+1}}
     +E_{t}u_{t}\cos\varphi_t
     \,\left( P+\frac{Q}{\cos\varphi_t}\right)
     \frac{\theta^t}{\cJ_{t-1}}\nonumber\\
     &\leq \frac{\brC\alpha_t}{\cos\varphi_{t+1}}
     +\frac{E_{t}u_{t}\cos\varphi_t}{R(\cos\varphi_{t})}
     \,\left( P+C_0 Q\right)\frac{\theta^t}{\cJ_{t-1}}
       \label{dalphadx}
\end{align}
Since $\alpha_{t}<C_1\theta^{t}$ for some $C_1>0$ and $\theta<1$,
by Lemma~\ref{Lmalphabeta}, the first term in (\ref{dalphadx}) can
be bounded as
$$
  \frac{\brC\alpha_t}{\cos\varphi_{t+1}}\leq
  \frac{E_{t}u_{t}\cos\varphi_t}{R(\cos\varphi_{t})}\,
  \frac{C_0C_1\brC\,\theta^{t+1}}{\theta
  E_{\min}u_{\min}\cos\varphi_{t+1}}
$$
We can choose $\theta$ in Lemma~\ref{Lmalphabeta} so that
$\theta^2 > \vartheta$, then $\theta^2 > 1/\cJ_{t-1}$. Now we
select $P$ and $Q$ so that
$$
  \frac{C_0C_1\brC}{\theta E_{\min}u_{\min}}<Q,
  \qquad{\rm and}\qquad
  (P+C_0Q)\,\theta<P
$$
which completes the proof of (\ref{alpha'ind}) by induction. Now
(\ref{alpha23}) follows from (\ref{alpha'ind}) due to
(\ref{cos23}), and hence Lemma~\ref{Lm3} is proven. \qed\medskip

\noindent{\em Remark}. For a future reference, we record a bound,
similar to (\ref{secder0}), on the second derivative of $\alpha_t$
taken along the corresponding fiber of $\cG^u_{n-k+t-1}$:
\beq
     \left|\frac{d^2\alpha_{t}}{dx^2_{t-1}}\right| \leq
     \frac{\Const\,\theta^t}{\cos^3\varphi_{t}}
   \leq \frac{\Const\,\theta^t
   \cJ_{t-1}}{|\gamma_{n-k+t}(x_t)|^{4/3}},
       \label{alpha''}
\eeq
here the second inequality follows from the first due to
(\ref{cos23}) and because $\cJ_{t-1}=\cO(1/\cos\varphi_t$). For
$t=0$, the bound (\ref{alpha''}) reduces to (\ref{secder0}), and
for $t\geq 1$ one can use an inductive argument as above, we leave
the details out.

\medskip
Our bounds (\ref{3-2}) and (\ref{3-3}) in Lemma~\ref{Lm3} are too
weak for small values of $k-j$. The next lemma provides stronger
bounds for that case:

\begin{lemma}
\label{LmSmalldivj} For some constant $\theta<1$, we have
\beq
    \left|\int_{\Omega}\alpha\,(\ln\rho)'\,A_-\,B_+\,d\mu\right|
    \leq \Const\, \theta^{j}
       \label{3-4}
\eeq
\beq
    \left|\int_{\Omega}\alpha'\,A_-\,B_+\,d\mu\right|
    \leq \Const\, \theta^{j}
       \label{3-5}
\eeq
\end{lemma}

The proof is based on a more general lemma, which will be useful later
as well:

\begin{lemma}
\label{LmAux} Let $\cG_{\ast}=\{\ell\}$ be a family of standard
pairs \index{Standard pair} $\ell=(\gamma_{\ell},\rho_{\ell})$ and
$\lambda_{\ast}$ a probability measure on $\cG_{\ast}$ such that
\beq
      \lambda_{\ast}\{\ell\colon
      |\gamma_{\ell}|< \varepsilon\}
      \leq\Const\, \varepsilon
      \qquad \forall \varepsilon>0.
       \label{Aux-1}
\eeq
Let $A$ be a $C^1$ function on $\Omega$ such that $\int
A\,d\mu=0$. Let $B_{\ell}\colon\gamma_{\ell}\to\reals$ be a family
of functions such that
\beq
   \|B_{\ell}\|_{\infty} < b|\gamma_{\ell}|^{-\beta}
       \label{Aux-2}
\eeq
for some $\beta\in (0,1)$ and $b>0$, and for every $\ell$ and any
$x,y\in\gamma_{\ell}$
\beq
    |B_{\ell}(x) -B_{\ell}(y)|\leq
    b|\gamma_{\ell}|^{-\beta}\,[\dist(x,y)]^{\zeta}
       \label{Aux-3}
\eeq
for some $\zeta>0$. Then for some $\theta\in(0,1)$ we have
$$
  \left|\int_{\cG_{\ast}}d\lambda_{\ast}\int_{\gamma_{\ell}}
  (A\circ\cF^{n})\,B_{\ell}\,\rho_{\ell}\,dx
  \right|\leq\Const\,b\, \theta^n
$$
for all $n\geq 0$.
\end{lemma}

\proof Let $k=n/2$ and
$$
      \cG_{\ast}^0=\{\ell\in\cG_{\ast}\colon
        |\gamma_{\ell}|<e^{-k/K}\}
$$
where $K>0$ is the constant from Proposition~\ref{PrDistEq0}. Observe
that
$$
  \left|\int_{\cG_{\ast}^0}
  d\lambda_{\ast}\int_{\gamma}
  (A\circ\cF^{n})\,B_{\ell}\,\rho_{\ell}\,dx
  \right|\leq\Const\, b\int_0^{e^{-k/K}} t^{-\beta}\, dt
  \leq \Const\,b\,\theta^n
$$
for some $\theta<1$. Then we apply Proposition~\ref{PrDistEq0} to
every pair $(\gamma_{\ell} ,\rho_{\ell})\in\cG_{\ast}^1\colon
=\cG_{\ast}\setminus\cG_{\ast}^0$ in the following way. Denote by
$\gamma_1',\gamma_2',\dots$ the H-components of \index{H-components}
$\cF^k(\gamma_{\ell})$. On each curve
$\gamma_j''=\cF^{-k}(\gamma_j')\subset\gamma_{\ell}$, we pick a
point $x_j\in\gamma_{j}''$ and replace $B_{\ell}(x)$ with a constant
function $\brB_{\ell}(x)=B_{\ell}(x_j)$ on the curve $\gamma_j''$.
This replacement gives us an error term
\begin{align*}
   \int_{\cG_{\ast}^1} d\lambda_{\ast}\int_{\gamma}
  |A\circ\cF^{n}|\,|B_{\ell}-\brB_{\ell}|
  \,\rho_{\ell}\,dx &\leq
  \Const\,\|A\|_{\infty}\,b\,\vartheta^{k\zeta}\,
   \int_0^{1} t^{-\beta}\, dt\\ &\leq
   \Const\,b\,\vartheta^{n\zeta/2}
\end{align*}
here $\vartheta<1$ is the minimal contraction factor of
\index{u-curves (unstable curves)} u-curves under $\cF^{-1}$.
Lastly, the constant function $\brB_{\ell}$ can be factored out, and
we can apply Proposition~\ref{PrDistEq0} to the \index{H-components}
H-components of $\cF^k(\gamma_{\ell})$ and get
\begin{align*}
  \left|\int_{\cG_{\ast}^1} d\lambda_{\ast}\int_{\gamma}
  (A\circ\cF^{n})\,\brB_{\ell}
  \,\rho_{\ell}\,dx\right| &\leq
  \Const\,b\,\theta^{k}\,
   \int_0^{1} t^{-\beta}\, dt\\ &\leq
   \Const\,b\,\theta^{n/2} \qquad\qquad\square
\end{align*}

\noindent{\em Remark}. In the above lemma, it is obviously enough
to require (\ref{Aux-3}) only for $x,y\in\gamma_{\ell}$ such that
$\cF^k(x)$ and $\cF^k(y)$ belong to the same H-component of
$\cF^k(\gamma_{\ell})$. In fact, the requirements of the lemma can
be relaxed even further in the following way: (\ref{Aux-2}) may be
replaced by
\beq
   |B_{\ell}(x)|<\Const\bigl(
   |\gamma_{\ell}|^{-\beta}+
   |\gamma_{\ell,x}'|^{-\beta}\bigr)
       \label{Aux-2a}
\eeq
where $\gamma_{\ell,x}'$ denotes the \index{H-components}
H-component of $\cF(\gamma_{\ell})$ that contains the point
$\cF(x)$, and (\ref{Aux-3}) may be replaced by
\beq
    |B_{\ell}(x) -B_{\ell}(y)|\leq \Const\biggl(
    \frac{[\dist(x,y)]^{\zeta}}{|\gamma_{\ell}|^{\beta}}+
    \frac{[\dist(\cF(x),\cF(y))]^{\zeta}}{|\gamma_{\ell,x}'|^{\beta}}
    \biggr)
      \label{Aux-3a}
\eeq
for every $x,y\in\gamma_{\ell}$ such that
$\cF(y)\in\gamma_{\ell,x}'$. The proof only requires minor changes
that we leave to the reader.
\medskip

\noindent {\em Proof of Lemma~\ref{LmSmalldivj}}. It suffices to
apply Lemma~\ref{LmAux} to two functions,
$B_{1,\ell}=\alpha(\ln\rho)'A_-$ and $B_{2,\ell}=\alpha'A_-$. The
family $\cG_{\ast}$ consists of fibers of the foliation
$\cG_{n-j-1}^u$, and (\ref{Aux-1}) follows from the growth
\index{Growth lemma} lemma~\ref{prgrow}. Next, (\ref{Aux-2a}) for
the functions $B_{1,\ell}$ and $B_{2,\ell}$ follows from
(\ref{rho23}) and (\ref{alpha23}), respectively. To verify
(\ref{Aux-3a}), it is enough to show that for $r=1,2$
\beq
        |B_{r,\ell}'(x)|\leq \Const\,\Big (|\gamma_{\ell}|^{-q}
        + |\gamma_{\ell,x}'|^{-q}\cJ_{\gamma_{\ell}}\cF(x)\Big )
          \label{B'q}
\eeq
with some $q<2$ (here $\cJ_{\gamma_{\ell}}\cF(x)$ stands for the
Jacobian of the map $\cF\colon \gamma_{\ell}\to \gamma_{\ell,x}'$
at the point $x$). Indeed, if (\ref{B'q}) holds, then for any
$x,y\in\gamma_{\ell}$
\begin{align*}
    |B_{r,\ell}(x) -B_{r,\ell}(y)| &\leq
    \frac{\dist(x,y)}{|\gamma_{\ell}|^{q}}
    +\frac{\dist(\cF x,\cF y)}{|\gamma_{\ell,x}'|^{q}}\\
    &\leq
    \frac{[\dist(x,y)]^{1-q/2}}{|\gamma_{\ell}|^{q/2}}
    +\frac{[\dist(\cF x,\cF y)]^{1-q/2}}{|\gamma_{\ell,x}'|^{q/2}}
\end{align*}
and we get (\ref{Aux-3a}). It remains to prove (\ref{B'q}) for
both functions $B_{1,\ell}$ and $B_{2,\ell}$. This is a
consequence of the following obvious facts: $|A_-'|\leq\Const$,
$\alpha \leq \Const$, and $\alpha' \leq \Const\, |\gamma_{\ell,x}
'|^{-2/3}$ by (\ref{alpha23}),
$$
   |\alpha''| \leq \Const\,
   |\gamma_{\ell,x}' |^{-4/3}\cJ_{\gamma_{\ell}}\cF(x)
$$
by (\ref{alpha''}), $|(\ln\rho)'| \leq \Const\,
|\gamma_{\ell}|^{-2/3}$ by (\ref{densC3}) and $|(\ln\rho)''|\leq
\Const\, |\gamma_{\ell}|^{-4/3}$ by (\ref{densC4}).
Lemma~\ref{LmSmalldivj} is now proved. \qed
\medskip

Combining Lemmas~\ref{Lm3} and \ref{LmSmalldivj} gives the
following upper bound on all non-boundary terms in the integral
formula (\ref{intA-A+}):

\begin{corollary}
$$
       \sum_{n=1}^N\sum_{k=0}^{n-1}\sum_{j=0}^k
       \left| I_{n,k,j}^{(v)} \right|\leq \Const\, N
$$
\end{corollary}

It remains to estimate the boundary terms $I_n^{(d)}$ and
$I_{n,k,j}^{(b)}$.

\subsection{Cancellation of large boundary terms} \label{subsecABT1}
Here we estimate the boundary terms
$I_n^{(d)}$ given by (\ref{Ind}) and $I_{n,k,j}^{(b)}$, see
(\ref{Inkjb}). First we rewrite them in a more explicit manner and
cancel out some of the resulting integrals.

\medskip\noindent{\em Convention}. Let $S\subset\Omega$ be a smooth
curve, $C$ a function and $\bv$ a vector field on $S$. Then we can
integrate
\beq
      \label{intS}
    \int_S C\, (\omega*\bv) =
    \int_S C\,\|\bv^\perp\|_0\cos\varphi\,dl
\eeq
were $\omega$ denotes the $\cF$-invariant volume form
$$
  \omega(dr,d\varphi) = \cos\varphi\,\,dr\wedge d\varphi
$$
and $(\omega*\bv)$ stands for the one form
$$
   (\omega*\bv)(\bw) = \omega(\bv,\bw)
$$
On the right hand side of (\ref{intS}), $\|\cdot\|_0$ stands for
the Euclidean norm $\bigl[ (dr)^2+(d\varphi)^2 \bigr]^{1/2}$ and
$\bv^{\perp}$ for the normal component of the vector $\bv$, and we
integrate with respect to the Lebesgue measure (length) $dl$ on
$S$.

The $\cF$-invariance of $\omega$ gives us a {\em change of
variables formula}
\beq
       \label{intAinv}
    \int_S C\, (\omega*\bv) =
    \int_{\cF^n(S)} (C\circ\cF^{-n})\, (\omega*d\cF^n\bv)
\eeq
provided $\cF^n$ is smooth on $S$.

First we consider $I_n^{(d)}$ given by (\ref{Ind}). Each
discontinuity curve $S\subset\cS_n\setminus\cS_0$ has the form
$S=\cF^{-k}S^+$, where $0\leq k\leq n-1$ and $S^+ \subset \cS_1
\setminus \cS_0$ is a discontinuity curve for $\cF$. Thus $S$
changes with velocity
$$
   v = d\cF^{-k}v_0-\sum_{m=0}^{k-1}d\cF^{-(k-m)}(X)
$$
where $v_0$ is the speed of $S^+$ as it changes with $Q$ (in the normal
direction). Therefore,
\begin{align*}
   I_n^{(d)} &= I_n^{(v)}-I_n^{(x)}\\
   &=  \sum_{k=0}^{n-1} I_{n,k}^{(v)}
   -\sum_{k=0}^{n-1}\sum_{m=0}^{k-1} I_{n,k,m}^{(x)}
\end{align*}
where
\beq
    \label{Inkv}
   I_{n,k}^{(v)}=\sum_{S^+}\int_{S^+}
  A_{-k}\,\Delta B_{n-k}\,
  \left(\omega*v_0\right)
\eeq
and
\beq
    \label{Inkmx}
  I_{n,k,m}^{(x)}=\sum_{S^+}\int_{S^+}
  A_{-k}\,\Delta B_{n-k}\,
  \left(\omega*d\cF^{m}(X)\right)
\eeq
(here the summation is performed over all smooth curves $S^+\subset
\cS_1\setminus\cS_0$).

Furthermore, by (\ref{dFkX}) we have
$$
   d\cF^m(X)=\sum_{j=0}^{m} \alpha_{k-m,m-j}^{(j)}
  +\beta_{k-m,m}
$$
Reindexing our formula by $r=k-m$, $s=m-j$, and $t=j$ gives
\begin{align}
   I_n^{(x)} &=
   \sum_{\substack{r,s,t\geqslant 0\\ r+s+t<n}}
   \sum_{S^+}\int_{S^+}
  A_{-(r+s+t)}\,\Delta B_{n-(r+s+t)}\,
  \left(\omega*\alpha_{r,s}^{(t)}\right)\nonumber\\
   &\quad +
   \sum_{\substack{r,s\geqslant 0\\ r+s<n}}
   \sum_{S^+}\int_{S^+}
  A_{-(r+s)}\,\Delta B_{n-(r+s)}\,
  \left(\omega*\beta_{r,s}\right)
    \label{Inx}
\end{align}
The first sum contains exponentially growing (with $t$) integrals,
but they will be cancelled shortly. At this moment we estimate the
total contribution of the second sum
\begin{align}
   T_N\colon &=
   \sum_{n=1}^N\sum_{\substack{r,s\geqslant 0\\ r+s<n}}
   \sum_{S^+}\int_{S^+}
  A_{-(r+s)}\,\Delta B_{n-(r+s)}\,
  \left(\omega*\beta_{r,s}\right)\nonumber\\
      &=
   \sum_{\substack{r,s\geqslant 0\\ r+s<N}}
   \sum_{S^+}\int_{S^+}
  A_{-(r+s)}\,
  \left[\sum_{n=r+s+1}^N\Delta B_{n-(r+s)}\right]
  \left(\omega*\beta_{r,s}\right)
     \label{TN}
\end{align}

\begin{lemma} We have
$$
   |T_N| \leq \Const \, N
$$
\end{lemma}

\proof Observe that for any point $x\in S^+\setminus \left(\cup_{k\geq
2}S_k\right)$ we have
$$
  \left(B_{n-(r+s)}(x)\right)_+=
  \left(B_{n\pm 1-(r+s)}(x)\right)_-
$$
where $(\cdot)_+$ and $(\cdot)_-$ denote the one-sided limit
values of the corresponding functions, and the choice of the sign
(in $\pm 1$) in the subscript depends on the orientation of the
curve $S^+$. Since $\Delta (B)=(B)_+-(B)_-$ for any function $B$,
the sum in the bracket in (\ref{TN}) telescopes, hence
$$
   \left|T_N\right| \leq \Const
   \sum_{\substack{r,s\geqslant 0\\ r+s<N}} \|\beta_{r,s}\|_0
$$
Recall that the $\|\cdot\|_0$ norm is equivalent to $\|\cdot\|$
(Proposition~\ref{prdrdp}) and
$\|\beta_{r,s}\|\leq\Const\,\theta^s$ (Lemma~\ref{Lmalphabeta}),
hence $|T_N|\leq \Const\, N$. \qed
\medskip

We now turn to $I_{n,k,j}^{(b)}$ from (\ref{Inkjb}). The set
$(\cS_{-(n-j-1)}\cup\cS_{j+1})\setminus\cS_0$ consists of s-curves
$S\subset\cF^{-m}(\cS_1\setminus\cS_0)$, $0\leq m\leq j$ and
\index{u-curves (unstable curves)} u-curves
$S\subset\cF^{m}(\cS_{-1}\setminus\cS_0)$, $0\leq m\leq n-j-2$.
Accordingly,
$$
   I_{n,k,j}^{(b)}=I_{n,k,j}^{(bs)}+I_{n,k,j}^{(bu)}
$$
where (using change of variables)
\beq
  I_{n,k,j}^{(bs)}=\sum_{m=0}^{j} \sum_{S^+}
  \int_{S^+} A_{-(n-j-1+m)}\, \Delta B_{j+1-m}
  \left(\omega * \alpha_{n-k,k-j}^{(m-1)}\right)
    \label{Inkjbs}
\eeq
and
\beq
  I_{n,k,j}^{(bu)}=\sum_{m=0}^{n-j-2} \sum_{S^-}
  \int_{S^-} \Delta\left[A_{-(n-j-1-m)}\,
  \left(\omega * \alpha_{n-k,k-j}^{(-m-1)}\right)\right]\,
  B_{j+1+m}
    \label{Inkjbu-}
\eeq
(here the summation is performed over all the discontinuity curves
$S^-\subset\cS_{-1}\setminus\cS_0$ of the map $\cF^{-1}$).

First we analyze (\ref{Inkjbs}). The case $m=0$ is special, and we
combine all the terms with $m=0$ in a separate expression:
\beq
    \label{Inbs}
  I^{(bs,0)}_n=\sum_{k=0}^{n-1}\sum_{j=0}^{k} \sum_{S^+}
  \int_{S^+} A_{-(n-j-1)}\, \Delta B_{j+1}
  \bigl(\omega * \alpha_{n-k,k-j}^{(-1)}\bigr)
\eeq
To deal with the other terms ($m>0$) in (\ref{Inkjbs}), we change our
indexing system to $r=n-k$, $s=k-j$, and $t=m-1$, and obtain a total of
$$
   \sum_{\substack{r,s,t\geqslant 0\\ r+s+t<n}}
   \sum_{S^+}\int_{S^+}
  A_{-(r+s+t)}\,\Delta B_{n-(r+s+t)}\,
  \left(\omega*\alpha_{r,s}^{(t)}\right)\\
$$
which completely cancels the first sum in (\ref{Inx}), hence all the
large integrals are now gone.
% the previously developed expression for $I_n^{(x)}$.

Next we make a general remark. Every curve $S^-$ separates two
regions, one is mapped by $\cF^{-1}$ into a vicinity of $\cS_0$
and the other -- into a vicinity of some curve
$S^+\subset\cS_1\setminus\cS_0$. On the side of $S^-$ that is
mapped onto $\cS_0$, the map $\cF^{-1}$ has unbounded derivatives,
and we call that side of $S^-$ {\em irregular}. On the other side,
the map $\cF^{-1}$ has bounded derivatives, and we call that side
of $S^-$ {\em regular}. Thus, every curve $S^- \subset \cS_{-1}
\setminus \cS_0$ has one regular side and one irregular side.
Similarly we define regular and irregular sides for every curve
$S^+\subset\cS_1\setminus\cS_0$. Note that $\cF^{-1}$ maps the
regular sides of $\cS_{-1}\setminus\cS_0$ to the regular sides of
$\cS_{1}\setminus\cS_0$, and the map $\cF$ does the opposite.

Observe that the integrand in (\ref{Inkjbu-}) vanishes on the
irregular side of every curve $S^-$ (to see this, note that the
field $\alpha_{n-k,k-j}^{(-m-1)}\subset E^u_{n-j-1-m}$ is in fact
tangent to $S^-$ on its irregular side; or, equivalently, one can
approximate $S^-$ by a curve $S$ lying on its irregular side, apply
(\ref{intAinv}) with $n=-1$ to $S$, and note that the form $\omega$
vanishes on $\cF^{-1} (S)$ as that curve approaches $\cS_0$). Now we
can change variables $y=\cF^{-1}x$ and rewrite (\ref{Inkjbu-}) as
\beq
    \label{Inkjbu+}
  I_{n,k,j}^{(bu)}=\sum_{m=1}^{n-j-1} \sum_{S^+}
  \int_{S^+} A_{-(n-j-1-m)}\,
  B_{j+1+m}\,
  \left(\omega * \alpha_{n-k,k-j}^{(-m-1)}\right)
\eeq
where the integration is performed along the regular side of each
curve $S^+$ (again, on the irregular side of $S^+$ the integrand
in (\ref{Inkjbu+}) vanishes).

Now the integrals in (\ref{Inkjbu+}) can be naturally combined with
those in (\ref{Inbs}) and make a total of
\beq
    \label{Inkb}
  I_{n,k}^{(b)}=\sum_{j=0}^{k}
  \sum_{m=0}^{n-j-1} \sum_{S^+}
  \int_{S^+} A_{-(n-j-1-m)}\,
  B_{j+1+m}\,
  \left(\omega * \alpha_{n-k,k-j}^{(-m-1)}\right)
\eeq
Here the case $m=0$ corresponds to (\ref{Inbs}) and the case
$m\geq 1$ to (\ref{Inkjbu+}). (Note that the integrand in
(\ref{Inbs}) also vanishes on the irregular side of each curve
$S^+$.)

We also note that
\beq
     \left\|\alpha_{n-k,k-j}^{(-m-1)}\right\|
     \leq\Const\,\theta^{m+k-j}
       \label{dFmalpha}
\eeq
due to Lemma~\ref{Lmalphabeta}.

It remains to estimate the terms $I_{n,k}^{(v)}$ given by (\ref{Inkv})
and $I_{n,k}^{(b)}$ of (\ref{Inkb}).

\noindent{\em Remark.} Before proceeding with our estimates let us
compare the approach of the present section with that of
Chapter~\ref{SecSPE}. There are three types of terms corresponding
to the variation of $\mu(A(B\circ \cF_Q))$:
\begin{itemize}
\item ``stable continuous'' terms $I_{n,k}^s$ given by \eqref{Inks}.
They correspond to the term $\RmII$ in \eqref{SUSin} since they deal
with the difference of the values of the observable at the shadowed
and the shadowing points.
\item ``unstable continuous'' terms $I_{n,k,j}^{(i)}$ given by
\eqref{Inkji}. They correspond to the term $\RmIII$ in
\eqref{SUSin}, since the Jacobian of the holonomy map is a product
of unstable Jacobian ratios.
\item The terms containing integration over discontinuity
($T_N$ given by \eqref{TN}, $I_{n,k}^{(v)}$ given by \eqref{Inkv}
and $I_{n,k}^{(b)}$ given by  (\ref{Inkb})). They correspond to the
first term  in \eqref{SUSin} since they account for orbits where
shadowing is impossible as they pass too close to the singularities.
\end{itemize}

\subsection{Estimation of small boundary terms}
\label{subsecABT2} First we outline our strategy. All the
integrals in (\ref{Inkv}) and (\ref{Inkb}) have a general form of
$$
   \int_{S^+} A_{-k_1}\,
  B_{k_2}\,
  (\omega * \bv) =
   \int_{S^+} A_{-k_1}\,
  B_{k_2}\,
  \|\bv^{\perp}\|_0\,\cos\varphi\, dl
$$
with $k_1+k_2=n$ and some vector fields $\bv$ on $S^+$. The curve
$S^+$ is strongly expanded by $\cF^{-k_1}$, as well as by
$\cF^{k_2}$, and so both functions $A_{-k_1}$ and $B_{k_2}$
rapidly oscillate on the curve $S^+$. However, if $k_1\ll k_2$,
then $B_{k_2}$ oscillates much faster than $A_{-k_1}$, and we will
approximate $A_{-k_1}$ by constants on appropriately chosen pieces
of $S^+$ and then use Proposition~\ref{PrDistEq0} to average
$B_{k_2}$ on each of those pieces. If $k_1\gg k_2$, then
$A_{-k_1}$ and $B_{k_2}$ switch places. In the remaining case
$k_1\approx k_2$ we simply bound the above integrand by
$\norm{A}_{\infty}\norm{B}_{\infty}\,\sup_{\Omega} \|\bv\|$, and
then summing up over $n\leq N$ and using (\ref{dFmalpha}) will
give us the desired $\cO(N)$ estimate.

When applying Proposition~\ref{PrDistEq0}, we will treat the
function $\rho=\|\bv^{\perp}\|_0\,\cos\varphi$ as a ``density'' on
the corresponding pieces of $S^+$, so that they become standard
pairs. However, while $v_0$ in (\ref{Inkv}) is bounded and smooth
\index{Standard pair} (which can be easily verified directly, we
omit details), the vector fields (and hence, the corresponding
$\rho$) in (\ref{Inkb}) are badly discontinuous: their
discontinuities lie on the set $\cS_{-k_1}$, which is very dense on
$S^+$. Our first task is to approximate vector fields in
(\ref{Inkb}) by smooth enough functions. To this end we develop a
general approach.

Let $S\subset\Omega$ be a \index{u-curves (unstable curves)} u-curve
or an s-curve, $a_1\in (0,1]$ and $a_2\geq 0$. We denote by
$\cH^{a_1,a_2}(S)$ the class of functions $\rho\colon S\to\reals$
that are well approximated by H\"older continuous functions in the
following sense:

\medskip\noindent{\bf Definition}.
$\rho\in\cH^{a_1,a_2}(S)$ iff there is a $L_{\rho}>0$ such that
for every $\varepsilon\in (0,1)$ there exists a function
$\rho_{\varepsilon}\colon S\to\reals$ satisfying two requirements:
\beq
      \label{ApHold1}
    \int_S |\rho-\rho_{\varepsilon}|\, dl
    \leq \varepsilon
\eeq
and for all $x,y\in S$
\beq
       \label{ApHold2}
  |\rho_{\varepsilon}(x)-\rho_{\varepsilon}(y)|
  \leq L_{\rho}\,\varepsilon^{-a_2}\,|S(x,y)|^{a_1}
\eeq
where $S(x,y)$ denotes the segment of the curve $S$ between the
points $x$ and $y$. We always take the smallest $L_{\rho}$ for
which (\ref{ApHold2}) holds for all $\varepsilon\in (0,1)$ and put
$$
     \|\rho\|_{a_1,a_2}\colon =L_{\rho}.
$$

\begin{lemma}[H\"older approximation]
\label{MidSmall} There exist $a_1\in (0,1]$ and $a_2\geq 0$ such that
$$
   \rho=\Bigl\|\left(\alpha_{n-k,k-j}^{(-m-1)}\right)^{\perp}\Bigr\|_0
   \cos\varphi\in\cH^{a_1,a_2}(S^+)
$$
and
$$
   \norm{\rho}_{a_1,a_2} \leq 1
$$
uniformly in $n,k,j,m$.
\end{lemma}

We postpone the proof of Lemma~\ref{MidSmall} until
Section~\ref{subsecHA} and continue our analysis of the integrals
(\ref{Inkv}) and (\ref{Inkb}).

As we mentioned already, $v_0$ is a bounded and smooth vector field,
hence
\beq
    \norm{v_0}_{\infty}\leq\Const\qquad{\rm and}\qquad
    \norm{v_0}_{a_1,a_2} \leq \Const
\eeq
for any $a_1\in (0,1]$ and $a_2\geq 0$.

\begin{proposition}[Two-sided integral sums]
\label{PrTwoSide} Given $a_1\in (0,1]$, $a_2\geq 0$, and $L>0$,
there are constants $C, c, \xi>0$ such that for each curve
$S^+\subset\cS_1\setminus\cS_0$ and all $m_1, m_2$ such that
$m_j<L\ln N,$ for any $\delta>0,$ and for any functions
$\rho_{k_1}\in\cH^{a_1,a_2}(S^+)$ such that
\beq
  \norm{\rho_{k_1}}_{\infty}\leq \delta
  \qquad{\rm and}\qquad
  \norm{\rho_{k_1}}_{a_1,a_2}
  \leq 1
    \label{rhoass}
\eeq
we have
$$
      \biggl|\sum_{\substack{k_1>m_1,\, k_2>m_2\\ k_1+k_2\leq N}}
      \int_{S^+}  A_{-k_1}
      \,B_{k_2}\, \rho_{k_1}\, dl\, \biggr|\leq
      C\Bigl(N\delta|\ln \delta| +N^2 e^{-c N^\xi}\Bigr)
$$
where the integral can be taken on either side of $S^+$ (but this
should be done consistently).
\end{proposition}

We prove Proposition~\ref{PrTwoSide} in Section~\ref{subsecTSIS}.

\begin{corollary}
\label{ContrSing}
$$
  \left| \sum_{n=1}^N\sum_{k=0}^{n-1} I_{n,k}^{(v)}\right|
  \leq\Const\, N,
   \qquad
  \left| \sum_{n=1}^N\sum_{k=0}^{n-1} I_{n,k}^{(b)}\right|
  \leq\Const\, N
$$
\end{corollary}

\proof We prove the second bound (the first one is easier).
Introduce new indices $(k_1, k_2, r)$ where $k_1=n-j-1-m,$
$k_2=j+1+m,$ and $r=k-j.$ Due to (\ref{dFmalpha}) we can choose
$L$ so large that the sum over quadruples with $m>L\ln N$ or $r>L
\ln N$ will be less than 1. Now Proposition~\ref{PrTwoSide}
and Lemma \ref{Lmalphabeta}  imply that for fixed $m$ and $r$ such that $m\leq L\ln N$ and
$r\leq L \ln N$, the sum over $k_1$ and $k_2$ is bounded by
$\Const\,[(r+m)\theta^{r+m} N + N^2 e^{-c N^\xi}].$ Summation over
$m$ and $r$ gives the desired bound. \qed \medskip

This completes the proof of our main estimate (\ref{MaindDN}) and hence
that of Proposition~\ref{PrRegDif} (modulo Lemma~\ref{MidSmall} and
Proposition~\ref{PrTwoSide}). \qed

\subsection{Two-sided integral sums}
\label{subsecTSIS} Here we prove Proposition~\ref{PrTwoSide}. For the
sake of brevity we shall call any set of the form
$$
    \{(k_1,k_2)\colon k_1\geq m_1,\ k_2\geq m_2,\
    (k_1-m_1)+(k_2-m_2)\leq R\}
$$
a {\em triangle} with side $R$, and any set of the form
$$
   \{(k_1,k_2)\colon
   m_1\leq k_1\leq m_1+R,\
   m_2\leq k_2\leq m_2+R\}
$$
a {\em square} with side $R.$ For brevity, we denote
$$
    \cI_{k_1,k_2}=\int_{S^+}
    A_{-k_1}\, B_{k_2}\, \rho_{k_1}\, dl
$$

\begin{lemma}
\label{SqB} For any square $\bS_R$ with side $R$
$$
   \biggl\lvert\sum_{(k_1, k_2)\in \bS_R}
   \cI_{k_1,k_2}\biggr\rvert
    \leq \Const\, R\, \delta |\ln \delta|
$$
\end{lemma}

\proof For simplicity, we will set $m_1=m_2=0$ (the general case
only requires minor modifications). Now we have
\begin{align*}
   \biggl\lvert\sum_{(k_1, k_2)\in \bS_R} \cI_{k_1,k_2}\biggr\rvert
     & \leq
   \int_{S^+} \biggl\lvert\biggl( \sum_{k_1} A_{-k_1} \rho_{k_1}\biggr)
   \biggl(\sum_{k_2} B_{k_2} \biggr)\biggr\rvert\, dl \\
     & \leq
  \left[\left(\int_{S^+} \biggl[ \sum_{k_1} A_{-k_1}\rho_{k_1}
  \biggr]^2 dl\right)
%\right)^{1/2}
\left(\int_{S^+} \biggl[\sum_{k_2}
  B_{k_2}\biggr]^2 dl\right) \right]^{1/2}
\end{align*}
To estimate the first factor we expand
\begin{align*}
   \int_{S^+} \biggl[\sum_{k_1} A_{-k_1}\rho_{k_1}
   \biggr]^2 dl
     & =
   \int_{S^+} \sum_{j_1, j_2} A_{-j_1}A_{-j_2}
   \rho_{j_1}\rho_{j_2}\, dl\\
     & =
   \int_{S^+} \sum_{j} A_{-j}^2\rho_{j}^2\, dl
   + 2 \int_{S^+}\sum_{j_2>j_1}
   A_{-j_1}A_{-j_2}
   \rho_{j_1}\rho_{j_2}\, dl
\end{align*}
The first term here is $\cO(R\,\delta^2)$. To estimate the second
sum we choose a large $K>0$ and divide the domain of summation
$\{j_1<j_2\}\subset \bS_R$ into two parts: a smaller one
$$
   \bS_R'=\{j_1<2K|\ln\delta|\}\cup\{|j_1-j_2|<2K|\ln\delta|\}
$$
and a larger one $\bS_R''= \{j_1<j_2\}\setminus \bS_R'$.
Obviously,
$$
   \biggl|\int_{S^+} \sum_{(j_1,j_2)\in \bS_R'}
   A_{-j_1}A_{-j_2}\rho_{j_1}\rho_{j_2}\, dl \biggr|
   \leq \Const\,R\,\delta^2|\ln\delta|
$$
To estimate the larger sum, we need to approximate
$\rho_{j_1,j_2}=\rho_{j_1} \rho_{j_2}$ by a H\"older continuous
function: (\ref{rhoass}) implies $\rho_{j_1,j_2}\in\cH^{a_1,a_2}(S^+)$
and $\|\rho_{j_1,j_2}\|_{a_1,a_2}\leq 1$, hence we can set
$\varepsilon=e^{-j_1/K}$ and find $\brrho_{j_1,j_2}$ such that
$$
  \int_{S^+}|\rho_{j_1,j_2}-\brrho_{j_1,j_2}|\,dl\leq e^{-j_1/K}
$$
and for any $x,y\in S^+$
$$
  |\brrho_{j_1,j_2}(x)-\brrho_{j_1,j_2}(y)|\leq e^{a_2j_1/K}[\dist(x,y)]^{a_1}
$$
The error of approximation can be bounded by
\begin{align*}
   \int\limits_{S^+} \sum_{(j_1,j_2)\in \bS_R''}
   |A_{-j_1}A_{-j_2}|\,|\rho_{j_1,j_2}
   -\brrho_{j_1,j_2}|\, dl
     & \leq
   \norm{A}_{\infty}^2\sum_{j_2>j_1\geq 2K|\ln\delta|}e^{-j_1/K}\\
     & \leq
   \Const\,R\,\delta^2
\end{align*}
It remains to bound the integrals
$$
   \bar{\cI}_{j_1,j_2} = \int_{S^+}
   A_{-j_1}A_{-j_2}\brrho_{j_1,j_2}\, dl
$$
for $(j_1,j_2)\in \bS_R''$. We denote by $S^+_q$, $q\geq 1$, all the
\index{H-components} H-components of $\cF^{-j_1}(S^+)$ (i.e.\ the
maximal curves $S_q^+ \subset \cF^{-j_1} (S^+)$ such that
\index{Homogeneity strips} $\cF^i(S^+_q)$ lies in one homogeneity
strip for each $i=0,\dots,j_1$), and by $m_q$ the image of the
Lebesgue measure $dl$ under $\cF^{-j_1}$ on $S_q^+$. Then
\beq
         \label{barcI}
     \bar{\cI}_{j_1,j_2} = \sum_q
     \int_{S^+_q} (A\,A_{-(j_2-j_1)})
    (\brrho_{j_1,j_2}\circ\cF^{j_1})\, dm_q
\eeq
We claim that if $K>0$ is large enough, then the function
$g=\brrho_{j_1,j_2}\circ\cF^{j_1}$ is H\"older continuous on each curve
$S^+_q$ with exponent $a_1$ and a uniformly bounded norm. Indeed, for
any $x,y\in S^+_q$
$$
  |g(x)-g(y)|\leq
  e^{a_2j_1/K}\vartheta^{a_1j_1}[\dist(x,y)]^{a_1}
  \leq[\dist(x,y)]^{a_1}
$$
where $\vartheta<1$ is the minimal factor of expansion of s-curves
under $\cF^{-1}$, and $e^{a_2/K}\vartheta^{a_1}<1$ provided $K$ is
large enough. Hence we can apply Lemma~\ref{LmAux} to the map
$\cF^{-(j_2-j_1)}$ (using time reversibility) on the set $\cup_q
S_q^+$, and thus estimate (\ref{barcI}) as
$$
    \left|  \sum_q\int_{S^+_q} (A\,A_{-(j_2-j_1)})
    (\brrho_{j_1,j_2}\circ\cF^{j_1})\, dm_q\right|
   \leq \Const\,\theta^{j_2-j_1}
$$
with some constant $\theta<1$. Therefore
$$
   \sum_{(j_1,j_2)\in \bS_R''}
   \left|\bar{\cI}_{j_1,j_2}\right|\leq\Const\,R\,
   \theta^{2K|\ln\delta|}\leq\Const\,R\,\delta^2
$$
provided $K$ is large enough (say, $K>1/|\ln\theta|$). Combining
all the previous estimates gives
$$
 \int_{S^+} \biggl( \sum_{k_1} A_{-k_1}\rho_{k_1}
  \biggr)^2 dl \leq \Const\,R\,\delta^2|\ln\delta|
$$
The same argument yields
$$
   \int_{S^+} \biggl( \sum_{k_2}
  B_{k_2}\biggr)^2 dl
  \leq \Const\, R
$$
(in fact, this is easier since there is no $\rho$'s to approximate).
This completes the proof of Lemma~\ref{SqB}. \qed

\begin{lemma}
\label{TrB} There exists a constant $C>0$ such that for any
triangle $\bT$ with side $R$
$$
    \biggl|\sum_{(k_1, k_2)\in \bT} \cI_{k_1,k_2}\biggr|
    \leq C (R \ln R)\, \delta |\ln\delta|
$$
\end{lemma}

\proof (See Figure~\ref{FigTriangles}.) We decompose $\bT$ into
the union of a square and two triangles with sides $R/2.$ Then we
apply a similar decomposition to each of the two smaller
triangles, an so on. In this way we get a decomposition of $\bT$
into squares of variable size, so that for each $k\geq 1$ there
are $2^k$ squares with side about $R/2^k.$ Applying
Lemma~\ref{SqB} to each square yields the required bound. \qed
\medskip

\begin{figure}[htb]
    \centering
    \includegraphics{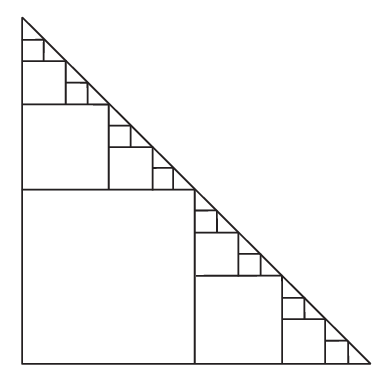}
    \caption{Proof of Lemma~\ref{TrB}}
    \label{FigTriangles}
\end{figure}

Lemma \ref{TrB} falls short of the estimate claimed in
Proposition~\ref{PrTwoSide}, because of the extra $\ln R$ factor
here, but it has the advantage of being applicable to an arbitrary
triangle. To upgrade Lemma \ref{TrB} to the estimate claimed in
Proposition~\ref{PrTwoSide} we need to bound off-diagonal terms.

\begin{lemma}[Off-diagonal bounds]
\label{OffD} Fix some $0<\zeta<1/2.$ Then there are constants $C, c,
\xi>0$ such that if
\beq
      \max\{k_1,k_2\}>R/2\qquad{\rm and}\qquad
         |k_1-k_2|>R^{1/2+\zeta}
           \label{OffDass}
\eeq
then
$$
    \left|\cI_{k_1,k_2}\right|\leq C\exp(-cR^{\xi})
$$
\end{lemma}

We prove Lemma~\ref{OffD} in Section~\ref{subsecBODT} and first
derive Proposition~\ref{PrTwoSide}.

\medskip\noindent{\em Proof of Proposition~\ref{PrTwoSide}.}
(See Figure~\ref{FigTriPart}.) Let $\bT$ be the triangle of
Proposition~\ref{PrTwoSide}, $\bS$ the inscribed square and
$\bT_1$, $\bT_2$ the triangles with side $N^{1/2+\zeta}$ whose one
vertex is the midpoint of the hypotenuse of $\bT$. Then
$$
     \sum_{(k_1, k_2)\in \bT} \cI_{k_1, k_2}=\sum_{\bS} \cI_{k_1, k_2}+
      \sum_{\bT_1\bigcup \bT_2} \cI_{k_1, k_2}+
      \sum_{\bT\setminus(\bS\bigcup \bT_1\bigcup \bT_2)} \cI_{k_1, k_2}
$$
The first sum here is $\cO(N\delta|\ln\delta|)$ by Lemma~\ref{SqB},
the second one is $\cO(N^{1/2+\zeta} \ln N\,\delta|\ln\delta|)$ by
Lemma~\ref{TrB}, and the last one is $\cO \bigl( N^2 \times
\exp(-cN^\xi)\bigr)$ by Lemma~\ref{OffD}, because every pair
$(k_1,k_2)\in \bT\setminus(\bS\cup \bT_1\cup \bT_2)$ satisfies
(\ref{OffDass}). \qed

\begin{figure}[htb]
    \centering
    \psfrag{t1}{$\bT_1$}
    \psfrag{t2}{$\bT_2$}
    \psfrag{s}{$\bS$}
    \includegraphics{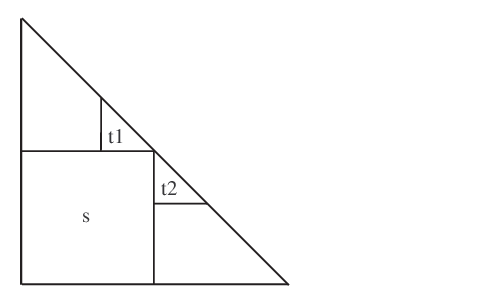}
    \caption{Proof of Proposition~\ref{PrTwoSide}}
    \label{FigTriPart}
\end{figure}

\subsection{Bounding off-diagonal terms}
\label{subsecBODT} Here we prove Lemma~\ref{OffD}. Our main idea
is that if $k_2-k_1\gg \sqrt{R}$, then we can partition $S^+$ into
subintervals such that the preimages under $\cF^{-k_1}$ are
predominantly small whereas their images under $\cF^{k_2}$ are
mostly large (as it will follow from moderate deviation bounds of
Section~\ref{subsecMD}). Thus we can approximate $A_{-k_1}$ and
$\rho_{k_1}$ by constants on each interval and average the value
of $B_{k_2}$ by using Proposition~\ref{PrDistEq0}.

Without loss of generality, we suppose that
$$
       k_2-k_1>\Delta=R^{1/2+\zeta}
$$
(the case $k_1-k_2>\Delta$ is completely symmetric by
time-reversibility).

Let $k_3=k_2-\Delta/4$ and $k_4=k_2-\Delta$. Denote by $S_p^-$,
$p\geq 1$, all the \index{H-components} H-components of the set
$\cF^{k_3}(S^+)$, then the curves $S_p^+=\cF^{-k_3}(S_p^-)\subset
S^+$ make a partition of $S^+$. Let $\bchi$ denote the Lyapunov
exponent of the map $\cF$. We say that a curve $S_p^+$ is {\em good}
if it satisfies three requirements:
\begin{itemize}\item[(a)] $\left|S_p^+\right|<\exp(-\bchi
k_3+\bchi\Delta/4)$, \item[(b)] $\cF^{-k_4}(S_p^+)$ belongs in one
\index{H-components} H-component of the set $\cF^{-k_4}(S^+)$,
\item[(c)] $\left|\cF^{-k_4}(S_p^+)\right| <\exp(-\bchi\Delta/4)$,
\end{itemize} and denote $G^+=\cup\{S_p^+\colon S_p^+$ is
good$\}$.

\begin{lemma}
\beq
    l(S^+\setminus G^+)\leq\Const\,\exp(-cR^{2\zeta})
\eeq
for some constant $c>0$.
\end{lemma}

\proof We note that the distortions \index{Distortion bounds} of the
map $\cF^{k_3}$ on each curve $S_p^+$ are bounded, i.e.\ for any
$x,y\in S_p^+$
$$
      0<\brC^{-1}\leq \frac{\cJ_{S_p^+}\cF^{k_3}(x)}
      {\cJ_{S_p^+}\cF^{k_3}(y)}\leq \brC<\infty
$$
where $\cJ_S\cF^k(x)$ denotes the Jacobian (the expansion factor) of
the map $\cF^k$ restricted to the curve $S$ at the point $x$. Now by
Proposition~\ref{PrMD} on moderate deviations
$$
   l\bigl(\cup S_p^+\colon \left|S_p^+\right|
   >\exp(-\bchi k_3+\bchi\Delta/4)\bigr)<\Const\,\exp(-cR^{2\zeta})
$$
for some $c>0$, hence we may ignore the curves on which (a) fails.

Similarly, the distortions of the maps $\cF^{-k}$, $k\geq 1$, on
each curve $S_p^+$ remain bounded as long as the preimage
$\cF^{-k}(S_p^+)$ lies within one \index{H-components} H-component
of the set $\cF^{-k}(S^+)$. If the condition (b) fails, then there
is a (smallest) $k<k_4$ such that $\cF^{-k}(S_p^+)$ crosses either a
singularity line of the map $\cF^{-1}$ or the boundary of a
homogeneity strip. \index{Homogeneity strips} Now we distinguish two
cases:
\begin{itemize}\item[(b1)]
$\left|\cF^{-k}(S_p^+)\right|<\exp(-\bchi\Delta/4)$ (a
\emph{short} curve), \item[(b2)]
$\left|\cF^{-k}(S_p^+)\right|\geq\exp(-\bchi\Delta/4)$ (a
\emph{long} curve).
\end{itemize}
Observe that every short curve $\cF^{-k}(S_p^+)$ lies within a
distance less than $\exp(-\bchi\Delta/4)$ of an endpoint of an
\index{H-components} H-component of $\cF^{-k}(S^+)$. Therefore, by
the growth \index{Growth lemma} lemma~\ref{prgrow}
$$
   l\left(\cup S_p^+\colon {\rm (b)}\ {\rm fails}\ {\rm and}\
   \cF^{-k}(S_p^+)\ {\rm is}\ {\rm short} \right)
      <\Const\,k_4 \exp(-\bchi\Delta/4)
$$
On the other hand, if $\cF^{-k}(S_p^+)$ is long and (a) holds, then
by bounded distortion \index{Distortion bounds}
$$
        \cJ_{S_p^+}\cF^{-k}(x)\geq
        \brC^{-1}\exp\left(\bchi k_3-\bchi\Delta/2\right)
$$
for every point $x\in S_p^+$. Note that $k<k_4=k_3-3\Delta/4$,
hence by Proposition~\ref{PrMD} on moderate deviations
$$
   l\left(\cup S_p^+\colon {\rm (b)}\ {\rm fails}\ {\rm and}\
   \cF^{-k}(S_p^+)\ {\rm is}\ {\rm long} \right)
      <\Const\,k_4 \exp(-cR^{2\zeta})
$$
with some $c>0$, hence we may ignore the curves on which (b) fails.

Lastly, if (a) and (b) hold but (c) fails, then we can apply the
previous argument to $\cF^{-k_4}(S_p^+)$, since its length exceeds
$\exp(-\bchi\Delta/4)$. \qed \medskip

Now, consider a good curve $S_p^+\subset S^+$. Observe that
$k_4\geq k_1$, hence $\left|\cF^{-k_1}(S_p^+)\right|
<\exp(-\bchi\Delta/4)$. Therefore, the oscillations of the
function $A_{-k_1}$ on $S_p^+$ does not exceed
$\Const\,\exp(-\bchi\Delta/4)$, so we can approximate $A_{-k_1}$
by a constant, $\hat{A}_{-k_1}$, on each good curve.

Next, we use a H\"older continuous approximation to $\rho_{k_1}$. We
set $\varepsilon=\exp(-\Delta)$ and find a $\brrho_{k_1}$ such that
$$
  \int_{S^+}|\rho_{k_1}-\brrho_{k_1}|\,dl\leq e^{-\Delta}
$$
and for any $x,y\in S^+_p$
$$
  |\brrho_{k_1}(x)-\brrho_{k_1}(y)|\leq
  e^{a_2\Delta}[\dist(x,y)]^{a_1}
  \leq e^{a_2\Delta-a_1\bchi k_3}
$$
so that the oscillations of $\brrho_{k_1}$ on $S_p^+$ do not exceed
$e^{-cR}$ with some $c>0$. Now we approximate $\brrho_{k_1}$ by a
constant, $\hat{\rho}_{k_1}$, on each good curve $S_p^+$. The errors of
this and other approximations above are all bounded by $\exp(-c\Delta)$
with some $c>0$.

Lastly, we apply Proposition~\ref{PrDistEq0} to the
\index{H-components} H-components of the set $\cF^{k_3}(S^+)$ and
average the function $B\circ\cF^{k_2-k_3}$ on every such component.
This gives us the estimate
$$
   \biggl|\int_{G^+}\hat{A}_{-k_1}B_{k_2}\,\hat{\rho}_{k_1}\,dl\biggr|
   \leq \Const\,\theta_0^{k_2-k_3} = \Const\,\theta_0^{\Delta/4}
$$
with a constant $\theta_0<1$. This proves Lemma~\ref{OffD}. \qed

\subsection{H\"older approximation}
\label{subsecHA} Here we prove Lemma~\ref{MidSmall}. The problem
we face is that the discontinuities of our vector field
$\alpha_{n-k,k-j}^{(-m-1)}$ on the curve $S^+$ exponentially grow
with $n-j$, whereas we need a bound independent of $n,j,m$. The
discontinuities of our vector field are generated by those of the
unstable foliation $\cG_{n-j-m-1}$ obtained by iterating the
original smooth foliation $\cG$, see Section~\ref{subsecIEGS}.
However, the action of $d\cF$ on the projective tangent space is
contractive within the unstable cone. Since contractions improve
smoothness, the influence of the singularities that have occurred
far back in the past decays exponentially allowing us to get a
uniform estimate in the end.

The singularity curves $S^+\subset\cS_1\setminus\cS_0$ are known
to be $C^2$ smooth with uniformly bounded curvature (see
\cite{C3}, where this fact is proved even for a more general class
of billiards, those in small external fields), hence it is enough
to prove that the restriction of the function
$$
   \hrho=\bigl\|\alpha_{n-k,k-j}^{(-m-1)}\bigr\|_0
$$
to $S^+$ has the required properties, i.e.\
$\hrho\in\cH^{a_1,a_2}(S^+)$ and $\norm{\hrho}_{a_1,a_2} \leq
\Const$. Next, it suffices to construct approximating functions
$\hrho_{\varepsilon}$ for
\begin{equation}
     \label{Epsmall}
   \varepsilon\leq \varepsilon_0\colon
     =\norm{\hrho}_{\infty}^{d}
\end{equation}
for some fixed large $d$. Indeed, if we can construct
$\hrho_\eps$ satisfying
(\ref{ApHold1})--(\ref{ApHold2})
for some $a_1,a_2$ and all $\varepsilon\leq \eps_0$,
then for $\eps>\eps_0$ we can define
$$
   \hrho_{\varepsilon}=
   \begin{cases} 0 & \text{if} \quad \varepsilon\geq
   \norm{\hrho}_{\infty},\cr
   \hrho_{\varepsilon_0} &
   \text{if} \quad \norm{\hrho}_{\infty}>\varepsilon>
    \varepsilon_0,
    \end{cases}
$$
and the resulting family $\{\hrho_\eps\}$ will satisfy
(\ref{ApHold1})--(\ref{ApHold2}) for all $\varepsilon>0$ but
with a different $a_2$. (In fact, the method we present here works also
for $\varepsilon\geq\varepsilon_0$, but we consider only small
$\varepsilon$ in order to avoid dealing with too many different cases.)

Now let
\begin{equation}
   \label{AprWu}
     r=K|\ln \varepsilon|
\end{equation}
where $K>0$ is a sufficiently large constant. Denote by $S^+_p$,
$p\geq 1$, all the \index{H-components} H-components of
$\cF^{-r}(S^+)$ and let $\xi^+$ be the partition of $S^+$ into the
curves $\cF^{r}(S^+_p)$. Next, $\cF$ maps $S^+$ onto a
\index{u-curves (unstable curves)} u-curve
$S^-\subset\cS_{-1}\setminus\cS_0$, and we denote by $S^-_q$, $q\geq
1$, all the \index{H-components} H-components of $\cF^{r-1}(S^-)$ .
Denote by $\xi^-$ the partition of $S^+$ into the curves
$\cF^{-r}(S^-_q)$. Let $\xi=\xi^+\vee\xi^-$ and denote by $\xi(x)$
the element of the partition $\xi$ that contains the point $x$.

We say that an element $W\in\xi$ is {\em large} if
$\length(W)>\varepsilon^{K^2}$ and {\em small} otherwise. We claim that

\begin{lemma}
\label{SmallInt}
The total Lebesgue measure of small intervals is less than
$\Const\,\varepsilon^2$ if $K$ is large enough.
\end{lemma}

\begin{proof}
It is enough to check that
$$
   l\{x\colon \dist(x,\partial\xi(x)) < \varepsilon^{K^2}\}
    < \Const\, \varepsilon^2
$$
This in turn follows from the estimates
$$
   l\bigl\{x\colon \dist(x,\partial\xi^\pm(x))
   <\varepsilon^{K^2}\bigr\}
   < \Const\, \varepsilon^2
$$
and the last bound holds by Proposition~\ref{PrLD} on large
deviations and the growth \index{Growth lemma} lemma~\ref{prgrow}.
\qed
\end{proof}

Next, we consider an element $W\in\xi$. Observe that (\ref{Epsmall}),
(\ref{AprWu}) and (\ref{dFmalpha}) imply $r\gg k-j+m$ (in fact, even
$r\gg 2(k-j+m)$), so that $n-k\gg n-j-m-r$. Hence, the projectors
$\Theta_p^s$ and $\Theta_p^u$ used in the construction of the field
$\alpha_{n-k,k-j}^{(-m-1)}$, cf.\ Section~\ref{subsecIEGS}, are smooth
on the curve $W$ and its preimages, up to $\cF^{-(n-k)}(W)$. Hence all
the discontinuities of $\alpha_{n-k,k-j}^{(-m-1)}$ on $W$ come from the
discontinuities of the field $E^u_{n-j-m-1-r}$ on the curve
$\cF^{-r}(W)$.

We claim that there exist constants $\ttheta<1$ and $a_1>0$ such
that the restriction of $\alpha_{n-k,k-j}^{(-m-1)}$ to each
$W\in\xi$ can be $\ttheta^r$ approximated in the $L^\infty$ metric
by a vector field whose $a_1$-H\"older norm is uniformly bounded.
To prove this claim, consider first the case $r\leq n-j-1-m.$ We
take an arbitrary smooth family of unstable directions $\hE^u$ on
the curve $\cF^{-r}(W)$ (whose derivative along $\cF^{-r}(W)$ is
uniformly bounded), for example, we take the restriction of the
family $E^u$ defined in Section~\ref{subsecIEGS} to $\cF^{-r}(W)$.
Then we use $\hE^u$, instead of $E^u_{n-j-1-m-r}$, to construct an
approximation $\halpha_{n,k,j,m,r}$ to
$\alpha_{n-k,k-j}^{(-m-1)}$. Our claim now follows from a general
fact:

\medskip\noindent{\bf Fact}.
Given a smooth field $\hE^u$ of unstable directions, the
derivatives of $\cF^n\hE^u$ along stable curves grow exponentially
with $n$, but the field $\cF^n\hE^u$ remains H\"older continuous
with some fixed exponent $a>0$ and a H\"older norm bounded
uniformly over all $n\geq 0$.
\medskip

This fact is known as the Invariant Section Theorem \cite[Theorem
5.18]{Sh}, and it has been proven for quite general hyperbolic
systems. For dispersing \index{Dispersing billiards} billiards, we
outline a direct proof in Section~\ref{subsecIST}.

Hence we obtain an $a_1$-H\"older continuous vector field
$\halpha_{n,k,j,m,r}$ with a uniformly bounded H\"older norm on each
$W$. Besides,
$$ \label{HLI}
   \bigl\lVert\halpha_{n,k,j,m,r}-
   \alpha_{n-k,k-j}^{(-m-1)}\bigr\rVert_{\infty}
    \leq\Const\,\ttheta^r
$$
because the angle between $\cF^r(\hE^u)$ and $E^u_{n-j-1-m}$ at every
point $x\in S^+$ is $\cO(\ttheta^r)$.

This proves the claim in the case $r\leq n-j-1-m.$ If the opposite
inequality holds, then the field $\alpha_{n-k,k-j}^{(-m-1)}$ itself is
smooth on $W$ and the claim follows by a direct application of the
Invariant Section Theorem.

Lastly, we need to make our approximative vector field H\"older
continuous on the entire curve $S^+$, which will be done in two
steps. First, let $\halpha_{n,k,j,m,r}^{(1)}$ coincide with
$\halpha_{n,k,j,m,r}$ on large intervals $W\subset S^+$ and be 0
on small ones (see the definition above). We have
$$
    \int_{S^+}\bigl\lVert\halpha_{n,k,j,m,r}^{(1)}-
    \alpha_{n-k,k-j}^{(-m-1)}\bigr\rVert_0\, dl
      \leq\Const(\varepsilon^2 \theta^{k-j+m}+ \ttheta^r)
$$
where the first term estimates the contribution of the small intervals
via Lemma \ref{SmallInt} and (\ref{dFmalpha}) and the second term estimates
the contribution of the large intervals via (\ref{HLI}).
This modification eliminates most of the discontinuities of
$\halpha_{n,k,j,m,r},$ however $\halpha_{n,k,j,m,r}^{(1)}$ is not
yet globally H\"older continuous -- it can have jumps of size
$\cO(\theta^{m+k-j})$ at the endpoints of each large interval. The
total number of jumps is twice the number of large intervals,
i.e.\ $\leq \Const\, \varepsilon^{-K^2}$. Now we further modify
$\halpha_{n,k,j,m,r}^{(1)}$ by replacing it with a linear function
in the $\varepsilon^{K^2+1}$ neighborhood of each jump, so that
the new modification, we call it $\halpha_{n,k,j,m,r}^{(2)}$,
becomes continuous on $S^+$. It is easy to see that
$$
    \int_{S^+}\bigl\lVert\halpha_{n,k,j,m,r}^{(2)}-
    \alpha_{n-k,k-j}^{(-m-1)}\bigr\rVert_0\, dl
      \leq\Const(\varepsilon\theta^{k-j+m}+ \ttheta^r)
$$
and the $a_1$-H\"older norm of the new approximation is
$\leq\Const\,\varepsilon^{-a_1(K^2+1)}$. So we set $a_2=a_1(K^2+1)$ and
obtain the required approximation to $\alpha_{n-k,k-j}^{(-m-1)}$ on
$S^+$. \qed
\newpage

\chapter{Moment estimates}
\label{ScME} \setcounter{section}{6}\setcounter{subsection}{0}

The main goal of this section is to prove Proposition~\ref{Weak}
and thus establish Theorem~\ref{ThVelBM}. Our proof is based on
various moment estimates of the underlying processes. In addition,
we obtain some more estimates to be used in the proofs of
Theorems~\ref{Root} and \ref{ThSmall} presented in the subsequent
sections.

\subsection{General plan} \label{subsecMom} Here we formulate several
propositions that constitute the basis of our arguments. Their
proofs are provided in Sections~\ref{subsecME-ST}--\ref{subsecMart}.

The proof of Theorem \ref{ThVelBM} uses martingale approach of
Stroock and Varadhan. Let us briefly recall the main ideas of this
approach postponing the details till Section~\ref{subsecMart}.
According to \cite{SV} in order to show that $\cX_\tau$ is a
diffusion process with generator $\cL$ it is enough to check that
for a large set of observables $B$ the process
$$ \cM_\tau=B(\cX_\tau)-\int_0^\tau \cL B(\cX_\sigma) d\sigma $$
is a martingale. Thus one has to check that for sufficiently smooth functions
$B_1, B_2\dots B_m$ and for all
$s_1\leq s_2\dots \leq s_m\leq \tau_1\leq \tau_2$
$$ \EXP\left(\left[\prod_{k=1}^m B_k(\cX_{s_k})\right]
(\cM_{\tau_2}-\cM_{\tau_1})\right)=0 . $$
Therefore in order to show that a family of random processes
$\{\cX_\tau,M\}$ converges to $\cX_\tau$ as $M\to\infty$
we need to show that
$$ \EXP\left(\left[\prod_{k=1}^m B_k(\cX_{s_k,M})\right]
(\cM_{\tau_2,M}-\cM_{\tau_1,M})\right)\to 0 . $$
To derive this one usually divides the segment
$[\tau_1, \tau_2]$ into small segments
$$\tau_1=t_1\leq t_2\dots \leq t_N=\tau_2$$
and uses Taylor development of $B$ to estimate
$B(\cX_{t_{j+1},M})-B(\cX_{t_{j},M}).$ Thus one needs to control the moments
$\EXP([\cX_{t_{j+1},M}-\cX_{t_j,M}]^p).$ Our first task is to control
the moments of $Q$ and $V.$

We need some notation.
Let $\bn=\kappa_M\sqrt{M},$ where for the proof of
Theorem~\ref{ThVelBM} we set $\kappa_M=M^{-\delta}$ with some
$\delta>0$, whereas for the proof of Theorem~\ref{Root} we will
need $\kappa_M$ to be a small positive constant (independent of
$M$).

The estimates of Propositions~\ref{PrDistEq1} and \ref{PrDistEq2}
require that $(Q,V)\in\Upsilon_{\delta_1}$, see (\ref{Upsilon}), so
we have to exclude the orbits leaving this region. Fix a
$\delta_2\in (\delta_1,\delta_0)$ and let $\ell=(\gamma, \rho)$ be
\index{Standard pair} a standard pair such that $\length(\gamma)\geq
M^{-100}.$ We define, inductively, subsets
$$
  \emptyset=I_0\subset I_1\subset
  \dots I_k\subset I_{k+1} \subset \dots \subset \gamma,
$$
which we will exclude from $\gamma$, as follows. Suppose that
$I_k$ is already defined so that

(i) $\cF^{k\bn} (\gamma\setminus I_k)= \bigcup_{\alpha}
\gamma_{\alpha,k}$, where for each $\alpha$ we have
$\length(\gamma_{\alpha,k})>M^{-100}$, and
$\ell_{\alpha,k}=(\gamma_{\alpha,k}, \rho_{\alpha,k})$ is a standard
pair (here $\rho_{\alpha,k}$ is the density of the measure
$\cF^{k\bn}(\mes_\ell)$ conditioned on $\gamma_{\alpha,k}$);

(ii) $\pi_1(\gamma\setminus I_k)\subset\Upsilon_{\delta_2}$, cf.\
(\ref{Upsilon}).

\noindent Now, by Proposition~\ref{PrDistEq1}, for each $\alpha$
$$
         \cF^{k\bn} \gamma_{\alpha,k}=
         \Big(\bigcup_{\beta}
         \gamma_{\alpha, \beta, k+1}\Big)
         \bigcup \tgamma_{\alpha,k+1}
$$
where $\mes_{\ell_{\alpha,k}}(\tgamma_{\alpha,k+1})<M^{-50}$ and
$\length(\gamma_{\alpha,\beta,k+1})>M^{-100}$ for each $\beta$.

We now define
$$
   I_{k+1} = I_k \cup \left (\cup_{\alpha}
   \cF^{-(k+1)\bn}(\tgamma_{\alpha,k+1})\right )
   \cup \left (\cup_{\alpha,\beta}^\ast
   \cF^{-(k+1)\bn}(\gamma_{\alpha,\beta, k+1}) \right )
$$
where $\cup_{\alpha,\beta}^\ast$ is taken over all pairs
$\alpha,\beta$ such that $\pi_1(\gamma_{\alpha,\beta,
k+1})\notin\Upsilon_{\delta_2}$. For each $x\in\ell$ let
$\bk(x)=\min\{k\colon\, x\in I_k\}$ (we set $\bk(x)=\infty$ if
$x\nsubseteq I_k$ for any $k$).

For brevity, for each $x\in\Omega$ we denote the point $\cF^n(x)$ by
\index{Standard pair} $x_n=(Q_n,V_n,q_n,v_n)$. Given a standard pair
$\ell=(\gamma,\rho)$ as above, we define for every $x\in\gamma$
\beq
  \hQ_n=\left \{
  \begin{array}{cc} Q_n & {\rm for}\ \ n<\bk\bn\\
    Q_{\bk\bn} & {\rm for}\ \ n\geq \bk\bn \end{array}\right .\qquad
  \hV_n=\left \{
  \begin{array}{cc} V_n & {\rm for}\ \ n<\bk\bn\\
  0 & {\rm for}\ \ n\geq \bk\bn \end{array}\right .
   \label{hQhVn}
\eeq

Recall that Theorem~\ref{ThVelBM} claims a weak convergence of the
stochastic processes $\cQ(\tau)$ and $\cV(\tau)$ on any finite
time interval $0<\tau<c$. From now on, we fix $c>0$ and set
$\brc=c/\brL$, where $\brL$ is the mean free path (\ref{FPL}).

Let $A\in\fR$ be a function satisfying two additional
requirements:\\ (a) $\mu_{Q,V}(A)=0$ for all $Q,V$;\\ (b)
$\mu_{Q,V}(A^k)$ is a Lipschitz continuous function of $Q$ and $V$
for $k=2,3,4.$

Given a standard pair $\ell=(\gamma,\rho)$ as above and
\index{Standard pair} $x\in\gamma$, we put $A(x_j) = A(\cF^jx)$ and
$S_n(x) = \sum_{j=0}^{n-1} A(x_j)$. We also set
$$
  \hA(x_j)=\left\{ \begin{array}{ll}
  A(x_j) & \text{if}\ \ x\not\in I_{[j/\bn]}\\
  0 & \text{otherwise}\end{array}\right.
  \qquad \text{and} \qquad
  \hS_n=\sum_{j=0}^{n-1} \hA(x_j)
$$
cf.\ (\ref{hQhVn}). In Section~\ref{subsecMEAp} we will prove

\begin{proposition}
\label{PrMomA} The following bounds hold uniformly in $M^{1/2}\leq n
\leq \brc M^{2/3} $ and for all standard pairs $\ell=(\gamma, \rho)$
such that \index{Standard pair}
\beq
     \label{Upsilon*}
   \pi_1(\gamma)\subset
   \Upsilon^\ast_{\delta_2,a}\colon=
   \{\dist(Q,\partial\cD)>\br +\delta_2,
   \ \ \|V\|<a M^{-2/3}\}
\eeq
and  $\length(\gamma)>M^{-100}$:
$$ {\rm (a)} \quad \EXP_\ell\bigl(\hS_n\bigr)=
\cO\left(M^\delta n/\bn\right). \Quad $$
$${\rm (b)} \quad \EXP_\ell\bigl(\hS_n^2\bigr)=\cO(n).\Quad $$
$${\rm (c)} \quad \EXP_\ell\bigl(\hS_n^4\bigr)=\cO(n^2).\Quad $$
{\rm (d)} The last estimate can be specified as follows. Let
\beq
   \label{AAsmall}
   \fS_A=\max_{(Q,V)\in\Upsilon_{\delta_2}}
   \max\left\{\mu_{Q,V} (A^2),D_{Q,V} (A)\right\}
\eeq
where
\beq
    D_{Q,V}(A)=\sum_{j=-\infty}^\infty \mu_{Q,V}
    \left(A\, (A\circ \cF_{Q,V}^j)\right)
       \label{DQVA}
\eeq
Then
$$
    \EXP_\ell\left(\hS^4_n\right)\leq
    2\fS_A^2n^2+\cO\left(n^{1.9}\right).
$$
\end{proposition}

We remark that part {\rm (d)} will only be used in
Chapter~\ref{ScSLP}, in the proof of Theorem~\ref{ThSmall}.

Note that our functions $\hQ_n$, $\hV_n$, $\hA_n$ and $\hS_n$ are
only defined on the selected standard pair $\ell$ and not on the
entire phase space $\Omega$ yet. Given an auxiliary \index{Auxiliary
measures} measure $m\in\fM$, we can use the corresponding partition
of $\Omega$ into standard pairs $\{\ell=(\gamma,\rho)\}$, see
Section~\ref{subsecMS}, and define our functions on all the pairs
$\{(\gamma,\rho)\}$ with $\length(\gamma)>M^{-100}$, and then simply
set these functions to zero on the shorter standard pairs. Now our
functions are defined on $\Omega$ (but they depend on the measure $m
\in \fM$ and the decomposition (\ref{aux1})). \index{Standard pair}

Next, given $m\in\fM$ and $x\in\Omega$, we define continuous
functions $\tQ(\tau)$, $\tV(\tau)$, and $\tS(\tau)$ on the
interval $(0,\brc)$ by
\begin{equation}
   \label{Rescale}
      \tQ(\tau)=\hQ_{\tau M^{2/3}},\quad
      \tV(\tau)=M^{2/3} \hV_{\tau M^{2/3}},\quad
      \tS(\tau)=M^{-1/3} \hS_{\tau M^{2/3}}
\end{equation}
(these formulas apply whenever $\tau M^{2/3}\in\integers$, and
then we use linear interpolation in between). In a similar way,
let $t_n$ be the time of the $n$th collision, $\hht_n=t_{\min\{n,
\bk\bn\} }$ the modified time, and then we define a continuous
function
\beq
    \ttt(\tau)=  M^{-1/3}\Bigl[\hht_{[\tau M^{2/3}]}
    -\brL\, \min\{\tau M^{2/3},\bk\bn\}\Bigr]
      \label{sigmatilde}
\eeq
where $\brL$ is the mean free path, cf.\ (\ref{FPL}). We note that
our normalization factors in (\ref{Rescale})--(\ref{sigmatilde})
are chosen so that the resulting functions typically take values
of order one, as we prove next.

Let us fix $a>0$ and for each $M>1$ choose an auxiliary
\index{Auxiliary measures} measure $m\in\fM$ such that
\beq
    m\left(\pi_1^{-1}\left(\Upsilon_{\delta_2,a}^\ast\right )
    \right)= 1 ,
       \label{mgood}
\eeq
see (\ref{Upsilon*}). Then each function $\tS(\tau)$, $\tQ(\tau)$,
$\tV(\tau)$, $\ttt(\tau)$ induces a family of probability measures
(parameterized by $M$) on the space of continuous functions
$C[0,\brc]$. We will investigate the tightness of these families.
All our statements and subsequent estimates will be uniform over the
choices of auxiliary \index{Auxiliary measures} measures $m\in\fM$
satisfying (\ref{mgood}).

In Section~\ref{subsecTight} we will prove

\begin{proposition}
\label{PrTight} $\ $ \\ {\rm (a)} For every function $A\in\fR$
satisfying the assumptions of Proposition~\ref{PrMomA}, the family
of
functions $\tS(\tau)$ is tight;\\
{\rm (b)} the families $\tQ(\tau)$, $\tV(\tau)$, and $\ttt(\tau)$
are tight.
\end{proposition}

\begin{corollary}
\label{CrTight} For every sequence $M_k\to\infty$ there is a
subsequence $M_{k_j}\to\infty$ along which the functions
$\tQ(\tau)$ and $\tV(\tau)$ on the interval $0<\tau<\brc$ weakly
converge to some stochastic processes $\hbQ(\tau)$ and
$\hbV(\tau)$, respectively.
\end{corollary}

Our next step is to use the tightness of $\hQ$ and $\hV$ to improve
the estimates of Proposition~\ref{PrMomA}(b). In
Section~\ref{subsecM2} we will establish

\begin{proposition}
\label{PrMom} Let $\varkappa\ll\brc$ be a small positive constant
and $n=\varkappa M^{2/3}$. The following estimates hold uniformly
\index{Standard pair} for all standard pairs $\ell=(\gamma, \rho)$
such that $\pi_1(\gamma)\subset\Upsilon^\ast_{\delta_2,a}$ and
$\length(\gamma)>M^{-100}$, and all $(\brQ,\brV)\in\pi_1(\gamma)$:
$${\rm (a)}\quad \EXP_\ell\bigl(\hV_n-\brV\bigr)=
\cO\left(\frac{1}{M^{1/3-\delta}\bn} \right).\Quad $$
$${\rm (b)}\quad \EXP_\ell\left ( (\hV_n-\brV)(\hV_n-\brV)^T
\right)= \bigl(\brsigma^2_{\brQ}(\cA)+o_{\varkappa\to
0}(1)\bigr)\, \varkappa\, M^{-4/3} \Quad .$$
$${\rm (c)}\quad \EXP_\ell\left(\norm{\hV_n-\brV}^4\right)
=\cO\left(\varkappa^2 M^{-8/3}\right)\Quad$$
$${\rm (d)}\quad \EXP_\ell\left(\hQ_n-\brQ\right) = (1+o_{\varkappa\to 0}(1))\,
\varkappa M^{2/3} \brL \brV +
\cO\left(\varkappa^{3/2}\right),\Quad
$$
$${\rm (e)}\quad \EXP_\ell\left(\norm{\hQ_n-\brQ}^2\right)
=\cO\left(\varkappa^2\right).\Quad $$
\end{proposition}

Let us now fix some $\delta_3\in (\delta_2,\delta_0)$. Let
$B(Q,V)$ be a $C^3$ smooth function of $Q$ and $V$ with a compact
support whose projection on the $Q$ space lies within the domain
$\dist(Q,\dcD)>\br+\delta_3$. Define a new function $\cL B(Q,V)$
by
$$
    (\cL B)(Q,V)= \brL \la V, \nabla_Q B\ra +\frac{1}{2}
    \sum_{i,j=1}^2 \left(\brsigma^2_Q(\cA)\right)_{ij}\,
    \partial^2_{V_i,V_j} B
$$
where $V_1$ and $V_2$ denote the components of the vector $V$, and
$\left(\brsigma^2_Q(\cA)\right)_{ij}$ stand for the components of
the matrix $\brsigma^2_Q(\cA)$. As before, for each $M>1$ we choose
a measure $m\in\fM$ satisfying (\ref{mgood}). In
Section~\ref{subsecMart} we will prove

\begin{proposition}
\label{PrGen}
Let $(\hbQ, \hbV)$ be a stochastic process that is a
limit point, as $M\to\infty$, of the family of functions
$(\tQ(\tau),\tV(\tau))$ constructed above. Then the process
$$
       \bM(\tau) = B(\hbQ(\tau), \hbV(\tau))
       -\int_0^\tau (\cL B)
       (\hbQ(s), \hbV(s))\, ds
$$
is a martingale.
\end{proposition}

Proposition \ref{PrGen} implies, in virtue of \cite[Theorem
4.5.2]{SV}, that any limit process $(\hbQ, \hbV)$ satisfies
$$
     \begin{array}{ll}
     d\hbQ = \brL \hbV\, d\tau, & \hbQ(0) = Q_0 \\
     d\hbV = \brsigma_{\hbQ}(\cA)\, d\bw(\tau),\ \ \ \  & \hbV(0) = 0
     \end{array}
$$

We will need an analogue of the above result for continuous time.
For each $x\in\Omega$ consider two continuous functions on $(0,
\brc \brL)$ defined by
$$
  \tQ_\ast(\tau)=\left \{
  \begin{array}{cc} Q(\tau M^{2/3}) & {\rm for}\ \ \tau<\tilde{\tau}\\
    Q(\tilde{\tau}) & {\rm for}\ \ \tau\geq \tilde{\tau}
    \end{array}\right.
$$
and
$$
  \tV_\ast(\tau)=\left \{
  \begin{array}{cc} M^{2/3} V(\tau M^{2/3}) & {\rm for}\ \ \tau<\tilde{\tau}\\
  0 & {\rm for}\ \ \tau\geq\tilde{\tau} \end{array}\right .
$$
where
\beq
      \label{tautilde}
   \tilde{\tau}=M^{-2/3}\inf\{t>0\colon\,
    (Q(t),V(t))\notin\Upsilon_{\delta_2}\}
\eeq
In Section~\ref{subsecMart} we will derive

\begin{corollary}
\label{CrCTMart} Suppose that $(\tQ(\tau), \tV(\tau))$ converges
along some subsequence $M_k \to \infty$ to a process $(\hbQ(\tau),
\hbV(\tau)).$ Then $(\tQ_\ast(\tau), \tV_\ast(\tau))$ converges
along the same subsequence $\{M_k\}$ to $$(\bQ_\ast(\tau),
\bV_\ast(\tau))=(\hbQ(\tau/\brL), \hbV(\tau/\brL)).$$
\end{corollary}

Corollary \ref{CrCTMart} implies that $(\bQ_\ast, \bV_\ast)$
satisfies (\ref{TIBM}) up to the moment when $\dist(\bQ_\ast(\tau)
,\partial\cD)=\br+\delta_3 .$ We can now prove
Proposition~\ref{Weak} (a). Observe that the difference between
the limit process $\bQ_\ast(\tau)$ above and the $\bQ(\tau)$
involved in Theorem~\ref{ThVelBM} is only due to the different
stopping rules (\ref{tautilde}) and (\ref{tast}), respectively. In
particular, $\bQ_\ast$ can be stopped earlier than $\bQ$ if for
some $t\leq \brc \brL M^{2/3}$ we have $MV^2(t)\geq 1-\delta_2$
but $\dist(Q(s),\partial\cD)\geq \br+\delta_0$ for all $s\leq t.$
By Proposition~\ref{PrTight} (b) the probability of this event
vanishes as $M\to\infty.$ Thus any $(\bQ,\bV)$ is obtained from
the corresponding $(\bQ_\ast,\bV_\ast)$ by stopping the trajectory
of the latter as soon as
$\dist(\bQ_\ast(\tau),\partial\cD)=\br+\delta_0$. This fact
concludes the proof of the main claim of Proposition~\ref{Weak}
(a). Its part (b) will be proved in Section~\ref{subsecUSDE}. \qed

\subsection{Structure of the proofs} After having completed a formal
description of all the intermediate steps in the proof of
Proposition~\ref{Weak}, let us give an informal overview of the
underlying ideas.

Our argument derives from the martingale method of Stroock and
Varadhan \cite{SV}, which is based on the estimation of the first
two moments of $V_n-V_0.$ These are provided by
Proposition~\ref{PrMom}, especially its part (b) saying that
$\forall u\in\reals^2$
\beq
       \label{M1}
      \EXP_{\ell}\left(\la V_n-V_0, u\ra^2\right)
      \sim n M^{-2} \la\brsigma^2_{Q_0}(\cA) u, u\ra
\eeq
for $n\sim \varkappa M^{2/3}$ (see our discussion in
Section~\ref{subsecSR2} for the motivation of this identity). Our
proof of (\ref{M1}) proceeds in two steps: first we show that
\beq
         \label{M2}
   \EXP_{\ell}\left(\la V_n-V_0, u\ra^2\right)\sim M^{-2}
   \sum_{j=0}^{n-1} \la\brsigma^2_{Q_j}(\cA) u, u\ra
\eeq
(cf.\ Lemma \ref{AP1St}), and then we approximate
\beq
       \label{M3}
    Q_j\sim Q_0
\eeq
(cf. Proposition \ref{PrTight}).

For any fixed $j$, the approximation (\ref{M3}) follows from
\beq
       \label{M4}
    \EXP_{\ell}\left(\|V_n-V_0\|^2\right)\leq
    \Const\,n/M^2
\eeq
by the Cauchy--Schwartz inequality. In order to get (\ref{M3}) for
all $j$ uniformly, we need to control the fourth moment, which can
be derived from (\ref{M2}) with little difficulty. In turn,
(\ref{M4}) itself follows from (\ref{M2}), hence the key step is
to establish (\ref{M2}).

The proof of (\ref{M2}) is essentially  based on the
\index{Equidistribution} equidistribution
(Proposition~\ref{PrDistEq2} and Corollary~\ref{PrDistEq3}). For
small $n$, Proposition~\ref{PrDistEq2} suffices. However, for $n\sim
M^{2/3}$ the term $n\|\brV\|$ becomes of order 1, so the estimates
of Proposition~\ref{PrDistEq2} alone are too crude. In that case we
first establish (\ref{M2}) for ``short term'', $n\sim
M^{1/2-\delta}$ (Section~\ref{subsecME-ST}), and then derive
(\ref{M2}) for ``long term'', $n\sim M^{2/3}$, via
Corollary~\ref{PrDistEq3} and the inductive estimate
$$
    \EXP_{\ell}\Bigl(\bigl\lVert V_{(j-1)M^{1/2-\delta}}
    \bigr\rVert\Bigr)\leq \Const\,
    \frac{\sqrt{(j-1) M^{1/2-\delta}}}{M},
$$
see details in Section~\ref{subsecMEAp}.

Finally let us comment on the above inductive step. Denote
$A^{(u)}=\la\cA,u\ra$ for $u \in \reals^2$ and consider the
expansion
\beq
      \label{M5}
     \EXP_{\ell}\Biggl(\biggl[\sum_{j=1}^n A^{(u)}(x_j)\bigg]^2\Biggr)
       =\sum_{i,j=1}^n
     \EXP_{\ell}\left(A^{(u)}(x_i) A^{(u)}(x_j)\right)
\eeq
Our early estimate (\ref{M2}) effectively states that the main
contribution to (\ref{M5}) comes from nearly diagonal ($i\approx
j$) terms. Thus to prove (\ref{M2}), it will suffice to bound the
contribution of the off-diagonal terms in (\ref{M5}). There are
two possible approaches to this task:

(I) Use Corollary~\ref{PrDistEq3} to estimate
$\EXP_{\ell}\left[(A^{(u)}(x_i) A^{(u)}(x_j)\right].$ Since we
expect the change of $V$ to be of order $M^{-2/3}$, the best
estimate we can get in this way is $\EXP_{\ell}\left[(A^{(u)}(x_i)
A^{(u)}(x_j)\right] =\cO(M^{-2/3} \ln M).$ Because there are
$[M^{2/3}]^2$ terms in (\ref{M5}), this approach would provide an
off-diagonal bound of $M^{2/3} \ln M$, which is way too crude --
it is even larger than the main term $\cO(M^{2/3})$.

(II) For a fixed $i$, we can try to get the inductive bound
$$
     \EXP_{\ell}\biggl(\sum_j A^{(u)}(x_j) A^{(u)}(x_i)\biggr)\leq
     \cO\left(\sqrt{\text{\# of terms}}\right)=\cO(M^{1/3}).
$$
This would give an off-diagonal bound of $\cO(M^{2/3+1/3}) =
\cO(M)$, which is even worse...

Hence neither approach alone seems to handle the task, but they can
be combined together to produce the necessary bound, in the
framework of the so called ``big-small \index{Big-small block
technique} block techniques''. Namely, we divide the interval $[1,
n]$ into ``big'' blocks of size $M^{1/2-\delta}$ separated by
``small'' blocks of size $M^\delta.$ The total contribution of the
small blocks is negligible, and denoting by $P_j'$ the contribution
of the $j$th big block and setting $U_k'=\sum_{j=1}^k P_j'$ we can
get
$$
       \EXP_{\ell}\left[(U_{k+1}')^2-(U_k')^2\right]=
       \EXP_{\ell}\left[(P_{k+1}')^2\right]+
       2\,\EXP_{\ell}\left[U_k' P_{k+1}'\right] .
$$
The first term here can be handled by the method (I), while for
the cross-product term we get, by Proposition~\ref{PrDistEq2},
$$
        \left|\EXP_{\ell}(U_k' P_{k+1}')\right|\leq
        \Const\, \EXP_{\ell}(|U_k'|)\,\EXP_\ell(P_k') ,
$$
and apply the method (II) to show that the first factor is of
order $\sqrt{k M^{1/2+\delta}}$, while the second factor is of
order 1 by the method (I). This approach yields the necessary
bound on the off-diagonal terms in (\ref{M5}) and thus proves
(\ref{M2}).

\subsection{Short term moment estimates for $V$}
\label{subsecME-ST} Here we estimate the moments of the velocity $V$
during time intervals of length $n=\cO(\sqrt{M})$, which are much
shorter than $\cO(M^{2/3})$ required for Theorem~\ref{ThVelBM}. Our
estimates will be used later in the proof of
Proposition~\ref{PrMomA}. The main result of this subsection is

\begin{proposition}
\label{PrST} Let $\ell=(\gamma,\rho)$ be a standard pair such that
\index{Standard pair} $\pi_1(\gamma)\subset\Upsilon_{\delta_1}$ and
$\length(\gamma)>M^{-100}$. Then for all $(\brQ,\brV)\in
\pi_1(\gamma)$ and $M^{1/3}\leq n \leq \delta_\diamond M^{1/2}$ we
have
$${\rm (a)}\quad \EXP_\ell(V_n-\brV)=
\cO\left(M^{\delta-1}\right)
\Quad.$$
$${\rm (b)} \quad \EXP_\ell(\|V_n-\brV\|^2)=
\cO\left(n/M^2\right).\Quad$$ Here $\delta_\diamond \ll \delta_1$
is the constant of Proposition~\ref{PrDistEq2}.
\end{proposition}

{\em Proof}. Let $\Delta V_j=V_{j+1}-V_j.$ Then by (\ref{CR2a})
\beq
   \Delta V_j=\frac{\cA\circ\cF^j}{M}+
   \cO\left(\frac{1}{M^{3/2}}\right)
     \label{DeltaVexp}
\eeq
where $\cA$ is as defined by (\ref{MEx}). Hence
$$
   V_n-\brV= \frac 1M\,\sum_{i=0}^{n-1}
   \cA\circ \cF^i+\cO\left(\frac{n}{M^{3/2}}\right)
$$
and then
$$
   \norm{V_n-\brV}^2\leq
   \frac{2}{M^2}\,\biggl\lVert\sum_{i=0}^{n-1}
   \cA\circ \cF^i\biggr\rVert^2+\cO\left(\frac{n^2}{M^3}\right).
$$
Therefore Proposition \ref{PrST} follows from the next result:

\begin{proposition}
\label{GenST}
Let $A\in\fR$ be a function satisfying
$\mu_{Q,V}(A)=0$ for all $Q,V$, and $\ell$ and $n$ be as in
Proposition \ref{PrST}. Then
$${\rm (a)}\quad \EXP_\ell(S_n)=
\cO\left(M^\delta\right)
\Quad.$$
$${\rm (b)} \quad \EXP_\ell(S_n^2)=\cO(n) .\Quad$$
\end{proposition}

The proof uses the big small block techniques \cite{Br}. For each
$k=0,\dots,[n/M^{1/3}]$ denote
$$
     R_k'=\sum_{j=kM^{1/3}+M^\delta}^{(k+1)M^{1/3}} A(x_j),
      \qquad
     R_k''=\sum_{j=k M^{1/3}}^{kM^{1/3}+M^\delta-1}A(x_j),
$$
$$
    Z_k'=\sum_{j=0}^{k-1} R_j',
      \qquad
    Z_k''=\sum_{j=0}^{k-1} R_j''.
$$
Observe that $Z_k''\leq \|A\|_{\infty}n/M^{1/3-\delta}.$ Next we prove
two lemmas:

\begin{lemma}
\label{Lm1ST} For every $k$,
$${\rm (a)} \quad \EXP_{\ell}(R_k')=
\cO\left(M^{1/3+\delta}\left(\|\brV\|+n/M\right)\right), \Quad
$$
$${\rm (b)} \quad \EXP_{\ell}\bigl([R_k']^2\bigr)
=\cO\left(M^{1/3}\right). \Quad $$
\end{lemma}

\begin{lemma}
\label{Lm2ST} Given $A$ as above, there exists $D>0$ such that
$${\rm (a)} \quad \EXP_{\ell}(Z_{k+1}')=\EXP_{\ell}(Z_k')+
\cO\left(M^{1/3+\delta}\left(\|\brV\|+n/M\right)\right), \Quad$$
$${\rm (b)} \quad \EXP_{\ell}\bigl([Z_{k+1}']^2\bigr)
=\EXP_{\ell}\bigl([Z_k']^2\bigr)+ \cO\left(M^{1/3}\right)+
\cO\big(\sqrt{k+1}\, M^{1/2+\delta} (\|\brV\|\qquad
$$
$\qquad +n/M)\big)$ and $\EXP_{\ell}\bigl([Z_{k}']^2\bigr)\leq
DM^{1/3}(k+1) .$
\end{lemma}

\noindent {\it Proof of Lemma \ref{Lm1ST}.} Applying
Corollary~\ref{PrDistEq3} to $n_1\leq n$ iterations of $\cF$,
setting $j=M^{\delta/4}$, and using the obvious bound
$$
  \|V_{n_1-j}\|\leq \|\brV\|+\Const\, n_1/M
$$
we get
$$
   \EXP_\ell(A(x_{n_1})) =
   \cO\left(\left(\|\brV\|+
   n_1/M\right)M^{\delta/2} \right)
   + \cO\left(M^{3\delta/4-1}\right)
$$
Now (a) follows by summation over $kM^{1/3}+M^{\delta} \leq n_1
\leq (k+1)M^{1/3}$.

To prove (b) we write
\beq
    (R'_k)^2 = \sum_{i,j} A(x_i)A(x_j)=
    2 \sum_{i<j} A(x_i)A(x_j)+\cO\left(M^{1/3}\right).
      \label{RRdec}
\eeq
Thus it suffices to show that
\begin{equation}
     \label{2Cor}
     \left|\EXP_\ell\left(A(x_i)A(x_j)\right)\right|<\Const\,
     \left(\theta_A^{j-i}+ \left(\|\brV\|+n/M\right) M^\delta\right)
\end{equation}
for some $\theta_A<1.$ To prove (\ref{2Cor}) we apply
Proposition~\ref{PrDistEq1} with $n=(i+j)/2$. Denoting $m=(j-i)/2$
we obtain
$$
     \EXP_\ell\left(A(x_i)A(x_j)\right)=
     \sum_\alpha c_\alpha \EXP_{\ell_\alpha}
     \left(A(x_{-m}) A(x_m)\right).
$$
If $\length(\gamma_\alpha)>\exp(-m/K)$ where $K$ is the constant
of Proposition~\ref{PrDistEq1}, choose $\brx_\alpha\in
\gamma_\alpha.$ Due to (\ref{expclose}) and the H\"older
continuity of $A$, for any $x_\alpha\in \gamma_\alpha$ we have
$|A(\cF^{-m} x_\alpha)-A(\cF^{-m} \brx_\alpha)| =\cO
(\theta_A^{m})$ for some constant $\theta_A <1$, therefore
\beq
   \EXP_{\ell_\alpha}\left(A(x_{-m})A(x_m)\right)=
   A(\cF^{-m} \brx_\alpha)\, \EXP_{\ell_\alpha}\left(A(x_m)\right)+
   \cO(\theta_A^m).
     \label{2Cortemp}
\eeq
By the argument used in the proof of Lemma~\ref{Lm1ST} (a)
$$
      \EXP_{\ell_\alpha}A(x_m)=
      \cO\left( (\|\brV\|+m/M)M^{\delta/2}\right ),
$$
hence
$$
   \EXP_{\ell_\alpha}\left(A(x_{-m})A(x_m)\right)=
   \cO\left(\theta_A^{m}+
   \left (\|\brV\|+m/M\right ) M^\delta\right) .
$$
On the other hand, the contribution of $\alpha$'s which satisfy
$\length(\gamma_\alpha)\leq\exp(-m/K)$ is exponentially small due to
Proposition \ref{PrDistEq1}. Summation over $\alpha$ gives
(\ref{2Cor}). Lastly, the summation over $i,j$ and remembering that
$n\leq \delta_\diamond \sqrt{M}$ and $\|\brV\|<1/\sqrt{M}$ yields
Lemma~\ref{Lm1ST} (b). \qed \medskip

\noindent{\em Proof of Lemma \ref{Lm2ST}.} The part (a) follows
directly from Lemma~\ref{Lm1ST} (a). To prove part (b) we expand
\beq
   \EXP_\ell\bigl([Z_{k+1}']^2\bigr)=
   \EXP_\ell\bigl([Z_{k}']^2\bigr)+
   \EXP_\ell\bigl([R'_{k+1}]^2\bigr)+
   \EXP_\ell\bigl(Z'_k R_{k+1}'\bigr).
     \label{ZZthree}
\eeq
The second term is $\cO(M^{1/3})$ by Lemma~\ref{Lm1ST} (b). We will
show that the last term is much smaller, precisely
\beq
   \EXP_\ell\bigl(Z'_k R_{k+1}'\bigr) =
   \cO\bigl(M^{\frac{1}{12}+\delta}\bigr)
     \label{ZZthird}
\eeq
The argument used in the proof of (\ref{2Cor}) gives
$$
   \bigl|\EXP_\ell(Z'_k A_{j})\bigr|\leq
   \Const\, \EXP_\ell |Z_k'|
   \Bigl(\theta_A^{M^\delta}+
   \bigl(\|\brV\|+n/M\bigr)
   M^\delta \Bigr)
$$
for $(k+1) M^{1/3}\leq j \leq (k+2) M^{1/3}$. Hence
\begin{align*}
   \bigl|\EXP_\ell\bigl(Z'_k R_{k+1}'\bigr)\bigr|
   & \leq
   \Const\, \EXP_\ell |Z_k'|\left(\theta_A^{M^\delta}+
   \left (\|\brV\|+n/M\right ) M^\delta\right) M^{1/3}\\
   & \leq
    \Const\, \sqrt{\EXP_\ell\bigl([Z'_k]^2\bigr)}
    \left (\|\brV\|+n/M\right ) M^{1/3+\delta}.
\end{align*}
(where we used the Cauchy-Schwartz inequality). By induction
$$
   \EXP_\ell(Z'_k R_{k+1}')\leq
   C\sqrt{D}\Bigl( \sqrt{(k+1) M^{1/3}}\,
   \bigl(\|\brV\|+n/M\bigr) \Bigr) M^{1/3+\delta}.
$$
Since $k\leq \delta_\diamond M^{1/6}$, the the right hand is
$\cO(M^{1/12+\delta})$. If $D$ is sufficiently large, this implies
both inequalities of part (b) for $k+1$ and thus completes the
proof of Lemma~\ref{Lm2ST}. \qed
\medskip

\noindent{\em Proof of Proposition \ref{GenST}.} To simplify our
analysis we assume that $n=k M^{1/3}$ for some integer $k$, so that
$S_n=Z_k'+Z_k''.$ Similarly to the proof of Lemma~\ref{Lm1ST} (a) we
get
$$
     \EXP_\ell\bigl(R_j''\bigr)=
     \cO\left(M^\delta ( \|\brV\|+n/M) \right)
$$
for all $1\leq j\leq k-1$ and
$$
     \EXP_\ell\bigl(R_0''\bigr)=
     \cO\left(M^\delta\right).
$$
(The difference between the first term and the others is due to
the restriction $n>K|\ln\length(\gamma)|$ in
Proposition~\ref{PrDistEq1}.) Combining the above estimates with
Lemma~\ref{Lm1ST} (a) we obtain part (a) of
Proposition~\ref{GenST}. To prove part (b) we estimate
$$
         \EXP_\ell \left (S_n^2\right )\leq 2\,
         \EXP_\ell\left ([Z_k']^2\right )+ 2\,
         \EXP_\ell\left([Z_k'']^2\right )=\cO(n+k^2 M^{2\delta}).
$$
This completes the proof of Proposition~\ref{GenST} and hence that
of \ref{PrST}.

\subsection{Moment estimates--{\it a priori} bounds}
\label{subsecMEAp} Here we prove Proposition \ref{PrMomA}.

First we get a useful bound on multiple correlations. Let
$A_1,\dots, A_p$ and $B_1,\dots, B_q$ be some functions from our
class $\fR$ and $c_1,c_2$ some constants. Consider the functions
$$   A(x) =
      \sum_{*}  A_1(x_{i_1})\cdots
      A_p(x_{i_p})-c_1,\quad
      B(x)=\sum_{**} B_1(x_{j_1})\cdots
      B_q(x_{j_q})-c_2
$$
where the summations $\sum_*$ and $\sum_{**}$ are performed over two
different sets of indices (time moments). Let $m_*$ be the {\em
maximal} index in the first set (denoted by $*$) and $m_{**}$ the {\em
minimal} index in the second set. We suppose that $m_*\leq
m-M^{\delta}<m\leq m_{**}$ for some $m$, i.e.\ there is a ``time gap''
of length $\geq M^{\delta}$ between $m_{*}$ and $m_{**}$.

Now, let $\ell=(\gamma,\rho)$ be a standard pair such that
\index{Standard pair} $\length(\gamma)> M^{-100}$. For any function
$C$, we can decompose the expectation
\begin{equation}
  \label{Middle}
   \EXP_\ell(C\circ\cF^{m-\frac{1}{2} M^\delta})
   =\sum\nolimits_\alpha \EXP_{\ell_\alpha}(C)
\end{equation}
where $\ell_{\alpha}$ denote the components of the image of $\ell$
under $\cF^{m-\frac{1}{2} M^\delta}$.

\begin{lemma}[Multiple correlations]
\label{INDEP} We have
$$
         \bigl|\EXP_\ell(A(x)B(x))\bigr|\leq
         \EXP_\ell \bigl|A(x)\bigr| \,
         \max_\alpha \Bigl|\EXP_{\ell_\alpha}
         \bigl( B(x_{-m+\frac 12 M^{\delta}})
         \bigr)\Bigr|+\cO\left(M^{-50}\right)
$$
where the maximum is taken over $\alpha$'s in (\ref{Middle}) with
$\length(\gamma_\alpha)>M^{-100}$. We note that the remainder term
$\cO\left(M^{-50}\right)$ here depends on the choice of the
functions $A_i,B_j$ and the constants $c_1,c_2$.
\end{lemma}

The proof of Lemma~\ref{INDEP} is similar to that of (\ref{2Cor}),
in which the factorization (\ref{2Cortemp}) plays a key role, we
omit details. \qed\medskip

We now turn to the proof of Proposition~\ref{PrMomA}. Using big
small blocks again, we put for all $k=0,\dots,[\brc M^{2/3}/\bn]$
$$
      P_k'=\sum_{j=k \bn+M^\delta}^{j=(k+1) \bn} \hA(x_j),
      \qquad P_k''=\sum_{j=k\bn}^{k\bn+M^\delta} \hA(x_j),
$$
and then
$$
   U_k'=\sum_{j=0}^{k-1} P_j', \qquad
   U_k''=\sum_{j=0}^{k-1} P_j'' .
$$
Note that $U_k''=\cO\left((k+1) M^{\delta}\right).$

\begin{lemma}
\label{AP1St} Under the assumptions of Proposition~\ref{PrMomA}
and uniformly in $k$
$${\rm (a)}\quad \EXP_\ell(P_k')=
\cO\left(M^{\delta}\right).\Quad $$
$${\rm (b)}\quad \EXP_\ell\left ( [P_k']^2\right ) =
\Bigl(\EXP_\ell\bigl(\hD_{Q_{k\bn},V_{k\bn}}\bigr) + g \Bigr)\,
\bn \Quad
$$
where
$$
     \hD_{Q_{k\bn}, V_{k\bn}}(x) = \left\{ \begin{array}{cc}
     D_{Q_{k\bn},V_{k\bn}}(A) & {\rm if}\ \ x\notin I_k \\
     0 & {\rm otherwise} \end{array} \right .
$$
see (\ref{DQVA}), and $g\to 0$ as $M\to \infty$ and $\kappa_M\to
0$, see a remark below.
$${\rm (c)}\quad \EXP_\ell\left ([P_k']^4\right )=\cO\left(\bn^2\right) .\Quad$$
{\rm (d)} In the notation of (\ref{AAsmall}), we have
$$
   \EXP_\ell\left ([P_k']^4\right )\leq 2\,\fS_A^2\bn^2
   +\cO\left(\bn^{1.9}\right).
$$
\end{lemma}

\noindent{\em Remark}. In the proof of Theorem~\ref{ThVelBM} we
set $\kappa_M=M^{-\delta}$, hence $\kappa_M\to 0$ follows from
$M\to\infty$, and so we can replace $g$ in (b) by $o(1)$. In the
proof of Theorem~\ref{Root}, however, $\kappa_M$ will be a small
constant (independent of $M$), hence the condition $\kappa_M\to 0$
will be necessary.

\proof We first note that $\hA$ is different from 0 only on
\index{Standard pair} standard pairs where (\ref{Away}) holds, see
the construction of $\hA$ in Section~\ref{subsecMom}. Thus,
Proposition~\ref{PrDistEq2} applies to each standard pair
$\gamma_{\alpha, k}$ where $\hA\neq 0$. Therefore it is enough to
verify Lemma~\ref{AP1St} for $k=0$ (but we need to establish it for
all $\ell=(\gamma,\rho)$ such that
$\pi_1(\gamma)\subset\Upsilon_{\delta_2}$).

Part (a) follows from Proposition \ref{GenST} (a).

The proof of (b) is based on the following claim: for each
$\varepsilon>0$ there exists $K(\varepsilon)$ (it is enough to set
$K(\varepsilon)=\Const\, |\ln\varepsilon|$) such that
\beq
      \EXP_\ell\biggl(\sum_{|i-j|>K(\varepsilon)}
      \hA(x_i) \hA(x_j)\biggr)<\varepsilon \bn
        \label{AAepsilon}
\eeq
(here, of course, $M^\delta\leq i,j\leq\bn$). To prove
(\ref{AAepsilon}), we apply the big small block decomposition, as
in Lemmas~\ref{Lm1ST} and \ref{Lm2ST}, to $[P_k']^2$ (with big
blocks of length $M^{1/3}$ and small blocks of length
$M^{\delta}$), then we use Eqs.\ (\ref{ZZthree})--(\ref{ZZthird}),
the induction on $k$, and finally the estimate (\ref{2Cor})
applied to each big block will yield (\ref{AAepsilon}).

Therefore, to get the asymptotics of $\EXP_\ell\left([P_0']^2\right)$
we need to get the asymptotics of
$$
      \EXP_\ell\biggl(\sum_{i=M^\delta}^{\bn}
      \hA(x_i) \hA(x_{i+m})\biggr)
$$
for each fixed $m$. Applying Proposition~\ref{PrDistEq1} to
$j=i-M^\delta$ iterations of $\cF$ we get
$$
        \EXP_\ell \bigl(\hA(x_i) \hA(x_{i+m})\bigr) =\sum_\alpha c_\alpha
        \EXP_{\ell_\alpha} \bigl(\hA(x_{M^\delta})\hA(x_{M^\delta+m})\bigr).
$$
where $\ell_\alpha=(\gamma_{\alpha},\rho_{\alpha})$ denote the
components of the image of $\ell$ at time $j$.
Proposition~\ref{PrDistEq2} applies to each $\gamma_{\alpha}$
where $\hA(x_{M^\delta})\neq 0$, hence for each $\alpha$ such that
$\length(\gamma_\alpha)>\exp(-M^\delta/K)$ we have
$$
      \EXP_{\ell_\alpha} \bigl(\hA(x_i) \hA(x_{i+m})\bigr)=
      \mu_{\brQ, \brV}\bigl(\hA(x_0) \hA(x_m)\bigr)+
      \cO\left (\|\brV \|M^\delta+M^{2\delta-1}\right )
$$
where $(\brQ,\brV)\in\pi_1(\gamma_\alpha)$ is an arbitrary point.
As before, the contribution of small $\gamma_\alpha$ is well
within the error bounds of our claim (b), hence summing over
$\alpha$ and using the fact that the oscillations of $Q$ and $V$
over $\gamma_\alpha$ are of order $1/M$ we obtain
\begin{align*}
      \EXP_\ell \left(\hA(x_i) \hA(x_{i+m})\right) &=
      \EXP_\ell \left (\mu_{Q_j, V_j} (\hA(x_0) \hA(x_{m}))\right )\\
      &\quad +
      \cO\left (\EXP_\ell(\|\hV_j\|)M^\delta+M^{2\delta-1}\right ).
\end{align*}
(recall that $j=i-M^{\delta}$). By Proposition \ref{PrST}
$$
      \EXP_\ell\bigl(\|\hV_j\|\bigr)=
      \cO\bigl(\|\brV\|+\sqrt i/M\bigr)
      =\cO\bigl(\|\brV\|+M^{-3/4}\bigr)
$$
whereas by Lemma~\ref{lmAA}
\begin{align*}
       \EXP_\ell \Bigl(\mu_{Q_j, V_j} \bigl(\hA(x_0)
       \hA(x_{m})\bigr)\Bigr)
       & =
       \EXP_\ell \Bigl(\mu_{\brQ, \brV} \bigl(\hA\,
       (\hA\circ \cF_{\brQ,\brV}^m)\bigr)\Bigr)\\
       &\quad +\cO\bigl(\EXP_\ell\|\hQ_j-\brQ\|\bigr)\\
       &\quad +\cO\bigl(\EXP_\ell\|\hV_j-\brV\|\bigr)+
       \cO\bigl(\|\brV\|+M^{-1}\bigr)
\end{align*}
Proposition~\ref{PrST} (b) gives
$$
    \EXP_\ell\bigl(\|\hV_j-\brV\|\bigr) \leq
    \Const\,\sqrt{j}/M \leq \Const\, M^{-3/4}
$$
(we note that $\bn=\kappa_M M^{1/2}< \delta_\diamond M^{1/2}$,
hence Proposition~\ref{PrST} indeed applies in our context). Also,
since $M\|V_j\|^2<1-\delta_1$, then $\|v_j\|\geq \Const>0$, hence
intercollision times are uniformly bounded above for all $j\leq
\bk\bn$. Therefore,
\begin{align*}
   \EXP_\ell\bigl(\|\hQ_j-\brQ\|\bigr)
   & \leq  \Const \sum_{p=0}^{j-1}
   \EXP_\ell \bigl(\|\hV_p\| \bigr) \\
   & \leq  \Const\, j\, \|\brV\|+\sum_{p=0}^{j-1}
   \EXP_\ell \bigl(\|\hV_p-\brV\|\bigr) \\
   & \leq \Const \bigl(j\, \|\brV\|+j\, M^{-3/4}\bigr).
\end{align*}
Note that
$$
    j\,\|\brV\|\leq \Const\, \bn/\sqrt{M}=\Const\,\kappa_M,
    \qquad j\, M^{-3/4}\leq \kappa_M M^{-1/4}.
$$
This gives
$$
     \EXP_\ell\bigl(\hA(x_i) \hA(x_{i+m})\bigr)=
     \mu_{\brQ, \brV}\bigl(\hA\,
     (\hA\circ \cF_{\brQ,\brV}^m)\bigr)+
     \cO\left(\kappa_M\right).
$$
We note that all our constants and the $\cO(\cdot)$ terms depend,
implicitly, on $m$ which takes values between $0$ and $K_\varepsilon$.
Summing over $i,m$ we get
$$
         \EXP_\ell\left([P_0']^2\right)=\bn \biggl(\sum_{|m|<K(\varepsilon)}
         \mu_{\brQ, \brV}\left(\hA\, (\hA\circ\cF_{\brQ,\brV}^m)\right)+
         \cR+\cO\left(\kappa_M\right)\biggr),
$$
where $|\cR|\leq \varepsilon$ and the $\cO(\cdot)$ term implicitly
depends on $\varepsilon$. Now for any $\varepsilon> 0$ we can
choose a small enough $\kappa_M$ so that the
$|\cO(\kappa_M)|<\varepsilon$. This concludes the proof of part
(b).

We proceed to the proof of part (d). We write
$$
        \EXP_\ell\left([P'_0]^4\right)=
        \sum_{i_1, i_2, i_3, i_4}
        \EXP_\ell \left(\hA(x_{i_1})\hA(x_{i_2})
        \hA(x_{i_3})\hA(x_{i_4})\right).
$$
For convenience, we order the indices in each term so that
\beq
       i_1\leq i_2 \leq i_3\leq i_4
         \label{Order}
\eeq
There are eight cases depending on the choice of ``$<$'' or
``$=$'' in (\ref{Order}), but we will be able to handle several
cases together. First we separate the terms in which $i_1\leq
i_2<i_3<i_4$ and get
\beq
        \EXP_\ell\left([P'_0]^4\right)=\sum_{m=M^{\delta}+1}^{\bn}
        \EXP_\ell\biggl(\hS_m^2\, \hA(x_m)
        \sum_{j=m+1}^{\bn} \hA(x_j) \biggr)+\cR
           \label{E4mainterms}
\eeq
where we denote, for convenience, $j=i_4$, $m=i_3$ and
$\hS_m=\sum_{i=M^{\delta}}^{m-1}\hA(x_i)$, while $\cR$ correspond
to all the remaining terms. Denote
\begin{align*}
       \hS^{(a)}&=\hS_{m-M^\delta},\quad
       &\hS^{(b)}=\hS_m-\hS_{m-M^\delta},\\
       \hS^{(c)}&=\hS_{m+M^\delta}-\hS_m,\quad
       &\hS^{(d)}=\hS_{\bn}-\hS_{m+M^\delta},
\end{align*}
then we have
$$
    \EXP_\ell\left(\hS_m^2 \hA(x_m)\sum_{j>m} \hA(x_j)\right) =
    \EXP_\ell\left([\hS^{(a)}]^2 \hA(x_m)\sum_{j>m} \hA(x_j)\right)
$$
$$
    + \EXP_\ell\left([\hS^{(b)}]^2 \hA(x_m)\sum_{j>m} \hA(x_j)\right)+
    2\EXP_\ell\left(\hS^{(a)} \hS^{(b)} \hA(x_m)\sum_{j>m} \hA(x_j)\right)
$$
$$
       =I+\RmII+\RmIII.
$$
We can assume here that $m>M^{3\delta}$, because terms with
$m<M^{3\delta}$ make a total contribution of order $\bn
M^{9\delta}$. It is clear that part (b) of Lemma~\ref{AP1St} can
be applied to any number of iterations $M^{3\delta}<m\leq\bn$,
hence $\EXP_\ell\left([\hS^{(a)}]^2\right) < (\fS_A+o(1))\,m.$
Therefore by Lemma~\ref{INDEP}
$$
    |I| <  m (\fS_A+o(1))\max_\alpha \biggl| \EXP_{\ell_\alpha}
    \biggl( A(x_{M^{\delta}/2}) \sum_{j=1}^{\bn-m}
    A(x_{j+M^{\delta}/2})\biggr) \biggr|
    + \cO\left ( M^{-50}\right ),
$$
where $\ell_\alpha$ denote the components of the image of $\ell$ at
time $m-M^{\delta}/2$. Similar to the proof of Lemma~\ref{AP1St} (b),
for each $\alpha$ we have
\beq
      \biggl|\EXP_{\ell_\alpha}\biggl(\hA(x_{M^{\delta}/2})
      \sum_{j=1}^{\bn-m} \hA(x_{j+M^{\delta}/2})\biggr)\biggr|
      \leq
      \sum_{j=1}^\infty \mu_{\brQ, \brV}
      \left(A\, (A\circ \cF_{\brQ,\brV}^j)\right)+o(1).
        \label{fourdiff}
\eeq
We observe that
$$
   \sum_{j=1}^\infty \mu_{\brQ, \brV}
   \left(A\, (A\circ \cF_{\brQ,\brV}^j)
   \right) =
   \frac{D_{\brQ, \brV}(A)-
   \mu_{\brQ,\brV}(A^2)}{2}
   \leq \fS_A,
$$
hence $I=\cO\left(\fS_A^2m\right).$ Next,
\begin{align*}
      \RmIII &=
      \EXP_\ell\left(\hS^{(a)} \hS^{(b)} \hA(x_m)
      [\hS^{(c)}+\hS^{(d)}]\right)\\
      &=\EXP_\ell\left(\hS^{(a)} \hS^{(b)}
      \hA(x_m) \hS^{(c)}\right)+
       \EXP_\ell\left(\hS^{(a)} \hS^{(b)}
       \hA(x_m) \hS^{(d)}\right)\\
       &= \RmIII_c+\RmIII_d.
\end{align*}
Now
$$
     |\RmIII_c|\leq \Const\, M^{2\delta}\,
     \EXP_\ell\left (\left|S^{(a)}\right|\right )\leq
     \Const\, M^{2\delta}\sqrt{m}
$$
where the last inequality is based on Lemma~\ref{AP1St} (b) and
Cauchy-Schwartz. On the other hand, due to Lemma~\ref{INDEP} and
Proposition~\ref{GenST} (a)
\begin{align*}
      |\RmIII_d| &\leq \Const\,M^\delta
        \EXP_\ell\Bigl(\Bigl |\hS^{(a)} \hS^{(b)} \hA(x_m)\Bigr|\Bigr)\\
        & \leq
     \Const\, M^{2\delta} \EXP_\ell(|\hS^{(a)}|)\\
     & \leq
     \Const\, M^{2\delta} \sqrt{m}
\end{align*}
(here again the last inequality follows from Lemma~\ref{AP1St} (b)).
Thus $ |\RmIII|\leq \Const\, M^{2\delta} \sqrt{m}.$ Similar estimates
show that $ |\RmII| \leq \Const\, M^{3\delta}. $ Combining these
results we get
$$
      \biggl|\EXP_\ell \biggl(\hS_m^2 \hA(x_m)
      \sum_{j>m} \hA(x_j)\biggr)\biggr|
      \leq 2\,\fS_A^2m +\cO\left(M^{2\delta}\sqrt{m}+M^{3\delta}\right).
$$
Summation over $m$ gives
$$
   \biggl|\EXP_\ell \biggl(\sum_m \hS_m^2 \hA(x_m)
   \sum_{j>m} \hA(x_j) \biggr)\biggr|\leq
    2\,\fS_A^2\bn^2 +\cO\left(\bn^{3/2+6\delta}\right).
$$
It remains to estimate the term $\cR$ in (\ref{E4mainterms}), which
corresponds to the cases where $i_2=i_3$ or $i_3=i_4$. The cases where
$i_1\leq i_2<i_3=i_4$ can be treated in the same way as above, except
(\ref{fourdiff}) now takes form
$$
     \EXP_{\ell_\alpha}\left(\hA^2(x_{M^{\delta}/2})\right )
     =\mu_{\brQ,\brV}(A^2)+o(1)\leq \fS_A
$$
and the term $\RmIII_d$ is missing altogether.

The case $i_1=i_2=i_3<i_4$ and that of $i_1<i_2=i_3=i_4$ can be handled
as follows:
$$
       \biggl|\EXP_\ell\biggl( \sum_{i\neq k}
       \hA(x_k)^3 \hA(x_i)\biggr)\biggr|
       \leq \Const \sum_k \EXP\left(|\hS_\bn|+1\right)
       \leq \Const\, \bn^{3/2}.
$$
Consider the case $i_1<i_2=i_3<i_4$. Using the same notation as in the
analysis of the first term in (\ref{E4mainterms}) we get
$$
   \sum_m \EXP_\ell\left[\bigl(\hS^{(a)}+\hS^{(b)}\bigr)
   \hA^2(x_m) \bigl(\hS^{(c)}+\hS^{(d)}\bigr)\right]=
   \sum_m \left[I^{ac}+I^{ad}+I^{bc}+I^{bd}\right]
$$
where we denoted $I^{\alpha\beta} =\EXP_\ell\left
(\hS^{(\alpha)}\hS^{(\beta)}\hA^2(x_m)\right).$ The estimation of
each term here is similar to the ones discussed above, and we
obtain
$$
      I^{ad}=\cO\Bigl(\sqrt{\EXP_\ell(\hS^{(a)})^2} M^\delta
      +M^{-50} \Bigr)=\cO\left(\sqrt{m}\, M^\delta\right) ,
$$
$$
     I^{ac}=\cO\left(\sqrt{m}\, M^\delta\right) ,
$$
$$
     I^{bc}=\cO\left(M^{2\delta}\right) ,
$$
$$
     I^{bd}=\cO\left(M^{2\delta}\right).
$$
Hence
$$
     \sum_{i_1<m<i_4}\EXP_\ell\left(\hA(x_{i_1})\hA^2(x_m)\hA(x_{i_4})\right)=
     \cO\left(\bn^{3/2}\right).
$$
The only remaining case $i_1=i_2=i_3=i_4$ is simple:
$$
      \sum_j \EXP_\ell\bigl(\hA^4(x_j)\bigr)\leq \Const\, \bn .
$$
This proves part (d). Obviously, part (c) follows from (d), which
completes the proof of Lemma~\ref{AP1St}. \qed \medskip

We now return to the proof of Proposition~\ref{PrMomA}. Denote
\beq
      \EXP_{\max}(\dots)=\max_\ell \left|\EXP_\ell(\dots)\right|
         \label{bmmax}
\eeq
where the maximum is taken over all standard pairs $\ell=(\gamma,
\rho)$ with $\length(\gamma)>M^{-100}. $ \index{Standard pair}

We are now going to prove by induction that
\begin{equation}
         \label{IndM1}
         \EXP_{\max}(U_k')\leq G_1 k M^{\delta} ,
\end{equation}
\begin{equation}
         \label{IndM2}
         \EXP_{\max}([U_k']^2)\leq G_2 k \bn,
\end{equation}
\begin{equation}
         \label{IndM4}
         \EXP_{\max}([U_k']^4)\leq G_4 k^2 \bn^2
\end{equation}
provided the constants $G_1, G_2, G_4$ are sufficiently large.
%The induction hypothesis is that these inequalities hold both for given
%$A$ and for function $\cA$ given in the introduction up to time $k.$
%In particular it is a part of inductive hypothesis that
%$$ \EXP_\ell(|V_{k\bn}|)\leq
%|\brV|+\sqrt{\EXP_\ell \left(V_{k\bn}-\brV\right)^2}\leq
%|\brV|+C\sqrt{G(\cA)} \frac{\sqrt {k\bn}}{M}$$
%Since $k\bn<M^{2/3}$ we get
%\begin{equation}
%\label{GrowthV}
%\EXP_\ell(|V_{k\bn}|)\leq
%|\brV|+C\sqrt{G(\cA)} M^{-2/3} .
%\end{equation}
Let us rewrite the estimates of Lemma~\ref{AP1St} in a simplified way:
\beq
    \EXP_\ell\left(P_k'\right)\leq C_1 M^\delta,\quad
    \EXP_\ell\left([P_k']^2\right)\leq C_2 \bn,\quad
    \EXP_\ell\left([P_k']^4\right)\leq C_4 \bn^2.
       \label{Lm412simple}
\eeq
Now, by the inductive assumption (\ref{IndM1}) we have
$$
           \left|\EXP_\ell\left (U_{k+1}'\right) \right|
           \leq G_1 k M^{\delta}+C_1 M^{\delta},
$$
hence (\ref{IndM1}) holds for $k+1$ provided that $G_1> C_1$.

Next, by Lemma~\ref{INDEP}, (\ref{Lm412simple}), the inductive
assumption (\ref{IndM2}), and the Cauchy-Schwartz inequality we have
\begin{align}
          \EXP_{\max}\left([U_{k+1}']^2\right) & \leq
          \EXP_{\max}\left([U_{k}']^2\right)+
          2\,\EXP_{\max}\bigl(U_{k}'P_k'\bigr)+
          \EXP_{\max}\left([P_k']^2\right)\nonumber\\
          & \leq
          G_2 k \bn + 2\,C_1 M^\delta\sqrt{G_2 k\bn }+C_2\bn.
              \label{MainContrM2}
\end{align}
Since $k\bn<\brc M^{2/3},$ the second term here is
$\cO(M^{-1/6+2\delta} \bn)$, hence (\ref{IndM2}) holds provided
that $G_2> C_2$.

Lastly, by the inductive assumption (\ref{IndM4}) we have
\begin{align*}
        \EXP_{\max}\left([U_{k+1}']^4\right) &\leq
        \EXP_{\max}\left([U_{k}']^4\right)+
        4 \EXP_{\max}\left([U_{k}']^3 P_k'\right)+
        6\EXP_{\max}\left([U_{k}']^2 [P_k']^2\right)\\
        &\quad +4\EXP_{\max}\left(U_{k}' [P_k']^3\right)
        +\EXP_{\max}\left([P_k']^4\right)\\
        & \leq  G_4 k^2 \bn^2 +I+\RmII+\RmIII+\RmIV.
\end{align*}
Using Lemma~\ref{INDEP}, (\ref{Lm412simple}), and the H\"older
inequality we get
\begin{align*}
     I &\leq 4C_1 G_4^{3/4} [k \bn]^{3/2}
     M^{\delta}=4C_1 G_4^{3/4} k \bn^2
     \sqrt{\frac{k}{\bn}}\, M^{\delta}, \\
     \RmII &\leq 6(G_2 k \bn)\, (C_2 \bn)
     =6(G_2 C_2) k \bn^2, \\
     \RmIII &\leq 4\sqrt{G_2 k\bn} \,
     (C_4^{3/4} \bn^{3/2})=
     \sqrt{G_2} C_4^{3/4} \sqrt{k}\, \bn^2, \\
     \RmIV &\leq C_4 \bn^2 .
\end{align*}
Hence, (\ref{IndM4}) holds provided that $G_4 > 6C_2 G_2$. This
completes the proof of (\ref{IndM1})--(\ref{IndM4}) establishing
the parts (a)--(c) of Proposition \ref{PrMomA} (the contribution
from $U_k''$ is well within our error bounds). The proof of (d) is
similar to (c), but we have to use Lemma~\ref{AP1St} (d) in place
of Lemma~\ref{AP1St} (c). Proposition~\ref{PrMomA} is proved. \qed
\medskip

Let us also note, for future reference, that by
(\ref{MainContrM2}) the main difference between
$\EXP_\ell\left([U_{k+1}']^2\right)$ and
$\EXP_\ell\left([U_{k}']^2\right)$ comes from the $[P_k']^2$ term.
Hence we have
\begin{equation}
      \label{PreM2}
      \EXP_\ell\left(\hS^2_{\brc M^{2/3}}\right)=
      \sum_{k\leq \brc M^{2/3}/\bn}
      \EXP_\ell \left([P_k']^2\right)+
      \cO\left(M^{1/3+2\delta}\right) .
\end{equation}

\subsection{Tightness} \label{subsecTight}
We precede the proof of Proposition~\ref{PrTight} with a few
general remarks.

To establish the tightness of a family of probability measures
$\{P_M\}$ on the space of continuous functions $C[0,\brc]$ we need
to show that for any $\varepsilon>0$ there exists a compact subset
$K_\varepsilon\subset C[0,\brc]$ such that $P_M(K_{\varepsilon})>
1-\varepsilon$ for all $M$. The compactness of $K_\varepsilon$
means that the functions $\{F\in K_\varepsilon\}$ are uniformly
bounded at $\tau=0$ and equicontinuous on $[0,\brc]$. All our
families of functions in Proposition~\ref{PrTight} are obviously
uniformly bounded at $\tau = 0$, hence we only need to worry about
the equicontinuity. For any $M_0>0$ all our functions $\tS(\tau)$,
$\tQ(\tau)$, $\tV(\tau)$, and $\ttt(\tau)$ corresponding to
$M<M_0$ have uniformly bounded derivatives (with a bound depending
on $M_0$), hence they trivially make a compact set. Thus, to prove
the tightness for these functions, it is enough to construct a
compact set $K_{\varepsilon}$ such that $P_M(K_{\varepsilon})>
1-\varepsilon$ for all $M>M_\varepsilon$ with some
$M_{\varepsilon}>1$, hence in our proofs we can (and will) assume
that $M$ is large enough.

Lastly, recall that each $m\in\fM$ satisfies
(\ref{aux1})--(\ref{aux3}). Since now we can assume that
$M_{\varepsilon}^{-50}<\varepsilon/2$, it will be enough to prove
all necessary measure estimates for measures $\mes_{\ell}$ on
\index{Standard pair} individual standard pairs $\ell=(\gamma,\rho)$
with $\length(\gamma)>M^{-100}$ (but our estimates must be uniform
over all such standard pairs).

First we prove the part (a) of Proposition~\ref{PrTight}. Let
$\cC_N$ be the space of continuous functions $S(\tau)$ on
$[0,\brc]$ such that
\begin{equation}
      \label{FSpace}
       \left|S\left(\frac{k+1}{2^m}\right) -S\left(\frac{k}{2^m}\right)\right|
\leq 2^{-\frac m8}
\end{equation}
for all $m\geq N$ and $k<2^m\brc$. Observe that functions in
$\cC_N$ are equicontinuous since they are uniformly H\"older on a
dense set (of binary rationals) and by continuity they are
globally H\"older continuous. We claim that for each $\cE>0$ there
exists $N$ such that for all $\ell = (\gamma, \rho)$ with
$\length(\gamma)>M^{-100}$
$$
      \mes_\ell(\tS \in \cC_N)>1-\cE
$$
uniformly in $M$, where $\tS$ is defined by (\ref{Rescale}). Note
that
\beq
        \label{smallint}
    \left|\tS \left(\frac{k+1}{2^m}\right) -
    \tS \left(\frac{k}{2^m}\right)\right|
    \leq\frac{\|A\|_{\infty} M^{1/3}}{2^m}
\eeq
so (\ref{FSpace}) holds for all $m$ such that $2^{-m}< \Const\,
M^{-8/21}.$ Assume now that $2^{-m}\geq \Const\, M^{-8/21}$.
Equivalently, we need to estimate $|\hS_{n_2}-\hS_{n_1}|$ for
$|n_2-n_1|\geq \Const\, M^{2/7}.$

\begin{lemma}
\label{ShiftMom} For all $n_1, n_2$ such that $|n_2-n_1|>\Const\,
M^{2/7}$ and for all $\ell = (\gamma, \rho)$ with
$\length(\gamma)>M^{-100}$
$$
      \EXP_\ell\Bigl(\bigl[\hS_{n_2}-\hS_{n_1}\bigr]^4\Bigr)
      \leq \Const\, (n_2-n_1)^2.
$$
\end{lemma}

\proof For $n_2-n_1\geq cM^{1/2}$, our estimate follows from
Lemma~\ref{AP1St} (c) and the argument used in the proof of
Proposition~\ref{PrMomA} (c). For smaller $n_2-n_1$, the proof is
similar to that of Lemma~\ref{AP1St} (c). \qed
\medskip

Lemma~\ref{ShiftMom} implies that for fixed $k, m$
\begin{align*}
    \Delta_m & \colon =
    \mes_\ell\left(\left|\tS\left(\frac{k+1}{2^m}\right)
    -\tS\left(\frac{k}{2^m}\right)\right| > 2^{-m/8} \right) \\
    &= \mes_\ell\left(\left|\tS\left(\frac{k+1}{2^m}\right)
    -\tS\left(\frac{k}{2^m}\right)\right|^4> 2^{-m/2} \right) \\
    &\leq 2^{m/2}\, \EXP_\ell\left( \left|\tS\left(\frac{k+1}{2^m}\right)
    -\tS\left(\frac{k}{2^m}\right)\right|^4 \right) \\
    &\leq \Const\, \frac{2^{m/2}}{2^{2m}}=\Const\, 2^{-3m/2} .
\end{align*}
Summation over $k$ and $m$ completes the proof of part (a) of
Proposition~\ref{PrTight}.

We now prove part (b). The tightness of $\tV(\tau)$ follows from
(\ref{CR2a}), (\ref{DeltaVexp}) and Proposition~\ref{PrTight} (a)
applied to the function $\cA$ (the contribution of the correction
term $\cO\left(M^{-3/2}\right)$ in (\ref{DeltaVexp}) is well
within our error bounds).

The equicontinuity of $\tQ(\tau)$ follows from a simple estimate:
$$
   \|\tQ(\tau_2)-\tQ(\tau_1)\|
   \leq (\tau_2-\tau_1)\,\max_\tau \|\tV(\tau)\|
$$
Hence the function $\tQ(\tau)$ is Lipschitz continuous with
Lipschitz constant $\max_{[0,T]}\|\tV(\tau)\|^2$ that can be
bounded by using the tightness of $\tV(\tau)$. Hence the tightness
of $\tQ(\tau)$.

To prove the tightness of $\ttt(\tau)$ we consider intercollision
times
$$s_j=\hht_{j+1}-\hht_j=\hd_j/\|v_j\|, $$
where $d_j$ is the
distance between the points of the $j$th and $(j+1)$st collisions
and $\hd_j=d_j 1_{j\leq \bk\bn}$, for all $0\leq j\leq
M^{2/3}\brc$. Note that $\|v_j\|\geq \Const>0$ for all $j<\bk\bn$,
hence $s_j\leq \Const$. Let $\hL_j$ equal $\brL$ if $j<\bk\bn$ and
0 otherwise.

Consider the function $d(x),$ $x\in\Omega$, equal to the distance
between the positions of the light particle at the points $x$ and
$\cF(x)$ (the distance between its successive collisions). In
Section~\ref{subsecA3} we prove the following:

\begin{proposition}
The function $d$ belongs in our space $\fR$. The average
$\mu_{Q,V}(d^k)$ is a Lipschitz continuous function of $Q,V$ for
$k\in\naturals.$ In particular, we have
$$
    \mu_{Q,V}(d) = \brL + \cO(\|V\|)
$$
where $\brL$ is the mean free path defined by (\ref{FPL}). \label{propd}
\end{proposition}

Let $A(x)=d(x)-\mu_{Q,V}(d)$ and $B(x)=\mu_{Q,V}(d)-\brL$. Then
$d(x) = A(x)+ \brL +B(x)$ and, accordingly, $\hd(x) = \hA(x) +
\hL(x) + \hB(x)$. Therefore
\begin{align*}
   \ttt(\tau) &= \frac{1}{M^{1/3}} \sum_{j=0}^{n} \frac{\hA(x_j)}{\|v_j\|}\\
   &\quad +
   \frac{1}{M^{1/3}} \sum_{j=0}^{n} \hL(x_j) \left(\frac{1}{\|v_j\|}-1\right)
   + \frac{1}{M^{1/3}} \sum_{j=0}^{n} \frac{\hB(x_j)}{\|v_j\|}\\
   &= \ttt_1(\tau)+\ttt_2(\tau)+\ttt_3(\tau)
\end{align*}
where $n=M^{2/3}\tau$. The function $A(x)/\|v(x)\|$ satisfies the
conditions of Proposition~\ref{PrMomA}, in particular
$\mu_{Q,V}(A/\|v\|)=0$, hence $\ttt_1(\tau)$ is tight due to
Proposition~\ref{PrTight} (a). Next
\beq
   \frac{1}{\|v_j\|}-1 =
   \frac{1-\sqrt{1-M\|V_j\|^2}}{\|v_j\|}=\cO\left (M\|V_j\|^2\right )
      \label{vvj}
\eeq
To prove the equicontinuity of $\ttt_2(\tau)$ we observe that
\begin{align*}
   \left|\ttt_2(\tau_2)-\ttt_2(\tau_1)\right| & \leq
   \Const\, \frac{M^{2/3}\left|\tau_2-\tau_1\right|}{M^{1/3}}
   \max_{n\leq \brc M^{2/3}} \left(M\|V_n\|^2\right)\\
   & =  |\tau_2-\tau_1|\,
   \max_{\tau\leq\brc} \|\tV(\tau)\|^2 .
\end{align*}
Hence, as before, the function $\ttt_2(\tau)$ is Lipschitz
continuous with Lipschitz constant $\max_{[0,T]}\|\tV(\tau)\|^2$
that can be bounded due to the tightness of $\tV(\tau)$. To prove
the equicontinuity of $\ttt_3(\tau)$ we use
Proposition~\ref{propd} and write
$$
   \left|\mu_{Q,V}(d)- \brL \right|=
   \left|\mu_{Q,V}(d)-\mu_{Q,0}(d)\right|\leq
   \Const\,\|V\| ,
$$
hence
$$
   \left|\ttt_3(\tau_2)-\ttt_3(\tau_1)\right| \leq
   \frac{\Const}{M^{1/3}}\,|\tau_2-\tau_1|\,
   \max_{\tau\leq\brc} \|\tV(\tau)\|,
$$
which is not only bounded due to the tightness of $\tV$, but can
be made arbitrarily small.

Proposition~\ref{PrTight} is proved. \qed

\subsection{Second moment}
\label{subsecM2} Here we prove Proposition~\ref{PrMom}. We work in
the context of Theorem~\ref{ThVelBM}, hence $\kappa_M=M^{-\delta}$
and $\bn=M^{1/2-\delta}$. The context of Theorem~\ref{Root} will be
discussed in the next section.

Recall that $n=\varkappa M^{2/3}.$ Our first step is to show that
under the conditions of Proposition~\ref{PrMomA}
\beq
         \EXP_\ell \bigl(\hS_n^2\bigr)= n
         \left[D_{\brQ,\brV}(A)+g\right] ,
           \label{CrM2}
\eeq
where $g\to 0$ as $M\to\infty$ and $\varkappa\to 0$ uniformly over
all standard pairs with \index{Standard pair}
$\pi_1(\gamma)\subset\Upsilon_{\delta_2,a}^{\ast}$ and
$\length(\gamma)>M^{-100}.$

Indeed, by (\ref{PreM2}) and Lemma~\ref{AP1St} (b) we have
$$
    \EXP_\ell\bigl(\hS_{n}^2\bigr) =
    \bn \sum_{k=0}^{n/\bn}
    \EXP_\ell(\hD_{Q_{k\bn}, V_{k\bn}}) +
    \cO\left(M^{1/3+2\delta}\right) + o(\bn).
$$
By Proposition~\ref{PrTight}, for most of the initial conditions
the quantity
$$
     \max_{k<n/\bn} \left\{
     \|Q_{k\bn}-\brQ\|, M^{2/3} \|V_{k\bn}-\brV\|\right\}
$$
is small if $\varkappa$ is small, hence for most of the initial
conditions $x\in\gamma$ we have $\bk(x)\geq n/\bn$ thus
$\hD_{Q_{k\bn}, V_{k\bn}}(x)=D_{Q_{k\bn}, V_{k\bn}}(A)$, and so we
would only make small error if we replace $\hD_{Q_{k\tau},
V_{k\tau}}$ by $D_{\brQ,\brV}(A)$ (note that $D_{Q,V}(A)$ is a
bounded and continuous function of $Q,V$ on the domain
$\dist(Q,\dcD)>\br+\delta_1$). Thus we obtain (\ref{CrM2}).

Now the parts (a)--(c) of Proposition \ref{PrMom} easily follow from
Proposition \ref{PrMomA} (a), (c), and (\ref{CrM2}). To prove (d), we
write
\begin{align*}
      \EXP_\ell \bigl(\hQ_n-\brQ\bigr)
      &=\EXP_\ell\biggl(\sum_{j=0}^{n-1}
      s_j \hV_j\biggr)\\
      &= \brV\EXP_\ell\biggl(\sum_{j=0}^{n-1}
      s_j\biggr)+\EXP_\ell\biggl(\sum_{j=0}^{n-1} s_j
      (\hV_j-\brV)\biggr)=I+\RmII
\end{align*}
where $s_j=\hht_{j+1}-\hht_j$ is the intercollision time. To
estimate $I$ we use the notation of the proof of
Proposition~\ref{PrTight} (b) and write:
$$
  \EXP_\ell\left (\sum s_j\right)=\EXP_\ell\left(s_j-\hL_j/\|v_j\|\right)
  +\EXP_\ell\left(\hL_j/\|v_j\|\right)=I_a+I_b
$$
As we noted earlier, the function $s_j-\hL_j/\|v_j\|$ satisfies
the assumptions of Proposition~\ref{PrMomA}, hence its part (a)
implies $I_a=\cO\left (M^{1/6+ \delta}\right )$. By using
(\ref{vvj}) and Proposition~\ref{PrTight} (b) we get
$$
     I_b = (1+o_{\varkappa\to 0}(1))\, \varkappa \brL M^{2/3}.
$$
Next, by the Cauchy-Schwartz inequality and Proposition~\ref{PrMom} (b)
\begin{align*}
      |\RmII| &\leq  \Const\, \sum_j
      \sqrt{\EXP_\ell\bigl(\|\hV_j-\brV\|^2\bigr)}\\
      &\leq  \Const\, \sum_j \frac{j^{1/2}}{M}\\
      &\leq  \Const\, \frac{(\varkappa M^{2/3})^{3/2}} M
      \leq \Const\, \varkappa^{3/2}.
\end{align*}
This implies (d). To prove (e), we write, in a similar manner,
\begin{align*}
        \EXP_\ell\left(\|\brQ_n-\brQ\|^2\right) & \leq
        2\, \|\brV\|^2\, \EXP_\ell\left (\left [\sum
        s_j\right ]^2\right )+2\,\EXP_\ell\left(
        \left\|\sum s_j(\hV_j-\brV)\right\|^2\right) \\
        &= I+\RmII.
\end{align*}
Then we have
$$
   |I| \leq \Const\, \varkappa^2 \|\brV\|^2 M^{4/3}=\cO(\varkappa^2)
$$
and
\begin{align*}
      |\RmII| &\leq  \Const\, \varkappa M^{2/3} \sum
      \EXP_\ell\Bigl(s_j^2\bigl\|\hV_j-\brV\bigr\|^2\Bigr)\\
      &\leq  \Const\, \varkappa M^{2/3} \sum
      \frac{j}{M^2}\\
      &\leq \Const\, \varkappa M^{2/3}
      \frac{(\varkappa M^{2/3})^2}{M^2}=\cO\left(\varkappa^3\right) .
\end{align*}
Proposition~\ref{PrMom} is proven. \qed

\subsection{Martingale property}
\label{subsecMart} To prove Proposition~\ref{PrGen} we need to show
that for every $m\geq 1$, all bounded and Lipschitz continuous
functions $B_1,\dots, B_m$ on the $(Q,V)$ space, and all times $s_1
< s_2 < \dots < s_m\leq \tau_1<\tau_2$ we have
$$
   \EXP\left( \biggl[\prod_{i=1}^m B_i(\hbQ(s_i), \hbV(s_i))\biggr]
   \Bigl[\bM(\tau_2)-\bM(\tau_1)\Bigr]\right)= 0.
$$
where $\EXP$ denotes the expectation and
\begin{align*}
    \bM(\tau_2)-\bM(\tau_1)
    &= B(\hbQ(\tau_2), \hbV(\tau_2))-
     B(\hbQ(\tau_1), \hbV(\tau_1))\\
     &\quad -\int_{\tau_1}^{\tau_2} (\cL B)
       (\hbQ(s), \hbV(s))\, ds
\end{align*}
In other words, we have to show that
\beq
       \EXP_{\max}\left(\biggl[\prod_{i=1}^m
       B_i\left(\tQ(s_i), \tV(s_i)\right)\biggr]
       \biggl[\beta_{J_2}-\beta_{J_1}-
       M^{-2/3}\sum_{j=J_1}^{J_2}
       \zeta_j\biggr]\right)\to 0
         \label{bigEXP}
\eeq
as $M\to \infty$, where
$$
    \beta_j=B\bigl(\hQ_j, M^{2/3}\hV_j\bigr),\qquad
    \zeta_j=\cL B\bigl(\hQ_j, M^{2/3}\hV_j\bigr).
$$
and
$$
    J_1= M^{2/3}\tau_1,\qquad
    J_2= M^{2/3}\tau_2
$$
(see (\ref{bmmax}) for the definition of $\EXP_{\max}(\cdot)$ and note
that $B$, $\cL B$, and $B_i$ are bounded and continuous functions).
Lemma~\ref{INDEP} allows us to eliminate the first factor in
(\ref{bigEXP}) and reduce it to
\beq
            \label{bmone}
      \EXP_{\max}\biggl(
      \beta_{J}-\beta_0-M^{-2/3}
      \sum_{j=0}^{J} \zeta_j\biggr)\to 0
      \qquad {\rm as}\ \ M\to\infty
\eeq
where $J= M^{2/3}(\tau_2-\tau_1)$
(note that even if $s_m=\tau_1$, we can approximate
$$ B_m((\tQ(s_m), \tV(s_m))\approx
B_m(\hQ_{s_m M^{2/3}-M^\delta}, \hV_{s_m M^{2/3}-M^\delta}),$$
$$ B((\tQ(\tau_1), \tV(\tau_1))\approx
B(\hQ_{\tau_1 M^{2/3}}, \hV_{\tau_1 M^{2/3}}),$$
so Lemma~\ref{INDEP} applies).

%the factor $B_m(\tQ(s_m), \tV(s_m))$ will be almost constant on
%every standard pair taken at time $\tau_1$, so it can be dropped
%from (\ref{bigEXP})).
We will denote $\tau_2-\tau_1$ by $\tau$.

Next we prove (\ref{bmone}). Given a small constant $\varkappa>0$
and a large constant $R>0$, we define $\hQ', \tV'$ similarly to
$\hQ, \hV$ but with an additional stopping rule, defined in in the
notation of Section~\ref{subsecMom}: at any time moment $k$ that is
a multiple of $[\varkappa M^{2/3}/\bn]$, we ``remove from the
\index{Standard pair} circulation'' all the standard pairs
$\ell_{\alpha,k} = (\gamma_{\alpha,k} ,\rho_{\alpha,k})$ where
$\|V\| > M^{-2/3}R$ for some point $(Q,V) \in \pi_1
(\gamma_{\alpha,k})$ (technically, we add the corresponding curve
$\cF^{-k\bn}(\gamma_{\alpha,k})$ to the set $I_k$, see
\ref{subsecMom}), and we do not change the construction of
Section~\ref{subsecMom} for any time $k$ that is not a multiple of
$[\varkappa M^{2/3}/\bn]$. Thus, the set $I_k$ may get larger and
$\bk(x)$ may decrease, respectively. However, by
Proposition~\ref{PrTight} (b) we have, uniformly in $\varkappa$,
$$
     \sup_\ell \mes_\ell\left \{
     (\hQ', \hV')\neq (\hQ, \hV)\right\}\to 0
$$
as $R\to\infty, M\to\infty$, where the supremum is taken over all
\index{Standard pair} standard pairs $\ell=(\gamma, \rho)$ with
$\length(\gamma)>M^{-100}$. Hence it is enough to show that for all
large enough $R$
\begin{equation}
\label{StopMart}
      \lim_{\varkappa\to 0}\lim_{M\to\infty}
      \EXP_{\max}\biggl(
      \beta_{J}'-\beta_0'-M^{-2/3}
      \sum_{j=0}^{J} \zeta_j'\biggr)\to 0 ,
\end{equation}
where
\begin{equation*}
    \beta_j'=B\bigl(\hQ_j', M^{2/3}\hV_j'\bigr), \quad
    \beta_0'=B\bigl(\brQ, M^{2/3}\brV\bigr),
\end{equation*}
\begin{equation*}
\zeta_j'=
\begin{cases} \cL B\bigl(\hQ_j', M^{2/3}\hV_j'\bigr), &
\text{if } j\leq \bk \bn \cr
              0                                       &
\text{otherwise} \cr
\end{cases}
\end{equation*}
\begin{equation*}
    J=M^{2/3}\tau.
\end{equation*}
(note that both expressions in parentheses in Eqs.\ (\ref{bmone})
and (\ref{StopMart}) are uniformly bounded by a constant
independent of $R$, because $B$ has a compact support). To
establish (\ref{StopMart}) it is enough to check that for all
large $R$ and uniformly in $k\leq\tau/\varkappa$
\begin{equation}
         \label{InfMart}
      \lim_{M\to\infty}
      \EXP_{\max}\biggl(
      \beta_{(k+1)\bL}'-\beta_{k\bL}'-M^{-2/3}
      \sum_{j=k\bL}^{(k+1)\bL} \zeta_j'\biggr)
        = o(\varkappa) ,
\end{equation}
where $\bL= \varkappa M^{2/3}$. To verify (\ref{InfMart}), we can
assume, without loss of generality, that $k=0.$ Next we expand the
function $B$ into Taylor series about the point
$(\brQ,M^{2/3}\brV)$:
\begin{align}
\label{MT}
     \beta_{\bL}'-\beta_0'
     & = \la \nabla_Q B, dQ\ra + \la \nabla_V B, dV\ra +
     \tfrac 12\, (dV)^TB_{VV}\, dV \\
     & +\cO\left(\|dQ\|^2+\|dV\|^3
     +\|dQ\| \|dV\|\right) \label{RT}
\end{align}
where $dQ=\hQ_{\bL}'-\brQ$ and $dV=M^{2/3}(\hV_{\bL}'-\brV)$, and
$B_{VV}$ is a $2\times 2$ matrix with components
$\partial^2_{V_i,V_j}B$, $1\leq i,j \leq 2$.
We claim that
%Applying Proposition~\ref{PrMom} one can easily derive
\begin{align}
     \EXP_\ell(\beta_{\bL}'-\beta_0') &=
     M^{2/3} \brL \la \brV, \nabla_Q B\ra +\tfrac{1}{2}
    \sum_{i,j=1}^2 \bigl(\brsigma^2_{\brQ}(\cA)\bigr)_{ij}\,
    \partial^2_{V_i,V_j} B+o(\varkappa)\nonumber\\
    &= (\cL B)(\brQ, M^{2/3}\brV)\, \varkappa
     +o(\varkappa)
    \label{Bry}
\end{align}
for each standard pair $\ell= (\gamma ,\rho)$ with \index{Standard
pair} $\length(\gamma)>M^{-100}$ Indeed the terms in (\ref{MT}) are
handled by Proposition~\ref{PrMom}(a), (b) and (d) whereas the terms
in (\ref{RT}) are bounded as follows
$$
   \EXP(||dQ||^2) =\cO(\varkappa^2)
$$
by Proposition \ref{PrMom}(e),
$$
   \EXP(||dV||^3)=\cO(\varkappa^{3/2})
$$
by Proposition \ref{PrMom}(c) and H\"older inequality,
$$
   \EXP(||dQ||||dV||) =\cO(\sqrt{\varkappa^2
   \varkappa})=\cO(\varkappa^{3/2})
$$
by Proposition \ref{PrMom} and Cauchy-Schwartz (note that
$\|\brV\|<M^{-2/3}R$ due to our modified construction of $\hQ'$ and
$\hV'$, hence Proposition~\ref{PrMom} applies). On the other hand,
by Proposition~\ref{PrTight} (b)
$$
       \max_{j\leq \bL} \left|
       \EXP_\ell\left[\cL B(\hQ_j',
       M^{2/3}\hV_j')-\cL B(\brQ, M^{2/3}\brV)
       \right]\right|=o_{M\to\infty, \varkappa\to 0} (1) ,
$$
hence
\begin{equation}
    \label{Int}
    \EXP_\ell\biggl(M^{-2/3} \sum_{j=0}^{\bL}
    \zeta_j'\biggr)=\cL B(\brQ, \brV)\,
    \varkappa\, (1+o_{M\to\infty, \varkappa\to 0}(1)) .
\end{equation}
Now (\ref{Bry}) and (\ref{Int}) imply (\ref{InfMart}).
Proposition~\ref{PrGen} is proved. \qed \medskip

\subsection{Transition to continuous time} \label{subsecTCT}
Here we prove Corollary~\ref{CrCTMart}. Pick a $\tau\in (0, \brc
\brL)$ and denote $t=M^{2/3}\tau$. For every $x\in\Omega$ choose
$n$ so that $t_n\leq t< t_{n+1}$. Then
\begin{align*}
   \tQ_*(\tau) &= \hQ_n+\cO\bigl(1/\sqrt{M}\bigr)\\
   &= \hQ_{[t/\brL]}+
   \bigl(\hQ_n-\hQ_{[t/\brL]}\bigr)
   +\cO\bigl(1/\sqrt{M}\bigr)
\end{align*}
By Proposition~\ref{PrTight} (b)
$$
   \mes_\ell \left(\|\hQ_n-\hQ_{[t_n/\brL]}\|>
   \max_{|n_1-n_2|<M^{1/3+\delta}}
   \|\hQ_{n_1}-\hQ_{n_2}\| \right)\to 0
$$
as $M\to\infty.$ By the tightness of $\tV(\tau)$
$$
   \mes_\ell \left(\max_{|n_1-n_2|<M^{1/3+\delta}}
   \|\hQ_{n_1}-\hQ_{n_2}\|>M^{-1/3+2\delta}
   \right)\to 0.
$$
Combining these estimates gives
$$
   \Delta_Q\colon=\mes_\ell \left(\sup_\tau
   \|\tQ_*(\tau)-\tQ(\tau/\brL)\|>\varepsilon\right)\to 0
$$
as $M\to\infty$. We also claim that
\beq
   \Delta_V\colon=\mes_\ell \left(\sup_\tau
   \|\tV_*(\tau)-\tV(\tau/\brL)\|>\varepsilon\right)\to 0
      \label{velDC}
\eeq
but this requires a slightly different argument. The tightness of
$\tV(\tau)$ means that for any $\varepsilon>0$ and $\varepsilon'>0$
there is $\varepsilon''>0$ such that
$$
   \mes_{\ell}\biggl(
   \sup_{|n_1-n_2|<M^{2/3}\varepsilon''}
   \|\hV_{n_1}-\hV_{n_2}\|>M^{-2/3}\varepsilon\biggr)
   < \varepsilon' .
$$
Hence, as before,
\begin{align*}
   \Delta_V
%&=\mes_\ell \Bigl(\sup_\tau \|\tV_*(\tau)-\tV(\tau/l)\|>\varepsilon\Bigr)  \\
    &<  \mes_\ell \biggl(\sup_{|n_1-n_2|<M^{1/3+\delta}}
   \|\hV_{n_1}-\hV_{n_2}\|>M^{-2/3}\varepsilon\biggr)+o(1) \\
    &<\varepsilon'+o(1)
\end{align*}
as $M\to\infty$. The arbitrariness of $\varepsilon'$ implies
(\ref{velDC}).

Thus each $(\bQ_\ast, \bV_\ast)$ can be obtained form the
corresponding $(\hbQ, \hbV)$ by the time change $\tau\to
\tau/\brL.$ \qed \medskip

\noindent{\em Remark}. In the proof of Corollary~\ref{CrCTMart} we used
the tightness of $\ttt(\tau)$, but it would be enough if the following
function
\beq
    \ttt_{\Diamond}(\tau)=  M^{-1/3-\delta/2}
    \left[\hht_{[\tau M^{2/3}]}
    -\brL\, \min\{\tau M^{2/3},\bk\bn\}\right]
      \label{sigmatildeD}
\eeq
was tight for some $\delta>0$. We will refer to this observation in
Chapter~\ref{ScSLP}.

\subsection{Uniqueness for stochastic differential equations}
\label{subsecUSDE} Here we establish the uniqueness of solutions of
(\ref{TIBM}) under the assumption that $\sigma_Q(\cA)$ satisfies
(\ref{sigQ1Q2}).

There are two types of uniqueness for stochastic differential
equations. \emph{Pathwise uniqueness} means, in our terms, that
given a Brownian motion $\bw(\tau)$, any two solutions
$(\bQ_1(\tau), \bV_1(\tau))$ and $(\bQ_2(\tau), \bV_2(\tau))$ such
that $(\bQ_1,\bV_1)(0)= (\bQ_2,\bV_2)(0)$ coincide almost surely.
\emph{Uniqueness in distribution} means that any two solutions of
the SDE have equal distributions provided their initial
distributions coincide. We need the uniqueness in distribution, but
according to \cite[Section IX.1]{RY} it follows from the pathwise
uniqueness, so we shall establish the later.

Our argument follows \cite[Section III]{H}. Let
$(\bQ_1(\tau),\bV_1(\tau))$ and $(\bQ_2(\tau),\bV_2(\tau))$ be two
solutions with the same initial conditions. Denote
$$
   \Delta\bQ(\tau)=\bQ_1(\tau)-\bQ_2(\tau), \qquad
   \Delta\bV(\tau)=\bV_1(\tau)-\bV_2(\tau)
$$
We need to show that $(\Delta\bQ,\Delta\bV)(\tau)\equiv 0$ with
probability one. Given $k>0$ let
$$
  \brtau_k=\sup\bigl\{\tau\colon\, \|\Delta\bQ(\tau)\|<0.1,\
  \|\bV_j(\tau)\|<k,\ j=1,2\bigr\}
$$
and for every $\tau\geq 0$ we set $\tau_k=\min\{\tau,\brtau_k\}$. Let
$$
  a(\tau)=\EXP\Bigl(\max_{s\leq\tau_k} \|\Delta\bV(s)\|^2\Bigr),
  \qquad
  b(\tau)=\EXP\Bigl(\max_{s\leq\tau_k} \|\Delta\bQ(s)\|^2\Bigr)
$$
where $\EXP$ denotes the mean value. Since the coefficients of
(\ref{TIBM}) are bounded due to our cutoffs, the functions
$a(\tau)$ and $b(\tau)$ are continuous. Our goal is to establish
that $a(\tau)=b(\tau)\equiv 0$ for each $k>0$.
% For brevity, we will drop the subscript $k$.

Observe that
$$
   \Delta\bV(\tau) =
   \int_0^{\tau}\left[\sigma_{\bQ_1(s)}(\cA)
   -\sigma_{\bQ_2(s)}(\cA)\right]\, d\bw(s)
$$
due to (\ref{TIBM}), hence $\Delta\bV(\tau)$ is a martingale. By Doob's
maximal inequality
$$
    a(\tau)\leq C_1\, \EXP[\Delta\bV(\tau_k)]^2
$$
(here and on $C_i>0$ are independent of $(\bQ_1,\bV_1)$ and
$(\bQ_2, \bV_2)$). By $L^2$--isomorphism
property of stochastic integration
$$
   a(\tau)\leq C_2\,\EXP\int_0^{\tau_k}\left\|\sigma_{\bQ_1(s)}(\cA)
   -\sigma_{\bQ_2(s)}(\cA)\right\|^2\, ds
$$
Now (\ref{sigQ1Q2}) yields
$$
   a(\tau) \leq C_3\, \int_0^{\tau_k}
   \EXP\left( \left\|\Delta\bQ(s)\right\|^2
   \ln^2 \left\|\Delta\bQ(s)\right\|^2\right)\,ds
$$
Observe that the function $G(s)=s\ln^2 s$ is convex on the interval
$0\leq s\leq 0.1$, and $\|\Delta\bQ(s)\|\leq 0.1$ for all
$s\leq\tau_k$. Thus, Jensen's inequality yields
\beq
        \label{a<b}
   a(\tau) \leq C_3 \int_0^{\tau_k} b(s) \ln^2 b(s)\, ds
\eeq
On the other hand,
$$
   \|\Delta\bQ(\tau)\|^2=\left\|\int_0^{\tau_k}\Delta\bV(s)\,
   ds \right\|^2
   \leq C_4\tau \int_0^{\tau_k} \|\Delta\bV(s)\|^2\, ds
$$
hence
\beq
       \label{b<a}
   b(\tau) \leq C_5 \int_0^{\tau_k} a(s)\, ds
\eeq
Our next goal is to show that (\ref{a<b}) and (\ref{b<a}), along
with initial conditions $a(0)=b(0)=0$, imply $a(\tau)=b(\tau)\equiv
0$. We use the following form of Gronwell inequality (see e.g.\
\cite[Chapter III]{H} for the proof of such results):

\begin{lemma}
\label{Gronwell} Let $f$ and $g$ be monotone functions on a rectangle
$R=[a_1, a_2]\times [b_1, b_2]$ and continuous functions $a(t)$ and
$b(t)$ satisfy
$$
   a(t)\leq \int_0^t f(a(s), b(s))\, ds,\qquad
   b(t)\leq \int_0^t g(a(s), b(s))\, ds
$$
Let $A$ and $B$ be solutions of the differential equations
$$
      A'=f(A,B), \qquad B'=g(A,B)
$$
If $(a(s),b(s))\in R$ and $(A(s),B(s))\in R$ for $0\leq s\leq t$ and
$$
       a(0)\leq A(0), \qquad b(0)\leq B(0)
$$
then
$$
       a(s)\leq A(s), \qquad b(s)\leq B(s)
$$
for all $0\leq s \leq t.$
\end{lemma}

This lemma (and the fact that $b<0.01$) allows us to compare the
functions $a(\tau)$ and $b(\tau)$ with the solutions of the
differential equations
\begin{equation}
   \label{Ham}
        A'=C_4 B \ln^2 B, \qquad B'=C_5 A
\end{equation}
with initial conditions $A(0)=B(0)=0$. Our goal is to show that
$A(\tau)=B(\tau)\equiv 0$ is the only nonnegative solution of the
above initial value problem, i.e.\ there is no branching at
$\tau=0$.

Observe that (\ref{Ham}) is a Hamiltonian-type system whose Hamiltonian
\begin{align*}
   H &= \tfrac 12\, C_5A^2
   -C_4\int_0^Bu\ln^2 u\, du\\
    &= \tfrac 12\, C_5A^2
   -C_4\bigl(\tfrac 12\, B^2\ln^2B - \tfrac 12\, B^2\ln B
   +\tfrac 14\, B^2 \bigr)
\end{align*}
remains constant on all solutions (i.e.\ $H'\equiv 0$). On every
solution originating at $(0,0)$, we have $H(\tau)\equiv 0$.
Therefore, for small $A,B$ we have $A\sim B|\ln B|$, hence
$$
       |B'| \leq C_6 B|\ln B|
$$
It remains to show that any such function $B$ must be identically
zero. Indeed, if $B_0=B(\tau_0)>0$ for some $\tau_0>0$, then
$$
  \tau_0 \geq C_6^{-1}\int_0^{B_0}\frac{dB}{B|\ln B|}
$$
which is impossible because this integral diverges.
\newpage

\chapter{Fast slow particle}
\label{ScFSP} \setcounter{section}{7}\setcounter{subsection}{0}

Here we prove Theorem~\ref{Root}, which allows the slow particle
(the disk) to move faster than Theorem~\ref{ThVelBM} does. Our
arguments are similar to those presented in Chapter~\ref{ScME}, in
fact now they are easier, because we only need to control the
dynamics during time intervals $\cO\left(M^{1/2}\right)$, instead of
$\cO\left(M^{2/3}\right)$.

Recall that for the proof of Theorem~\ref{Root} we set $\kappa_M$ to
a small constant independent of $M.$ Observe that
Propositions~\ref{PrST} and \ref{GenST}, as well as
Lemma~\ref{AP1St}, are applicable in the context of
Theorem~\ref{Root}, but in the rest of Chapter~\ref{ScME} we assumed
$\|\brV\|\leq a M^{-2/3}$, which is not the case anymore. Instead of
that, we will now assume that $M\|\brV\|^2\leq 1-\delta_2$ (and
$1-\delta_2>\chi$). We consider the dynamics up to $n\leq
\brc\sqrt{M}$ collisions, where $\brc=c\sqrt{1-\chi^2}/\brL$ and $c$
is defined in Theorem~\ref{Root} (note that $\sqrt{1-\chi^2}$ is the
initial speed of the light particle, hence $\brL /\sqrt{1-\chi^2}$
will approximate the mean intercollision time).

The following statement is analogous to Proposition~\ref{PrMomA}.

\begin{proposition}
\label{Pr2Tight} Assume the conditions of Proposition \ref{PrMomA}
but with a modified bound on the initial velocity:
$M\|\brV\|^2\leq 1-\delta_2.$ Then, uniformly for $n\leq
\brc\sqrt{M}$, we have
$$ {\rm (a)} \quad \EXP_\ell\bigl(\hS_n\bigr)=
\cO\left(M^\delta \right). \Quad $$
$${\rm (b)} \quad \EXP_\ell\bigl(\hS_n^2\bigr)=\cO(n).\Quad $$
$${\rm (c)} \quad \EXP_\ell\bigl(\hS_n^4\bigr)=\cO(n^2).\Quad $$
\end{proposition}

\proof It is enough to divide $[0,\brc]$ into intervals of length
$\varkappa$ and apply Lemma~\ref{AP1St} to each of them. \qed
\medskip

Next we define certain continuous functions on the interval
$[0,\brc]$, in a way similar to
(\ref{Rescale})--(\ref{sigmatilde}), but with scaling factors
specific to Theorem~\ref{Root}:
\begin{align*}
      \tQ(\tau)&=M^{1/4}\biggl[
      \hQ_{\tau M^{1/2}} - Q_0 - \frac{\brL\, \min\{\tau M^{1/2},\bk\bn\}}
      {\sqrt{1-\chi^2}}\, V_0\biggr], \\
      \tV(\tau)&=M^{3/4} \left [\hV_{\tau M^{1/2}}
      -V_0\right ],\\
      \tS(\tau)&=M^{-1/4} \hS_{\tau M^{1/2}}, \\
    \ttt(\tau)&=  M^{-1/4}\biggl[\hht_{[\tau M^{1/2}]}
    -\frac{ \brL\, \min\{\tau M^{1/2},\bk\bn\}}
    {\sqrt{1-\chi^2}}\biggr] .
\end{align*}

The next result is analogous to Proposition~\ref{PrTight}, and the
proof only requires obvious modifications:

\begin{proposition}
\label{Pr3Tight} {\rm (a)} For every function $A\in\fR$ satisfying
the assumptions of Proposition~\ref{PrMomA}, the family of
functions $\tS(\tau)$ is tight;\\
{\rm (b)} the families $\tQ(\tau)$, $\tV(\tau)$, and $\ttt(\tau)$
are tight.
\end{proposition}

The following result is similar to Proposition~\ref{PrMom}:

\begin{proposition}
\label{Pr2Mom}  Let $\varkappa$ be a small positive constant and
$n=\varkappa \sqrt{M}.$ The following estimates hold uniformly for
\index{Standard pair} all standard pairs $\ell=(\gamma,\rho)$ with
$\length(\gamma)>M^{-100}$ and
$\pi_1(\gamma)\subset\Upsilon_{\delta_2}$, and all
$(\brQ,\brV)\in\pi_1(\gamma)$:
$${\rm (a)}\quad \EXP_\ell\bigl(\hV_n-\brV\bigr)=
\cO\bigl(M^{-1+\delta}\bigr).\Quad $$
$${\rm (b)}\quad \EXP_\ell\bigl((\hV_n-\brV)(\hV_n-\brV)^T\bigr)=
\bigl(\brsigma^2_{\brQ,\brV}(\cA)+o_{\varkappa\to
0}(1)\bigr)\,\varkappa\, M^{-3/2}.\Quad$$
$${\rm (c)}\quad \EXP_\ell\bigl(\|\hV_n-\brV\|^4\bigr)=
\cO\left(\varkappa^2 M^{-3} \right) . \Quad $$
$${\rm (d)}\quad \EXP_\ell\bigl(\hQ_n-\brQ-\hht_n\brV\bigr)
=\cO\left(\varkappa^{3/2} M^{-1/4}\right). \Quad $$ In particular,
if $\brV=V_0+u M^{-3/4}$, for a $u\in\reals^2$, then
$$ \EXP_\ell\left(\hQ_n-\brQ-\hht_n V_0\right)
= \frac{(1+o_{\varkappa\to 0}(1))\,\brL nu}{\sqrt{1-\chi^2}\,
M^{3/4}} +\cO\left(\varkappa^{3/2} M^{-1/4}\right).$$
$${\rm (e)}\quad \EXP_\ell\left(\|\hQ_n-\brQ-\hht_n\brV\|^2\right)=
\cO\left(\varkappa^3 M^{-1/2}\right).\Quad $$
\end{proposition}

The proof goes along the same lines as that of
Proposition~\ref{PrMom}. We only note that the proofs of parts (d)
and (e) do not have to deal with the term $\sum_j s_j \brV$ since
it is included in $\hht_n \brV$, whereas the bound on $\sum_j s_j
(V_j-\brV)$ is obtained exactly as before. Also note that the
second estimate of part (d) follows from the first one and the
fact that, by Lemma~\ref{AP1St}, $\EXP_\ell(\hht_n)\sim
n\brL/\sqrt{1-\chi^2}.$ \qed \medskip

The next statement is an analogue of Proposition~\ref{PrGen}:

\begin{proposition}
\label{Pr2Gen} The function $\tV(\tau)$ weakly converges, as
$M\to\infty$, to a Gaussian stochastic process $\tilde{\bV}(\tau)$
with independent increments, zero mean, and the covariance matrix
$$
   {\rm Cov}\,\tilde{\bV}(\tau)=(1-\chi^2)\int_0^\tau
   \brsigma^2_{Q^\dag\bigl(s\brL/\sqrt{1-\chi^2}
   \bigr)}(\cA)\, ds.
$$
\end{proposition}

\proof Since $\tV(\tau)$ is tight, we only need to prove the
convergence of finite dimensional distributions. Fix a $\tau<\brc$
and choose $\varkappa\ll\tau$ so that
$\tau/\varkappa\in\naturals$. Denote
$$
    R_k'=\hV_{(k+1)\varkappa\sqrt{M}-M^\delta}
    -\hV_{k\varkappa\sqrt{M}},
$$
and
$$
    \tV'(\tau)=\sum_{k=0}^{\tau/\varkappa} R_k'.
$$
Note that $\tV(\tau)-\tV'(\tau) =\cO(M^{\delta-1/4})\to 0$ as
$M\to\infty$, hence the random processes $\tV'(\tau)$ and
$\tV(\tau)$ must have the same finite dimensional limit
distributions. By the continuity theorem, it is enough to prove
the pointwise convergence of the corresponding characteristic
functions, which we do next.

For every vector $\bz\in\reals^2$ we write Taylor expansion
\begin{align}
    \Phi_k(\bz) \colon &=
    \exp\left(iM^{3/4} \la \bz,R_k'\ra\right) \nonumber\\
    &=1+iM^{3/4}\la \bz,R_k'\ra-
    \tfrac 12\, M^{3/2}\, \la \bz,R_k'\ra^2
    +\cO\left(M^{9/4}\la \bz,R_k'\ra^3\right).
       \label{charfunct}
\end{align}

\begin{lemma}
\label{lmPhiTaylor} For any standard pair $\ell=(\gamma,\rho)$
\index{Standard pair} satisfying the conditions of
Proposition~\ref{Pr2Mom} we have
$$
     \EXP_\ell\bigl(\Phi_k(\bz)\bigr)=
     1-\tfrac 12\,(1-\chi^2)\,\varkappa\,\bz^T D_k\,\bz+o(\varkappa).
$$
where
$$
    D_k =
    \brsigma^2_{Q^\dag\bigl(k\varkappa\brL
    /\sqrt{1-\chi^2}\bigr)}(\cA)
$$
\end{lemma}

\proof We apply Proposition~\ref{Pr2Mom} (a) and (b) to the linear
and quadratic terms of (\ref{charfunct}), respectively, and bound
the remainder term by the H\"older inequality:
$$
  \EXP_\ell\left(M^{9/4}\left| \la \bz,R_k'\ra\right|^3\right)\leq
  M^{9/4}\left [\EXP_\ell\left(
  \la \bz,R_k'\ra^4\right)\right ]^{3/4}
$$
and then use Proposition~\ref{Pr2Mom} (c). A delicate point here is
to deal with the matrix $\brsigma^2_{\brQ,\brV}(\cA)$ that comes
from Proposition~\ref{Pr2Mom} (b). According to
Proposition~\ref{Pr3Tight}, for most of the standard pairs
\index{Standard pair} $\ell=(\gamma,\rho)$
\begin{align*}
        \brQ &=Q^\dag\bigl(k\varkappa\brL/\sqrt{1-\chi^2}\bigr)
        +\cO\left(M^{-1/4+\delta}\right), \\
        \brV &=V_0+\cO\left(M^{-3/4+\delta}\right),
\end{align*}
where $k\varkappa$ is the time moment at which
Proposition~\ref{Pr2Mom} (b) was applied. Since
$\brsigma^2_{\brQ,\brV}(\cA)$ is a bounded continuous function on
the domain $\{(Q,V)\colon\,\dist(\brQ,\partial\cD)>\br
+\delta_2\}$, see Lemma~\ref{lmdiffcont}, we can replace
$\brsigma^2_{\brQ,\brV}(\cA)$ with
$$
   \brsigma^2_{Q^\dag\left(k\varkappa\brL/\sqrt{1-\chi^2}\right),V_0}
   (\cA)=(1-\chi^2)D_k
$$
the last equation follows from (\ref{EqSigmabar2}). \qed \medskip

Now by (\ref{charfunct})
\begin{align}
   E_k(\bz) \colon &=\ln \EXP_\ell\bigl(\Phi_k(\bz)\bigr)\nonumber\\
    &= -\tfrac 12\,(1-\chi^2)\,\varkappa\,\bz^TD_k\,\bz+o(\varkappa).
    \label{LocGauss}
\end{align}
Let $0\leq\tau'<\tau''\leq\brc$ be two moments of time such that
$k'=\tau'/\varkappa\in\naturals$ and
$k''=\tau''/\varkappa\in\naturals$. Then
\begin{align*}
    E_{\tau',\tau''}\colon &=
    \ln\EXP_\ell\left(\exp\left(iM^{3/4}
    \bigl\la\bz, \tV'(\tau'')-\tV'(\tau')\bigr\ra
    \right)\right)\\
    &= \ln \EXP_\ell\biggl(
    \prod_{k=k'}^{k''}\Phi_k(\bz)\biggr) \\
    &= \sum_{k=k'}^{k''}
    E_k(\bz) + o_{\varkappa\to 0}(1),
\end{align*}
where we used the same trick as in the proof of Lemma~\ref{INDEP}.
By using (\ref{LocGauss}) and letting $\varkappa\to 0$ we prove
that for any $0\leq\tau'<\tau''\leq\brc$
$$
   \lim_{M\to\infty} E_{\tau',\tau''} =
   -\frac{1-\chi^2}{2} \int_{\tau'}^{\tau''}
    \bz^T\,\brsigma^2_{Q^\dag\bigl(s
    \brL /\sqrt{1-\chi^2}\bigr)}(\cA)\, \bz\, ds.
$$
This shows that the increments of the limit process are Gaussian.

Next, let $0\leq \tau_1<\dots<\tau_{m+1}\leq\brc$ be arbitrary
time moments and $\bz_1,\dots,\bz_m\in\reals^2$ arbitrary vectors.
A similar computation as in Lemma~\ref{lmPhiTaylor} shows that the
joint characteristic function of several increments
$$
    \EXP_\ell\biggl(\exp \biggl(i  M^{3/4}
    \sum_{j=1}^{m} \left\la\bz_j,\tV'(\tau_{j+1})
    -\tV'(\tau_j)\right\ra\biggr)\biggr)
$$
converges to
$$
   \exp\biggl(-\frac{1-\chi^2}{2} \sum_{j=1}^{m}
   \int_{\tau_{j}}^{\tau_{j+1}}
    \bz_j^T\,\brsigma^2_{Q^\dag\bigl(s
    \brL / \sqrt{1-\chi^2}\bigr)}(\cA)\, \bz_j\, ds\biggr).
$$
as $M\to\infty$, hence the increments of the limiting process are
independent. This completes the proof of
Proposition~\ref{Pr2Gen}.\qed
\medskip

Lastly, the same argument as in the proof of
Corollary~\ref{CrCTMart} shows that the velocity function
$\cV(\tau)$ defined in Section~\ref{subsecSR1} converges to the
stochastic process $\bV(\tau)=\tilde{\bV}(\tau \sqrt{1-\chi^2}/
\brL)$. The properties of $\bV$ listed in Theorem~\ref{Root}
immediately follow from those of $\tilde{\bV}$, which we proved
above. The convergence of $\cQ(\tau)$ to
$\bQ(\tau)=\int_0^{\tau}\bV(s)\, ds$ follows from the fact that the
integration is a continuous map on $C[0,\brc\brL].$
Theorem~\ref{Root} is proved. \qed
\newpage

\chapter{Small large particle}
\label{ScSLP} \setcounter{section}{8}\setcounter{subsection}{0}

Here we prove Theorem~\ref{ThSmall}, which requires the larger
particle (the disk) shrink as $M\to\infty$.

First of all, the results of Chapter~\ref{ScME} apply to every
$\br\in(0,\br_0)$, where $\br_0$ is a sufficiently small constant,
and every time interval $(0,\brc)$. We now fix $\brc_0>0$ and for
each $\br\in(0,\br_0)$ apply those results to the time interval
$(0,\brc)$ with
\beq
        \brc = \brc_{\br} = \br^{-1/3}\brc_0
          \label{bTr}
\eeq
In other words, we consider a family of systems
$\cF_{\br}\colon\Omega_{\br}\to\Omega_{\br}$ (parameterized by
$\br$), and for each of them obtain the results of
Chapter~\ref{ScME} on the corresponding interval $(0,\brc_{\br})$
with $\brc_{\br}$ given by (\ref{bTr}). Of course, all the
$\cO(\cdot)$ estimates in Chapter~\ref{ScME} will now implicitly
depend on $\br$.

Next, for each $\br$ we define continuous functions on the
interval $[0,\brc_0]$, by the following rules that modify
(\ref{Rescale}) and (\ref{sigmatildeD}):
\begin{equation}
   \label{Rescale1}
      \tQ(\tau)=\hQ_{\tau \br^{-1/3}M^{2/3}},\qquad
      \tV(\tau)=\br^{-1/3}M^{2/3} \hV_{\tau \br^{-1/3}M^{2/3}},
\end{equation}
and
\beq
    \ttt_{\Diamond}(\tau)= \br^{1/6}M^{-1/3-\delta/2}
    \Bigl[\hht_{[\tau \br^{-1/3}M^{2/3}]}
    -\brL\, \min\{\tau \br^{-1/3}M^{2/3},\bk\bn\}\Bigr]
      \label{sigmatilde1}
\eeq
Now for each $\br\in(0,\br_0)$ we pick a function
$A_{\br}\in\fR_{\br}$ (where $\fR_{\br}$ denotes the space $\fR$
defined in Section~\ref{subsecMS} corresponding to $\br>0$),
satisfying the assumptions of Proposition~\ref{PrMomA} with $\brc
= \br^{-1/3}\brc_0$. Denote by $A = \{A_{\br}\}$ the family of
just selected functions. Assume, additionally, that
\beq
            \label{Aunif}
            \fS_{A} \colon =
        \sup_{0<\br<\br_0}
        \max\{\|A_{\br}\|_{\infty},\fS_{A_{\br}}\} < \infty
\eeq
where $\fS_{A_\br}$ is computed according to (\ref{AAsmall}). Now
we define
\begin{equation}
   \label{Rescale2}
      \tS(\tau)=\br^{1/6}M^{-1/3} \hS_{\tau \br^{-1/3}M^{2/3}}.
\end{equation}

\begin{proposition}
\label{PrTight1} $\ $ \\
{\rm (a)} Given a family of functions $A=\{A_{\br}\}$ as above,
there is a function $M_{A}(\br)$ such that for $\br<\br_0$,
$M>M_{A}(\br)$, the family $\tS(\tau)$ is tight;\\ {\rm (b)} There
is a function $M_0(\br)$ such that for $\br<\br_0$, $M>M_0(\br)$,
the families $\tQ(\tau)$, $\tV(\tau)$, and $\ttt_{\Diamond}(\tau)$
are tight.
\end{proposition}

The proof follows the same lines as that of Proposition~\ref{PrTight},
and we only discuss steps which require nontrivial modifications. The
inequality (\ref{smallint}) now implies (\ref{FSpace}) whenever
$2^{-m}<\br^{4/21}M^{-8/21}$, cf.\ (\ref{Aunif}). For the case
$2^{-m}<\br^{4/21}M^{-8/21}$ we need the following sharpened version of
Lemma~\ref{ShiftMom}:

\begin{lemma}
\label{ShiftMom1} Given a family $A=\{A_{\br}\}$ as above, there is
a function $M_{A}(\br)$ such that for all $\br<\br_0$,
$M>M_{A}(\br)$, all $n_1, n_2$ such that $|n_2-n_1|>\br^{-1/7}
M^{2/7}$ and all standard pairs $\ell=(\gamma,\rho)$ with
\index{Standard pair} $\length(\gamma)>M^{-100}$ and
$\pi_1(\gamma)\subset\Upsilon^{\ast}_{\delta_2,a}$ we have
$$
      \EXP_\ell\Bigl(\bigl[\hS_{n_2}-\hS_{n_1}\bigr]^4\Bigr)
      \leq 3\,\fS_{A}^2 (n_2-n_1)^2.
$$
\end{lemma}

\proof This bound follows from Lemma~\ref{AP1St} (d) and the
argument used in the proof of Proposition~\ref{PrMomA} (d). Note
that the term $\cO\left(\bn^{1.9}\right)$ in Lemma~\ref{AP1St} (d)
implicitly depends on $\br$, i.e.\ it is $<C(\br)\,\bn^{1.9}$, but
we can always increase $M_0(\br)$ so that $\bn^{0.1} > M^{0.04} >
C(\br) / \fS_{A}^2$, hence $C(\br)\,\bn^{1.9} < \fS_{A}^2 \bn^2$,
as desired. \qed
\medskip

Now the tightness of $\tS(\tau)$ follows due to (\ref{Aunif}).

To prove the tightness of $\tV(\tau)$, we need to modify the above
argument slightly. Due to (\ref{CR2a}), (\ref{DeltaVexp}) and
Lemma~\ref{ShiftMom1}
$$
      \EXP_\ell\Bigl(\bigl[\hV_{n_2}-\hV_{n_1}\bigr]^4\Bigr)
      \leq \fS_{\cA}^2M^{-4}(n_2-n_1)^2
      \leq \Const\,\br^2 M^{-4}(n_2-n_1)^2,
$$
where we used (\ref{rsmall}).
Therefore,
$$
      \EXP_\ell\Bigl(\bigl[\tV(\tau_2)-\tV(\tau_1)\bigr]^4\Bigr)
      \leq \Const\,(\tau_2-\tau_1)^2 ,
$$
which is sufficient to prove the equicontinuity of $\tV(\tau)$.

The tightness of $\tQ(\tau)$ and $\ttt_{\diamond} (\tau)$ follows
by the same argument as the one in the proof of
Proposition~\ref{PrTight}. This involves the verification of
(\ref{Aunif}) for the function $A(x)=d(x)-\mu_{Q,V}(d)$, which
requires some effort. Fortunately, we can bypass this step by
using the extra factor $M^{-\delta}$ included in the formula for
$\ttt_{\Diamond}(\tau)$ and only verifying (\ref{Aunif}) for the
function $A_{\Diamond}(x)=M^{-\delta/2}A(x)$, which is much
easier: it suffices to observe that $\fS_{A_{\Diamond}}=
M^{-\delta}\fS_A$ and choose $M_0(\br)$ so that
$M_0^{\delta}(\br)>\fS_A$ for every $\br<\br_0$. This gives
$\fS_{A_{\Diamond}}<1$. \qed
\medskip

Next, Lemma~\ref{AP1St} (b) still holds, with some $g\to 0$,
because we can use the same trick as above -- increase $M_0(\br)$,
if necessary, to suppress the terms depending on $\br$. Having
proved the tightness and Lemma~\ref{AP1St} (b), we can derive
estimates similar to those of Proposition~\ref{PrMom}.

The following statement is analogous to Proposition~\ref{Pr2Gen}:

\begin{proposition}
\label{Pr2Gen1} The function $\tV(\tau)$ weakly converges, as
$\br\to 0$ and $M\to\infty$, $M>M_0(\br)$, to the random process
$\sigma_0 w_{\cD}(\brL \tau)$, in the notation of
(\ref{Vlimsmall}).
\end{proposition}

The proof is similar to that of Proposition~\ref{Pr2Gen}. A slight
complication comes from the fact that, unlike Theorem~\ref{Root},
we have to stop the heavy particle when it comes too close to the
border $\dcD$. Thus we cannot argue the independence as before,
since the increments depend on whether we have already stopped our
particle or not. To overcome this complication, we let $\bw(\tau)$
be the standard two dimensional Brownian motion (independent of
our dynamical system) and define
$$
      V_n^\Diamond=\begin{cases}
      V_n & \text {if } n\leq \bk\bn \cr
      V_{\bk\bn}+\br^{1/3} M^{-2/3} \sigma_0
      [\bw(\brL n)-\bw(\brL \bk\bn)] & \text{otherwise} \cr
\end{cases} $$
(in other words, rather than terminating the velocity process once the
particle comes too close to the border, we switch to an auxiliary
Brownian motion). After this modification, the limiting process will
have independent increments, and we can proceed as in the proof of
Proposition~\ref{Pr2Gen}. \qed
\medskip

Lastly, the same argument as in the proof of
Corollary~\ref{CrCTMart} (see also the remark after it) shows that
the limit of the functions $\cV(\tau)$ defined in
Section~\ref{subsecSR3} and that of $\tV(\tau)$ above only differ by
a time rescaling, $\tau\mapsto\tau/\brL$, hence $\cV(\tau)$
converges to the stochastic process $\sigma_0 w_{\cD}(\tau)$, as
claimed by Theorem~\ref{ThSmall}. Finally, (\ref{Qlimsmall}) follows
by the fact that the integration is a continuous map on
$C[0,\brc\brL].$ Theorem~\ref{ThSmall} is proved. \qed
\newpage

\chapter{Open problems}
\label{SecOP} \setcounter{section}{9}\setcounter{subsection}{0}

Here we mention possible extensions of our results. More detailed discussion
can be found in our survey \cite{CDol}.

%\subsection{Averaging Theory.}
%One important question is to extend Theorem \ref{Root} to
%a more general class of systems with hyperbolic fast motion.
%For the discrete time case some results are obtained in
%\cite{Ba2}, however the continuous time case is more complicated.
%In our proof we use several times the fact that the mean free path
%does not depend on the position of the heavy particle
%(see (\ref{FPL})). In the general situation it is not even clear
%that one could choose a section of the flow in such a way that
%the mean free path depends smoothly on values of the slow
%variables.

\subsection{Collisions of the massive disk with the wall}
An important problem is to understand the behavior of
$\brsigma^2_Q(\cA)$ as the disk $\cP(Q)$ approaches the boundary of
$\cD$, since this would allow one to extend our results beyond the
moment of the first collision of the heavy disk with the wall. It is
natural to assume that this behavior should be controlled by the
billiard dynamics in the domain where the heavy particle just
touches $\dcD$ at some point. This domain is still a dispersing
\index{Dispersing billiards} billiard table, but two of its boundary
components are tangent to each other (make two cusps). Therefore,
one has to understand the mixing properties of dispersing billiards
with cusps, which is a long standing open problem in billiard
theory. There is a heuristic argument \cite{Ma} that leads us to
believe that discrete time correlations should decays as $\cO(1/n)$,
hence the diffusion \index{Diffusion matrix} matrix
$\brsigma^2_Q(\cA)$ might be infinite or behave very irregularly.
%On the
%other hand, the continuous time correlations in dispersing billiards
%with cusps may still decay very fast, making the corresponding real
%time diffusion matrix finite.
In any case, the dynamics in billiard tables with cusps appears to be
quite delicate and requires further investigation. See \cite{CM1} for
recent results.

\subsection{Longer time scales}
\label{subsLTS}
In all the results of our paper,
the velocity $v$ of the light particle does not change
significantly during the time intervals we consider, in fact its
fluctuations converge to zero in probability as $M\to\infty$. On
the basis of heuristic analysis of Section~\ref{subsecBA}, we
expect that $v$ would experience changes of order one after
$\cO(M)$ collisions with the heavy disk. However, we are currently
unable to treat such long intervals, since the error bounds we
have in Proposition~\ref{PrDistEq2} and Corollary~\ref{PrDistEq3}
would accumulate beyond $\cO(1)$, so we need to improve upon this
proposition in order to proceed further. We note that as the
velocity of the light particle experiences changes of order one,
the system starts approaching its natural equilibrium (its
behavior is described by the invariant ergodic measure).

\subsection{Stadia and the piston problem} The question of
approaching a thermal equilibrium was recently considered by
several authors for the piston model \cite{CL}. In that model a
cubic container is divided into two compartments by a heavy
insulating piston, and these compartments contain ideal gases at
different temperatures. If the piston were infinitely heavy, it
would not move and the temperature in each compartment would
remain constant. However, if the mass of the piston is finite the
temperatures would change slowly due to the energy and momenta
exchanges between the particles and the piston. So far not much is
known about the thermalization time needed for the temperatures to
converge to a common limit value.

There is an obvious analogy between the motion of the piston and
that of the heavy disk in our model. The dynamics of ideal gas
particles in each compartment of in the piston model can be made
hyperbolic by appropriate boundary conditions (say, let the
container have a form of the Bunimovich stadium \cite{Bu79}). Then
the methods of our paper could be used. Let us point out, however,
that in our case the fluctuations about the averaged dynamics are
diffusive, while in the piston case nondiffusive fluctuations may
develop as follows. Some particles may move almost parallel to the
piston bouncing back and forth between the flat walls of the
container for a long time. If that happens on one side of the
piston but not the other, the pressure balance will be broken, and
the piston may be forced to move on a macroscopic scale.

\subsection{Finitely many particles}
\label{subsecFMP} The analysis
of our paper extends without changes to systems with several heavy
disks and one light particle. Of course, we need to prevent the disks
from colliding with each other or the boundary of the table by
restricting our analysis to a sufficiently short interval of time. Let
us, for example, formulate an analogue of Theorem~\ref{ThVelBM} in this
situation (similar generalizations are possible for Theorem~\ref{Root}
and \ref{ThSmall}). Let $k$ be the number of heavy disks which are
initially at rest. Then after rescaling time by $M^{2/3}$, the velocity
of the limiting process satisfy
$$
         d\left(\begin{array}{c}
         V_1\cr \vdots\cr V_k\cr \end{array}\right)
         = \sigma_{Q_1\dots Q_k}\, d\bw
$$
where $\bw$ is a standard $2k$-dimensional Brownian motion. Notice
that even though the heavy disks are not allowed to approach each
other, each one ``feels'' the presence of the others through the
diffusion \index{Diffusion matrix} matrix $\sigma_{Q_1\dots Q_k}$
which depends on the positions of all the disks.

In order to extend our results to systems with several light
particles, one needs to generalize Proposition~\ref{PrDistEq2}.
Here we have two possibilities. One is to work with a discrete
time dynamics, then the multiparticle system is a semidispersing
billiard in a higher dimensional space. Very little is known about
mixing rates in such systems, see some results in \cite{C2}.
Alternatively, we may work directly with a continuous time system,
and in this case we get a direct product of 2D billiards. This
would require obtaining the bounds on continuous time correlation
functions, which should be possible in view of recent results
\cite{C1.5, L}.

\subsection{Growing number of particles}
\label{subsecGNP} A more realistic model of Brownian motion consists
of one heavy disk and many light particles, whose number grows with
$M.$ It is also quite reasonable to make the size of the heavy disk
decrease as $M$ grows. Let the diameter of the disk be
$\br=M^{-\alpha}$ for a small $\alpha>0$ and the number of light
particles $N=M^\beta$ for a small $\beta>0$. Since, in view of
(\ref{CR2}), the heavy disk ``remembers'' only the last $\cO(M)$
collisions, it is natural to assume that its velocity will be of
order $\sqrt{M}/M=1/\sqrt M$. Hence it covers a distance of order
one during a time interval of order $\sqrt{M}.$ Let
$\tau=t\sqrt{M}.$ According to the calculations of
Chapter~\ref{ScSLP}, the expected number of collisions during this
time interval is of order
$$ N\sqrt{M} \br=M^{1/2+\alpha-\beta}.$$
On the other hand, (\ref{CR2}) tells that $\cO(M)$ is a critical number
of collisions. Hence the following conjecture seems reasonable:

\begin{conjecture}
Suppose that the initial state of each light particle is chosen
independently, so that the position and velocity direction are
uniformly distributed and the speed has an initial distribution with
smooth density $\rho_0(v).$ Denote $a_j(\rho)=\int |v|^j \rho(v)\, dv.
$ Then the limiting process $\bQ(\tau)$ is

{\rm (a)} straight motion if $\beta=\alpha+1/2-\epsilon,$

{\rm (b)} the integral of an Ornstein-Uhlenbeck process
\beq
   \label{OU}
   d\bQ=\bV\, d\tau, \qquad
   d\bV=-\nu \bV\, d\tau+\sigma\, d\bw
\eeq
where
$$ \nu=c_1 a_1(\rho_0), \qquad
\sigma^2=c_2 a_3(\rho_0) $$
if $\beta=\alpha+1/2$

{\rm (c)} a Brownian motion if $\beta=\alpha+1/2+\epsilon.$
\end{conjecture}

The justification of this conjecture is straightforward. In fact, part
(a) is in direct analogy with Theorem~\ref{Root}. There are too few
collisions to produce significant changes of the velocity of the
massive disk. Part (b) is similar to Theorem~\ref{ThVelBM}, with one
notable difference: in Theorem~\ref{ThVelBM}, there is no drift for the
velocity of the disk since the number of collisions was too small for
the factor $\frac{M-1}{M+1}$ in (\ref{CR2}) to take effect. Under the
setting of the above conjecture, it is this factor that determines the
drift of the Ornstein-Uhlenbeck process. The factor $a_1$ in the drift
term comes from the fact that the number of collisions of the massive
disk with any given particle is proportional to the speed of that
particle. The reason for the factor $a_3$ in the diffusion term is
explained before Theorem \ref{Root} (see also \cite{DGL1}). Also,
observe that in the two particle model treated in this paper, the
velocity of the massive disk has a maximal value, $\frac{1}{\sqrt{M}}$,
hence when it gets close to this value it is more likely to decrease
than to increase. In this sense, we have a ``superdrift'' in the two
particle model. Finally, in the case (c) the Ornstein-Uhlenbeck regime
should take effect on time intervals which are much shorter than
$\tau$, hence (c) is quite natural in view of the fact that
Ornstein-Uhlenbeck process satisfies the central limit theorem.

We believe that the cases of several light particles of the
previous subsection and a growing number of particles discussed
here are similar, on a technical level. Indeed, our arguments are
based on the estimation of the first four moments. For arbitrary
many particles, the computation of the fourth moment contains only
the contribution of all 4-tuples of collisions, but each 4-tuple
involves at most four different light particles, hence an
extension of Proposition~\ref{PrDistEq2} to only four light
particles should be enough for the study of systems with arbitrary
many particles. In the case of a growing number of particles,
there is also an additional complication because there are,
inevitably, slow particles for which there is not enough time for
mixing to take effect. However, we expect the contribution of
those particles be small, since their collisions with the heavy
disk will result in relatively small changes of the velocity of
the latter (cf.\ also \cite{DGL1}).

\subsection{Particles of positive size} The results of our paper
obviously remain valid if the light particle has a positive
diameter, $2r_0$, which is smaller than the shortest distance
between scatterers \index{Scatterer} $\BAN_i$. Indeed, that model
can be reduced to ours by enlarging the massive disk and all the
scatters by $r_0$. However, a model of several light particles of
positive diameter becomes more interesting, since the particles can
interact with each other. To fix our ideas, consider the situation
of the previous subsection with $\beta=\alpha+1/2$ but now let us
assume that instead of $r_0=0$ we have $r_0=M^{-\gamma}.$ Then each
light particle is expected to collide with $N r_0 \sqrt{M}$ other
particles. Since momentum transferred during each collision is of
order 1, now an interesting scaling regime is $N r_0 \sqrt{M}\sim
1.$ In this case we can expect $\rho$ to change according to the
kinetic theory. Thus the following statement seems reasonable.

\begin{conjecture}
The limiting process satisfies
\beq
d\bQ=\bV\, d\tau, \qquad d\bV=-\nu(\tau)\, \bV\,
d\tau+\sigma(\tau)\, d\bw
\eeq
where \medskip

\noindent {\rm (a)} $\quad \nu=c_1 a_1(\rho_0), \quad \sigma^2=c_2
a_3(\rho_0)\ \ $ if $\gamma>\alpha+1$, \medskip

\noindent {\rm (b)} $\quad \nu=c_1 a_1(\rho_\tau), \quad
\sigma^2=c_2 a_3(\rho_\tau),\ \ $ and $\rho_t$ satisfies the
homogeneous Boltzmann equation
$$
   \frac{d\rho_\tau}{d\tau}=Q(\rho_\tau, \rho_\tau),
$$
where $Q$ is the Boltzmann collision kernel, if
$\gamma=\alpha+1$,\medskip

\noindent{\rm (c)} $\quad \nu=c_1 a_1(\rho_{\rm Max}(a_2(\rho_0)))\
\ $ and $\ \ \sigma^2=c_2 a_3(\rho_{\rm Max}(a_2(\rho_0))),\ \ $
where $\rho_{\rm Max}(a_2(\rho_0))$ is the Maxwellian distribution
with the same second moment as $\rho_0$, if $\gamma<\alpha+1.$
\end{conjecture}

\medskip\noindent{\bf Acknowledgment.}
We are deeply indebted to Ya.~G.~Sinai who proposed to us many of
the problems discussed in this paper.
%, as a continuation of earlier studies of the dynamics of a
% massive piston in an ideal gas \cite{CLSlong}.
He outlined a general strategy of the proof of Theorem~\ref{Root},
suggested the main idea of the proof of Proposition~\ref{PrTwoSide},
and we benefited from long conversations with him about the entire
work. This work was done when the authors, in turn, stayed at the
Institute for Advanced Study. We are grateful to the institute staff
and especially to our host T. Spencer for the excellent working
conditions. D.~Dolgopyat would like to thank D.~Ruelle for
explaining to him the results of \cite{R2}. A local version of these
results constitute an important tool in this paper. We thank Peter
Balint for his comments on the preliminary version of this paper.
N.~Chernov was partially supported by NSF grants DMS-9729992,
DMS-0098788. D.~Dolgopyat was partially supported by NSF grant
DMS-0245359, the Sloan Fellowship and IPST.
\newpage

%\chapter*{Appendices}
\appendix
%\addcontentsline{toc}{section}{Appendices}

\renewcommand{\thechapter}{\Alph{chapter}}
\renewcommand{\thesection}{\Alph{section}}
\setcounter{chapter}{0}

\chapter[Statistical properties]{Statistical
properties of dispersing billiards}
%\addcontentsline{toc}{section}{Appendix A. Dispersing billiards}
%\renewcommand{\theequation}{A.\arabic{equation}}
%\renewcommand{\thelemma}{A.\arabic{lemma}}
%\renewcommand{\thesubsection}{A.\arabic{subsection}}
%\setcounter{equation}{0}
%%\setcounter{lemma}{0}
\setcounter{section}{1} \setcounter{subsection}{0}

Throughout the paper, we have made an extensive use of statistical
properties of dispersing \index{Dispersing billiards} billiards
obtained recently in \cite{BSC2,C2,Y}. On several occasions, though,
those results were insufficient for our purposes, and we needed to
extend or sharpen them. Here we adjust the arguments of
\cite{BSC2,C2,Y} to obtain the results we need. The reader is
advised to consult those papers and a recent book \cite{CM2} for relevant details.

\subsection{Decay of correlations: overview}
\label{subsecA1} To fix our notation, let $\cD=\Tor^2\setminus
\cup_{i=0}^r \BAN_i$ be a dispersing \index{Dispersing billiards}
billiard table, where $\BAN_0,\BAN_1,\dots,\BAN_r$ are open convex
scatterers \index{Scatterer} with $C^3$ smooth boundaries and
disjoint closures (the scatterer $\BAN_0$ will play a special role,
it corresponds to the disk $\cP(Q)$ in our main model). Denote by
$\Omega_{\cD}=\dcD\times [-\pi/2,\pi/2]$ the collision space,
$\cF_{\cD}\colon\Omega_{\cD}\to\Omega_{\cD}$ the collision map, and
$\mu_{\cD}$ the corresponding invariant measure.

Assume that the horizon \index{Finite horizon} is finite, i.e.\ the
free path between collisions is bounded by $L_{\max}<\infty$. In
this case for every $k\geq 1$ the map $\cF_{\cD}^k$ is discontinuous
on a set $\cS_k\subset\Omega_{\cD}$, which is a finite union of
smooth compact curves. The complement $\Omega_{\cD}\setminus\cS_k$
is a finite union of open domains which we denote by
$\Omega_{\cD,k,j}$, $1\leq j\leq J_k$.

Now let $\cH_{k,\eta}$ denote the space of functions on
$\Omega_{\cD}$ which are H\"older continuous with exponent $\eta$
on each domain $\Omega_{\cD,k,j}$, $1\leq j\leq J_k$:
\begin{align*}
   f\in \cH_{k,\eta} \ \ \Leftrightarrow \ \ &\exists K_f: \;
   \forall j\in [1,J_k]\ \forall x,y\in\Omega_{\cD,k,j}\\
   &\quad |f(x)-f(y)|\leq K_f\,[\dist(x,y)]^{\eta}.
\end{align*}
One of the central results in the theory of dispersing
\index{Dispersing billiards} billiards is

\begin{proposition}[Exponential decay of correlations \cite{Y}]
\label{PrDecCor}
For every $\eta\in(0,1]$ and $k\geq 1$ there is a
$\theta_{k,\eta}\in(0,1)$ such that for all $f,g\in\cH_{k,\eta}$ and
$n\in\integers$
\beq
        \bigl|\mu_{\cD}(f\cdot (g\circ\cF_{\cD}^n))
        -\mu_{\cD}(f)\mu_{\cD}(g)\bigr|
        \leq C_{f,g}\, \theta_{k,\eta}^{|n|}
          \label{expbound}
\eeq
where
\beq
    C_{f,g} = C_0 \bigl(K_f+\|f\|_{\infty}\bigr)
    \bigl(K_g+\|g\|_{\infty}\bigr)
    \label{Cfg}
\eeq
and $C_0 = C_0 (\cD) >0$ is a constant.
\end{proposition}

The exponential bound (\ref{expbound}) is stated and proved in
\cite{C2,Y}. The formula (\ref{Cfg}), which we also need for our
purposes, is not explicitly derived there, but it follows from the
estimates on pages 608--609 of \cite{Y}.

The arguments in \cite{Y} can be used to derive the following
analogue of our Proposition~\ref{PrDistEq2}:

\begin{proposition}[Equidistribution for billiards]
\label{PrDistEq0} For every $\eta\in(0,1]$ and $k\geq 1$ there is a
$\theta_{k,\eta}\in(0,1)$ such that for any $f\in\cH_{k,\eta}$ and
\index{Equidistribution} \index{Standard pair} any standard pair
consisting of an \index{H-curves} H-curve $W \subset \Omega_\cD$ and
a smooth probability measure $\nu$ on it we have
\begin{equation}
\label{BEqEq}
         \left |\int_W f\circ \cF^n_{\cD}\, d\nu
         - \mu_{\cD}(f)\,\right |
         \leq C_f \theta_{k,\eta}^n
         \qquad\forall n\geq K\big |\ln |W|\,\big |
\end{equation}
where $C_f = C_0 (K_f + \|f\|_{\infty})$ and $C_0,K>0$ are
constants. In addition, by time reversibility, a similar property
holds for stable curves and negative iterations of $\cF_{\cD}$.
\end{proposition}

Below we sketch alternative proofs of both
Propositions~\ref{PrDecCor} and \ref{PrDistEq0} using a
\index{Coupling} `coupling method' \cite{BL,Y2}. This is done in
order to make our presentation self-contained, as well as to
emphasize the central role played by shadowing-type
\index{Shadowing} arguments in the whole theory.

Since the rest of this subsection deals with a fixed domain, we
drop $\cD$ in $\cF_\cD.$ First we derive
Proposition~\ref{PrDecCor} from \ref{PrDistEq0}. We may assume
that $\mu(g)=0$ (otherwise we replace $g$ with $g-\mu(g)$).

Let $\cG = \{ \gamma_\alpha \}$ be a smooth foliation of
\index{Standard pair} $\Omega_\cD$ by \index{H-curves} H-curves on
which standard pairs can be defined. (Sufficiently smooth H-curves
will do, alternatively such foliations are constructed in
\cite{C4}.) Denote by $\cG' = \{ \gamma_\beta' \}$ the foliation of
$\Omega_\cD$ into the \index{H-components} H-components of the sets
$\cF^{n/2} (\gamma_\alpha)$, $\gamma_\alpha \in \cG$. For every
curve $\gamma_{\beta}' \in \cG'$ its preimage $\cF^{-n/2}
(\gamma_{\beta}')$ has length smaller than $C\vartheta^{n/2}$, where
$\vartheta^{-1} >1$ denotes the minimal expansion factor of unstable
curve, cf.\ (\ref{thetamin}). Hence we can approximate the function
$f$ by a constant function on every curve $\cF^{-n/2}
(\gamma_{\beta}')$, $\gamma_{\beta}' \in \cG'$, and this
approximation results in an error term $\cO \bigl( \|g\|_{\infty}
K_f \vartheta^{n\eta/2} \bigr)$ (for all $n/2 > k$). Then we apply
(\ref{BEqEq}) to $n/2$ iterations of $\cF$, the function $g$ and
every curve $\gamma_{\beta}' \in \cG'$ whose length is at least
$e^{-n/2K}$, and obtain a bound $\|f\|_{\infty}
(K_g+\|g\|_{\infty})\, \theta_{k,\eta}^{n/2}$. Lastly, the total
measure of the curves $\gamma_{\beta}' \in \cG'$ whose length is
shorter than $e^{-n/2K}$ is $\cO(e^{-n/2K})$ due to
Lemma~\ref{prgrow} (b), so their contribution will be bounded by
$\cO \bigl( \|f\|_{\infty} \|g\|_{\infty} e^{-n/2K} \bigr)$. Thus
Proposition~\ref{PrDecCor} follows. \qed
\medskip

Next we prove Proposition~\ref{PrDistEq0} in several steps.
\medskip

\noindent{\bf Step 1}. We may assume that $W$ is long enough, i.e.\
$|W|$ is bounded away from zero, otherwise we apply
Lemma~\ref{prgrow} (c) to transform $W$ into \index{H-components}
H-components of length $\geq \eps_0$. Hence we assume that $|W| \geq
\eps_0$ (in this case (\ref{BEqEq}) will hold for all $n\geq 1$). We
will say that an \index{H-curves} H-curve $W$ is \emph{long} if
$|W|\geq \eps_0$.

Next, to establish (\ref{BEqEq}) it is enough to show that the
distribution of the image of $\cF^n W$ is almost independent of
$W$, that is
\begin{equation}
\label{A1Ind}
         \left|\int_{W_1} f\circ \cF^n\, d\nu_1
         - \int_{W_2} f\circ \cF^n\, d\nu_2 \right|
         \leq C_f \theta_{k,\eta}^n
\end{equation}
where $(W_i, \nu_i)$ satisfy the assumptions of
Proposition~\ref{PrDistEq0}, and both $W_1,W_2$ are long. We will
prove a (slightly) more general fact:
\begin{equation}
\label{A1IndAux}
         \left|\int_{M} f\circ \cF^n\, d\mu_1
         - \int_{M} f\circ \cF^n\, d\mu_2 \right|
         \leq C_f \theta_{k,\eta}^n
\end{equation}
where $\mu_1,\mu_2$ are measures of the form
$$
    \mu_i=\int \mes_{\ell_\alpha} d\lambda_i(\alpha)
$$
with $\cG= \{\ell_\alpha\}$ being some family of standard pairs
\index{Standard pair} and $\lambda_i$ factor measures on $\cG$
satisfying
\begin{equation}
\label{A1Short}
    \lambda_i\bigl(\length(\gamma_\alpha)\leq\eps
    \bigr)\leq\Const\, \eps.
\end{equation}
We will say that a family of standard pairs with a factor measure
$\lambda_i$ is \emph{proper} if it satisfies (\ref{A1Short}).

Note that (\ref{A1IndAux}) implies Proposition~\ref{PrDistEq0} if we
set $\lambda_1$ to an atomic measure (concentrated on a single long
\index{H-curves} H-curve) and $\mu_2=\mu_{\cD}$, as $\mu_{\cD}$
satisfies (\ref{A1Short}) by our discussion in
Section~\ref{subsecMS} and \cite{C4}.
\medskip

\noindent{\bf Step 2}. The proof of (\ref{A1Ind}) will be
accomplished by the so called \index{Coupling} coupling algorithm
developed in \cite{Y2}. Its main idea is to divide $\cF^n W_1$ and
$\cF^n W_2$ into pieces, which can be paired so that the elements of
each pair are close to each other (we used a similar idea to prove
Proposition~\ref{PrDistEq1}, but there we coupled the images of the
same curve under different maps). However, since the expansion is
not uniform in different regions of $\Omega_\cD$, some pieces of
$\cF^n W_i$ may carry more weight than others, so we may have to
couple a heavy piece with several light ones. This can be done by
splitting a heavy piece into several `thinner' curves, each coupled
to a different partner. It is actually convenient to split each
curve $W_i$ into uncountable many `fibers'. Namely, given a standard
\index{Standard pair} pair $(W, \nu)$, we consider $Y=W \times
[0,1]$ and equip $Y$ with a probability measure
\begin{equation}
\label{RecMes}
    dm(x,t)=d\nu(x) \, dt = \rho(x)\, dx\, dt
\end{equation}
where $\rho(x)$ is the density of $\nu$ and $0\leq t\leq 1$. We
call $Y$ a {\it rectangle} with {\it base} $W$. The map $\cF^n$
can be naturally defined on $Y$ by $\cF^n(x,t)=(\cF^n x,t)$ and
the function $f$ by $f(x,t)=f(x)$.

%\noindent{\bf Step 2}.
The \index{Coupling} coupling method developed in \cite{Y2} will
give us the following:

\begin{lemma}
\label{LmCoupl} Let $W_1$ and $W_2$ be two long \index{H-curves}
H-curves, and $Y_1$ and $Y_2$ the corresponding rectangles. Then
there exist a measure preserving map (coupling map) $\xi \colon
Y_1\to Y_2$ and a function $R: Y_1\to \naturals$ such that

{\rm (A)} For all $(x,t) \in Y_1$ and $\xi(x,t) = (y,s) \in Y_2$
and all $n>R(x,t)$ the points $\cF^n(x)$ and $\cF^n(y)$ lie on the
same stable manifold in the same connected component of
$\Omega_\cD \setminus \cS_{-n+R(x,t)}$; in particular
$$
   \dist (\cF^n(x), \cF^n(y))\leq C \theta^{n-R(x,t)}
$$
where $C>0$ and $\theta<1$ are constants.

{\rm (B)} For all $n$ we have
$m_1\bigl( (x,t)\colon R(x,t)>n\bigr)\leq C \theta^n$.
\end{lemma}

We postpone the proof untill step 3 and first derive (\ref{A1Ind})
from Lemma~\ref{LmCoupl}:
\begin{align*}
         \Delta\colon &=
         \int_{W_1} f\circ \cF^n\, d\nu_1
         - \int_{W_2} f\circ \cF^n\, d\nu_2\\
         &= \int_{Y_1} f(\cF^n(x,t))\, dm_1
         - \int_{Y_2} f(\cF^n(y,s))\, dm_2\\
         &= \int_{Y_1} \bigl[f(\cF^n(x,t))-
         f(\cF^n(\xi(x,t)))\bigr]\, dm_1.
\end{align*}
The last integral can be decomposed as
$$
   \int_{Y_1} [\dots]=\int_{R>n/2} [\dots]
   +\int_{R \leq n/2} [\dots]=I+\RmII,
$$
and it is easy to see that $|I|\leq 2C\|f\|_\infty \theta^{n/2}$
and $|\RmII|\leq \Const\, K_f \theta^{n\eta/2}$. \qed \medskip

\noindent{\bf Step 3}. Here we begin the proof of
Lemma~\ref{LmCoupl}. It is fairly long and technical; we describe
all the major steps here, but a little more detailed presentation
can be found in \cite[Appendix]{C4}.

First we construct a special family of stable manifolds that will be
used to `couple' points of $Y_1$ and $Y_2$. Let $\tW \subset
\Omega_\cD$ be an \index{H-curves} H-curve and $\kappa>0$; define
$$
    \tW_{\kappa} = \tW \setminus
    \cup_{n\geq 0} \cF^{-n} \cU_{\kappa \vartheta^n}(\cS_1)
$$
where $\cU_{\eps}(\cS_1)$ denotes the $\eps$-neighborhood of
$\cS_1$. It is standard that through every point $x\in
\tW_{\kappa}$ there is a stable manifold $W^s_x$ extending at
least the distance $\kappa$ on both sides of $\tW$. We denote this
family of stable manifolds by $\cG^s_\kappa(\tW)$.

Furthermore, $\bigl|\tW \setminus \cup_{\kappa>0} \tW_\kappa
\bigr|=0$. Hence by reducing $\tW$ we can ensure that, given any
$D,\delta>0$, we can find a curve $\tW$ and $\kappa >0$ such that
\beq \label{tWkappa}
  \kappa > D|\tW|
  \qquad\text{and}\qquad
  |\tW_{\kappa}|/|\tW|>1-\delta.
\eeq
Moreover, for every $x\in \tW_{\kappa}$ the set of points $y\in
W^s_x$ such that the unstable manifold $W^u_y$ intersects all the
stable manifolds $W^s \in \cG_\kappa^s (\tW)$ has positive Lebesgue
measure on $W^s_x$. For the rest of this section, we fix a small
$\delta>0$, such a curve $\tW$, the family $\cG^s = \cG^s_\kappa
(\tW)$, and denote their union by $\Lambda^s = \cup_{\cG^s} W^s$. We
will say that an \index{H-curves} H-curve $W$ \emph{fully crosses}
$\Lambda^s$ if it intersects \emph{all} the stable manifolds $W^s
\in \cG^s$. We note that if $D$ is large enough then the first
inequality in (\ref{tWkappa}) guarantees that any sufficiently long
\index{H-curves} H-curve $W$ that satisfies dist$(W,\tW) < |\tW|$
will  fully cross $\Lambda^s$ (because the `height' of $\Lambda^s$
is much larger than its `length'). Observe that $\tW_\kappa = \tW
\cap \Lambda^s$. For any \index{H-curves} H-curve $W$ fully crossing
$\Lambda^s$ we set $W_\kappa \colon = W\cap \Lambda^s$.

Next, for any standard pair $\ell =(\gamma, \rho)$ and any $n \geq
0$ denote by $\gamma_{n,i}$ the \index{H-components} H-components of
$\cF^n (\gamma)$ \index{Standard pair} that fully cross $\Lambda^s$
and put
\beq \label{Wstar}
    \gamma_{n,\ast} = \cup_i \cF^{-n} (\gamma_{n,i} \cap \Lambda^s).
\eeq
We claim that there are constants $n_0 \geq 1$ and $d_0>0$ such
that for any long standard pair (i.e.\ $|\gamma|\geq
\varepsilon_0$) and any $n\geq n_0$ we have
\beq \label{mWstar}
    \mes_\ell (\gamma_{n,\ast}) \geq d_0.
\eeq
This follows from the mixing property of $\cF$ and the compactness
of the set of long \index{H-curves} H-curves, the proof of
(\ref{mWstar}) is essentially given in \cite[Theorem 3.13]{BSC2}.

Now let $\ell = (\gamma,\rho)$ be a standard pair such that
\index{Standard pair} $\gamma$ fully crosses $\Lambda^s$, then
$\gamma_\kappa = \gamma\cap \Lambda^s$ is a Cantor set on $\gamma$,
and its complement $\gamma \setminus \gamma_{\kappa}$ consists of
infinitely many intervals; we call them \emph{gaps in}
$\gamma_{\kappa}$. These gaps naturally correspond to the intervals
of $\tW \setminus \tW_{\kappa}$ (gaps in $\tW_{\kappa}$), which are
created by the removal of the $\cF^{-n}$-images of the
$c\vartheta^n$-neighborhoods of $\cS_1$ from $\tW$. We call $n$ the
\emph{rank} of the corresponding gap (if a gap is made by several
overlapping intervals with different $n$'s, then its rank is the
smallest such $n$).

If a gap $\tilde{V} \subset \tW \setminus \tW_{\kappa}$ has rank
$n$, then $\cF^n(\tilde{V})$ will have length $\geq c\vartheta^n$.
It corresponds to a gap $V \subset \gamma \setminus
\gamma_{\kappa}$, to which we also assign rank $n$; observe that
$\cF^n(V)$ lies in the $\eps$-vicinity of $\cF^n(\tilde{V})$ with
some $\eps \ll \vartheta^n$, hence $\cF^n(V)$ has length $\geq \frac
12 c\vartheta^n$. Then the set $\cF^{n(1+ \beta_3|\ln\vartheta|)}
(V)$, equipped with the image of the conditional measure $\mes_V =
\mes_\ell (\cdot|_V)$ on $V$, will be a proper family of standard
pairs, in the sense of \index{Standard pair} (\ref{A1Short}), as it
follows from Lemma~\ref{prgrow} (b). Accordingly, we define a
`recovery time' function $r_\gamma (x)$ on $\gamma \setminus
\gamma_{\kappa}$ by setting $r_\gamma(x) =
n(1+\beta_3|\ln\vartheta|)$, where $n$ is the rank of the gap
containing the point $x$ (note that the function $r_\gamma (x)$ is
constant on every gap). Lemma~\ref{prgrow} (b) implies that for some
$\theta<1$ and all $n>0$
\beq \label{lWtheta}
   \mes_\ell \bigl(x\in\gamma \setminus \gamma_{\kappa}\colon
   r_\gamma (x) > n\bigr)/\mes_\ell (\gamma \setminus \gamma_{\kappa})
   \leq \Const\, \theta^n.
\eeq
Next, let $s_\ell(x)$ be another function on $\gamma \setminus
\gamma_{\kappa}$ that is constant on every gap and such that
$s_\ell (x) \geq r_\gamma (x) + n_0$. Then $\mes_V
(V_{s_\ell(V),\ast}) \geq d_0$ for each gap $V \subset \gamma
\setminus \gamma_{\kappa}$, in the notation of (\ref{Wstar}). We
call $s_\ell$ a `stopping time' function.

\begin{lemma} \label{LmStop}
We can define the stopping time function $s_\ell (x)$ on $\gamma
\setminus \gamma_{\kappa}$ so that for all $n \geq 1$
\beq \label{stop1}
  \mes_\ell (x\in\gamma \setminus \gamma_{\kappa}\colon
   s_\ell (x)=n)/\mes_\ell(\gamma \setminus \gamma_{\kappa})=q_n,
\eeq
where $\{ q_n \}$ is a sequence satisfying
\beq \label{stop2}
   \sum q_n =1\quad{\rm and}\quad
   q_n < \Const\, \theta^n.
\eeq
Furthermore, the sequence $\{q_n\}$ is independent of $\ell$,
\index{Standard pair} i.e.\ it is the same for all standard pairs
$\ell = (\gamma, \rho)$ that fully cross $\Lambda^s$.
\end{lemma}

\proof Due to (\ref{lWtheta}), it is easy to define $s_\ell$ so
that that for all $n>0$
\beq \label{selltheta}
   \mes_\ell(x\in\gamma \setminus \gamma_{\kappa}\colon
   s_\ell (x)>n)/  \mes_{\ell} (\gamma \setminus
   \gamma_{\kappa}) \leq \Const\, \theta^n.
\eeq
We still have a considerable flexibility in defining $s_\ell$, and
we want to adjust it so that it will satisfy (\ref{stop1}) with a
sequence $\{q_n\}$ independent of $\ell$. This seems to be a rigid
requirement, but it can be fulfilled by splitting gaps $V$ into
`thinner' curves with the help of rectangles $V \times [0,1]$
described in Step~2: precisely, we can replace each gap $V$ with a
rectangle $V \times [0,1]$, divide the latter into subrectangles
$V \times I_j$, where $I_j \subset [0,1]$ are some subintervals,
and define $s_\ell$ differently on each subrectangle $I_j$. The
sizes of the subintervals $I_j \subset [0,1]$ must be selected to
ensure (\ref{stop1}), as well as (\ref{stop2}). \qed
\medskip

\noindent {\bf Step 4}. We now turn to the construction of the
coupling map $\xi \colon Y_1\to Y_2$ for Lemma~\ref{LmCoupl}, which
will be done recurrently. Given two rectangles $Y_1,Y_2$ with long
bases $W_1,W_2$, we define the first stopping time to be constant
$s_0(x)=n_0$ on both rectangles. At the time $s_0=n_0$ some of the
\index{H-components} H-components of each curve $W_i$ will fully
cross $\Lambda^s$. For every \index{H-components} H-component
$W_{1,s_0,i}$ of $\cF^{s_0} (W_1)$ that fully crosses $\Lambda^s$ we
consider the corresponding rectangle $Y_{1,s_0,i}= W_{1,s_0,i}
\times [0,1]$. We will  split off a subrectangle $W_{1,s_0,i}\times
[0,\btau_{1,i}]$ with some $\btau_{1,i} \leq 1/2$ so that
$m_1(\tY_{1,1}) = d_0/2$, where
$$
   \tY_{1,1} = \bigl\{(x,t) \in Y_1
   \colon \cF^{s_0} (x) \in W_{1,s_0,i} \cap \Lambda^s
   \ \&\ t\in [0,\btau_{1,i}]\ \ \text{for some}\ i\bigr\}
$$
(this is possible due to (\ref{mWstar})).

Suppose we define, similarly, the set $\tY_{2,1} \subset Y_2$. Then
the sets $\tY_{1,1}$ and $\tY_{2,1}$ will have the same overall
measure ($=d_0/2$), and their $\cF^{s_0}$-images will intersect the
same stable manifolds $W^s \in \cG^s$, but for every $W^s \in \cG^s$
the intersections $W^s \cap \cF^{s_0} (\tY_{1,1})$ and $W^s \cap
\cF^{s_0} (\tY_{2,1})$ may carry different `amount' of measures
$m_1$ and $m_2$, respectively. This happens for two reasons: (i) the
densities of our measures may vary along \index{H-components}
H-components and (ii) the Jacobian of the holonomy \index{Holonomy
map} map may also vary and differ from one. To deal with these
problems, we need to assume that the diameter of $\Lambda^s$ is
small, so that the corresponding oscillations of the densities are
small (say, the ratio of the densities at different points on the
same \index{H-components} H-component is between 0.99 and 1.01), and
the Jacobian takes values in a narrow interval, say, $[0.99,1.01]$.

Now we define the set $\tY_{2,1}$ as follows. For every
\index{H-components} H-component $W_{2,s_0,j} \subset \cF^{s_0}
(W_2)$ that fully crosses $\Lambda^s$ we will construct a function
$\btau_{2,j}(y) \leq 0.6$ on $W_{2,s_0,j} \cap \Lambda^s$ and then
put
\begin{align*}
   \tY_{2,1} &= \bigl\{(y,t)\in Y_2\colon
   \cF^{s_0} (y) \in W_{2,s_0,j} \cap \Lambda^s
   \\ &\qquad \&\ t\in [0,\btau_{2,j}(\cF^{s_0}y)]\ \ \text{for some}\ j\bigr\}
\end{align*}
The functions $\btau_{2,j}$ can be constructed so that for every
$W^s \in \cG^s$ the intersections $W^s \cap \cF^{s_0} (\tY_{1,1})$
and $W^s \cap \cF^{s_0} (\tY_{2,1})$ carry the same `amount' of
measures $m_1$ and $m_2$ (this is why we allow $\btau_{2,j}$ to
take values up to $0.6$). Now we naturally define the coupling map
$\xi \colon \tY_{1,1} \to \tY_{2,1}$ that preserves measures and
couples points whose $\cF^{s_0}$-images lie on the same stable
manifold of the $\cG^s$ family. Note that
\beq \label{d0over2}
   m_{r} (\tY_{r,1}) = d_0/2
   \qquad \text{for}\ \ r=1,2.
\eeq
Lastly we set $R(x,t)=s_0$ on $\tY_{1,1}$. This concludes the
first round of our recurrent construction of $\xi$. \medskip

\noindent {\bf Step 5}. Before we start the second round, we need to
`inventory' the remaining parts of $Y_r$, $r=1,2$, and represent
each of them as a countable union of rectangles. To this end we
define a function $\btau_{r,i}$ on every \index{H-components}
H-component $W_{r,s_0,i}$ of $\cF^{s_0} (W_r)$ that fully crosses
$\Lambda^s$: for $r=1$ we set $\btau_{1,i}(x)$ to be constant equal
to $\btau_{1,i}$ defined in Step~4, and for $r=2$ we extend the
function $\btau_{2,i} (x)$ defined in Step~4 on $W_{2,s_0,i} \cap
\Lambda^s$ continuously and linearly to every gap $V_{2,s_0,i,j}
\subset W_{2,s_0,i} \setminus \Lambda^s$.
% and overall continuous on $W_{2,s_0,i}$.
The graph of $\btau_{r,i}$ divides the rectangle $W_{r,s_0,i}
\times [0,1]$ into two parts (`subrectangles' whose one side may
be curvilinear).

\begin{figure}[htb]
    \centering
    \psfrag{t}{$t$}
    \psfrag{j}{$\btau_{2,i}(x)$}
    \psfrag{0}{$0$}
    \psfrag{1}{$1$}
    \includegraphics{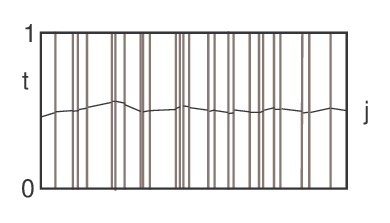}
    \caption{The partition of a rectangle over an H-component
    $W_{2,s_0,i}$: the irregular line in the middle is
    the graph of the function  $\btau_{2,i}(x)$; it separates
    the `upper subrectangle' (of the second type) from
    lower trapezoids (of the third type).}
    \label{FigHolder}
\end{figure}

Now the set $\cF^{s_0} (Y_r \setminus \tY_{r,1})$ consists of
connected components of three types. First, there are rectangles
corresponding to the \index{H-components} H-components of $\cF^{s_0}
(W_r)$ that do not fully cross $\Lambda^s$. Second, the `upper
subrectangles'
$$
 \{(x,t)\colon x\in W_{r,s_0,i}\ \&\ t\in [\btau_{r,i}(x),1]\}.
$$
These are genuine rectangles for $r=1$ and figures with one
`jagged' side for $r=2$, see Fig.~\ref{FigHolder}. All of them
have sufficiently long bases (longer than the size of $\Lambda^s$
in the unstable direction). Third, the `lower subrectangles'
$$
 \{(x,t)\colon x\in V_{r,s_0,i,j}\ \&\ t\in [0,\btau_{r,i}(x)]\}.
$$
constructed over gaps $V_{r,s_0,i,j} \subset W_{r,s_0,i} \setminus
\Lambda^s$. These are true rectangles for $r=1$ and trapezoids for
$r=2$, see Fig.~\ref{FigHolder}.

The shape of the functions $\btau_{2,i}$ is determined by the
densities on our \index{H-components} H-components, which are
Lipschitz continuous, see (\ref{densbound}), and the Jacobian of the
holonomy \index{Holonomy map} map, which is only weakly regular in
the following sense. For any two nearby \index{H-curves} H-curves
$W',W''$ and $x,y\in W'$ that belong to one connected component of
$\Omega_\cD \setminus \cS_n$, the Jacobian of the holonomy map
$h\colon W'\to W''$ satisfies
$$
   |\ln \cJ_{W'}h(x)-\ln \cJ_{W'}h(y)| \leq \Const\, \theta^n
$$
for some $\theta<1$, see \cite[Theorem 3.6]{BSC2} (this property
is sometimes called `dynamically defined H\"older continuity'
\cite[p.\ 597]{Y}). Thus, the function $\btau_{2,i}$ will be
dynamically H\"older continuous, i.e.\ it will satisfy
\beq \label{dynHc}
   |\ln \btau_{2,i}(x)-\ln \btau_{2,i}(y)| \leq
   C\theta^n
\eeq
whenever $x$ and $y$ belong to the same connected component of
$\Omega_\cD \setminus \cS_n$.

Next we rectify the rectangles of the second and third type as
follows. Given a `rectangle' $Y=\{(x,t)\colon x\in W\ \&\ t\in
[0,\btau(x)]\}$, where $\btau(x)$ is a dynamically H\"older
continuous function on $W$, equipped with a probability measure
$dm(x,t) = \rho(x)\, dx\, dt$, we transform $[0,\btau(x)]$ onto
$[0,1]$ linearly at every point $x\in W$, and thus obtain a
full-height rectangle $\hat{Y}=W\times [0,1]$ with measure
$$
   d\hat{m}(x,t)= \hat{\rho}(t)\, dx\, dt, \qquad
   \hat{\rho}(x) = \btau(x)\rho(x).
$$
Since $\hat{\rho}(x)$ is dynamically H\"older continuous (rather
than Lipschitz), we have to generalize our notion of standard pairs
\index{Standard pair} (for the proof of Lemma~\ref{LmCoupl} only!)
to include such densities. This will not do any harm, though, since
it is only oscillations of these densities that matters in our
proof, and the oscillations are well controlled by the dynamical
H\"older continuity; observe also that our densities will smooth out
before the next stopping time, thus they will always remain
uniformly H\"older continuous -- with the same $C$ and $\theta$ in
(\ref{dynHc}); this is intuitively clear, but see \cite[page
1089]{C4} for the exact argument.

Thus the remaining set $Y_{r,1} \colon = \cF^{s_0} (Y_r \setminus
\tY_{r,1})$ for $r=1,2$ is a (countable) union of rectangles of
the full (unit) height, which we denote by $\{Y_{r,1,i}\}$; it
carries a probability measure $m_{r,1}$ induced by the
$\cF^{s_0}$-image of the measure $m_r$. The family $\{Y_{r,1,i}\}$
may not be proper, i.e.\ it may fail to satisfy (\ref{A1Short}).
However, if we condition the measure $m_{r,1}$ onto the union of
rectangles of the first and second type, it will obviously recover
and become a proper family in just a few iterations of $\cF$. On
the rectangles of the third type, the recovery time may vary
greatly, see Step~3, and we define the stopping time function
$s_1(x,t)$ on the rectangles of the third type as described in
Lemma~\ref{LmStop}. We can clearly define the stopping time $s_1$
on the rectangles of the first and the second types as well, so
that its overall distribution matches that described in Step~3,
i.e.\
\beq \label{ddqn}
   m_{r,1}\bigl(Y_{r,1}^{(n)}\bigr)= q_n,
   \qquad
   Y_{r,1}^{(n)} \colon = \bigl\{ \cup_i Y_{r,1,i} \colon
   s_1 = n\bigr \}
\eeq
with the same sequence $\{q_n\}$ as in
(\ref{stop1})--(\ref{stop2}). Of course, $s_1$ must be constant on
every rectangle, so to ensure (\ref{ddqn}) we may need to split
rectangles $Y_{r,1,i}$ of the first and second type into `thinner'
subrectangles, as we did in the end of Step~3, and define $s_1$
separately on every subrectangle.

Now for every rectangle $Y_{r,1,i}$ the set $\cF^{s_1} (Y_{r,1,i})$
will contain \index{H-components} H-components fully crossing
$\Lambda^s$, and in the notation of (\ref{Wstar}) we have $m_{r,1}
(Y_{r,1,i,s_1,\ast})/ m_{r,1}(Y_{r,1,i}) \geq d_0$, due to
(\ref{mWstar}), hence
\beq \label{mdYd0}
     m_{r,1} \bigl( Y^{(n)}_{r,1,s_1,\ast} \bigr) \geq
     d_0\, m_{r,1}\bigl( Y^{(n)}_{r,1} \bigr) = d_0q_n,
\eeq
where $Y^{(n)}_{r,1,s_1,\ast} = \cup_i Y_{r,1,i,s_1,\ast}$.

Based on (\ref{ddqn}) and (\ref{mdYd0}), for every $n \geq 1$ we can
apply our \index{Coupling} coupling procedure (Step 4) to the sets
$Y_{1,1}^{(n)}$ and $Y_{2,1}^{(n)}$ and define the coupling map
$\xi$ on a subset of relative measure $d_0/2$, see (\ref{d0over2}),
i.e.
\beq \label{qnd02}
   m_{r,1} \bigl((x,t)\in Y_{r,1}^{(n)}
   \ \&\ \cF^n (x,t)\ \text{is coupled}\bigr)
%   {m_{d,1} \bigl(\cup_i Y_{d,1,i}\colon s_1=n\bigr)}
    = d_0q_n/2
\eeq
We denote by $\tY_{r,2} \subset Y_r$ the set of the preimages of
just `coupled' points and put $R(x) =s_0(x) +s_1(\cF^{s_0}x)$ on
$\tY_{1,2}$. We do this for every $n\geq 1$, and this concludes
the second round of our construction. It then proceeds
recursively, by repeating Steps 4 and 5 alternatively.

At the $k$th round, we define a stopping time function $s_{k-1}$
on the set $Y_{r,k-1}$ of yet uncoupled points for $r=1,2$,
%(that carries a probability measure $m_{r,k-1}$)
then we `couple' some points of the images $\cF^{s_{k-1}}
(Y_{r,k-1})$, denote by $\tY_{r,k} \subset Y_r$ the set of the
preimages of just `coupled' points, and define
$$
  R(x)=s_0(x) + \cdots +
  s_{k-1}(\cF^{s_0+\cdots+s_{k-2}}x)
$$
on $\tY_{1,k}$. Observe that the point $\cF^{R(x)} (x)$ and its
partner $\cF^{R(x)} \bigl(\xi(x)\bigr)$ lie on the same stable
manifold, which proves the claim (A) of Lemma~\ref{LmCoupl}.
\medskip

\noindent{\bf Step 6}. It remains to prove the claim (B), which
will also imply that the coupling map $\xi$ is defined almost
everywhere on $Y_1$. For brevity, we identify the set $Y_r$ (for
each $r=1,2$) with its images, i.e.\ we consider all our stopping
time functions as defined on $Y_r$. We then have two conditional
probability formulas:
\beq \label{msss1}
   m_{r}(s_{k}=n\vert s_{k-1}=n_{k-1},
   \ldots,s_1=n_1,s_0=n_0)=q_n
\eeq
due to (\ref{ddqn}) and
\beq \label{msss2}
   m_{r} \bigl(\tY_{r,k}\vert s_{k-1}=n_{k-1},\ldots,s_1=n_1,
   s_0=n_0\bigr) = \delta\colon = d_0/2
\eeq
due to (\ref{qnd02}). The following argument is standard in the
\index{Standard pair} studies of random walks. Let $\bar{p}_n =
m_1\bigl( (x,t)\in Y_1\colon R(x,t) =n\bigr)$ denote the fraction of
points coupled exactly at time $n$ (i.e., at the $n$th iteration of
$\cF$, rather than at the $n$th round). Note that $\bar{p}_i=0$ for
$i<n_0$ and $\bar{p}_{n_0}=\delta$. Now $p_n = \bar{p}_n/\delta$ is
the fraction of points \emph{stopped} at time $n$, i.e. \
$$
   p_n = m_1\bigl( (x,t) \in Y_1 \colon
   s_0+s_1+\cdots+s_k=n\ \text{for some}\ k\bigr).
$$
Due to (\ref{msss1}) and (\ref{msss2}) we have the following
`convolution law':
\beq \label{convolution}
  p_{n+n_0} = (1-\delta)\Bigl( q_{n}+ (1-\delta)
  \sum_{i=1}^{n-1} q_{n-i}p_{n_0+i}\Bigr)
  \qquad \forall n\geq 1.
\eeq
Now consider two complex analytic functions
$$
  P(z)=\sum_{n=1}^\infty p_{n_0+n}z^n
  \quad\text{and}\quad
  Q(z)=\sum_{n=1}^\infty q_{n}z^n,
$$
then (\ref{convolution}) implies $P(z)=(1-\delta)\, Q(z)+
(1-\delta)^2\,P(z)\,Q(z)$, hence
\beq \label{PQ}
   P(z)=\frac{(1-\delta)\,Q(z)}{1-(1-\delta)^2\,Q(z)}
\eeq
Due to (\ref{stop2}), we have $|Q(z)| \leq 1$ for all $|z|\leq 1$,
and the function $Q(z)$ is analytic in the complex disk $\{z\colon
|z| < 1+\eps\}$ for some $\eps>0$. Hence $P(z)$ is also analytic
in a complex disk of radius greater than one, which implies an
exponential tail bound on $p_n$. A similar bound then follows for
$\bar{p}_n = \delta p_n$. \qed

\subsection{Decay of correlations: extensions}
In this section we will extend the mixing results in several ways:

\medskip\noindent{\bf Extension 1}.
Suppose one scatterer \index{Scatterer} (specifically, $\BAN_0$) is
removed from the construction of $\Omega_{\cD}$, i.e.\ we redefine
$\tOmega_{\cD}=\cup_{i=1}^r\partial \BAN_i\times [-\pi/2,\pi/2]$,
and respectively the return map
$\tcF_{\cD}\colon\tOmega_{\cD}\to\tOmega_{\cD}$ and the invariant
measure $\tilde{\mu}_{\cD}$. Note that we do not change the dynamics
-- the billiard particle still collides with the scatterer $\BAN_0$,
we simply skip those collisions in the construction of the collision
map. Assume, additionally, that the billiard particle cannot
experience two successive collisions with $\BAN_0$ without colliding
with some other scatterer(s) \index{Scatterer} in between (this
follows from our finite horizon \index{Finite horizon} assumption in
Section~\ref{subsecDass}, provided $\br$ is small enough). In this
case the analysis done in Section~\ref{subsecA1} (as well as the
earlier one \cite{Y}) goes through and Propositions~\ref{PrDistEq0}
and \ref{PrDecCor} now hold for the dynamical system
$(\tOmega_{\cD},\tcF_{\cD} ,\tilde{\mu}_{\cD})$ and functions $f,g$
defined on $\tOmega_{\cD}$.
\medskip

The value of $\theta_{k,\eta}$ in (\ref{expbound}) depends on the
following quantities characterizing the given billiard table:
\begin{itemize} \item[(a)] the minimal and maximal
free path (called $L_{\min}$ and $L_{\max}$), \item[(b)] the minimal
and maximal curvature of the boundary of the scatterers,
\index{Scatterer}
\item[(c)] the upper bound on the derivative of the curvature of
the scatterers, \item[(d)] the value of $\btheta_1$ in the one-step
expansion estimate (\ref{alpha01}).
\end{itemize}

According to (\ref{bthetabound}), the value of $\btheta_1$ will be
bounded away from one because $L_{\max}/L_{\min}$ remains bounded.

We note that the earlier work \cite{Y} does not use $\btheta_1$.
Instead, it uses the \emph{complexity bound}, i.e.\ the smallest
$n\geq 1$ for which
\beq  \label{complexboundF}
     K_n \vartheta^n < 1,
\eeq
where $K_n$ denotes the maximal number of pieces into which $\cS_n$
can partition arbitrary short unstable curves. It is known
\cite[Section 8]{BSC1} and \cite[p.\ 634]{Y} that $K_n\leq
C_1n+C_2$, where $C_1$ and $C_2$ are constants determined by the
number of possible tangencies between successive collisions, i.e.\
by the maximal number of points at which a straight line segment
$I\subset\cD$ can touch some scatterers \index{Scatterer} $\BAN_i$.
We note that this number does not exceed $L_{\max}/L_{\min}$, thus
$C_1$, $C_2$, and $n$ in (\ref{complexboundF}) are effectively
determined by $L_{\max}/L_{\min}$ which remains bounded.

\medskip\noindent{\bf Extension 2}.
Consider a family of dispersing \index{Dispersing billiards}
billiard tables obtained by changing the position of one of the
scatterers \index{Scatterer} (specifically, $\BAN_0$) continuously
on the original dispersing billiard table. We only allow such
changes that the maximal free path $L_{\max}$ remains bounded away
from infinity, and the minimal free path $L_{\min}$ remains bounded
away from zero. Then all the characteristic values (a)--(d) of the
billiard tables in our family will effectively remain unchanged, and
therefore the bound (\ref{expbound}) will be {\em uniform}. (Note
that the space $\Omega_{\cD}$ does not depend on the position of the
movable scatterer $\BAN_0$, hence the functions $f,g$ in
(\ref{expbound}) do not have to change with the position of
$\BAN_0$).

\medskip\noindent{\bf Extension 3}.
Suppose we not only change the position of the scatterer
\index{Scatterer} $\BAN_0$, but also reduce its size homotetically
(namely, suppose $\BAN_0$ is a disk of radius $\br_0$, and we
replace it with a disk of radius $\br<\br_0$). Hence we consider a
larger family of dispersing billiard tables than in Extension~2. Now
the collision space $\Omega_{\cD}$ depends on the size of $\BAN_0$,
but we restrict the analysis to the space $\tOmega_{\cD}$
constructed exactly as we did in Extension~1, by skipping collisions
with $\BAN_0$. Then the space $\tOmega_{\cD}$ will be the same for
all billiard tables in our family, so we can speak about the
uniformity of the exponential bound on correlations for the map
$\tcF$. Again, we assume a uniform upper bound on $L_{\max}$ and a
uniform positive lower bound on $L_{\min}$. There are several new
problems now:

The curvature of $\partial \BAN_0$ will not be uniformly bounded
anymore, it will be proportional to $1/\br$. The upper bound on the
curvature is used to prove a uniform transversality of stable and
unstable cones, see \cite[pp.\ 534--535]{C2}. Those cones are not
uniformly transversal anymore, the angle between them is $\cO(\br)$
on the part of the phase space $\Omega_{\cD}$ corresponding to the
boundary of the scatterer \index{Scatterer} $\BAN_0$, but this part
is specifically excluded from the construction of $\tOmega_{\cD}$,
hence the cones are still uniformly transversal on $\tOmega_{\cD}$.
The upper bound on the curvature is also used in the distortion
\index{Distortion bounds} and curvature estimates, similar to those
in Appendix~C, but we will show that those estimates remain uniform
over all $\br>0$, see a remark after the proof of Lemma~\ref{lm0}.
Next, the curvature of the disk $\BAN_0$ is constant, so its
derivative is zero.

Lastly, the complexity $K_n$ of the singularity set $\cS_n$ will be
affected by $\br$, too, if $\br$ is allowed to be arbitrarily small.
Indeed, if all the scatterers \index{Scatterer} had fixed size, one
considers \cite{C2,Y} short enough unstable curves that can only
break into two pieces at any collision (one piece collides, the
other passes by, as it is explained in the proof of
Lemma~\ref{prgrow}). But now, no matter how small an unstable curve
is, the scatterer $\BAN_0$ may be even smaller, and then the
unstable curve may be torn by $\BAN_0$ into \emph{three} pieces. The
middle piece hits $\BAN_0$, gets reflected, and by the next
collision its image will be of length $\cO(1)$. It is not hard to
see then that the sequence $K_n$ will grow exponentially fast, and
therefore the complexity bound (\ref{complexboundF}) may easily fail
for all $n\geq 1$.

The complexity bound is only used in the proof of the growth
\index{Growth lemma} lemma \cite[Theorem 3.1]{C2}, which is
analogous to our Lemma~\ref{prgrow}. We have seen in
Section~\ref{subsecSP} that the growth lemma follows from the
one-step expansion estimate (\ref{alpha01}). In fact, it suffices to
establish the one-step expansion estimate for any iteration of the
given map, see \cite[Proposition 10.1]{C2} and \cite[Theorem
10]{CZ}, i.e.\ in our case it is enough to prove that
\beq
     \exists n\geq 1\colon\ \
   \btheta_n \colon = \liminf_{\delta\to 0}\
   \sup_{W\colon\, \length(W)<\delta} \sum_i \vartheta_{i,n} <1
      \label{alpha00}
\eeq
Here $W\subset\tOmega_{\cD}$ denotes an \index{H-curves} H-curve and
$\vartheta_{i,n}^{-1}$ the smallest local factor of expansion of
$\tcF_{\cD}^{-n}(W_{i,n})$ under the map $\tcF_{\cD}^n$, where
$W_{i,n}$, $i\geq 1$, denote the \index{H-components} H-components
of $\tcF_{\cD}^n(W)$.

Next we prove (\ref{alpha00}). First we consider the case $n=1$.
Collisions of $W$ with the fixed scatterers \index{Scatterer}
$\BAN_j$, $j\geq 1$, are described in the proof of
Lemma~\ref{prgrow}. Now if $W$ collides with the disk $\BAN_0$ of a
very small radius $\br$, say $\br<\cO(\length(W))$, then $W$ may be
torn into three pieces as described above. The middle piece
(reflecting off $\BAN_0$) will be further subdivided into countably
many \index{H-components} H-components lying in all the homogeneity
\index{Homogeneity strips} strips $\bbH_{\pm k}$ for $k\geq k_0$, as
well as $\bbH_0$. Those \index{H-components} H-components will be
expanded by factors $\geq ck^2/\br$ and $>c/\br$, respectively, see
our estimates in Section~\ref{subsecSUV}. Hence the contribution of
all these \index{H-components} H-components to the sum $\sum
\vartheta_{i,1}$ will be $\br/c+2\br\sum_{k\geq k_0}(ck^2)^{-1}\leq
\Const\, \br$.

Thus, the image $\tcF_{\cD}(W)$ may consist, generally, of the
following \index{H-components} H-components $W_{i,1}$: countably
many $W_{i,1}$'s produced by a collision with $\BAN_0$, at most
$L_{\max}/L_{\min}$ countable sets of $W_{i,1}$'s produced by almost
tangential reflections off some fixed scatterers \index{Scatterer}
(cf.\ the proof of Lemma~\ref{prgrow}), and at most two
\index{H-components} H-components that miss the collision with
$\BAN_0$ and all the grazing collisions -- these land somewhere else
on $\dcD$. The last two \index{H-components} H-components are only
guaranteed to expand by a moderate factor of $\vartheta^{-1}$, which
gives an estimate
\beq
     \btheta_1 \leq 2\vartheta + \frac{L_{\max}}{L_{\min}}\,
     \frac{\Const}{k_0}+\Const\,\br
        \label{alpha01a}
\eeq
Note that if $\br > \length(W) = \cO(\delta)$, then there is at most
one (not two) \index{H-components} H-component expanding by
$\vartheta^{-1}$, and then (\ref{alpha01a}) could be easily handled
as in the proof of Lemma~\ref{prgrow} (a). Thus, we may assume that
$\br = \cO(\delta)$, and taking $\limsup_{\delta\to 0}$ we can
simplify (\ref{alpha01a}) as
$$
     \btheta_1 \leq 2\vartheta + \Const/k_0
$$
The last term can be made arbitrarily small by selecting $k_0$
large, as in the proof of Lemma~\ref{prgrow} (a), but the first
term may already exceed one, hence the estimate (\ref{alpha00})
would fail for $n=1$.

Therefore, we have to consider the case $n\geq 2$. Our previous
analysis shows that the image $\tcF_{\cD}^n(W)$ will consist of
\index{H-components} H-components $W_{i,n}$ of two general types:
(a) countably many \index{H-components} H-components that have
either collided with $\BAN_0$ at least once or got reflected almost
tangentially off some fixed scatterer \index{Scatterer} at least
once, and (b) all the other \index{H-components} H-components.
Respectively, we decompose $\sum_i\ \vartheta_{i,n} =
\sideset{}{^{(a)}}{\textstyle \sum} + \sideset{}{^{(b)}}{\textstyle
\sum}$.

First, we estimate $\sideset{}{^{(a)}}{\textstyle \sum}$. The above
estimate (\ref{alpha01a}) can be easily extended to a more general
bound:
$$
   \Theta\colon=\sup_{W} {\textstyle \sum}_i\
   \vartheta_{i,1} < \Const
$$
where the supremum is taken over all \index{H-curves} H-curves
$W\subset\tOmega_{\cD}$. Now the chain rule and the induction on $n$
gives
\beq
         \sideset{}{^{(a)}}{\textstyle \sum}
         \leq \Const\, \Theta^n(\br+1/k_0)
              \label{Sigmaa}
\eeq

We now turn to $\sideset{}{^{(b)}}{\textstyle \sum}$. First, we need
to estimate the maximal number of \index{H-components} H-components
of type (b), we call it $\tilde{K}_n$. Suppose for a moment that
$\BAN_0$ is removed from the billiard table. Then any short
\index{u-curves (unstable curves)} u-curve will be cut into at most
$K_n'\leq C_1n+C_2$ pieces by the singularities of the corresponding
collision map during the first $n$ collisions, see above, where
$C_1$ and $C_2$ only depend on the fixed scatterers
\index{Scatterer} $\BAN_i$, $i\geq 1$. Now we put $\BAN_0$ back on
the table. As we have seen, each unstable curve during a free flight
between successive collisions with the fixed scatterers can be cut
by $\BAN_0$ into three pieces, of which only two (the middle one
excluded) can produce \index{H-components} H-components of type (b),
thus adding one more piece to our count. Therefore the total number
of pieces of type (b), after $n$ reflections, will not exceed
$\tilde{K}_n\leq nK_n'\leq C_1n^2+C_2n$. This gives a quadratic
bound on $\tilde{K}_n$, and it is important that this bound is
independent of the location or the size of the variable scatterer
$\BAN_0$, i.e.\ our bound is uniform over all the billiard tables in
our family.

Now, since $\vartheta_{i,n}\leq \vartheta^n$ for every
\index{H-components} H-component of type (b), then
$$
         \sideset{}{^{(b)}}{\textstyle \sum}
         \leq \tilde{K}_n \vartheta^n
         \leq (C_1n^2+C_2n)\,\vartheta^n
$$
thus
$$
     \btheta_n \leq (C_1n^2+C_2n)\,\vartheta^n
     + \Const\,\Theta^n/k_0
$$
where the $\limsup_{\delta\to 0}$ is already taken to eliminate
$\br$ from (\ref{Sigmaa}). Clearly the first term here is less
than one for some $n\geq 1$, and then the second term can be made
arbitrarily small by choosing $k_0$ large, hence we obtain
(\ref{alpha00}).

This proves that the exponential bound on correlations for the map
$\tcF_{\cD}$ will be {\em uniform} for all the billiard tables in
the family constructed in Extension~3.

We need to make yet another remark: the one-step expansion estimate
(\ref{alpha00}) implies the analogue of the growth \index{Growth
lemma} lemma~\ref{prgrow} for the map $\tcF$, with all the constants
$\beta_1,\dots,\beta_6$ and $q$ independent of the location or the
size of $\BAN_0$.

\subsection{Large deviations}
\label{subsecLD} Consider an unstable curve $W$ with the Lebesgue
measure $d\nu$ on it. Denote by $\cJ_{W}\cF^n(x)$ the Jacobian
(the expansion factor) of the map $\cF^n$ restricted to $W$ at the
point $x\in W$.

\begin{proposition}[Large deviations]
\label{PrLD}
There are constants $K>0$ and $\theta<1$ such that
uniformly in $W$ and $n\geq 1$
$$
   \nu\bigl(x\in W\colon\ \ln\cJ_W\cF^n(x)>K n\bigr)
   \leq \Const\, \theta^{n}.
$$
\end{proposition}

Note: by time reversibility, a similar estimate holds for stable curves
and negative iterations of $\cF$.

\begin{lemma}
\label{LmExpJac} There is $A>1$ such that for any $\zeta
\in(0,1/2)$ there is $C_{\zeta}>0$ such that uniformly in $W$ and
$n$

$$
     \int_W \bigl|\cJ_W\cF^n(x)\bigr|^{\zeta}\, d\nu
     \leq C_{\zeta} A^n.
$$
\end{lemma}

\proof For every $x\in W$ let $m_1,\dots,m_{n}$ be the indices of
the homogeneity \index{Homogeneity strips} strips where the first
$n$ images of $x$ belong, i.e.\ let $\cF^i(x)\in\bbH_{m_i}$ for
$1\leq i\leq n$. To avoid zeroes, let us relabel the set $\bbH_0$ by
$\bbH_1$ here. Now, as we mentioned in Section~\ref{subsecA1}, the
expansion factor of $\cF$ on \index{u-curves (unstable curves)}
u-curves $W \subset \cF^{-1}(\bbH_m)$ is $\cO(m^2)$, hence
$$
      C_1^n\, m_1^2\cdots m_n^2<
      \cJ_W\cF^n(x) < C_2^n\, m_1^2\cdots m_n^2
$$
for some constants $C_2>C_1>0$. On the other hand, there is a
constant $B>1$ such that for any given sequence $m_1,\dots,m_{n}$
(``itinerary''), there is at most $B^n$ \index{H-components}
H-components $W_k\subset\cF^n(W)$ so that the points $x\in
\cF^{-n}(W_k)$ have exactly this itinerary. This fact can be proved
by induction on $n$: given an \index{H-components} H-component
$W_k$, its image has at most $B$ \index{H-components} H-components
in every single homogeneous strip $\bbH_m$, cf.\
Section~\ref{subsecA1}, where $B=L_{\max}/L_{\min}$.

Therefore, denoting by $W_{m_1\dots m_n}$ the set of points $x\in
W$ with the given itinerary $m_1,\dots,m_n$ we obtain
$$
     \nu\bigl(W_{m_1\dots m_n}\bigr) <
     (B/C_1)^n\, m_1^{-2}\cdots m_n^{-2}
$$
Hence
$$
     \int_W \bigl|\cJ_W\cF^n(x)\bigr|^{\zeta}\, d\nu
     \leq \sum_{m_1,\dots, m_n} (BC_2^{\zeta}/C_1)^n\,
      m_1^{-2+2\zeta}\cdots m_n^{-2+2\zeta}
$$
and the series converges for any $\zeta<1/2$. \qed \medskip

\noindent{\em Proof of Proposition~\ref{PrLD}} is based on
Lemma~\ref{LmExpJac} and Markov inequality:
\begin{align*}
    \nu\bigl(x\in W\colon\ \ln\cJ_W\cF^n(x)>K n\bigr)
    &=\nu\Bigl(x\in W\colon\ \bigl|\cJ_W\cF^n(x)\bigr|^\zeta
    >e^{\zeta K n}\Bigr)\\
   &\leq C_{\zeta}\bigl[A\exp(-\zeta K)\bigr]^n.
\end{align*}
It remains to choose $K$ so large that $A\exp(-\zeta K)<1$. \qed
\medskip

\subsection{Moderate deviations}
\label{subsecMD}

We use the notation of the previous section and
denote by $\bchi$ the positive Lyapunov exponent of the map $\cF$.

\begin{proposition}[Moderate deviations]
\label{PrMD} Given $\delta>0$, there are constants $C,a>0$ such that
$$
   \nu\Bigl(x\in W\colon \bigl|
   \ln \cJ_W\cF^n(x)-n\bchi\bigr|>k \Bigr)
   \leq C \exp(-a k^2/n)
$$
uniformly in $W$, $n>0$ and $\sqrt{n}\leq k\leq n^{2/3-\delta}$
\end{proposition}

Note: by time reversibility, $\bchi$ is also the positive Lyapunov
exponent of the map $\cF^{-1}$, and the above estimate holds for
stable curves and negative iterations of $\cF$.

\proof We can assume that $W$ is long enough (for example, $|W| \geq
\varepsilon_0$, see Section~\ref{subsecA1}) and replace $\nu$ with a
smooth probability measure on $W$; i.e.\ we replace $(W,\nu)$ with a
standard pair $\ell = (\gamma, \rho)$. Let $W^u_x$ denote the
\index{Standard pair} unstable manifold through $x \in \Omega$.
Since the tangent lines $\cT_{\cF^ix} (\cF^i \gamma)$ and
$\cT_{\cF^ix} \bigl(W^u_{\cF^ix} \bigr)$ are getting exponentially
close to each other as $i\to\infty$, the difference between $ \ln
\cJ_\gamma \cF^n(x) = \sum_{i=0}^{n-1} \ln \cJ_{\cF^i \gamma} \cF
(\cF^ix)$ and $\sum_{i=0}^{n-1} \ln \cJ_{W^u_{\cF^ix}} \cF (\cF^ix)$
is bounded uniformly in $n$; so it is enough to prove
\beq
   \mes_\ell \bigl(x\in \gamma \colon |S_n|>k\bigr)
   \leq C \exp(-a k^2/n),
\eeq
where
\beq \label{SnA}
   S_n = \sum_{i=0}^{n-1} A\circ \cF^i,
   \qquad A (x) = \ln \cJ_{W^u_x} \cF (x) - \bchi.
\eeq
Next we pick $m = m(n)$ such that
\beq \label{m13}
     k^2/n \ll m \ll n/k \ll m^{100}
\eeq
where $P \ll Q$ means that $P/Q = \cO (n^{-\eps})$ for some
$\eps>0$. For example, $m= n^{1/3}$ will suffice.

Next we divide the time interval $[0,n]$ into segments of length
$m$; we will estimate the sums over odd-numbered intervals and those
over even-numbered intervals separately\footnote{Thus our method
resembles the big-small \index{Big-small block technique} block
technique of probability theory, except our blocks have the same
length. It seems that using blocks of variable lengths may help to
optimize the value of $a$ in Proposition~\ref{PrMD}, but we do not
pursue this goal.}. Accordingly, we define
$$
   R_j^{(1)} \colon = \sum_{i=2jm-m}^{2jm-1} A \circ \cF^i,
   \qquad
   R_j^{(2)} \colon = \sum_{i=2jm}^{2jm+m-1} A \circ \cF^i,
$$
for $1\leq j<L \colon = \frac{n}{2m}$. Then we denote
$Z_r^{(1)}=\sum_{j=1}^{r} R_{j}^{(1)}$ and
$Z_r^{(2)}=\sum_{j=1}^{r} R_{j}^{(2)}$ for $r\ \leq L$, and obtain
$S_n = S_m + Z_L^{(1)} + Z_L^{(2)}$, so
\beq \label{ModDev3}
   \mes_\ell \bigl(|S_n|>k\bigr)
   \leq\mes_\ell \bigl(|S_m|>k/3\bigr)+
   \sum_{j=1,2}\mes_\ell \bigl(|Z_L^{(j)}|>k/3\bigr).
\eeq
The first term with $S_m$ will be handled later. Our analysis of
$Z_L^{(1)}$ and $Z_L^{(2)}$ is completely similar, so we will do
it for $Z_L^{(1)}$ only (and omit the superscript $(1)$ for
brevity).

\begin{lemma} \label{LmLaplace}
There exists a subset $\hgamma \subset \gamma$ such that
\beq \label{meshgamma}
   \mes_\ell (\hgamma) \leq \Const\,
   e^{-k^2/n}
\eeq
and for every $m^{-100} < |t| < m^{-1}$
\beq \label{Lapl}
    \int_{\gamma\setminus\hgamma}
    e^{t Z_L}\, d\mes_\ell
    \leq e^{Dt^2n}
\eeq
where $D>1$ is a constant.
\end{lemma}

This lemma implies
$$
  \mes_\ell (Z_L>k) \leq \mes_\ell (\hgamma)
  +e^{Dt^2n-tk}.
$$
Substitution $t=\frac{k}{2Dn}$ (which is between $m^{-100}$ and
$m^{-1}$ due to (\ref{m13})) gives $\mes_\ell (Z_L>k) \leq
\Const\, e^{-ak^2/n}$ with $a=1/4D$. Similarly we obtain
$\mes_\ell (Z_L<-k) \leq \Const\, e^{-ak^2/n}$, and combining we
get
\begin{equation}
\label{MDBB}
   \mes_\ell (|Z_L|>k) \leq
   \Const\, e^{-ak^2/n},
\end{equation}
which takes care of $Z_L = Z_L^{(1)}$ in (\ref{ModDev3}).

\medskip\noindent \emph{Proof of Lemma~\ref{LmLaplace}}.
We construct, inductively, sets $\emptyset = \hgamma_0 \subset
\hgamma_1 \subset \cdots \subset \hgamma_L =\colon \hgamma$ such
that (i) the image $\cF^{2mr} (\hgamma_r)$ is a union of some
\index{H-components} H-components of the set $\cF^{2mr} (\gamma)$,
(ii) $\mes_\ell (\hgamma_r \setminus \hgamma_{r-1})\leq \Const\,
\theta^m$ for some constant $\theta<1$, and (iii) we have
\begin{equation}
\label{LaplI}
    \int_{\gamma\setminus\hgamma_r}
    e^{t Z_r}\, d\mes_\ell
    \leq e^{Dt^2mr}.
\end{equation}
Then (ii) implies (\ref{meshgamma}), since $m(\hgamma) =\cO(
L\theta^m) =\cO\bigl (e^{-k^2/n} \bigr)$ due to (\ref{m13}).

Suppose $\hgamma_r$ is constructed. Let $\gamma_{r,\alpha}$ denote
all the \index{H-components} H-components of the set $\cF^{2mr}
(\gamma \setminus \hgamma_r)$ and $c>0$ a small constant. We put
\beq \label{gammarc}
  \gamma_r^{(c)}=\cup_{\alpha}
  \{\gamma_{r,\alpha} \colon
  |\gamma_{r,\alpha}|<e^{-cm}\},
  \qquad
  \hgamma_r^{(1)} = \cF^{-2mr}
  \bigl(\gamma_r^{(c)}\bigr).
\eeq
By Lemma~\ref{prgrow} (b), $\mes_\ell (\hgamma_r^{(1)}) =\cO
\bigl( e^{-cm} \bigr)$.

Next let $\gamma_{r,\alpha} \nsubseteq \gamma_r^{(c)}$ be one of
the `longer' components, denote by $\trho_{r,\alpha}$ the induced
density on $\gamma_{r,\alpha}$ and put
\beq \label{newdens}
   \rho_{r,\alpha,t}=
   \frac{\trho_{r,\alpha}\,
   e^{t Z_r \circ \cF^{-2mr}}}
   {\int_{\gamma_{r,\alpha}}
   \trho_{r,\alpha}\,e^{t Z_r \circ \cF^{-2mr}}\,dx}.
\eeq
The function $A(x)$ defined by (\ref{SnA}) is smooth along unstable
manifolds, hence $\ell_{r,\alpha,t} = (\gamma_{r,\alpha},
\rho_{r,\alpha,t})$ is a standard pair, and the regularity of
\index{Standard pair} $\rho_{r,\alpha,t}$ is uniform in $r$,
$\alpha$, and $|t|< 1/m$. Even though $A(x)$ is not smooth over
$\Omega$, it is `dynamically H\"older continuous' in the sense of
(\ref{dynHc}), see \cite[Theorem 3.6]{BSC2}. Now the same argument
as in the proof of Proposition~\ref{PrDistEq0}, which is based on
Lemma~\ref{LmCoupl}, implies $| \EXP_{\ell_{r,\alpha,t}} (A \circ
\cF^i) | \leq \Const\, \theta^i$ for some $\theta<1$ and all $i \geq
m$, provided $c$ in (\ref{gammarc}) is small enough, namely we need
$(1-c)/c > K$, where $K$ is the constant from
Proposition~\ref{PrDistEq0} (observe that $\int_{\Omega} A\, d \mu
=0$). Hence we have
\beq \label{EXPtR2}
   \bigl|  \EXP_{\ell_{r,\alpha,t}} (\tR_{r+1}) \bigr|
   \leq \Const\, \theta^{m},
   \qquad
   \bigl|  \EXP_{\ell_{r,\alpha,t}} (\tR_{r+1}^2) \bigr|
   \leq \Const\, m
\eeq
where $\tR_{r+1} = R_{r+1} \circ \cF^{-2mr}$; the second bound
follows by the same argument as in Chapter~\ref{ScME}.

Next let $\tgamma_{\beta}$ denote all the \index{H-components}
H-components of $\cF^{2m(r+1)} (\gamma \setminus \hgamma_r)$ and
$$
   \gamma_{r}^{(K)} \colon =
   \cup_{\beta} \{ \cF^{-2m} (\tgamma_{\beta})\colon
   \ \max_{x\in \cF^{-2m} (\tgamma_{\beta})}
   |\tR_{r+1}(x)| \geq Km\},
$$
where $K>0$ is the constant from the Proposition~\ref{PrLD} on
large deviations. Put $\hgamma^{(2)} = \cF^{-2mr}
(\gamma_r^{(K)})$. Since the oscillations of $\tR_{r+1}$ on each
curve $\cF^{-2m} (\tgamma_{\beta})$ are $\cO(1)$, it easily
follows from Proposition~\ref{PrLD} that $\mes_\ell
(\hgamma_{r}^{(2)}) = \cO(\theta^m)$. Now the set $\hgamma_{r+1}
\colon = \hgamma_r \cup \hgamma_r^{(1)} \cup \hgamma_r^{(2)}$ will
satisfy the requirements (i) and (ii), so it remains to prove
(iii).

Let $\gamma_{r,\alpha} \nsubseteq \gamma_r^{(c)}$. For brevity,
denote $\gamma' = \gamma_{r,\alpha} \setminus \gamma_r^{(K)}$ and
$\gamma'' = \gamma_{r,\alpha} \cap \gamma_r^{(K)}$, as well as
$\rho = \rho_{r,\alpha,t}$. At every point $x\in \gamma'$ we have
$|\tR_{r+1}| < Km$, hence
$$
   e^{t\tR_{r+1}} \leq 1 + t\tR_{r+1} + At^2\tR_{r+1}^2
$$
with a constant $A>1$, uniformly in $|t| < 1/m$. Thus
\begin{align*}
   \int_{\gamma'} e^{t\tR_{r+1}} \rho\, dx
   &\leq
   \int_{\gamma'} (1 + t\tR_{r+1} + At^2\tR_{r+1}^2) \rho\, dx\\
   &\leq
   \int_{\gamma'\cup\gamma''} (1 + t\tR_{r+1} + At^2\tR_{r+1}^2) \rho\, dx\\
   &\leq
   1 + Bt^2m \leq e^{Bt^2m}
\end{align*}
with some constant $B>0$. To obtain the second line, we used
$$\int_{\gamma''} |t\tR_{r+1}| \rho\, dx \leq \int_{\gamma''}
(1+t^2\tR_{r+1}^2)\rho \, dx,$$ and for the third line we used
(\ref{EXPtR2}) (note that $|t|\theta^m \ll t^2m$ since $t \geq
m^{-100}$). Now using (\ref{newdens}) gives
\begin{align*}
  \int_{\gamma'} \trho_{r,\alpha}\,
  e^{tZ_{r+1} \circ \cF^{-2mr}}dx &=
  \int_{\gamma'} e^{t\tR_{r+1}} \rho\, dx \times
  \int_{\gamma_{r,\alpha}}
   \trho_{r,\alpha}\,e^{t Z_r \circ \cF^{-2mr}}\,dx\\
  &\leq e^{Bt^2m}\,\int_{\gamma_{r,\alpha}}
   \trho_{r,\alpha}\,e^{t Z_r \circ \cF^{-2mr}}\,dx
\end{align*}
Summation over $\alpha$ and using (\ref{LaplI}) implies
$$\int_{\gamma\setminus\hgamma_{r+1}} e^{t Z_{r+1}}\, d\mes_\ell
\leq e^{Dt^2m(r+1)}$$ (provided $D\geq B$), which proves
(\ref{LaplI}) inductively. \qed \medskip

It remains to handle the first term in (\ref{ModDev3}). Note that
$k \gg m$, and due to the uniform hyperbolicity $A \geq 0$, hence
$S_m \geq -\bchi m$. The necessary upper bound on $S_m$ will
follow from the next lemma, which is similar to
Proposition~\ref{PrLD} on large deviations, but it controls ``very
large deviations'':

\begin{lemma}
\label{LmDevn} We have $\mes_\ell (S_m > k)\leq \Const\,m\,
e^{-k/m}$ for all $k>0$.
\end{lemma}

\proof If $S_m(x) > k$, then $A\circ \cF^i(x) >k/m$ for some $0\leq
i<m$, therefore $\cF^{i+1}(x)$ lies in the $(e^{-k/m})$-neighborhood
of $\cS_0=\partial\Omega$, but for each $i$ the probability of this
event is $\leq\Const\,e^{-k/m}$ due to the growth \index{Growth
lemma} lemma~\ref{prgrow}. This completes the proof of the lemma and
that of Proposition~\ref{PrMD}. \qed

\subsection{Nonsingularity of diffusion matrix}\label{subsecA2a}
\index{Diffusion matrix}
Here we discuss the properties of the matrix $\brsigma^2_Q(\cA)$
defined by the Green-Kubo formula \index{Green-Kubo formula} (\ref{EqSigmabar}).

\begin{lemma}
The matrix $\brsigma^2_Q(\cA)$ depends on $Q$ continuously.
\label{lmdiffcont}
\end{lemma}

\proof Every term in the series (\ref{EqSigmabar}) depends on $Q$
continuously, and the claim now follows from a uniform bound proved in
Extension~1 of Section~\ref{subsecA1}. \qed\medskip

Next we describe the conditions under which the matrix
$\brsigma^2_Q(\cA)$ is nonsingular. For any vector $u \in
\reals^2$ we have
\begin{align*}
   u^T \brsigma^2_Q(\cA)\, u &=
   \sum_{n=-\infty}^{\infty}\int_{\Omega_Q}
   (u^T\cA)\bigl[(\cA\circ\cF_Q^n)^Tu\bigr]\, d\mu_Q \\
   &= \sum_{n=-\infty}^{\infty}\int_{\Omega_Q}
   g_u\, (g_u\circ\cF_Q^n)\, d\mu_Q
\end{align*}
where $g_u = \la \cA, u\ra$ is a smooth function on $\Omega_Q$.

\medskip\noindent{\bf Fact}.
For any smooth function $g\colon \Omega_Q \to \reals$, the
following three conditions are equivalent: \begin{itemize}
\item[1.] $\sum_{n=-\infty}^{\infty} \int g_u\, (g_u\circ\cF_Q^n)\, d\mu_Q = 0$;
\item[2.] $g = h\circ\cF_Q - h$ for some $h\in L^2(\Omega_Q)$;
\item[3.] For any periodic point $x\in \Omega_Q$ with period $k \geq 1$,
such that $g$ is smooth at $x, \cF(x), \ldots, \cF^{k-1}(x)$ we
have
$$
    S_g(x) \colon = \sum_{i=0}^{k-1} g(\cF_Q^ix) = 0
$$
\end{itemize}
The equivalence of 1 and 2 is a standard fact of ergodic theory (see
e.g.\ \cite[Lemma 2.2]{CLB}); for the equivalence of 2 and 3 see
\cite{E}, \cite[Section~7]{BSC2}, and \cite[Section~5]{BSp}.

Now, if the matrix $\brsigma^2_Q(\cA)$ is singular, it has an
eigenvector $u$ corresponding to the zero eigenvalue, so that
$\brsigma^2_Q(\cA) = 0$. Equivalently, for any periodic point
$x\in \Omega_Q$ of period $k\geq 1$,
$$
   \la S_{\cA}(x), u \ra = 0,\qquad
    S_{\cA}(x) \colon = \sum_{i=0}^{k-1} \cA(\cF_Q^ix) = 0
$$
Therefore, we obtain the following:

\medskip\noindent{\bf Criterion for nonsingularity of $\brsigma^2_Q(\cA)$}.
Suppose there are two periodic points, $x_1,x_2\in\Omega$ with
periods $k_1$ and $k_2$, respectively, such that the vectors
$S_{\cA}(x_1)$ and $S_{\cA}(x_2)$ are nonzero and noncollinear.
Then $\brsigma^2_Q(\cA)$ is nonsingular.
\medskip

Observe that there is at least one periodic point $x$ such that
$S_{\cA}(x) \neq 0$, it is made by an orbit running between $\cP(Q)$
and the closest scatterer \index{Scatterer} $\BAN_i$. On many
billiard tables, one can easily find several such trajectories,
which would guarantee the nondegeneracy of $\brsigma^2_Q(\cA)$.
%It is also known that periodic points
%are dense in $\Omega_Q$ \cite{BSC1}. By a perturbation argument,
%it is easy to show that the billiard tables for which
%$\brsigma^2_Q(\cA)$ is nonsingular make an open and dense set.

%Other orbits may be found by a standard variational argument, as
%local minima of the function $L(r)=\dist(r,\dcD)$ for
%$r\in\dcP(Q)$. It is clear that for most billiard tables $\cD$
%there are at least two nonparallel periodic orbits of that kind,
%and then the matrix $\brsigma^2_Q(\cA)$ will not be singular. One
%might find it surprising, however, that this is not always true --
%there are billiard tables where every scatterer $\BAN_i$ (except
%the one closest to $\cP(Q)$) is partially blocked from $\cP(Q)$ by
%another scatterer, so that the function $L(r)$ has no local
%minima! Such examples, though, are very special (we will not
%present them here).

\subsection{Asymptotics of diffusion matrix}\label{subsecA2b}
 \index{Diffusion matrix}
Here we discuss the asymptotics of $\sigma^2_Q(\cA)$ as $\br\to 0$
and prove (\ref{rsmall}), in fact a stronger version of it:
$$
   \sigma^2_Q(\cA)=\frac{8\br}{3\, \Area(\cD)}\, I
   +Z_Q\br^2 + o(\br^2),
$$
where $Z_Q$ is a $2\times 2$ matrix (independent of $\br$). By
virtue of (\ref{FPL}), this is equivalent to
\begin{equation}
   \label{rsmall2}
   \brsigma^2_Q(\cA)=\frac{8\pi\br}{3\, \length(\dcD)}\,
   I+\biggl(Z_Q - \frac{16\pi^2}{3\,
   [\length(\dcD)]^2}I\biggr)\br^2 + o(\br^2).
\end{equation}
We also provide an explicit algorithm for computing $Z_Q$.

First we fix our notation. To emphasize the dependence of our
dynamics on $\br$ we denote by $\Omega_{Q,\br}$ the collision space,
$\cF_{Q,\br}$ the collision map and $\mu_{Q,\br}$ the invariant
measure. We also use notation of Extension~3 of
Section~\ref{subsecA1}, after identifying our disk $\cP(Q)$ with the
variable scatterer \index{Scatterer} $\BAN_0$: thus we get the
collision space $\tOmega_{\cD}$, the collision map $\tcF_{\cD}$ on
it (however, we will denote this map by $\tcF_{Q,\br}$ to emphasize
its dependence on $Q$ and $\br$), and the corresponding invariant
measure $\tilde{\mu}_{\cD}$. Note that $\tilde{\mu}_{\cD}$ is
obtained by conditioning the measure $\mu_{Q,\br}$ on
$\tOmega_{\cD}$, the ratio of their densities is
\beq \label{Lbr}
    L_{\br}\colon=\frac{\length(\dcD)+2\pi \br}{\length(\dcD)},
\eeq
and $\tilde{\mu}_{\cD}$ is in fact independent of $Q$ and $\br$.

Consider the function $\tcA(x)\colon =\cA(\cF_Q(x))$ on
$\tOmega_{\cD}$ and the matrix
\begin{equation}
  \label{EqSigmatilde}
  \tilde{\sigma}^2_Q(\tcA)\colon= \sum_{n=-\infty}^{\infty}
  \int_{\tOmega_{\cD}}
  \tcA\, \left (\tcA\circ\tcF_{Q,\br}^n\right )^T\,d\tilde{\mu}_{\cD}.
\end{equation}

It follows from \cite[Theorem~1.3]{MT} that
$\tilde{\sigma}^2_Q(\tA)=L_{\br}\, \brsigma^2_Q(A)$. Hence it is
enough to prove that
\begin{equation}
   \label{rsmall33}
   \tilde{\sigma}^2_Q(\tcA)=
   \frac{8\pi\br}{3\, \length(\dcD)}\, I+
   Z_Q\br^2 + o(\br^2).
\end{equation}
First we will establish a weaker formula
\begin{equation}
   \label{rsmall3}
   \tilde{\sigma}^2_Q(\tcA)=
   \frac{8\pi\br}{3\, \length(\dcD)}\, I+
   \cO(\br^2\ln\br),
\end{equation}
which is, by the way, sufficient for our main purpose of proving
(\ref{rsmall}), and then outline a proof of the sharp estimate
(\ref{rsmall2}).

Due to the invariance of the measure $\mu_{Q,\br}$ under the map
$\cF_{Q,\br}$, we have $\tilde{\mu}_{\cD}(\tcA)=\mu_{Q,\br}(\cA)=0$. It
is easy to check that $\tcA$ is H\"older continuous with exponent
$\eta=1/2$ and coefficient $K_{\tcA}=\Const/\br$, hence
$\tcA\in\cH_{1,1/2}$, in the notation of Section~\ref{subsecA1}.
Therefore, the uniform bound on correlations proved in Extension~3
gives
\beq \label{Ithetar2}
  \left|\int_{\tOmega_{\cD}}
  \tcA\, \left (\tcA\circ\tcF_{Q,\br}^n\right )^T\,
  d\tilde{\mu}_{\cD}\right| \leq  \Const \, \br^{-2}
  \theta_{1,1/2}^{|n|}
\eeq
where $\Const$ is independent of $Q$ and $\br$. Let $K$ be such
that $\theta_{1,1/2}^{K|\ln\br|}=\br^5.$ Then
$$
   \tilde{\sigma}^2_Q(\tcA) =
   \sum_{|n|\leq K|\ln\br|} \int_{\tOmega_{\cD}}
  \tcA\, \left (\tcA\circ\tcF_{Q,\br}^n\right )^T\,
    d\tilde{\mu}_{\cD} + \cO(\br^3).
$$
Next we prove that for each $n\neq 0$
\beq
  \left|\int_{\tOmega_{\cD}}
  \tcA\, \left (\tcA\circ\tcF_{Q,\br}^n\right )^T\,
  d\tilde{\mu}_{\cD}\right| \leq  \Const \, \br^{2}
     \label{nbrbr}
\eeq
By the time symmetry it is enough to consider $n>0.$ Let
$\ttcA=\tcA\circ \tF_Q^{-1}.$ Then (\ref{nbrbr}) is equivalent
to
\beq
  \left|\int_{\tOmega_{\cD}}
  \ttcA\, \left (\tcA\circ\tcF_{Q,\br}^{n-1}\right )^T\,
  d\tilde{\mu}_{\cD}\right| \leq  \Const \, \br^{2}.
     \label{nbrbr1}
\eeq

Consider domains
$$
     \tPi_\br \colon = \{x\colon \tcA\neq 0\},
     \quad\text{and}\quad
     \ttPi_\br \colon = \{x\colon \ttcA\neq 0\}
$$
in the space $\tOmega_{\cD}$. Observe that $\tPi$ consists of points
that are about to collide with $\BAN_0$, and $\ttPi$ consists of
points that just collided with $\BAN_0$. Thus $\tPi$ is a finite
union of narrow strips of width $\cO(\br)$ stretching along some
s-curves, while $\ttPi$ is a finite union of strips of width
$\cO(\br)$ stretching along some \index{u-curves (unstable curves)}
u-curves, see Fig.~\ref{FigPiPi}. Observe also that $\cap_{\br>0}
\tPi_\br$ is a finite union of s-curves $\tgamma_i \subset
\tOmega_{\cD}$ (consisting of points $x \in \tOmega$ whose
trajectories run straight into the point $Q$), and $\cap_{\br>0}
\ttPi_\br$ is a finite union of \index{u-curves (unstable curves)}
u-curves $\ttgamma_i \subset \tOmega_{\cD}$ whose trajectories come
straight from the point $Q$.

\begin{figure}[htb]
    \centering
    \psfrag{O}{$\tOmega$}
    \psfrag{1}{$\tPi_{\br}$}
    \psfrag{2}{$\ttPi_{\br}$}
    \includegraphics{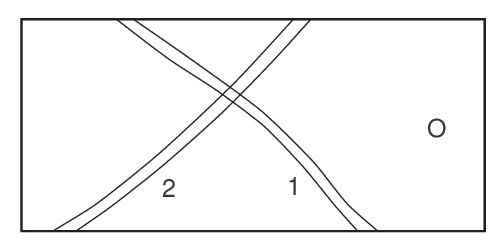}
    \caption{Two narrow strips in $\tOmega$}
    \label{FigPiPi}
\end{figure}

Now the estimate (\ref{nbrbr1}) is obvious for $n=1$. For $n>1$ we
apply the growth \index{Growth lemma} lemma~\ref{prgrow}. The domain
$\ttPi$ can be easily foliated by \index{H-curves} H-curves of
length $\cO(1)$ (independent of $\br$). If the foliation is smooth
enough, the conditional measures on its fibers will have homogeneous
densities, cf.\ Section~\ref{subsecSHUC}, thus they become standard
pairs.
\index{Standard pair} Then Lemma~\ref{prgrow} implies that at any
time $n>1$ the images of those fibers will consist, on average, of
curves of length $\cO(1)$. Thus the fraction of that image
intersecting $\tPi$ will be $<\Const\,\br$. Integrating over all the
fibers we obtain
$$
     \tilde{\mu}_{\cD}\Bigl( x\in\ttPi\colon\
     \tcF_{Q,\br}^{n-1}(x)\in\tPi \Bigr)
     \leq\Const\,\br\,\tilde{\mu}_{\cD}\bigl(\ttPi\bigr)
     \leq\Const\,\br^2
$$
This implies (\ref{nbrbr}).

It remains to compute the $n=0$ term
\beq
   \int_{\tilde{\Omega}_{\cD}}
    \tcA\, \tcA^T\,
    d\tilde{\mu}_{\cD} =
   L_{\br}\int_{\Omega^{\ast}_{Q,\br}}
    \cA\, \cA^T\,
    d\mu_{Q,\br}
      \label{n=0term}
\eeq
where $\Omega^{\ast}_{Q,\br}=\dcP(Q)\times [-\pi/2,\pi/2]$ is the
collision space of the disk $\cP(Q)$. The measure $\mu_{Q,\br}$
has density $c^{-1}\cos\varphi\, dr\, d\varphi$ in the coordinates
$r,\varphi$ introduced in Section~\ref{subsecSUC}, where
$c=2\,\length(\dcD)+4\pi\br$ is the normalization factor. For
convenience, we replace the arclength parameter $r$ on $\dcP(Q)$
with the angular coordinate $\psi\in [0,2\pi)$, then we get the
$\psi,\varphi$ coordinates on $\Omega^{\ast}_{Q,\br}$ and
$d\mu_{Q,\br}=c^{-1}\br\cos\varphi\, d\psi\, d\varphi$.

Due to an obvious rotational symmetry, the matrix (\ref{n=0term}) is a
scalar multiple of the identity matrix, so it is enough to compute its
first diagonal entry. The first component of the vector function $\cA$
is $2\cos\psi\cos\varphi$, hence the first diagonal entry of
(\ref{n=0term}) is
$$
  \frac{\br}{2\,\length(\dcD)}\,
  \int_{-\pi/2}^{\pi/2}\int_{0}^{2\pi}4\cos^2\psi
  \cos^3\varphi\, d\psi\, d\varphi=
  \frac{8\pi\br}{3\,\length(\dcD)}
$$
This completes the proof of (\ref{rsmall3}), and hence that of
(\ref{rsmall}). \qed \medskip

We have established the necessary result (\ref{rsmall}), but, as
T.~Spencer pointed out to us, the importance of the diffusion
\index{Diffusion matrix} matrix in physics justified further
analysis to obtain the more refined formula (\ref{rsmall33}), which
we do next.

The proof of (\ref{rsmall33}) requires more accurate calculation
of the integral (\ref{nbrbr1}). A crucial observation is that
for fixed $n$ the
intersection $\ttPi_\br \cap \tcF^{-n+1} \tPi_\br$ tends to
concentrate around finitely many points, which we call \emph{core
points} and denote by
$$
   \bigcap_{\br>0}
   \text{clos}\Bigl(\ttPi_\br \cap \tcF^{-n+1} \tPi_\br\Bigr)
   =\bigl\{x_1^{(n)},\ldots,x_{k_n}^{(n)}\bigr\}
$$
(here clos$(A)$ means the closure of $A$). These points corresponds
to billiard trajectories under the map $\cF_{Q,0}$ (on the table
with a ``degenerate'' scatterer \index{Scatterer} $\BAN_0$ of radius
$\br=0$) starting and ending at $\BAN_0=\{Q\}$. We distinguish
\emph{non-singular core points} (which do not experience tangential
collisions between their visits to $Q$) and other (\emph{singular})
core points. The values of $Z_Q$ in (\ref{rsmall33}) can be computed
by using the trajectories of the core points.

Our proof of (\ref{rsmall33}) consists of five steps. \medskip

\noindent{\bf Step 1}. We show that
$$
  \tilde{\cI}_n\colon= \int_{\tOmega_{\cD}}
  \ttcA\, \Bigl(\tcA\circ\tcF_{Q,\br}^{n}\Bigr)^T\,
  d\tilde{\mu}_{\cD} =
  a_n(Q)\,\br^2+b_n(Q,\br)+ \cO(\br^{2+\delta})
$$
for some $\delta>0.$ Here the first term, $a_n(Q)\,\br^2$,
corresponds to the contribution of non-singular core points, the
second term describes the contribution of singular core points,
and the third term accounts for the trajectories hitting $\BAN_0$
more than once before time $n$ and for non-linear effects.
\medskip

\noindent{\bf Step 2}. We establish an {\it a priori} bound
$b_n(Q,\br)\leq \Const\, \theta^n\, \br^2$ where $\theta<1$ for
the contribution of the singular core points.\medskip

\noindent{\bf Step 3}. From Steps 1 and 2 and the estimate
(\ref{Ithetar2}) we conclude that
$$
   a_n(Q) = \frac{\tilde{\cI}_n}{\br^2}
   +\cO\bigl(\theta^n+\br^\delta\bigr)=
   \cO\biggl(\frac{\theta^n}{\br^4}+\br^\delta\biggr) .
$$
Since the left hand side does not depend on $\br$, we can optimize
our bound in $\br$ to get $a_n(Q)=\cO(\ttheta^n)$ for some
$\ttheta<1$. \medskip

\noindent{\bf Step 4}. We fix $n$ and show that
$$
   \frac{b_n(Q,\br)}{\br^2}\to b_n(Q)
   \quad\text{ as } n\to\infty.
$$
We should note that $b_n(Q)$ describes the contribution of
singular core points, and the existence of such points is a
``codimension one event'' (there are only countably many orbits
starting from $Q$ and making a tangential collision), and there is
no reason for them to pass through $Q$ again). Thus for most $Q$
we expect $b_n(Q) = 0.$ However, since $Q$ varies over a
two-dimensional domain, we do expect a non-zero contribution for
some exceptional values of $Q$.
\medskip

\noindent{\bf Step 5}. The estimates of steps 2--4 and the
dominated convergence theorem imply
$$
   \sum_n \frac{\tilde{\cI}_n}{\br^2}
   \to \sum_n \left[a_n(Q)+b_n(Q)\right].
$$

Steps 3 and 5 are self-explanatory. We now describe estimates in
Steps 1, 2 and 4 in more detail. First we compute $a_n(Q)$. For
every point $q \in \dcD$ with coordinate $r$ we denote by
$\be_Q(r)$ the unit vector pointing from $q$ to $Q$, and by
$d_Q(r)$ the distance from $q$ to $Q$. Let $x_i^{(n)} =
\bigl(r_i^{(n)}, \varphi_i^{(n)}\bigr)$ be a nonsingular core
point. In its vicinity, i.e. for $\bigl|r-r_i^{(n)}\bigr|
<\varepsilon$, we have
$$
    \ttPi_{\br}=\left\{\left|\varphi-
    \varphi^*(r)\right| <
    \sin^{-1} \left(\frac{\br}{d(r)}
    \right)\right\},
$$
where $\varphi^*(r)$ denotes the reflection angle of the unique
trajectory arriving at $r$ straight from the point $Q$. For small
$\br$ we can approximate
$$
   \sin^{-1} \left(\frac{\br}{d(r)}\right)
   =\frac{\br}{d(r)}+\cO(\br^3)
$$
and so $\ttcA=\ttu\bigl(r,\frac{\varphi-\varphi^*(r)}{\br}\bigr)
+\cO(\br^2)$, where $ \left\Vert
\ttu(r,s)\right\Vert=2\sqrt{1-[d(r) s]^2} $, and $\ttu$ makes
angle $\pi-\sin^{-1}\left(d(r) s\right) $ with the vector $\be_Q
(r)$. Similar formulas apply to $\tcA$.

Observe that the set $\tcF_Q^{-n} \tPi_{\br} \cap \ttPi_{\br}$
consists of three (not necessarily disjoint) parts:

\begin{enumerate}
\item Vicinities of nonsingular core points that come back to $Q$
for the first time in exactly $n$ collisions;
\item Vicinities of nonsingular core points that come back to $Q$
more than once in the course of $n$ collisions;
\item Vicinities of singular core points, more precisely,
orbits passing in the $\br$--neighborhood of tangential collisions
before returning to $Q$.
\end{enumerate}

We claim that, at least for large $n,$ the main contribution to the
integral $\tilde{\cI}_n$ comes from orbits of the first type. To
make this statement precise we denote by $I_2(n)$ and $I_3(n)$ the
contribution of type 2 and type 3 orbits, respectively. To estimate
$I_2(n)$ we observe that for fixed $k<n$ the set of points having
collision with $\BAN_0$ immediately before the $k$-th return has
measure at most $\Const\, \br^2.$ The images of these points can be
foliated by \index{u-curves (unstable curves)} u-curves of length
$\cO(1)$, so the contribution of such points is bounded by
$$
   \left| I_2(k,n) \right| \leq \Const\, \br^3
$$
Summation over $k=1,\ldots,n$ gives
$$
    \left| I_2(n) \right| \leq \Const\, \br^3 |\ln\br| .
$$
Next we turn to the type 3 orbits. Let $x$ be such an orbit and
$k\in [1,n]$ denote the first moment of time when its image passes
in the $\br$--neighborhood of a tangential collision. Assume first
that $k\leq n/2.$ Again, the measure of the set of all such orbits
is $\Const\, \br^2$, and it can be foliated by \index{u-curves
(unstable curves)} u-curves of length $\cO(\br)$. Denote by
$\gamma_k(x)$ the u-curve containing the image of our point $x$.
Consider now the images of $\gamma_k(x)$ at time $\frac{3n}{4}$ and
denote by $\tr_{3n/4}(x)$ the distance from the corresponding image
of our point $x$ to the nearest endpoint of the u-curve it belongs
to. Pick $\lambda$ slightly larger than 1, then we have two cases:
\medskip

\noindent (1) $\tr_{3n/4}(x) < \lambda^n \br$. By
Lemma~\ref{prgrow}, if $\lambda$ is sufficiently close to 1, then
the measure of all such points is less than $\Const\, \theta^n
\br^2$ with some $\theta <1$. \medskip

\noindent (2) $\tr_{3n/4}(x)\geq\lambda^n \br.$ In this case we can
again use the growth \index{Growth lemma} lemma~\ref{prgrow} to
conclude that the conditional probability of later hitting $\BAN_0$
(i.e.\ coming to $\tPi_{\br}$) is at most
$$
   \Const\,\frac{\br}{\lambda^n \br}
   = \frac{\Const}{\lambda^n},
$$
and so the total contribution of such orbits is at most
$\Const\,\br^2/\lambda^n.$ Summation over $k\leq n/2$ gives the
combined contribution that can be expressed as $\Const\, n\,
\br^2\left(\theta^n+1/\lambda^n\right) .$

The case $k\geq n/2$ can be reduced to the previous one by using
the time reversal property of billiard dynamics. Hence
$\left|I_3(n)\right|\leq \Const\,\ttheta^n$ for some $\ttheta<1$,
thus establishing the estimate claimed in Step 2.

Let us now compute the contribution of type 1 orbits. We only
outline the argument leaving the (elementary but lengthy)
estimates of some higher-order terms out. Let $x = x_i^{(n)} =
\tcF^{-n}_Q\tgamma_{j'} \cap \ttgamma_{j''}$ be a nonsingular core
point. Choose a frame in $\cT_x\Omega$ consisting of a unit vector
tangent to $\ttgamma_{j''}$ and $\frac{\partial}
{\partial\varphi}.$ Consider a frame in $\cT_{\tcF^n_Q x}\Omega$
consisting of a unit vector tangent to $\tgamma_{j'}$ and
$\frac{\partial}{\partial\varphi}.$ Denote by $\zeta(x) = D_x
\cF^n_Q$ the $2\times 2$ matrix of the derivative of the map
$\cF_Q^n$ in these frames. Foliate a neighborhood of $x$ by curves
$$
   \sigma_c=\left\{\frac{\varphi-\varphi^{\ast}(r)}
   {\br}=c\right\}.
$$
Then $\ttcA$ is approximately constant on each such curve. When
$\br$ is sufficiently small the image $\tcF^n_Q\sigma_c$
intersects $\tPi_{\br}$ in a curve which is close to the straight
line $D_x \tcF^n_Q (\cT_x\ttgamma)$ and the length of its preimage
is about $ 2\br/\zeta_{21}(x)$. On the curve $\sigma_c$, we have
$\ttcA=\ttu(r,c)+\cO(\br) $ so the average value of $\ttcA \tcA$
over the above intersection is
$$
   \ttu(r(x),c) \int \tu\bigl(r(F^n x), \tc\bigr)
   \,\, d\tc+\cO(\br).
$$
Observe that the integral here is a vector parallel to
$\be_Q(r(\cF_Q^n x))$. To compute the magnitude of this vector
consider an angular coordinate $\psi$ on $\BAN_0$ such that its
value $\psi=0$ corresponds to the direction toward the point
$r(\cF^n_Qx).$ In this coordinate, possible angles of collision
range from $-\frac{\pi}{2}+\cO(\br)$ to $\frac{\pi}{2}+\cO(\br).$
The incoming vector is close to $(-1,0)$, so the outgoing vector
is close to $\bigl(\cos (2\psi), \sin (2\psi)\bigr)$, hence
$$
   \tu\approx \bigl(\cos (2\psi)+1, \sin (2\psi)\bigr)
   =\bigl(2(1-\sin^2 \psi), \sin (2\psi)\bigr) .
$$
Finally the incoming vector makes angle close to $\br \sin \psi /
d(r(\cF^n x)).$ Hence the length of the average momentum change is
close to
$$
    \frac{\int_{-\pi/2}^{\pi/2} (1-\sin^2\psi)\, d \sin \psi}
     {\int_{-\pi/2}^{\pi/2}\, d \sin \psi}
      =\frac{4}{3}.
$$
Next, averaging over $c$ and using the density of the invariant
measure $d\mu = [\length(\cD)]^{-1} \cos\varphi\, dr\, d\varphi$
gives the total contribution of the nonsingular core point $x$,
which we denote by $Z_Q(x)\,\br^2+\cO(\br^3)$, where
$$
   Z_Q(x)=-\;\;\frac{64 \cos \varphi^*(x)}
   {9 d(x) d(\cF^n x)\zeta_{21}(x)\, \length(\partial \cD)}
     \,\be_Q(r(x))\otimes \be_Q\bigl(r(\cF^n x)\bigr).
$$
Thus
$$
    \sum_n a_n(Q)=\sum_{x} Z_Q(x),
$$
where the sum is taken over all nonsingular core points of type 1.
This completes estimate claimed in Step 1.

It remains to compute the contribution of type 3 orbits around
singular core points (which only occur for some exceptional values
of $Q$, as we explained above).
%For a fixed $n$, we see that for sufficiently small $\br$ there is no
%contribution except there is a point $x$ such that $x\in\ttgamma,$
%$\cF^n x\in \tgamma$ and $\cF^k x\in \cS$ for some $k<n.$
This can be done similarly to $a_n(Q)$, except now the two parts
of $\tcF_Q^{-n} \tPi_{\br} \cap \ttPi_{\br}$ separated by the
discontinuity curve have to be treated differently. The part
experiencing an almost grazing collision makes no contribution to
$b_n(Q)$, since $\zeta_{21}=\infty$. The contribution of the part
avoiding the grazing collision is computed similarly to the type 1
orbits, but now the image of $\cF_Q^n \sigma_c$ will be cut by the
singularity curve, which needs to be approximated by its tangent
line. The resulting expression for $b_n(Q)$ is not very useful, so
we do not include it here. \qed
\newpage

\chapter[Growth and distortion]{Growth and distortion in
dispersing billiards} \setcounter{section}{2}
\setcounter{subsection}{0}
%\addcontentsline{toc}{section}{Appendix B. Growth and distortion properties of dispersing billiards}
%\renewcommand{\theequation}{B.\arabic{equation}}
%\renewcommand{\thelemma}{B.\arabic{lemma}}
%\renewcommand{\thesubsection}{B.\arabic{subsection}}
%\setcounter{equation}{0} \setcounter{subsection}{0}

\index{Dispersing billiards}

\subsection{Regularity of
\index{H-curves} H-curves} It is known that for $C^r$ smooth
uniformly hyperbolic maps, such as Anosov diffeomorphisms, unstable
manifolds are uniformly $C^r$ smooth and the conditional densities
of the SRB measure on unstable manifolds are uniformly $C^{r-1}$
smooth.

In the case of billiards, the collision map $T\colon
\Omega\to\Omega$ is $C^r$ smooth whenever the table border $\dcD$
is $C^{r+1}$ smooth. Then we also have the $C^r$ smoothness of
unstable manifolds and the $C^{r-1}$ smoothness of SRB densities,
but not uniformly over the space $\Omega$, since the corresponding
derivatives explode near the singularities. Here we establish
certain uniform bounds on the corresponding first and second
derivatives.

%Homogeneous unstable curves
%(H-curves) in dispersing billiards have been an object of study in
%many papers, but we need to improve and extend the existing
%results.
Let $W_0\subset\Omega$ be an \index{H-curves} H-curve, $x_0=(r_0
,\varphi_0)\in W_0$, and for every $n\geq 1$ denote by $W_n$ the
\index{H-components} H-component of $\cF^n(W_0)$ containing the
point $x_n=(r_n ,\varphi_n) =\cF^n (x_0)$. In the $r,\varphi$
coordinates, the curve $W_n$ is a function $\varphi(r)$ and we
denote its slope at the point $x_n$ by $\Gamma_n= d\varphi/dr$.
Recall that we use the metric (\ref{dxnormQV}) on \index{u-curves
(unstable curves)} u-curves, in which the norm of tangent vectors
$dx=(dr,d\varphi)$ to u-curves satisfies
\beq
     \|dx\|^2 = (dr\cos\varphi)^2 + (d\varphi+\cK\, dr)^2
       \label{dxnormQV0}
\eeq
Recall that by (\ref{CCs}) $0<c_1\leq |dx/dr| \leq c_2<\infty$ for
some constants $c_1,c_2$. Let
$$
  \cJ_{W_i}\cF^{-1}(x_{i})=
  \bigl[\cJ_{W_{i-1}}\cF(x_{i-1})\bigr]^{-1}=|dx_{i-1}|/|dx_{i}|
$$
denote the Jacobian (the contraction factor) of the map $\cF^{-1}\colon
W_{i}\to W_{i-1}$ at the point $x_i$.

\begin{proposition}
\label{PrDistCurv} Suppose the boundary $\dcD$ is of class $C^3$ and
$|d\Gamma_0/dx_0|\leq C_0$ for some $C_0>0$ and all $x_0\in W_0$. Then
there is a constant $C>0$ such that for all $n\geq 1$
\beq
   \left|\frac{d\Gamma_n}{dx_n}\right|\leq C
     \label{curvC3}
\eeq
and
\beq
 \left |\frac{d\ln \cJ_{W_n}\cF^{-1}(x_n)}{dx_{n}}\right |
 \leq \frac{C}{|W_{n}|^{2/3}}
   \label{distC3}
\eeq
Suppose, in addition, that the boundary $\dcD$ is of class $C^4$ and
moreover $|d^2\Gamma_0/dx_0^2|\leq C_0$. Then for all $n\geq 1$
\beq
   \left|{d^2\Gamma_n}/{dx_n^2}\right|\leq C
     \label{curvC4}
\eeq
and
\beq
 \left |\frac{d^2\ln \cJ_{W_n}\cF^{-1}(x_n)}{dx_{n}^2}\right |
 \leq \frac{C}{|W_{n}|^{4/3}}
   \label{distC4}
\eeq
\end{proposition}

The first part of this proposition (related to a $C^3$ boundary)
is known -- full proofs are provided in \cite{C3}, even for a more
general class of billiards, where a small external field is
permitted. The second part related to a $C^4$ boundary is new.

Before giving a proof, we derive a corollary. Let $\rho_0$ denote a
density on $W_0$ and $\rho_n$ the induced density on $W_n$:
\beq
   \rho_n(x_n) =
   \rho_0(x_0)\,\cJ_{W_n}\cF^{-n}(x_n)
     \label{rhon}
\eeq

\begin{corollary}
\label{CoDens1} Suppose the boundary $\dcD$ is of class $C^3$ and
$|d\Gamma_0/dx_0|\leq C_0$ and $|d\ln\rho_0 / dx_0| \leq C_0$.
Then there is a constant $C>0$ such that for all $n\geq 1$
\beq
 \left |\frac{d\ln \rho_n}{dx_{n}}\right |
 \leq \frac{C}{|W_{n}|^{2/3}}
   \label{densC3}
\eeq
Suppose, in addition, that the boundary $\dcD$ is of class $C^4$ and
moreover $|d^2\Gamma_0/dx_0^2|\leq C_0$. Then for all $n\geq 1$
\beq
 \left |\frac{d^2\ln \rho_n}{dx_n^2}\right |
 \leq \frac{C}{|W_{n}|^{4/3}}
   \label{densC4}
\eeq
\end{corollary}

\proof This follows by logarithmic differentiation of (\ref{rhon})
and using the following simple estimate:
$$
   \frac{1}{|W_i|^{2/3}}\,
   \left|\frac{dx_i}{dx_n}\right| =
   \frac{\cJ_{W_n}\cF^{i-n}(x_n)}{|W_i|^{2/3}}
   \leq\frac{\Const\, \vartheta^{\frac{n-i}{3}}}{|W_n|^{2/3}}
$$
where $\vartheta^{-1}>1$ denotes the minimal factor of expansion of
\index{u-curves (unstable curves)} u-curves. \qed \medskip

\medskip
\noindent{\em Proof of Proposition~\ref{PrDistCurv}}. The curve
$W=W_i$ corresponds to a family of trajectories of the billiard flow
$\Phi^t$. Let $t_i$ be the reflection time for the trajectory of the
point $x_i$. The tangent vector $(dr_i ,d\varphi_i )\in\cT_xW$
corresponds to a (time-dependent) tangent vector $(dq_t,dv_t)$ to
the orthogonal cross-section of that family, as it was shown in
Chapter~\ref{SecSPE} (note that both $dq_t$ and $dv_t$ here are
perpendicular to the velocity vector $v_t$ of the family, since $M =
\infty$). Denote by $\cB_t=|dv_t|/|dq_t|>0$ the curvature of the
family.

The following facts are standard in billiard theory \cite{C1,C2}
and can be obtained directly:
\beq
     \tfrac{d}{dt}\,dq_t=dv_t,\quad
     \tfrac{d}{dt}\,dv_t=0,\quad
     \tfrac{d}{dt}\bigl(\cB_t^{-1}-t\bigr)=0
       \label{dBdt}
\eeq
(provided $t$ is not a moment of collision) and
\beq \label{A25}
   \bigl|dq_{t_i^+}\bigr| = \bigl|dq_{t_i^-}\bigr|,\qquad
  \cB_{t_i^+}=\cB_{t_i^-}+\frac{2\cK(r_i)}{\cos\varphi_i}
\eeq
at a moment of collision (here $\cK(r)>0$ denotes the curvature of
$\dcD$ at the point $r$, and $t_i^-$, $t_i^+$ refer to the
precollisional and postcollisional moments, respectively). One can
see that
$$
   c_1\leq\cB_{t_i^-}\leq c_2,\qquad
   c_1\leq\cos\varphi_i\,\cB_{t_i^+}<c_2
$$
for some constants $0<c_1<c_2<\infty$. Note that $\tcB_i\colon
=\cB_{t_i^-}$ remains uniformly bounded. In fact, all our troubles
come from the unbounded factor $1/\cos\varphi_i$ in (\ref{A25}).

\begin{lemma}
\label{LmDB} Suppose the boundary $\dcD$ is of class $C^3$ and also
$|d\tcB_0/dx_0|\leq C_0'/\cos^2\varphi_0$ for some $C_0'>0$ and all
$x_0\in W_0$. Then there is a constant $C'>0$ such that for all
$n\geq 1$
\beq
     \left| {d\tcB_n}/{dr_n}\right| \leq C'
       \label{dBdr1}
\eeq
Suppose, in addition, that the boundary $\dcD$ is of class $C^4$ and
moreover $|d^2\tcB_0/dx_0^2|\leq C_0''/\cos^2\varphi_0$. Then for
all $n\geq 1$
\beq
     \left| {d^2\tcB_n}/{dr_n^2}\right| \leq C''
       \label{dBdr2}
\eeq
\end{lemma}

We postpone the proof of the lemma and complete the proof of
Proposition~\ref{PrDistCurv} first.

The slope $\Gamma_n = d\varphi_n/dr_n$ of the curve $W_n$ satisfies
\beq
    \Gamma_n = \tcB_n\cos\varphi_n + \cK(r_n)
       \label{fn}
\eeq
hence
$$
  \frac{d\Gamma_n}{dx_n} =
  \frac{\frac{d\tcB_n}{dr_n}\cos\varphi_n-\tcB_n\Gamma_n\sin\varphi_n
  +\frac{d\cK(r_n)}{dr_n}}{[\cos^2\varphi_n+(\Gamma_n+\cK(r_n))^2]^{1/2}}
$$
(the denominator equals $|dx_n/dr_n|$ according to
(\ref{dxnormQV0})). It is easy to see that our assumption
$|d\Gamma_0/dx_0|\leq C_0$ implies $|d\tcB_0/dx_0|\leq
C_0'/\cos\varphi_0$ for some constant $C_0'>0$. Now (\ref{curvC3})
follows from (\ref{dBdr1}).

Differentiating further gives an expression for $d^2\Gamma_n/dx_n^2$
(we leave it to the reader), and it shows that our assumption
$|d^2\Gamma_0/dx_0^2|\leq C_0$ implies $|d^2\tcB_0/dx_0^2|\leq
C_0''/\cos^2\varphi_0$ for some $C_0''>0$. Now (\ref{curvC4}) follows
from (\ref{dBdr2}).

It remains to prove (\ref{distC3}) and (\ref{distC4}). To compute the
Jacobian $\cJ_{W_n}\cF^{-1}(x_n)=|dx_{n-1}|/|dx_n|$ we note that
by (\ref{A25})
$$
  |dx_n|^2=|dq_{t_n^+}|^2+|dv_{t_n^+}|^2
  =|dq_{t_n^-}|^2(1+\cB_{t_n^+}^2)
$$
hence
$$
   \cJ_{W_n}\cF^{-1}(x_n)=
   \frac{\bigl|dq_{t_{n-1}^+}\bigr|}{\bigl|dq_{t_{n}^-}\bigr|}\,
   \left[\frac{1+\cB_{t_{n-1}^+}^2}{1+\cB_{t_{n}^+}^2}\right]^{1/2}
$$
where
\beq
   \frac{\bigl|dq_{t_{n-1}^+}\bigr|}{\bigl|dq_{t_{n}^-}\bigr|}
   =\frac{1}{1+(t_{n}-t_{n-1})\,\cB_{t_{n-1}^+}}
     \label{dqdqB}
\eeq
It follows that
\begin{align*}
  2\ln\cJ_{W_n}\cF^{-1}(x_n) &= -\ln\biggl[1+\left(\tcB_n+
  \frac{2\cK(r_n)}{\cos\varphi_n}\right)^2\biggr]\\
  &\quad +\ln\biggl[\tcB_n^2+\left(1-(t_n-t_{n-1})\tcB_n\right)^2\biggr]
\end{align*}
We also note that
\begin{equation}
\label{DerT}
dt_n/dr_n=\pm\sin\varphi_n
\end{equation}
and $d^2t_n/dr_n^2
= \pm \Gamma_n \cos\varphi_n$ (where the sign depends on the
orientation of the tangent vector $(dr_n,d\varphi_n)$). Now one
can differentiate $\ln \cJ_{W_n}\cF^{-1}(x_n)$ directly and use
Lemma~\ref{LmDB} to derive bounds
$$
 \left |\frac{d\ln \cJ_{W_n}\cF^{-1}(x_n)}{dx_{n}}\right |
 \leq \frac{\Const}{\cos\varphi_n}\quad{\rm and}\quad
 \left |\frac{d^2\ln \cJ_{W_n}\cF^{-1}(x_n)}{dx_{n}^2}\right |
 \leq \frac{\Const}{\cos^2\varphi_n}
$$
which imply (\ref{distC3}) and (\ref{distC4}) due to
(\ref{cos23}). \qed \medskip

\noindent{\em Proof of Lemma~\ref{LmDB}}. Our argument has an inductive
character. Observe that
\beq
  \tcB_n = \frac{1}{t_n-t_{n-1}+
  [\tcB_{n-1}+2\cK(r_{n-1})/\cos\varphi_{n-1}]^{-1}}
      \label{BB}
\eeq
$$ = \frac{1}{t_n-t_{n-1}} -
  \frac{1}{(t_n-t_{n-1})^2
  \left(\tcB_{n-1}+\frac{2\cK(r_{n-1})}{\cos\varphi_{n-1}}+\frac{1}{t_n-t_{n-1}}\right)} . $$

Next,
$$
   \frac{| dr_{n-1} |}{| dr_n |}=
   \frac{\bigl|dq_{t_{n-1}}\bigr|/\cos\varphi_{n-1}}
   {\bigl|dq_{t_{n}}\bigr|/\cos\varphi_{n}}
   =\frac{\cos\varphi_n}{w_{n-1}}
$$
where
$$
   w_{n-1}= 2\cK(r_{n-1})(t_n-t_{n-1})+
   \cos\varphi_{n-1}\bigl(1+(t_n-t_{n-1})\tcB_{n-1}\bigr)
$$
Note that $w_{n-1}$ is uniformly bounded above and below:
$$
   0<w_{\min}\leq w_{n-1}\leq w_{\max}<\infty
$$
Now a direct differentiation of (\ref{BB}) using (\ref{fn}) and (\ref{DerT}) gives
$$
  \biggl|\frac{d\tcB_n}{dr_n}\biggr|=
  \theta_{n-1}^2\theta_n\,
  \frac{w_n}{w_{n-1}}\,
  \biggl|\frac{d\tcB_{n-1}}{dr_{n-1}}\biggr|
  +\cR'
$$
where
$$
  \theta_{n-1}=
  \frac{\bigl|dq_{t_{n-1}^+}\bigr|}{\bigl|dq_{t_{n}^-}\bigr|}=
  \frac{1}{1+(t_{n}-t_{n-1})\,\cB_{t_{n-1}^+}}
  \leq \theta_{\max}
$$
with
$$
   \theta_{\max}\colon=\frac{1}{1+L_{\min}\cK_{\min}}<1
$$
(here $L_{\min}$ is the minimum free path between collisions), and
the remainder term $\cR'$ is uniformly bounded,
$|\cR'|\leq\cR_{\max}'$. It is now easy to see that
$$
    \biggl|\frac{d\tcB_n}{dr_n}\biggr|\leq
    \frac{w_{\max}}{w_{\min}}
    \biggl(\frac{\cR_{\max}'}{1-\theta_{\max}^3}
    +\theta_0^2\,\biggl|\frac{d\tcB_0}{dr_0}\biggr|\biggr)
$$
Also note that $\theta_0^2\leq\Const\,\cos\varphi_0^2$. This proves
(\ref{dBdr1}).

Differentiating one more time and using (\ref{dBdr1}) gives
$$
  \biggl|\frac{d^2\tcB_n}{dr_n^2}\biggr|=
  \theta_{n-1}^2\theta_n^2\,
  \frac{w_n^2}{w_{n-1}^2}\,
  \biggl|\frac{d^2\tcB_{n-1}}{dr_{n-1}^2}\biggr|
  +\cR''
$$
where $|\cR''|\leq\cR_{\max}''$. Now it is easy to see that
$$
    \biggl|\frac{d^2\tcB_n}{dr_n^2}\biggr|\leq
    \frac{w_{\max}^2}{w_{\min}^2}
    \biggl(\frac{\cR_{\max}''}{1-\theta_{\max}^4}
    +\theta_0^2\,\biggl|\frac{d^2\tcB_0}{dr_0^2}\biggr|\biggr)
$$
This proves (\ref{dBdr2}). \qed\medskip

\subsection{Invariant Section Theorem} \label{subsecIST}
Here we outline a proof of the general fact mentioned in
Section~\ref{subsecHA}. Let $E^u$ be a family of unstable
directions on an s-curve $S$, and $\Gamma(x)=d\varphi/dr >0$
denote the slope of the $E^u$ direction through the point $x\in
S$. We say that $E^u$ is H\"older continuous on $S$ with exponent
$a>0$ and norm $L>0$ if for all $x,y\in S$
$$
   |\Gamma(x)-\Gamma(y)|\leq L\,[\dist(x,y)]^a
$$

\begin{proposition}
\label{PrIST} Let $S$ be an s-curve such that $S_n=\cF^n(S)$ is an
s-curve for every $n=1,\dots,N$. If a family $E^u$ on $S$ is
smooth enough, then the family $E_n^u=\cF^n(E^u)$ on $S_n$ is
H\"older continuous with exponent $a=1/2$ and norm $\leq C$ for
all $n=1,\dots,N$, where $C>0$ is a constant independent of $N$
and $S$.
\end{proposition}

\proof We use the notation of the previous section. Let
$x=(r,\varphi)\in S$ and $x_n=(r_n,\varphi_n)=\cF^n(x)\in S_n$.
Denote by $\Gamma_n(x_n)$ the slope of $E^u_n$ direction at $x_n$.
Due to (\ref{fn})
$$
     \Gamma_n(x_n) = \tcB_n(x_n)\cos\varphi_n + \cK(r_n)
$$
where $\tcB_n(x_n)$ is the curvature of the incoming family of
trajectories corresponding to the $E^u_n$ direction at $x_n$.
Since $\cK(r_n)$ is uniformly $C^1$ smooth, it is enough to prove
the H\"older continuity for $\tcB_n$ with a uniformly bounded
norm. Let $x+dx=(r+dr,\varphi+d\varphi)\in S$ be a nearby point
and $x_n+dx_n=(r_n+dr_n,\varphi_n+d\varphi_n)=\cF^n(x+dx)\in S_n$.
We will prove by induction on $n$ that
\beq
   |\tcB_n(x_n+dx_n)-\tcB_n(x_n)|\leq \brC\, u(x_{n})\,
   |dr_n|^{1/2}
     \label{BB12}
\eeq
where $\brC>0$ is a large constant and $u(x)$ is a function (defined
below), which is uniformly bounded:
$$
       0<u_{\min}\leq u(x)\leq u_{\max}<\infty
$$
(here $u_{\min}$ and $u_{\max}$ do not depend on $N$ or $S$).
Since the distance $|dr_n|$ is equivalent to our metric
(\ref{dxnormQVst}) on stable curves, the bound (\ref{BB12})
implies Proposition~\ref{PrIST}.

Now we prove (\ref{BB12}). Due to (\ref{BB})
$$
  \tcB_{n+1}(x_{n+1})=\cfrac{1}{s(x_n)+
  \cfrac{1}{\tcR(x_n)+\tcB_n(x_n)}}
$$
where $s(x)$ denotes the free path between the collision points at
$x$ and $\cF(x)$, and $\tcR(x_n)=2\cK(r_n)/\cos\varphi_n$. For
brevity, we will use notation
$d\tcB_{n}=\tcB_{n}(x_{n}+dx_{n})-\tcB_{n}(x_{n})$,
$ds_n=s(x_n+dx_n)-s(x_n)$, etc. Now elementary calculations give
$$
  |d\tcB_{n+1}|\leq\frac{|ds_n|}{[s(x_n)]^2}+
  \frac{|d\tcR_n|+|d\tcB_n|}{[\cJ_n(x_n)]^2}
$$
where
$$
   \cJ_n(x_n) = 1 + s(x_n)\bigl[\tcR(x_n)+\tcB_n(x_n)\bigr]
$$
Observe that $|ds_n|\leq |dr_n|+|dr_{n+1}|$ and
$$
  |d\tcR_n|\leq\frac{\Const\,|dr_n|}
  {\cos^2\varphi_n}
$$
hence
$$
  \frac{|d\tcR_n|}{[\cJ_n(x_n)]^2}
    \leq\Const\,|dr_n|
$$
It is also easy to see that $|dr_n|\leq\Const\,|dr_{n+1}|^{1/2}$, thus
we obtain
\beq
  |d\tcB_{n+1}|\leq\Const\,|dr_{n+1}|^{1/2}+
  |d\tcB_n|/[\cJ_n(x_n)]^2
     \label{dBn+1}
\eeq
Now pick a vector $0\neq dr_n^u\in E_n^u(x_n)$ and put $dr_{n+1}^u
=d\cF(dr_n^u)\in E_{n+1}^u(x_{n+1})$. Using the notation of the
previous section, we introduce $|dq_n|=|dr_n^u|\cos\varphi_n$ and
$|dq_{n+1}|=|dr_{n+1}^u|\cos\varphi_{n+1}$, then
$$
         \cJ_n(x_n)=|dq_{n+1}|/|dq_n|
$$
due to (\ref{dqdqB}). The element of the Lebesgue measure
$dm=dr\,d\varphi$ at the point $x_n$ can be expressed by
$$
            dm(x_n)=|dr_n|\,|dr_n^u|\, w(x_n)
$$
where
$$
   w(x_n) = d\varphi_n^u/dr_n^u-d\varphi_n/dr_n.
$$
We note that $d\varphi_n^u/dr_n^u >0$ and $d\varphi_n/dr_n <0$.
Then by (\ref{CCu}) and (\ref{CCs}), $w(x_n)$ is a function
uniformly bounded above and below by positive constants:
$$
       0<w_{\min}\leq w(x)\leq w_{\max}<\infty
$$
cf.\ (\ref{CCu}). Since the measure $d\mu=\cos\varphi\, dm$ is
$\cF$-invariant, we can write
$$
   |dr_n|\,|dr_n^u|\, w(x_n)\cos\varphi_n =
   |dr_{n+1}|\,|dr_{n+1}^u|\, w(x_{n+1})\cos\varphi_{n+1}
$$
hence
$$
     |dr_n| = |dr_{n+1}|\,\cJ_n(x_n)\,\frac{w(x_{n+1})}{w(x_n)}
$$
Now we set $u(x_n)=[w(x_n)]^{1/2}$ and use (\ref{dBn+1}) and the
inductive assumption (\ref{BB12}) to get
$$
  |d\tcB_{n+1}|\leq\Const\,|dr_{n+1}|^{1/2}+
  \frac{\brC\, u(x_{n+1})\, |dr_{n+1}|^{1/2}}{[\cJ_n(x_n)]^{3/2}}
$$
Since $\cJ_n(x_n)\geq 1+L_{\min}\cK_{\min}>1$, we have
$$
  |d\tcB_{n+1}|\leq \brC\, u(x_{n+1})\, |dr_{n+1}|^{1/2}
$$
provided $\brC$ is large enough. This proves (\ref{BB12}) by induction.
\qed

\subsection{The function space $\fR$}\label{subsecA3}
Here we prove Lemma~\ref{lmfRprop}. Clearly, it is enough to prove
it for $B_1\equiv 1$. We use induction on $n_A$. For $n_A=1$ the
lemma reduces to the definition of $\fR$. For $n_A\geq 2$ we put
$B=B_2\circ\cF^{n_A-2}$ and $A=B_2\circ\cF^{n_A-1}$. The H\"older
continuity of $A$ on the connected components of
$\Omega\setminus\cS_{n_A}$ follows from (\ref{cFHolder}):
\begin{align*}
  |A(x)-A(x')| &= |B(\cF(x))-B(\cF(x'))|\\
  & \leq
  K_B\,[\dist(\cF(x),\cF(x'))]^{\alpha_B}\\
  & \leq
  K_B K_{\cF}^{\alpha_B}\,[\dist(x,x')]^{\alpha_B\alpha_{\cF}}
\end{align*}
It remains to estimate the local Lipschitz constant Lip$_x(A)$
defined by (\ref{Ax'}). First we note that ${\rm Lip}_x(A)\leq
\|D_x\cF\|\,{\rm Lip}_{y}(B)$, where $y=\cF(x)$. The derivative
$D_x\cF$ is unbounded in the vicinity of $\cS$, more precisely, on
one side of $\cS$ which corresponds to nearly grazing collisions,
i.e.\ where $y$ is close to $\partial\Omega$. Denote
$d_1=\dist(x,\cS)$, $d_2=\dist(x,\cS_{n_A}\setminus\cS)$,
$d_3=\dist(y,\partial\Omega)\sim\pi/2-|\varphi|$, where
$y=(r,\varphi)$ in the notation of Section~\ref{subsecSUC}, and
$d_4=\dist(y,\cS_{n_A-1})$. All these distances are measured along
some unstable curves, see Section~\ref{subsecPA}, and $D_x\cF$
attains its maximal expansion $\sim 1/\sqrt{d_1}$ along unstable
curves through $x$, hence $d_3\sim \sqrt{d_1}$. Note that
$d\colon=\dist(x,\cS_{n_A})=\min\{d_1,d_2\}$. We now have two cases:
\begin{itemize} \item[(a)] If $d_1<d_2$, then $d_3<\Const\,d_4$,
hence we have $\dist(y,\cS_{n_A-1})>\Const^{-1}d_3$ and
$$
   {\rm Lip}_x(A)\leq
   \frac{\Const}{\sqrt{d_1}}\,\frac{\Const}{d_3^{\beta_B}}
   \leq\frac{\Const}{d^{(1+\beta_B)/2}}
$$
\item[(b)] If $d_2\leq d_1$, then $d_4\leq\Const\, d_3$ and
$d_4\sim d_2/\sqrt{d_1}$, hence
$$
   {\rm Lip}_x(A)\leq
   \frac{\Const}{\sqrt{d_1}}\,\frac{\Const}{d_4^{\beta_B}}
   \leq\frac{\Const}{d_1^{1/2-\beta_B/2}d_2^{\beta_B/2}}
   \leq\frac{\Const}{d^{(1+\beta_B)/2}}
$$
\end{itemize}
In either case we obtain the required estimate with
$\beta_A=(1+\beta_B)/2$. Since $\beta_B<1$, we have $\beta_A<1$.
Lemma~\ref{lmfRprop} is proved. \qed
\medskip

Lastly, we prove Proposition~\ref{propd}. Its first claim follows
from the fact that the configuration space of our system is a four
dimensional domain bounded by cylindrical surfaces. To prove the
Lipschitz continuity of $\mu_{Q,V}(d)$, consider two nearby points
$(Q,V)$ and $(Q',V')$ and denote $h=\|Q-Q'\|+\|V-V'\|$. First
assume that the light particle starts at a point $(q,v)$ such that
$q\in\dcD$ and compare $d(Q,V,q,v)$ with $d(Q',V',q,v)$. It is
convenient to use the coordinate frame moving with velocity vector
$V$ (in this frame, the disk $Q,V$ is at rest). Then the light
particle moves with velocity $v-V$ and the disk $Q',V'$ moves with
velocity $V'-V$. At the time of the next collision, the moving
disk $Q',V'$ will be at distance $\cO(h)$ from the fixed disk
$Q,V$. One can check by direct inspection that the average
difference $d(Q,V,q,v)-d(Q',V',q,v)$ is $\cO(h)$. The other case
$q\in\dcP(Q)$ is easier, we leave it to the reader. \qed
\newpage

\chapter[Distortion bounds]{Distortion bounds for two particle
system} \setcounter{section}{3} \setcounter{subsection}{0}
%\addcontentsline{toc}{section}{Appendix C. Distortion bounds}
%\renewcommand{\theequation}{C.\arabic{equation}}
%\renewcommand{\thelemma}{C.\arabic{lemma}}
%\renewcommand{\thesubsection}{C.\arabic{subsection}}
%\setcounter{equation}{0} \setcounter{subsection}{0}

 \index{Distortion bounds}

Here we prove rather technical Propositions~\ref{prdist},
\ref{prcurv} and Lemma~\ref{LmJac} whose proofs were left out in
Chapter~\ref{SecSPE}.

We first outline our strategy. We have shown in Chapter~\ref{SecSPE}
that unstable \index{Unstable vectors} vectors $dx_t = (dq_t, dv_t,
dQ_t, dV_t)$ grow with $t$ through two alternating stages: free
motion between collisions expands $dq_t$, while at collisions $dv_t$
``jumps up''. The resulting transformation of the tangent vectors is
usually described by an operator-valued continued fraction
\cite{C0,C1}, and then distortion \index{Distortion bounds} bounds
can be proved by differentiating that fraction along unstable
directions. This approach is convenient for completely hyperbolic
billiards, because it treats all the components of unstable vectors
equally. In our case, the components $dq_t$ and $dv_t$ expand
uniformly, while $dQ_t$ and $dV_t$ change little and may not grow at
all (effectively, we deal with a partially hyperbolic dynamics). We
use a more explicit approach to prove distortion bounds here: pick
two almost equal unstable \index{Unstable vectors} vectors at nearby
points on one \index{u-curves (unstable curves)} u-curve and show
that the images of these vectors have almost the same length at
every iteration.

Let $x_0= (Q_0,V_0,q_0,v_0) \in \Omega$ and $x'_0 =(Q_0', V_0',
q_0', v_0') \in \Omega$ be two nearby points that belong in one
homogeneity \index{Homogeneity sections} section, say $\bbH_{m_0}$.
Assume that for each  $1 \leq i \leq n$ the points $x_i= (Q_i, V_i,
q_i, v_i) = \cF^i (x_0)$ and $x_i'= (Q_i', V_i', q_i', v_i') = \cF^i
(x_0')$ also belong in one homogeneity \index{Homogeneity sections}
section, call it $\bbH_{m_i}$. We assume that for $0\leq i\leq n$
$$
   \| Q_i - Q_i' \| \leq C\, \|q_i - q_i'\| / M
$$
and
$$
   \| V_i - V_i' \| \leq C\, \|v_i - v_i'\| / M
$$
where $C>1$ is a large constant. Denote by $(r_i,\varphi_i)$ and
$(r_i',\varphi_i')$ the coordinates of the points $\pi_0 (x_i)$
and $\pi_0 (x_i')$, respectively, and put
$$
    \delta_i = \sqrt{ (r_i-r_i')^2 + (\varphi_i-\varphi_i')^2}
$$
Assume that for all $i \leq n$
\beq
         \| q_i - q_i' \|
          \leq C \delta _i
            \label{assq}
\eeq
and
\beq
         \| v_i - v_i' \|
         \leq C \delta_i
            \label{assv}
\eeq
where $C>1$ is a large constant. These assumptions hold, for
example, when $x_i$ and $x_i'$ belong in one unstable curve (this
follows from Propositions~\ref{prabc} and \ref{prdrdp}).

Let $dx_0 = (dQ_0, dV_0, dq_0, dv_0)$ be a postcollisional unstable
\index{Unstable vectors} vector at $x_0$, and $dx_0' = (dQ_0',
dV_0', dq_0', dv_0')$ a similar vector at $x_0'$. For $i \geq 1$,
denote by $dx_i = (dQ_i, dV_i, dq_i, dv_i)$ and $dx_i' = (dQ_i',
dV_i', dq_i', dv_i')$ their postcollisional images at the points
$x_i$ and $x_i'$, respectively. We say that the unstable vectors
$dx_i$ and $dx_i'$ are ($\varepsilon_i,
\tilde{\varepsilon}_i$)-close, if the following four bounds hold:
\begin{align*}
   \|dq_i - dq_i'\| &\leq \varepsilon_i\, \|dq_i\|, \\
   \|dv_i - dv_i'\| &\leq \varepsilon_i\, \|dv_i\|, \\
   \|dQ_i - dQ_i'\| &\leq \tilde{\varepsilon}_i\, \|dq_i\| / M, \\
   \|dV_i - dV_i'\| &\leq \tilde{\varepsilon}_i\, \|dv_i\| / M.
\end{align*}

\begin{lemma}[One-step distortion control]
Assume that the unstable \index{Unstable vectors} vectors $dx_0$ and
$dx_0'$ are ($\varepsilon_0,\tilde{\varepsilon}_0$)-close for some
small $\varepsilon_0,\tilde{\varepsilon}_0>0$. Then their images
$dx_1$ and $dx_1'$ will be
($\varepsilon_1,\tilde{\varepsilon}_1$)-close, where
\beq
   \varepsilon_1 = \left ( \varepsilon_0 +
   \frac{2\tilde{\varepsilon}_0}{M} \right )\,
   \left (1+\frac{C}{\sqrt{M}}\right ) +
   C\, \frac{\delta_0 + \delta_1}{\cos \varphi_1}
     \label{key0}
\eeq
\beq
   \tilde{\varepsilon}_1 = \Bigl( 2\varepsilon_0 +
   \tilde{\varepsilon}_0\Bigr)\,
   \left (1+\frac{C}{\sqrt{M}}\right ) +
   C\, \frac{\delta_0 + \delta_1}{\cos \varphi_1}
     \label{key1}
\eeq
Here $C>0$ is a large constant. \label{lm0}
\end{lemma}

\noindent{\em Proof}. We first compare the precollisional vectors
$$
  dx_1^- = (dQ_1^-, dV_1^-, dq_1^-, dv_1^-)
$$
and
$$
  (dx_1')^- = \bigl((dQ_1')^-, (dV_1')^-, (dq_1')^-, (dv_1')^-\bigr)
$$
at the points $x_1$ and $x_1'$, respectively. Equation
(\ref{Qqts}) and the triangle inequality imply
\begin{align*}
    \|dq_{1}^- - (dq_{1}')^-\|
    &\leq \varepsilon_0\, \|dq_0\|
    +s\, \varepsilon_0 \, \|dv_0\| \\
    &\quad +|s-s'|\, \|dv_0'\|
\end{align*}
where $s$ (resp., $s'$) is the time between collisions at the
points $x_0$ and $x_1$ (resp., $x_0'$ and $x_1'$). Due to
Proposition~\ref{prabc} (d)--(e), the vectors $dq_0$ and $dv_0$
are almost parallel, for large $M$, hence we can combine the first
two terms in the above bound:
$$
  \varepsilon_0\, \|dq_0\|
    + s\, \varepsilon_0\, \|dv_0\|
    \leq \varepsilon_0 (1+C/\sqrt{M})\, \|dq_{1}^-\|
$$
Here and below we denote by $C=C(\dcD,\br)>0$ various constants.
Next, it is a simple geometric fact that
$$
    \| s v - s' v' \|
    \leq \delta_0 + \delta_1
    + \| Q_0 - Q_0' \|
    + \| Q_1 - Q_1' \|
$$
and by the assumptions (\ref{assq})--(\ref{assv}) we get
$$
    |s-s'| \leq C ( \delta_0 + \delta_1 )
$$
Using Proposition~\ref{prabc} (g) gives
\beq
    \|dq_{1}^- - (dq_{1}')^-\|
    \leq \varepsilon_0 (1+C/\sqrt{M})\,\|dq_{1}^-\|
    + C(\delta_0 + \delta_1)\, \|dq_1^-\|
      \label{dq1-}
\eeq
Similarly,
\begin{align}
    \|dQ_{1}^- - (dQ_{1}')^-\|
    &\leq \tilde{\varepsilon}_0 (1+C/\sqrt{M})\,
    \|dq_{1}^-\| / M\nonumber\\
    &\quad + C(\delta_0 + \delta_1)\, \|dq_1^-\| / M
       \label{dQ1-}
\end{align}
where the estimation of the last term involves Proposition~\ref{prabc}
(c).

Now the postcollisional vectors $dq_{1}$ and $dQ_{1}$ depend on
the precollisional vectors $dq_{1}^-$ and $dQ_{1}^-$ through
certain reflection operators defined in terms of the normal vector
$n$, see (\ref{Rdq}), (\ref{RdQ}). Since $\dcD$ and $\dcP (Q)$ are
$C^3$ smooth, those reflection operators depend smoothly on
$x_{1}$, with uniformly bounded derivatives, hence
\begin{align*}
    \|dq_{1} - dq_{1}'\|
    & \leq  \|dq_1^- - (dq_1')^-\|
    + 2\|dQ_1^- - (dQ_1')^-\|\\
    &\quad +C\delta_1\, \|dq_1^-\| +
    C\delta_1\, \|dQ_1^-\|
\end{align*}
Applying (\ref{dq1-})--(\ref{dQ1-}) and Proposition~\ref{prabc} (b)
gives
\begin{align*}
    \|dq_{1} - dq_{1}'\|
    &\leq  (\varepsilon_0 + 2\tilde{\varepsilon}_0/M)
    (1+C/\sqrt{M})\,\|dq_{1}\| \\
    &\quad + C(\delta_0 + \delta_1)\, \|dq_1\|
\end{align*}
Similarly,
\begin{align*}
    \|dQ_{1} - dQ_{1}'\|
    & \leq  \|dQ_1^- - (dQ_1')^-\|
    + 2\|dq_1^- - (dq_1')^-\|/M\\
    &\quad +C\delta_1\, \|dQ_1^-\| +
    C\delta_1\, \|dq_1^-\|/M
\end{align*}
Applying (\ref{dq1-})--(\ref{dQ1-}) and Proposition~\ref{prabc} (b)
gives
\begin{align*}
    \|dQ_{1} - dQ_{1}'\|
    & \leq   (2\varepsilon_0 + \tilde{\varepsilon}_0)\,
    (1+C/\sqrt{M})\,\|dq_{1}\| / M \\
    &\quad + C(\delta_0 + \delta_1)\, \|dq_1\| / M
\end{align*}

We now consider the velocity components $dv$ and $dV$. They do not
change between collisions. At collisions, these vectors are formed
by certain reflection operators defined in terms of the normal $n$
and acquire an addition involving the operator $\bTheta^-$, see
Section~\ref{subsecSUV}. In those equations all the operators and
vectors smoothly change with the point $x_1$ with bounded
derivatives (see also (\ref{assv})), except for the unbounded
factor $\|w^+\|/ \la w^+, n\ra$, which we later denoted by $1/\cos
\varphi$. In what follows, we apply an elementary estimate for
$\varphi$, $\varphi'$ in the same homogeneous section:
\beq
    \left |\frac{1}{\cos\varphi_1}-\frac{1}{\cos\varphi_1'}\right |
    \leq \frac{|\varphi_1-\varphi_1'|}{\cos\varphi_1\cos\varphi_1'}
    \leq \frac{C\delta_1}{\cos^2\varphi_1}
       \label{phi13}
\eeq
Consider first the (simpler) case of a collision of the light particle with
$\dcD.$
It follows from (\ref{Theta-0}) that
$$
  \|\bTheta^-\| = \frac{2\cK\|v^+\|^2}{\la v^+,n\ra}
$$
where all the vectors are taken at the point $x_1$.
Thus, we obtain
\begin{align}
    \|dv_{1} - dv_{1}'\|
    & \leq \|dv_{0} - dv_0'\|
    + \|\bTheta^-\|\, \|dq_1^- - (dq_1')^-\| \nonumber\\
    &\quad + C\,\delta_1 \, \|dv_{1}\|
    + C \|\bTheta^-\|\, \|dq_1^- \|\,\delta_1/\cos\varphi_1 .
      \label{vv1}
\end{align}
By using (\ref{dq1-}), the sum of the first two terms on
the right hand side in the above inequality can be bounded as
follows:
\begin{align*}
  A\colon &= \|dv_{0} - dv_0'\| + \|\bTheta^-\|\, \|dq_1^- -
  (dq_1')^-\|\\
  &\leq  \varepsilon_0\, \|dv_{0}\|
   + \varepsilon_0(1+C/\sqrt{M})\, \|\bTheta^-\|\, \|dq_1^-\|
  +C(\delta_0 + \delta_1)\, \|\bTheta^-\|\, \|dq_1^-\|
\end{align*}
It is also
clear that the operator $\bTheta^-$ attains its norm on the
vectors perpendicular to $v^- = v_1^-$. By Proposition~\ref{prabc}
the vector $dq_1^-$ is almost perpendicular to $v_1^-$, thus
$\|\bTheta^-\|\, \|dq_1^-\| = (1+\varkappa) \| \bTheta^-(dq_1^-)
\|$ with some $\varkappa = {\cO} (1/\sqrt{M})$. Now we can combine
the first two terms on the right hand side of the previous
inequality as follows:
\begin{align*}
  A'\colon &= \varepsilon_0\, \|dv_{0}\|
   + \varepsilon_0(1+C/\sqrt{M})\, \|\bTheta^-\|\,\|dq_1^-\|\\
   &\leq \varepsilon_0\, \|\bR_n (dv_{0})\|
   + \varepsilon_0(1+C/\sqrt{M})\, \|\bTheta^+(dq_1^+)\|\\
   &\leq
   \varepsilon_0(1+C/\sqrt{M})\, \|dv_1\|
\end{align*}
Finally, combining all our estimates gives
\beq
    \|dv_{1} - dv_{1}'\|
     \leq  \varepsilon_0(1+C/\sqrt{M})\, \|dv_1\|
    + C(\delta_0 + \delta_1)\, \|dv_1\| / \cos\varphi_1
      \label{vv1last}
\eeq
In the (more difficult) case of an interparticle collision we have
a few extra terms in the main bound (\ref{vv1}):
\begin{align*}
    \|dv_{1} - dv_{1}'\| &\leq
    \cdots + 2\| dV_0 - dV_0'\| + C\delta_1
    \|dV_0\|\\ &\quad
    + C\delta_1\|\bTheta^-\|\,\|dQ_1^-\|
     + \|\bTheta^-\|\,\|dQ_1^- - (dQ_1')^-\|
\end{align*}
where $\cdots$ denote the terms already shown in (\ref{vv1}). Now,
the first new term above is bounded by
\beq
    2\| dV_0 - dV_0'\| \leq
    2\tilde{\varepsilon}_0\,\|dv_0\| / M
      \label{1new}
\eeq
The following two terms can be easily bounded and incorporated
into the previous estimate (\ref{vv1last}). The last term
$\|\bTheta^-\|\,\|dQ_1^- - (dQ_1')^-\|$ can be bounded, with the
help of (\ref{dQ1-}), by
$$
  \tilde{\varepsilon}_0 (1+C/\sqrt{M})\,
  \|\bTheta^-\|\, \|dq_{1}^-\| / M
    + C(\delta_0 + \delta_1)\, \|\bTheta^-\|\, \|dq_1^-\| / M
$$
The second term in this expression can be easily incorporated into
the previous estimate (\ref{vv1last}). To the first term we apply
the same analysis of the operator $\bTheta^-$ as was made in the
case of a collision of the light particle with $\dcD$, and then
combine it with (\ref{1new}) and obtain the bound
\begin{align*}
   2\tilde{\varepsilon}_0\,\|dv_0\| / M
   +\tilde{\varepsilon}_0 &(1+C/\sqrt{M})\,
  \|\bTheta^-(dq_{1}^-)\| / M
  \\ &\leq  2\tilde{\varepsilon}_0 (1+C/\sqrt{M})\,
  \| dv_1\| / M
\end{align*}
Combining all these bounds gives
$$
    \|dv_{1} - dv_{1}'\|
     \leq \left ( \varepsilon_0 +
     2\tilde{\varepsilon}_0 / M \right )\,
     (1+C/\sqrt{M})\, \|dv_1\|
    + C(\delta_0 + \delta_1)\, \|dv_1\| / \cos\varphi_1
$$

Lastly, we consider the vectors $dV$ and $dV'$. These do not change
between collisions or due to collisions of the light particle with
$\dcD$. At an interparticle collision, we have, in a way similar to the
previous estimates
\begin{align}
    \|dV_{1} - dV_{1}'\|
    & \leq \|dV_0 - dV_0'\|
    + 2\|dv_0 - dv_0'\|/M\nonumber\\
    &\quad +C\delta_1\, \|dv_0\|/M +
    C\delta_1\, \|dq_1^-\|/M\nonumber\\
    &\quad + \|\bTheta^-\|\,\|dq_1^- - (dq_1')^-\|/M\nonumber\\
    &\quad + \|\bTheta^-\|\,\|dQ_1^- - (dQ_1')^-\|/M\nonumber\\
    &\quad +C\delta_1\|\bTheta^-\|\,\|dq_1^-\|/
    (M\cos\varphi_1)
      \label{dVdV}
\end{align}
Applying (\ref{dq1-}) and the same analysis of the operator
$\bTheta^-$ as before gives
\begin{align*}
    \|\bTheta^-\|\,\|dq_1^- - (dq_1')^-\|/M
    & \leq  \varepsilon_0
    (1+C/\sqrt{M})\,\|\bTheta^-(dq_{1}^-)\| / M\\
    &\quad + C(\delta_0 + \delta_1)\,
    \|dq_1^-\|/(M\cos\varphi_1)
\end{align*}
The first term on the right hand side can be combined with the
term $2\|dv_0 - dv_0'\|/M$ in (\ref{dVdV}), and we get
\begin{align*}
   A''\colon & =
   2\|dv_0 - dv_0'\|/M +\varepsilon_0
    (1+C/\sqrt{M})\,\|\bTheta^-(dq_{1}^-)\| / M \\
   & \leq  2\varepsilon_0 \|dv_0\|/M + 2\varepsilon_0
    (1+C/\sqrt{M})\,\|\bTheta^-(dq_{1}^-)\| / M \\
   & \leq  2\varepsilon_0
    (1+C/\sqrt{M})\,\|dv_{1}\| / M \\
\end{align*}
We collect the above estimates and obtain
\begin{align*}
    \|dV_{1} - dV_{1}'\|
    & \leq    (2\varepsilon_0 + \tilde{\varepsilon}_0)\,
    (1+C/\sqrt{M})\,\|dv_{1}\| / M \\
    &\quad + C(\delta_0 + \delta_1)\, \|dv_1\|
    / (M \cos\varphi_1)
\end{align*}
Lemma~\ref{lm0} is proved. \qed \medskip

\noindent{\em Remark}. By fixing $Q$ and setting $M = \infty$ we
obtain a version of the above lemma for the billiard map $\cF_Q$.
It becomes much simpler, of course, since $dQ_i=dV_i=0$ and
$\tilde{\varepsilon}_i=0$, so (\ref{key0}) reduces to
$$
   \varepsilon_1 = \varepsilon_0 +
   C\,\frac{\delta_0+\delta_1}{\cos\varphi_1}
$$
and (\ref{key1}) becomes obsolete. It is important to note also that
the constant $C$ is uniform over all $\br>0$, as long as the points
$x_i$ and $x_i'$ belong to the same \index{u-curves (unstable
curves)} u-curve or s-curve (see Section~\ref{subsecSUC}). To verify
the uniformity of $C$, assume that $q_1\in\dcP(Q)$, then
$|r_1-r_1'|\leq c\br|\varphi_1 -\varphi_1'|$ for some $c>0$, hence
$|n(q_1)-n(q_2)|\leq \br^{-1}|r_1-r_1'|\leq c|\varphi_1
-\varphi_1'|<c\delta_1$, which allows us to suppress the large
factor $\br^{-1}$. Hence, the resulting distortion and curvature
\index{Distortion bounds} bounds will be uniform over all $\br>0$.

\begin{corollary}
Suppose that the point $x_1$ in Lemma~\ref{lm0} belongs to an
\index{H-curves} H-curve $W_1$. In that case the key estimates
{\rm (\ref{key0})--(\ref{key1})} of Lemma~\ref{lm0} can be modified
as follows:
\beq
   \varepsilon_1 = \left ( \varepsilon_0 +
   \frac{2\tilde{\varepsilon}_0}{M} \right )\,
   \left (1+\frac{C}{\sqrt{M}}\right ) +
   C\, \frac{\delta_0 + \delta_1}{|W_1|^{2/3}}
     \label{key2}
\eeq
\beq
   \tilde{\varepsilon}_1 = \Bigl( 2\varepsilon_0 +
   \tilde{\varepsilon}_0\Bigr)\,
   \left (1+\frac{C}{\sqrt{M}}\right ) +
   C\, \frac{\delta_0 + \delta_1}{|W_1|^{2/3}}
     \label{key3}
\eeq
with some constant $C>1$ {\textup (}possibly different from the
constant in Lemma~\ref{lm0}{\textup )}. \label{cr0}
\end{corollary}

\noindent{\em Proof}. Indeed, if the points $x_1$ and $x_1'$ lie in
a homogeneity \index{Homogeneity sections} section $\bbH_k$, then
$|W_1| \leq \Const\, (\cos\varphi_1)^{3/2}$, see (\ref{cos23}). This
proves Corollary~\ref{cr0}. \qed
\medskip

We now extend the estimates of Lemma~\ref{lm0} and
Corollary~\ref{cr0} to an arbitrary iteration of $\cF$. We will
show that the tangent vectors $dx_n$ and $dx_n'$ are
($\varepsilon_n, \tilde{\varepsilon}_n$)-close with some
$\varepsilon_n$ and $\tilde{\varepsilon}_n$ that we will estimate.
For brevity, put $\bvarepsilon_i = (\varepsilon_i,
\tilde{\varepsilon}_i)^T$ for $i \leq n$. The bounds in
Lemma~\ref{lm0} and Corollary~\ref{cr0} can be rewritten in a
matrix form
\beq
    \bvarepsilon_{i} = \bA \bvarepsilon_{i-1} + \bb_{i}
      \label{beA}
\eeq
where $\bA$ is a fixed matrix
$$
    \bA=(1+C/\sqrt{M})\,\bB,\qquad
    \bB =  \begin{pmatrix}
     1   &   2/M \\
    2 &  1 \end{pmatrix}
$$
and $\bb_i = (b_i,b_i)^T$, where $b_i = C\, (\delta_{i-1} +
\delta_i) /\cos \varphi_i$ or $b_i = C\, (\delta_{i-1} + \delta_i)
/|W_i|^{2/3}$ depending on whether we are applying
(\ref{key0})--(\ref{key1}) or (\ref{key2})--(\ref{key3}).

Iterating (\ref{beA}) gives
\beq
   \bvarepsilon_n = \bA^n\, \bvarepsilon_0
   + \sum_{i=1}^n \bA^{n-i}\,\bb_i
      \label{enAn}
\eeq
The matrix $\bB$ has eigenvalues $\lambda_1=1+2/\sqrt{M}$ and
$\lambda_2=1-2/\sqrt{M}$. By using its eigenvectors, we find
$$
   \bB^k =\tfrac 12\,
   \left ( \begin{array}{cc}\lambda_1^k+\lambda_2^k &
   (\lambda_1^k-\lambda_2^k)/\sqrt{M}\\ (\lambda_1^k-\lambda_2^k)\sqrt{M} &
   \lambda_1^k+\lambda_2^k \end{array}\right)
$$
It is easy to see that $\|B^k\| \leq \Const\, k\, (1 +
2/\sqrt{M})^k$, thus
$$
   \|A^k\| \leq C_1\, k\,
   \biggl(1+ \frac{C_2}{\sqrt{M}} \biggr)^k.
$$
For some constants $C_1,C_2>0$. This gives us

\begin{lemma}[$n$ step distortion control]
Assume that the standard unstable \index{Unstable vectors} vectors
$dx_0$ and $dx_0'$ are ($\varepsilon_0, \varepsilon_0$)-close for
some small $\varepsilon_0 >0$. Then their images $dx_n$ and $dx_n'$
are ($\varepsilon_n, \varepsilon_n$)-close with
$$
     \varepsilon_n = C_1n\,
     \biggl(1+ \frac{C_2}{\sqrt{M}} \biggr)^n
     \varepsilon_0
     +  C_1
     \sum_{i=1}^n (n-i+1)
     \biggl(1+ \frac{C_2}{\sqrt{M}} \biggr)^{n-i+1} b_i
$$
for all $n\leq 1$. \label{lmn}
\end{lemma}

We see that the sequence $\varepsilon_n$ effectively grows
linearly with $n$.

Now we are ready to prove Propositions~\ref{prdist} and \ref{prcurv}.
For brevity, we will say $\varepsilon$-close instead of
$(\varepsilon,\varepsilon)$-close.

\begin{lemma}
Under the assumptions of Proposition~\ref{prdist}, for any $c>0$
there is a $C>0$ such that whenever the tangent vectors $dx_0$ and
$dx_0'$ are $\varepsilon_0$-close with $\varepsilon_0 =
c|W_0(x_0,x_0')|/ |W_0|^{2/3}$, then the tangent vectors $dx_i$
and $dx_i'$ are $\varepsilon_i$-close with $\varepsilon_i =
C|W_i(x_i,x_i')|/ |W_i|^{2/3}$ for all $i=1,\dots,n$.
\label{lmdist}
\end{lemma}

\proof Since the points $x_i$ and $x_i'$ belong to one
\index{H-curves} H-curve, we can redefine $\delta_i$ to be
$|W_i(x_i,x_i')|$, and all our previous estimates will hold (with
maybe different values of the constants). Next, since $\delta_0
\simeq |W_0|$, the initial tangent vectors $dx_0$ and $dx_0'$ are
$(c\delta_0^{1/3})$-close. Similarly, we have $\delta_{i-1}
+\delta_i\leq\,\Const\,\delta_i$, hence
$b_i\leq\,\Const\,\delta_i^{1/3}$ in the notation of
Lemma~\ref{lmn}. Since \index{H-curves} H-curves grow by a factor
$\vartheta^{-1} > 1$, cf.\ (\ref{theta}), we have $\delta_i \leq
\vartheta^{n-i} \delta_n$ for all $i<n$. We now employ
Lemma~\ref{lmn} and easily obtain that the tangent vectors $dx_n$
and $dx_n'$ are $(C\delta_n^{1/3})$-close with some $C>0$. Therefore
\beq
    \left | \ln \frac{\cJ_{W_0}\cF^n(x_0)}
     {\cJ_{W_0}\cF^n(x_0')} \right |
     \leq C \delta_n^{1/3}
       \label{dist0}
\eeq
for some $C>0$. This estimate is weaker than the distortion bound
\index{Distortion bounds} claimed in Proposition~\ref{prdist}, but
it provides us, at least, with a uniform bound on distortions in the
sense of (\ref{distrough}) with some $\tbeta>0$.

The exponential growth of \index{H-curves} H-curves (\ref{theta})
and the uniform bound (\ref{distrough}) imply that
$$
   \frac{\delta_i}{|W_i|^{2/3}}
   \leq C \vartheta^{\frac{n-i}{3}}\,
   \frac{\delta_n}{|W_n|^{2/3}}
$$
for all $i < n$ and some constant $C>0$. Now we apply
(\ref{beA})--(\ref{enAn}) with $b_i =\,\Const\,\delta_i/
|W_i|^{2/3}$ and easily obtain that the tangent vectors $dx_n$ and
$dx_n'$ are $(C\delta_n/|W_n|^{2/3})$-close with some $C>0$.
Lemma~\ref{lmdist} is proved. \qed \medskip

Now Propositions~\ref{prdist} and \ref{prcurv} follow directly.
\qed \medskip

\medskip\noindent{\em Remark}. It is clear that for a sufficiently
smooth unstable curve one can always choose tangent vectors at any
two points that are $\varepsilon$-close for arbitrarily small
$\varepsilon>0$, thus they will satisfy the assumptions of
Lemma~\ref{lmdist}. \medskip

Lastly, we prove Lemma~\ref{LmJac}. Let $dy'$ and $dy''$ be tangent
vectors to the curve $\gamma$ at the points $y'$ and $y''$,
respectively. According to the definition of standard
\index{Standard pair} pairs, we can assume that they are
$\varepsilon$-close with $\varepsilon = C\, \dist (y', y'')
/|\gamma|^{2/3}$. Then by Proposition~\ref{prdist}, which we just
proved, the vectors $dx_-' = D\cF^i (dy')$ and $dx_-'' = D\cF^i
(dy'')$ are $\varepsilon_-$-close with
$$
   \varepsilon_- = C\, \dist (x_-', x_-'')
   / |W|^{2/3} \leq \Const\,\varepsilon_{\gamma}
   / |W|^{2/3}
$$
We now compare the tangent vectors $dx' = (D\cF_{Q} \circ D \pi_0)
(dx_-')$ and $dx'' = (D\pi_0 \circ D\cF) (dx_-'')$.

\medskip\noindent{\bf Claim}. $dx'$ and $dx''$ are
$\varepsilon_0$-close with
\beq
   \varepsilon_0 = \frac{C\varepsilon_\gamma}{|W|^{2/3}}
     +  \frac{C\varepsilon_\gamma}{|W_0'|^{2/3}}
        \label{claim}
\eeq
where $C>0$ is a large constant.\medskip

\proof Our argument follows the same lines as the proofs of
Lemma~\ref{lm0} and Corollary~\ref{cr0}, and we only focus on the
novelty of the present situation. First, since dist$(x_-', x_-'')
=\cO (\varepsilon_\gamma )$ and dist$(x', x'') =\cO
(\varepsilon_\gamma )$, then both $\delta_0$ and $\delta_1$ in
(\ref{key2})--(\ref {key3}) will be $\cO (\varepsilon_\gamma)$. In
addition, we apply $D\pi_0$ to both vectors. Recall that the
projection $\pi_0\colon \Omega \to \Omega_0$, fixes the position
of the heavy particle, sets its velocity to zero, and normalizes
the vector $w$ defined by (\ref{w}). Accordingly, $D \pi_0$ sets
the components $dQ$ and $dV$ of the tangent vector to zero and
rescales the component $dw$ by the same factor as it rescales $w$,
i.e.\ it divides $dw$ by $\|w\|$. In addition, we need to project
both components $dq$ and $dw$ onto the line perpendicular to $w$,
so that the basic equations (\ref{Vort})--(\ref{Qort}) would hold.
Therefore, the map $D \pi_Q\colon (dQ,dV, dq,dw) \mapsto
(dQ_1,dV_1, dq_1,dw_1)$ acts according to the following rules:
$dQ_1 = dV_1 =0$, and
$$
    dw_1 = dw/\|w\| - \la dw,w\ra/\|w\|^3,\ \ \ \ \
    dq_1 = dq - \la dq,w\ra/\|w\|^2,\ \ \ \ \
$$
As we noted in Section~\ref{subsecSUV}, the estimates (a)--(g) of
Proposition~\ref{prabc} apply to the vectors $w$ and $dw$, just as
well as to $v$ and $dv$. Hence
$$
   \| dq_1 - dq \| \leq C\|V\|\, \|dq\|
   \leq C\varepsilon_\gamma \|dq\|
$$
and
$$
   \Big\|\,\|w\|\, dw_1 - dw\,\Big \| \leq C\|V\|\, \|dw\|
   \leq C\varepsilon_\gamma \|dw\|
$$
Such a difference can be incorporated into the right hand side of
(\ref{claim}). The division of $dw$ by $\|w\|$ results in a change
of order one, in general, but this will be matched by the
corresponding division by $\|w\|$ when $D \pi_0$ is applied to the
other vector, as one can easily verify. This completes the proof
of the claim. \qed \medskip

We now finish the proof of Lemma~\ref{LmJac}. For each $r\geq 1$ we
need to compare the tangent vectors $dx'_r = D\cF_{Q}^r (dx')$ and
$dx''_r = D\cF_{Q}^r (dx'')$. The map $\cF_{Q}$ on $\Omega_{Q}$
corresponds to the motion of the light particle when the heavy one
is fixed at $Q$, which is the limit case of our two-particle
dynamics as $M\to \infty$. Thus, our analysis in Appendix~B, in
particular Lemma~\ref{lm0}, Corollary~\ref{cr0}, and Lemma~\ref{lmn}
apply to the map $\cF_{Q}$ as well. In order to use them, though, we
need to verify the conditions they are based on. First, since for
each $r\geq 1$ the points $\cF_Q^r (x')$ and $\cF_Q^r (x'')$ belong
to one homogeneous stable manifold, they lie in one homogeneity
\index{Homogeneity sections} section. Second,
(\ref{assq})--(\ref{assv}) hold due to (\ref{CCs}). Now
Corollary~\ref{cr0} and Lemma~\ref{lmn} can be used, indeed, and
they directly imply Lemma~\ref{LmJac}. \qed
\newpage

\backmatter

\printindex

\end{document}